\documentclass[10pt]{article}
\usepackage[letterpaper, left=1in, right=1in, top=1in, bottom=1in]{geometry}
\usepackage[dvipsnames]{xcolor}
\usepackage{microtype}
\usepackage{xspace}
\usepackage{amsthm}
\usepackage{mathtools}
\usepackage{amsmath}
\usepackage{bbm}
\usepackage{amsfonts}
\usepackage{booktabs}
\usepackage{amssymb}
\usepackage{algorithm}
\usepackage{algorithmic}
\usepackage{wrapfig}
\usepackage{verbatim}
\usepackage{xpatch}
\usepackage{enumitem}
\usepackage{multirow}
\usepackage{thm-restate}
\usepackage[sort]{natbib}
\bibpunct{(}{)}{;}{a}{,}{,}

\newtheorem{lemma}{Lemma}

\newtheorem{theorem}{Theorem}
\newtheorem{definition}{Definition}

\newenvironment{proofsketch}{\noindent\textbf{Proof Sketch}}{}

\usepackage{prettyref}
\newcommand{\pref}[1]{\prettyref{#1}}

\newcommand{\savehyperref}[2]{\texorpdfstring{\hyperref[#1]{#2}}{#2}}
\newrefformat{eq}{\savehyperref{#1}{\textup{(\ref*{#1})}}}
\newrefformat{eqn}{\savehyperref{#1}{Equation~\ref*{#1}}}
\newrefformat{lem}{\savehyperref{#1}{Lemma~\ref*{#1}}}
\newrefformat{def}{\savehyperref{#1}{Definition~\ref*{#1}}}
\newrefformat{thm}{\savehyperref{#1}{Theorem~\ref*{#1}}}
\newrefformat{cor}{\savehyperref{#1}{Corollary~\ref*{#1}}}
\newrefformat{sec}{\savehyperref{#1}{Section~\ref*{#1}}}
\newrefformat{subsec}{\savehyperref{#1}{Section~\ref*{#1}}}
\newrefformat{app}{\savehyperref{#1}{Appendix~\ref*{#1}}}
\newrefformat{fig}{\savehyperref{#1}{Figure~\ref*{#1}}}
\newrefformat{alg}{\savehyperref{#1}{Algorithm~\ref*{#1}}}
\newrefformat{tab}{\savehyperref{#1}{Table~\ref*{#1}}}

\usepackage[colorlinks=true, allcolors=black]{hyperref}

\newcommand{\R}{\mathbb{R}}
\newcommand{\E}{\mathbb{E}}

\DeclarePairedDelimiter{\abs}{\lvert}{\rvert} %
\DeclarePairedDelimiter{\brk}{[}{]}
\DeclarePairedDelimiter{\crl}{\{}{\}}
\DeclarePairedDelimiter{\prn}{(}{)}
\DeclarePairedDelimiter{\nrm}{\|}{\|}

\DeclarePairedDelimiter{\ceil}{\lceil}{\rceil}
\DeclarePairedDelimiter{\floor}{\lfloor}{\rfloor}

\let\P\undefined
\DeclareMathOperator{\P}{\mathbb{P}}
\DeclareMathOperator{\spn}{\mathrm{span}}
\DeclareMathOperator{\prox}{\mathrm{prox}}

\DeclareMathOperator*{\argmin}{arg\,min}
\DeclareMathOperator*{\argmax}{arg\,max}             

\newcommand{\mc}[1]{\mathcal{#1}}

\def\ddefloop#1{\ifx\ddefloop#1\else\ddef{#1}\expandafter\ddefloop\fi}
\def\ddef#1{\expandafter\def\csname 
bb#1\endcsname{\ensuremath{\mathbb{#1}}}}
\ddefloop ABCDEFGHIJKLMNOPQRSTUVWXYZ\ddefloop
\def\ddefloop#1{\ifx\ddefloop#1\else\ddef{#1}\expandafter\ddefloop\fi}
\def\ddef#1{\expandafter\def\csname 
b#1\endcsname{\ensuremath{\mathbf{#1}}}}
\ddefloop ABCDEFGHIJKLMNOPQRSTUVWXYZ\ddefloop
\def\ddef#1{\expandafter\def\csname 
c#1\endcsname{\ensuremath{\mathcal{#1}}}}
\ddefloop ABCDEFGHIJKLMNOPQRSTUVWXYZ\ddefloop
\def\ddef#1{\expandafter\def\csname 
h#1\endcsname{\ensuremath{\widehat{#1}}}}
\ddefloop ABCDEFGHIJKLMNOPQRSTUVWXYZ\ddefloop
\def\ddef#1{\expandafter\def\csname 
hc#1\endcsname{\ensuremath{\widehat{\mathcal{#1}}}}}
\ddefloop ABCDEFGHIJKLMNOPQRSTUVWXYZ\ddefloop
\def\ddef#1{\expandafter\def\csname 
t#1\endcsname{\ensuremath{\widetilde{#1}}}}
\ddefloop ABCDEFGHIJKLMNOPQRSTUVWXYZ\ddefloop
\def\ddef#1{\expandafter\def\csname 
tc#1\endcsname{\ensuremath{\widetilde{\mathcal{#1}}}}}
\ddefloop ABCDEFGHIJKLMNOPQRSTUVWXYZ\ddefloop

\newcommand{\indicator}[1]{\mathbbm{1}_{#1}}
\newcommand{\inner}[2]{\left\langle #1,\, #2 \right\rangle}

\newcommand{\remove}[1]{}

\newcommand{\tm}[0]{\mathsf{Time}}

\newcommand{\ind}[1]{^{(#1)}}
\newcommand{\mathth}[0]{^{\textrm{th}}}

\newcommand{\bx}[0]{\bar{x}}
\newcommand{\bxt}[0]{\bar{x}_t}
\newcommand{\bgt}[0]{\bar{g}_t}
\newcommand{\pp}[1]{\prn*{#1}_+}
\newcommand{\sdiff}[0]{\zeta_*}

\newcommand{\prog}[1]{\pi_{#1}}

\newcommand{\xag}[0]{x^{\textrm{ag}}}
\newcommand{\xmd}[0]{x^{\textrm{md}}}

\newcommand{\eladreduction}[0]{\textsf{SC}\to\mathsf{Cvx}}
\newcommand{\myreduction}[0]{\textsf{Cvx}\to\mathsf{SC}}

\begin{document}
\begin{titlepage}
\begin{center}
$\ $\\\vspace{3cm}
\textbf{\LARGE The Minimax Complexity of Distributed Optimization}

\vspace{1cm}
BY \\
\vspace{0.2cm}
BLAKE WOODWORTH \\

\vfill

A thesis submitted \\
\vspace{0.2cm}
in partial fulfillment of the requirements for \\
\vspace{0.2cm}
the degree of \\
\vspace{0.75cm}
Doctor of Philosophy in Computer Science \\
\vspace{0.75cm}
at the \\
\vspace{0.75cm}
TOYOTA TECHNOLOGICAL INSTITUTE AT CHICAGO \\
\vspace{0.2cm}
Chicago, Illinois \\
\vspace{0.75cm}
September, 2021 \\
\vspace{2cm}
    
Thesis Committee: \\
\vspace{0.2cm}
Nathan Srebro (Thesis Advisor) \\
\vspace{0.2cm}
Ohad Shamir \\
\vspace{0.2cm}
Stephen Wright \\
\vspace{0.2cm}
Madhur Tulsiani \\
\vspace{3cm}
    
\end{center}
\end{titlepage}

\pagenumbering{roman}\setcounter{page}{2}

\section*{Abstract}\label{sec:abstract}
\addcontentsline{toc}{section}{\nameref{sec:abstract}}

In this thesis, I study the minimax oracle complexity of distributed stochastic optimization. First, I present the ``graph oracle model'', an extension of the classic oracle complexity framework that can be applied to study distributed optimization algorithms. Next, I describe a general approach to proving optimization lower bounds for arbitrary randomized algorithms (as opposed to more restricted classes of algorithms, e.g., deterministic or ``zero-respecting'' algorithms), which is used extensively throughout the thesis. For the remainder of the thesis, I focus on the specific case of the ``intermittent communication setting'', where multiple computing devices work in parallel with limited communication amongst themselves. In this setting, I analyze the theoretical properties of the popular Local Stochastic Gradient Descent (SGD) algorithm in convex setting, both for homogeneous and heterogeneous objectives. I provide the first guarantees for Local SGD that improve over simple baseline methods, but show that Local SGD is not optimal in general. In pursuit of optimal methods in the intermittent communication setting, I then show matching upper and lower bounds for the intermittent communication setting with homogeneous convex, heterogeneous convex, and homogeneous non-convex objectives. These upper bounds are attained by simple variants of SGD which are therefore optimal. Finally, I discuss several additional assumptions about the objective or more powerful oracles that might be exploitable in order to develop better intermittent communication algorithms with better guarantees than our lower bounds allow.

\newpage

\section*{Acknowledgements}\label{sec:acks}
\addcontentsline{toc}{section}{\nameref{sec:acks}}

I am extremely grateful for the support of many people for many things over the past six years, and there is no way that I can adequately express my gratitude in a short acknowledgements section here, but I will try. 

First and foremost, I want to thank my thesis committee members---Nati, Ohad, Steve, and Madhur---for their mentorship, advice, and excellent ideas throughout my PhD and especially in preparing my thesis. 

I especially want to thank Nati---I could not have asked for a better PhD advisor. Adjusting to becoming a PhD student can be difficult and intimidating, but you took a gradual and relaxed approach that made it much easier. I remember emailing you during the summer before I started asking for a list of textbooks and papers that I should read before arriving, and you told me (nicely) to chill out. This really set the tone for a great grad school experience that I know should not be taken for granted. I am inspired by your approach to computer science where it's the answers to the questions that matter much more than the papers written about them, and by your truly incredible committment to rigor in all its forms. At first, I was annoyed by the hour-long tangents in the reading group where the rigor police\textsuperscript{\textregistered} rolled in and demanded to know every detail of what, \emph{exactly}, the authors were saying in Theorem 4b. Now, having been around for a while longer, I have come to appreciate the value in being very careful about and attentive to the minute details of every mathematical statement---it's important! Finally, I have always appreciated your patience and support over the years. For the guy who is always a few minutes late for the next calendar event, you always managed to find a time to squeeze in a short meeting or, at least, a huge email full of equations, even in the middle of the night before a paper deadline. Thank you very much, Nati.

I have been extremely lucky to have worked with a long list of incredible collaborators. Thank you to all of you: Nati, Suriya, Mesrob, Srinadh, Behnam, Vitaly, Saharon, Andy, Maya, Heinrich, Karthik, Serena, Seungil, Jialei, Adam, Brendan, Dylan, Ayush, Ohad, Yossi, Yair, John, Ryan, Aaron, Om, Mark, Elad, Max, Jason, Edward, Pedro, Itay, Daniel, Kshitij, Sebastian, Zhen, Brian, Shahar, Mor, and Amir. In addition to learning a million things from you all, I was so happy to have such smart, kind, and interesting people around me all the time.

I thank the TTIC students, faculty, and staff for making TTIC such a wonderful place to be. My experience having you as peers and professors in courses, sitting with you and hearing from you at talks, and speaking with you in the hallways and during meals has been nothing but positive. TTIC is undoubtedly the most comfortable academic setting I've been in, and I am very sad to be leaving. I also want to give a special shout out to the TTIC administrators who made everything run so smoothly every day. Thanks in particular to Chrissy, Erica, and Mary for your patience all those times that I messed up my forms and forgot to get my course list approved. You all make it very easy to be a student at TTIC and we all really appreciate it!

To the Salonica breakfast club: while the French toast and coffee at Salonica are not actually very good at all, your company and support a few times a week at strangly early hours of the morning were the greatest. Sometimes our conversations were educational, other times philosophical, we laughed, we cried, and we had a great time. I will miss you and I hope we can get together for breakfast again soon.

To my roommates Nick, Shane, Philip, and (for too short a time!) Davis: it was a true pleasure living and working with you over the years. When things were going well, I always had people to celebrate with and if things weren't going well, you guys were always available for commiseration. We had a lot of fun together, running, climbing, gaming, cooking, etc.~and I will always remember 5655 Harper Ave fondly.

To my family: thank you for giving me a great life where I had the chance to do a PhD. You guys got me interested in school, in learning, and in setting high goals, and I owe basically everything ever to all of your love and support. Whether or not it's true, you always made me feel smart and gave me a sense of unlimited possibilities, which I know I am extremely lucky to feel. Thank you so much, I love you.

Kasia: thank you for sharing my life for the last 9(!) years. Before I met you, PhDs were not something that people did in real life, and without you, I wouldn't have had any idea what I was doing for the last six years. Your bottomless love, advice, wisdom, and companionship have been the best thing in my life. You've proofread papers, listened to me ramble on about boring computer science stuff, taken care of me during paper deadlines, taken care of me during all the other times, and all of this while getting your own PhD! Although I won't miss flying back and forth between Chicago and Maryland, we have had a lot of fun and interesting experiences together in the past few years and I am really excited for whatever happens next.

Finally, I thank the NSF Graduate Research Fellowship and Google Research PhD Fellowship programs for providing financial support for my graduate work.

\newpage

\tableofcontents
\newpage

\listoftables
\newpage

\listoffigures
\newpage

\pagenumbering{arabic}

\section{Introduction}

Large-scale optimization plays an vital role in many modern computational applications, and particularly in machine learning. Because of the diversity of use cases, there is great value in developing broadly applicable, general-purpose algorithms that can be readily applied wherever they are needed, without relying on any unique structure. For example, the stochastic gradient descent algorithm (SGD) and its variants can be applied to a huge variety of optimization objectives, and almost all of the recent accomplishments of machine learning owe some of their success to SGD. 

Optimization problems are growing increasingly large and therefore it is often necessary to develop algorithms that leverage parallelism. In machine learning, for example, it has become common to use models that have millions or billions of trainable parameters, and the datasets used for learning often include millions or billions of examples. Training such models gives rise to enormous, very high-dimensional optimization objectives which cannot be tackled on a single machine. 

In this thesis, we are motivated by these massive optimization problems and the need for new and better general-purpose, distributed optimization algorithms to solve them. We will begin by formulating one notion of an optimal algorithm in distributed optimization settings. We will then consider a number of distributed optimization problems and ask for optimal algorithms for them. Identifying optimal algorithms has obvious benefits and the pursuit of new and better algorithms naturally leads to better performance in downstream applications. At the same time, \emph{failing} to identify optimal algorithms highlights settings in which there is room for improvement over the status quo and it motivates further research to find better methods. Finally, even in settings where we can identify an optimal algorithm, there is always an opportunity to be ``better than optimal'' by figuring out additional structure in the problem that can be exploited to yield better methods.

\subsection{Machine Learning as an Optimization Problem}\label{subsec:intro-ml-motivation}

One of the most important applications of optimization is training machine learning models, and because machine learning will serve as a running example throughout, we will take a moment to conceptualize it as fundamentally a stochastic optimization problem.

In supervised machine learning, the user specifies a model---a mapping from parameters to a prediction function---and a loss function---an evaluation metric measuring the accuracy of each prediction---and then ``trains the model'' meaning they find a setting of the parameters that ``fits'' the data in the sense of minimizing the loss function. Naturally, for any given model and loss function, one could come up with a bespoke training algorithm that finds a good setting of the parameters by cleverly exploiting some special structure. However, the typical approach is much more general: we set up training as a continuous optimization problem ``minimize the loss function over the parameters,'' and then we apply a general purpose optimization algorithm, often stochastic gradient descent (SGD), to solve that optimization problem. 

This generalized approach to training machine learning models by reducing it to a generic continuous optimization problem has several advantages. First, this method (in combination with tools like automatic differentiation) allows users to easily change their model or loss function without needing to design a whole new training algorithm. This flexibility has allowed the machine learning community to rapidly switch between different models and loss functions as we learn more about what works for which problems. You can imagine that if we had a killer training algorithm for 2-layer neural networks with the square loss, specifically, then we would likely have been much slower to try deep learning approaches which have led to many of the recent triumphs of machine learning. Second, separating machine learning into orthogonal modelling and optimization components allows all machine learners to benefit from advances in optimization algorithms. In recent years, numerous general-purpose optimization algorithms have been proposed \citep[e.g.][]{duchi2011adaptive,kingma2014adam}, and these have proven highly successful in a wide variety of applications. 

In order for this scheme to work, optimization algorithms should be general-purpose and applicable to as many different objectives as possible. However, there is no optimization algorithm that can be guaranteed to work on every function, and any algorithm needs to exploit \emph{something} about the objective in order to succeed. Therefore, an important aspect of optimization research is identifying a small set of properties that (1) can be exploited by an algorithm to efficiently optimize the objective and (2) can be expected to hold for objectives of interest. Even for broad classes of optimization objectives, there is a lot of plausibly exploitable structure, and it is important to distinguish the relevant from the irrelevant. Simultaneously, there is also a degree to which we \emph{control} the properties of the optimization objectives, for example, in machine learning, our choices of model and loss function give rise to the training objective. Therefore, if we learned that some property A allows for very efficient optimization, that would motivate designing models/loss functions which produce this property.

\subsection{Minimax Optimality in Optimization}

One of the primary goals of optimization research is to seek out better and more efficient optimization algorithms. Doing so of course requires developing and analyzing new and more clever methods, but it is also important to identify where there is room for improvement over the status quo, and what that improvement would look like. Finding such opportunities requires posing and answering questions of optimality---what is the best we can do in a given situation, and do our current methods perform that well? A significant amount of work is necessary to properly formulate a useful notion of optimality, which we discuss in \pref{sec:formulating-the-complexity}. At a high level, we do this by specifying a family of possible optimization objectives and of possible optimization algorithms and then we ask what the \emph{best} of these algorithms can guarantee for the \emph{hardest} objective in the family. This notion of ``minimax complexity'' indicates what is the best we can hope for when trying to optimize those sorts of objectives using that type of optimization algorithm. 

Analyzing the minimax complexity has two parts: ``upper bounds'' and ``lower bounds.'' Whenever we analyze an optimization algorithm and guarantee that it achieves a certain level of accuracy for any objective in the family, this puts an ``upper bound'' on the minimax error because, of course, the best algorithm's guarantee is no worse. On the other hand, a lower bound is a proof that \emph{no} optimization algorithm in the family of algorithms being considered can guarantee better than a certain level of accuracy for \emph{every} objective in the family. 

Matching upper bounds and lower bounds specify the minimax error, and whichever optimization algorithm achieved the upper bound is optimal. There are obvious benefits to identifying optimal methods, after all, everyone wants to use the best possible algorithm. On the other hand, there are frequently gaps between the best known upper bounds and lower bounds on the minimax error and, in a certain sense, these gaps are more exciting because they identify opportunities to design new methods that improve over the current state of the art.

One of the most famous examples of a gap between upper and lower bounds was for optimizing smooth, convex objectives using first-order algorithms. For a very long time, the best known upper bound corresponded to the guarantee of the Gradient Descent algorithm (which dates all the way back to Cauchy in the 1840's), which was known to guarantee error of at most $O(1/T)$ after $T$ iterations. On the other hand, the best known lower bound showed that no first-order method could guarantee error less than $\Omega(1/T^2)$ after $T$ iterations \citep{nemirovskyyudin1983}\footnote{This lower bound was actually originally proven in Russian in 1978---the 1983 citation is for the book's English translation.}. This gap---between $1/T$ and $1/T^2$---persisted for several years, and at the time it was quite unclear what the minimax error would be. Many efforts were made both to design better algorithms with guarantees better than $1/T$ and to prove better lower bounds that showed that it is impossible to do better than $1/T$. It was not until several years later that Nesterov's famous Accelerated Gradient Descent algorithm was proposed and shown to converge at the $1/T^2$ rate after all \citep{nesterov1983method}. In this example, the existence of \citeauthor{nemirovskyyudin1983}'s $1/T^2$ lower bound played an important role in driving optimization research forward, despite the fact that the lower bound did not match anything at the time. 

It is important to properly interpret the meaning of a lower bound. It says that no optimization algorithm \emph{in the considered class of algorithms} is able to provide a better guarantee \emph{for all objectives in the considered family of objectives}. This does \emph{not} mean that continued progress is futile, and that we should give up and settle for whatever ``optimal'' algorithm we have. Instead, it means that additional progress requires identifying additional, useful structures that algorithms can exploit and modifying the classes of algorithms and objectives that we consider accordingly. In this sense, studying minimax optimality and proving lower bounds can be thought of as a task of modelling---out of the many possible properties that an objective might have, which ones are useful and exploitable, and which ones are not? What additional properties would allow for better methods? Conveniently, lower bounds typically identify a particular optimization objective that is hard to optimize, and show us precisely \emph{why} it is hard. Once we know the pitfalls in a given setting, we can identify additional structure that could be used to avoid them.

\subsection{Distributed Optimization}

The field of distributed optimization is marked by a huge diversity of possible forms of parallelism. Parallel optimization algorithms can be implemented on multi-core processors within a single computing device. They can also arise in data center setting where many, very powerful devices are arrayed in the same location. The parallel computers could also be spread around the world, leading to high-latency communication between them. These are just a few examples of the nearly unlimited possible parallelism scenarios that one could face. 

Given this variety and our interest in general-purpose algorithms, we make efforts to study optimization methods in a way that is broadly applicable to many different distributed optimization settings. Accordingly, our framework for studying the complexity of distributed optimization (see \pref{subsec:algorithm-class-and-graph-framework}) is based around the \emph{structure} of the parallelism---e.g.~there are $M$ parallel workers, or the machines communicate with each other every $K$ iterations, or the parallel workers have access to distinct datasets---rather than details of the setting---e.g.~the machines have a low-latency connection with each other, or each worker computes at $X$ petaflops. This allows us to understand distributed optimization in a greater variety of settings, and these general principles can often also be applied to answer questions about specific settings. 

Throughout, we will generally focus on understanding distributed optimization in particular, fixed settings, for example, we might study algorithms that use $M$ parallel workers which each compute $T$ stochastic gradient estimates. In much of our analysis, we would treat the quantities $M$ and $T$ as set in stone for several reasons. First, if we have an algorithm that is optimal for any given $M$ and $T$, this naturally tells us the minimax error as a function of $M$ and $T$, and we can easily tell what would happen if they were changed. Second, this allows us to better capture the tradeoffs that are inherent in distributed optimization. In particular, the answer to ``would using more parallel workers improve my algorithm's performance?'' is almost always ``yes, obviously.'' Similarly, running for more iterations, using larger minibatches, and communicating more frequently will always improve performance. However, the question in distributed optimization is often how can we manage tradeoffs between competing considerations. If I double the number of parallel workers, can I halve the number of iterations---and therefore the total runtime---without hurting performance? If communication between machines takes $X$ms and computing one stochastic gradient on each machine takes $Y$ms, how large of a batchsize would get us to error $\epsilon$ in the shortest amount of time? These are often the most important questions in distributed optimization, and the answer generally depends on the particulars of the situation. Finally, some aspects of the parallel environment are outside of our control---for instance, my department only has so many GPUs available---and it would not be so helpful to know what the best number of machines is when that choice is unavaiable. Nevertheless, again, this is largely a philosophical question since our approach also allows for answering many of these types of questions.

\subsection{Overview of Results}

In this thesis, we build a theory of minimax optimality for distributed stochastic optimization and apply it to several parallel settings. 

In \pref{sec:formulating-the-complexity}, we begin by describing an extension of the classical oracle model \citep{nemirovskyyudin1983} to the distributed setting, which allows us to rigorously pose questions of optimality. The basic oracle model, which allows for proving tight and informative lower bounds in the sequential (i.e.~not distributed) setting, is based on the idea of restricting the means through with an algorithm interacts with the objective function, but \emph{not} what the algorithm is allowed to do with the information it learns about the objective. This allows for strong lower bounds that apply to broad classes of optimization algorithms and give deep insight into the complexity of sequential optimization. However, we describe that the classic oracle model is insufficient for distributed optimization, and we describe in \pref{subsec:algorithm-class-and-graph-framework} an extension of the model to the parallel setting, the ``graph oracle model.'' The idea is to capture the distributed structure of an optimization algorithm using a graph structure, where each vertex in the graph corresponds to a single oracle access, and the edges describe the dependencies between different queries. This approach is highly flexible and allows us to formulate a notion of minimax oracle complexity for many different distributed optimization settings using a single framework. 

In \pref{sec:lower-bound-tools}, we present several generic tools for analyzing the minimax oracle complexity in the graph oracle model which prove useful for our other results and are likely of interest more broadly. First, many existing optimization lower bounds, even in the sequential setting, apply only to fairly resrictive classes of algorithms---typically only deterministic, ``span-restricted,'' or ``zero-respecting'' algorithms. While these families contain many algorithms of interest, they do not answer the question of whether we might be able to do better using other methods. In \pref{subsec:high-level-lower-bound-approach}, we sketch a general approach to proving lower bounds that apply to much larger classes of optimization algorithms, up to and including the class of all randomized algorithms corresponding to a particular graph oracle setting. In \pref{subsec:generic-graph-oracle-lower-bound}, we proceed to use this method to prove a lower bound in the graph oracle model that applies to any randomized distributed first-order method corresponding to any graph. This lower bound only depends on two generic properties of the graph---the number of vertices and its depth---and we apply it extensively in our later results. Finally, in \pref{subsec:reduction-section-all}, we describe a generic reduction that connects the complexity of optimizing convex objectives with the complexity of optimizing strongly convex objectives. In particular, we show that algorithms for convex optimization, when applied to strongly convex objectives, can automatically attain much faster rates of convergence without exploiting the strong convexity in any explicit way.

In \pref{sec:local-sgd}, we study the theoretical properties of the popular Local SGD algorithm. We begin in \pref{subsec:local-sgd-and-baselines} by identifying three natural baseline algorithms, corresponding to other variants of SGD that correspond to the same graph oracle setting. The conventional wisdom says that Local SGD should dominate these baselines, but little of the existing work makes any direct comparison with these methods. 
In \pref{subsec:local-sgd-homogeneous}, we study Local SGD in the ``homogeneous'' setting, where each parallel worker has access to data from the same distribution. We begin by showing that existing analysis of Local SGD fails to show any improvement over the baseline algorithms, which raises serious questions about the idea that Local SGD is uniformly better. We proceed to show that in the special case of least squares problems, Local SGD does indeed dominate the baselines; we prove a new guarantee for Local SGD for general convex objectives that is \emph{sometimes} better the baselines but \emph{sometimes} is not; and we conclude by showing that this was no accident, and Local SGD really is worse than the baselines in some regimes. 
In \pref{subsec:local-sgd-heterogeneous}, we turn to the ``heterogeneous'' setting, where each parallel worker has access to data from a \emph{different} distribution, but where the goal is optimize the average of the local objectives. We show that, as in the homogeneous setting, the existing guarantees for Local SGD fail to improve over the baselines. We also show that under the standard assumptions, Local SGD \emph{might} be able to improve over the baselines in a narrow regime, but will generally perform much worse than a Minibatch SGD baseline. We conclude by introducing a new assumption about the objective which allows for Local SGD to improve over the baselines in certain regimes which we identify.

In \pref{sec:intermittent-communication-setting} we study, more broadly, the ``intermittent communication setting,'' a natural distributed optimization setting that commonly arises in practice. The intermittent communication setting corresponds to the case where $M$ parallel workers collaborate to optimize an objective over the course of $R$ rounds of communication, and in each round of communication, each machine is able to compute $K$ stochastic gradients sequentially. In \pref{subsec:homogeneous-intermittent-minimax}, we study the minimax oracle complexity of optimization in the homogeneous intermittent communication setting. For convex, strongly convex, and non-convex objectives, we tighly characterize the minimax error and we identify optimal algorithms that are given by the combination of two accelerated SGD variants, a ``minibatch'' variant and a ``single-machine'' variant. These results highlight an interesting dichotomy in the homogeneous intermittent communication setting between exploiting the local computation (captured by $K$) and exploiting the parallelism (captured by $M$). In \pref{subsec:heterogeneous-convex-intermittent-minimax}, we look to the heterogeneous intermittent communication setting. Here, we also identify the minimax error and optimal algorithms for convex and strongly convex settings. This time, the optimal algorithm is just the minibatch algorithm, which exploits the parallelism but not the local computation, in contrast to the homogeneous case.

Finally, in \pref{sec:breaking-the-lower-bounds}, we revisit the intermittent communication setting with the goal of ``breaking'' the lower bounds presented in \pref{sec:intermittent-communication-setting}. Specifically, we identify several additional properties of the objective or oracle that allow for better methods whose guarantees are better than the lower bounds would allow. In \pref{subsec:homogeneous-nearly-quadratic-intermittent-minimax}, we show that in the homogeneous intermittent communication setting, it is possible to attain better error when the objective is ``nearly-quadratic.'' In \pref{subsec:bounded-heterogeneous-convex-intermittent-minimax}, we show that when the objective is only boundedly heterogeneous, meaning the local objectives are not arbitrarily different, it is possible to exploit this structure to outperform the optimal algorithm from \pref{subsec:heterogeneous-convex-intermittent-minimax}. Finally, in \pref{subsec:non-convex-with-MSS}, we show that when the oracle satisfies a certain smoothness property, then it is possible to circumvent the lower bound for homogeneous non-convex optimization presented in \pref{subsec:homogeneous-intermittent-minimax}.

\section{Formulating Distributed Stochastic Optimization}\label{sec:formulating-the-complexity}

Throughout this thesis, we consider a stochastic optimization objective, where the goal is to optimize
\begin{equation}\label{eq:intro-stochastic-opt-problem}
\min_x \crl*{F(x) = \E_{z \sim\mc{D}}f(x;z)}
\end{equation}

\paragraph{Machine Learning as Stochastic Optimization}
The optimization problem \eqref{eq:intro-stochastic-opt-problem} naturally captures many machine learning problems. For example, supervised learning corresponds to taking $f(x;z)$ to be the loss of the predictor parametrized by $x$ on the sample $z \sim \mc{D}$, then $F(x)$ corresponds to the expected risk, and our goal is to find parameters that minimize this risk\footnote{Unfortunately, standard notation differs between the optimization and machine learning literature. As is typical for optimization, we use ``$x$'' to denote the optimization variable of interest, and in machine learning other notation---e.g.~$w$, $\theta$, or $h$---are more common, and ``$x$'' is typically used for a feature representation of the data.}. For example, least squares regression from samples $z = (z_{\textrm{features}},z_{\textrm{target}}) \in \R^d \times \R$ would correspond to
\begin{equation}
f(x;z) = \frac{1}{2}\prn*{\inner{x}{z_{\textrm{features}}} - z_{\textrm{target}}}^2
\end{equation}
In the context of machine learning, there are two ways of thinking about the problem \eqref{eq:intro-stochastic-opt-problem} and, in particular, the role of $\mc{D}$. The first is a ``sample average approximation'' (SAA) viewpoint \citep{rubinstein1990optimization,kleijnen1996optimization}, where we take $\mc{D}$ to be the empirical distribution over a training set of i.i.d.~samples from the distribution of interest, and solving \eqref{eq:intro-stochastic-opt-problem} amounts to empirical risk minimization, that is, finding parameters that minimize the training loss. The second is a ``stochastic approximation'' (SA) viewpoint \citep{robbins1951stochastic}, where $\mc{D}$ is the population distribution of interest, from which a collection of i.i.d.~samples are available. There are advantages and disadvantages to both perspectives. In the SAA view, the objective $F$ has a special finite-sum structure and the distribution $\mc{D}$ is ``known'', which opens up various algorithmic possibilities that can allow for substantially faster convergence to a minimizer. For example, variance reduction methods can very efficiently exploit finite-sum structure \citep[e.g.][]{johnson2013accelerating}. On the other hand, solving \eqref{eq:intro-stochastic-opt-problem} in the SAA sense says nothing about how well the model would perform on unseen data, and a separate argument is required to show generalization (e.g.~via uniform convergence or algorithmic stability). Conversely, in the SA setting, solving \eqref{eq:intro-stochastic-opt-problem} directly implies strong performance on unseen data, but SA algorithms typically require fresh samples for each update which may result in worse sample complexity. 

\paragraph{Optimization and Learning}
There are two, mostly orthogonal, sources of difficulty in solving the stochastic optimization problem \eqref{eq:intro-stochastic-opt-problem}. First, there is the challenge of optimizing the function $F$, irrespective of the stochastic nature of the problem. Indeed, even for algorithms that ``know'' the distribution $\mc{D}$, it is far from trivial to find a minimizer of $F$, and the complexity of optimization using exact information about the objective has been studied extensively. Simultaneously, there is an issue of stochasticity---optimization algorithms need to optimize $F$ based on noisy information, and there are fundamental statistical limits to what can be learned about $F$ in this way. As a result of these two challenges, the complexity of stochastic optimization typically involves two pieces: an ``optimization term'' and a ``statistical term,'' which we will highlight in our results.

Our goal is to understand the complexity of stochastic optimization for different classes of optimization problems. However, significant care must be taken to formalize this complexity in a useful way. In this thesis, we define the complexity using three pieces: (1) a class of objectives, (2) a class of oracles, and (3) a class of optimization algorithms. These components together define an ``optimization problem,'' for which we proceed to define and study the minimax complexity. We will now discuss each of these pieces before defining our notion of complexity. 

\subsection{The Function Class}

To define an optimization problem, we first restrict our attention to a set of objective functions satisfying certain properties. There are innumerable conditions that we might impose of the objective---convexity, smoothness, Lipschitzness, etc.---any combination of which may be reasonable. However, we must make \emph{some} assumptions in order to have any hope of optimizing the function because it is possible to cast any number of intractable or even uncomputable problems as (perhaps extremely pathological) instances of \eqref{eq:intro-stochastic-opt-problem}.

Generally, we will try to consider broad function classes that make minimal restrictions on the objective. It is a stronger statement when an algorithm can guarantee good performance on a broader class of functions; and algorithms that rely on less structure are more broadly applicable. As described in \pref{subsec:intro-ml-motivation}, the frequent changes to state-of-the-art machine learning models and loss functions (which, together with the data distribution, determine $F$) means there is great value in general-purpose algorithms that can be readily applied to new objectives. However, there is a balance to be struck since there can also be value in imposing stronger restrictions on the function class, which allows for specialized optimization algorithms that exploit specific properties of the objective to achieve stronger performance. 

We will consider numerous function classes, most of which will be defined as they become relevant. However, there are two function classes to which we will return frequently: the class of smooth and convex objectives and the class of smooth and strongly convex objectives. We recall that a function $F$ is convex when
\begin{equation}
F(\theta x + (1-\theta) y) \leq \theta F(x) + (1-\theta)F(y) \qquad\forall_{x,y}\forall_{\theta\in[0,1]}
\end{equation}
We say that $F$ is $\lambda$-strongly convex when $F(x) - \frac{\lambda}{2}\nrm{x}^2$ is convex. Finally, a function $F$ is $H$-smooth if it is differentiable and its gradient is $H$-Lipschitz with respect to the L2 norm. For convex functions, this is equivalent to the inequality
\begin{equation}
F(y) \leq F(x) + \inner{\nabla F(x)}{y-x} + \frac{H}{2}\nrm{x-y}^2\qquad\forall_{x,y}
\end{equation}
We will often consider the following classes of smooth objectives
\begin{align}
\mc{F}_0(H,B) &= \crl*{F:\R^d\to\R \,:\, d \in \mathbb{N},\, F \textrm{ is convex and } H\textrm{-smooth},\, \exists x^* \in \argmin_x F(x) \textrm{ s.t.~} \nrm{x^*}\leq B} \\
\mc{F}_\lambda(H,\Delta) &= \crl*{F:\R^d\to\R \,:\, d \in \mathbb{N},\, F \textrm{ is } \lambda\textrm{-strongly convex and } H\textrm{-smooth},\, F(0) - \min_x F(x) \leq \Delta}
\end{align}
We note that for both of these function classes, we impose restrictions of $F$, the population objective, only. To provide any meaningful guarantee, it is necessary to bound in some way ``how far away'' the minimizer might be. We follow the standard practice of measuring this via the norm of the solution in the convex case, and the value of $F(0) - F^*$ in the strongly convex case.


\paragraph{Dimension-Free Complexity}
We focus on a function classes where the dimension, $d$, is not explicitly bounded, and throughout this thesis $d$ should be thought of as being ``large.'' Our notion of complexity will therefore be dimension-free, capturing what it is possible to guarantee without relying on the dimension being small in any way. Our complexity lower bounds will hold only in sufficiently high dimensions (typically polynomially-large in the other problem parameters), and our upper bounds hold even in unbounded, even infinite, dimensions. Of course, it is also important and interesting to study the complexity of optimization in a dimension-dependent manner, which opens up the possiblity of algorithms that can take advantage of a bound on the dimension to ensure better performance. Nevertheless, our focus on the dimension-free complexity is motivated by machine learning applications, where the dimension, i.e.~the parameter count, can easily run into the millions or billions. In this context, dimension-dependent rates are often weaker, and algorithms that depend on the dimension would typically incur unreasonably high computational costs.

\subsection{The Oracle}\label{subsec:the-oracle}

We study the complexity of optimization in the context of an oracle model, which specifies through which means the algorithm interacts with the optimization objective, and the oracle essentially specifies what form the ``input'' to the optimization algorithm takes \citep{nemirovskyyudin1983}. As an example, we mostly focus on optimization using a stochastic first-order oracle, which an algorithm can query at a point $x$ to receive a noisy estimate of the gradient $\nabla F(x)$. The classic notion of oracle complexity essentially amounts to counting the number of times that an algorithm needs to interact with the oracle before reaching an approximate solution to \eqref{eq:intro-stochastic-opt-problem}. 

Such oracle models have a long history in the study of optimization, and they generally serve as a proxy for the \emph{computational} complexity of optimization. The number of oracle accesses generally serves as a good proxy for the computational cost because the portions of the algorithm corresponding to oracle accesses typically constitute the bulk of the total computational cost. For instance, each iteration of stochastic gradient descent involves computing one stochastic gradient, one scalar-vector product, and one vector-vector addition, of which computing the stochastic gradient will almost always be the most costly. It is possible, in principle, to study the computational complexity of optimization directly, but there are a number of challenges to doing so, and across computer science it is notoriously difficult to formulate and prove bounds on the computational complexity of almost anything, even for \emph{much} simpler problems than optimizing high-dimensional, real-valued functions.

We focus on optimization using a stochastic first-order oracle which, given a point $x$, simply returns an unbiased, and bounded-variance stochastic estimate of the gradient $\nabla F(x)$. There is, however, some subtlety in the source of the stochasticity. 

The first and most general type of stochastic first-order oracle, which we will refer to as an ``\textbf{independent-noise}'' (I-N) oracle, simply returns any random vector $\mc{O}_g^{\sigma}(x)$ such that (1) $\E \mc{O}_g^{\sigma}(x) = \nabla F(x)$, (2) $\E \nrm{\mc{O}_g^{\sigma}(x) - \nabla F(x)}^2 \leq \sigma^2$, and (3) $\mc{O}_g^{\sigma}(x) | x$ is conditionally independent of both the state of the algorithm and the previous interactions between the algorithm and the oracle. For an I-N oracle, the stochasticity is almost completely unconstrained: it can depend arbitrarily on the query $x$, and repeated queries at the same point $x$ can yield estimates of the gradient with arbitrarily different distributions. The I-N oracle imposes almost no structure on the stochastic gradients besides unbiasedness and bounded variance, but it turns out that $\mc{O}_g^\sigma$ is often sufficient for solving \eqref{eq:intro-stochastic-opt-problem}, and algorithms like stochastic gradient descent require nothing more. Many of the results in this thesis will be stated in terms of an independent-noise oracle.

We will refer to a second variant as a ``\textbf{statistical learning}'' first-order oracle, $\mc{O}_{\nabla f}^\sigma$. This oracle also returns an unbiased and bounded variance estimate of the gradient, but with additional structure. In particular, when queried at $x$, the oracle returns $\nabla f(x;z)$ for an i.i.d.~$z\sim\mc{D}$. Without imposing any assumptions on the function $f$, this structure is little different than the I-N oracle described above. Nevertheless, $\mc{O}_{\nabla f}^\sigma$ opens the door to making assumptions about the components $f$ and/or the distribution $\mc{D}$ which can be exploited by optimization algorithms. For example, in many cases it is reasonable to assume that $f(x;z)$ is convex and smooth in its first argument for each $z$, which introduces a potentially non-trivial constraint on the stochastic gradients, see \pref{subsec:breaking-assumptions-on-components} for further discussion.

Finally, we define an ``\textbf{active statistical learning}'' first-order oracle, $\mc{O}_{\nabla f(\cdot;z)}^\sigma$. This is the same as the statistical learning first-order oracle above, but where the algorithm may either receive $\nabla f(x;z)$ for an i.i.d.~$z\sim\mc{D}$ \emph{or} it may receive $\nabla f(x;z)$ for some previously seen $z$ of its choice. As an example, finite sum optimization corresponds to the case where $\mc{D}$ is the uniform distribution over $\crl{1,\dots,n}$, and an active statistical learning oracle allows an optimization algorithm to calculate the gradient of a chosen component. We discuss active oracles further in \pref{subsec:repeated-accesses}.

\subsection{The Algorithm Class and Oracle Graph Framework}\label{subsec:algorithm-class-and-graph-framework}

The final piece of an ``optimization problem'' is a class of algorithms under consideration. A major advantage of oracle models as a concept is that it allows us to analyze very broad families of algorithms. Restricting how the algorithm gains information about the objective, but not restricting what it can do with that information, allows for proving strong lower bounds that can even apply to the class of \emph{all} optimization algorithms that only interact with the function through the oracle. 

In the context of sequential optimization, it has been common historically to consider a restricted class of algorithms for which each oracle query must be in the linear span of previous oracle responses. This class of span-restricted algorithms is conducive to proving lower bounds, and numerous classic results on the complexity of convex optimization study this family of optimization algorithms \citep[e.g.][]{nemirovskyyudin1983,nesterov2004introductory}. Indeed, this is a natural family of algorithms which contains the vast majority of known optimization methods including gradient descent, accelerated variants of gradient descent, variance reduction methods, etc.~and other methods like coordinate descent belong to a recent generalization of the class of span-restricted algorithms, termed ``zero-respecting'' algorithms \citep{carmon2017lower1}. Nevertheless, such results are still limited and they do not preclude the possibility that algorithms might be able to perform better by exploring points outside the span of the previously seen oracle responses. 

A related simplification is to consider the family of \emph{deterministic} algorithms that is not necessarily span-restricted or zero-respecting. It turns out that this family of algorithms is essentially no more powerful than the class of span-restricted or zero-respecting ones because of their determinism. Specifically, it is possible to prove nearly identical lower bounds by constructing a ``resisting oracle'' which adversarially rotates the objective function so that any time the algorithm deviates from the span of previous oracle responses, those deviations happen only along invariant directions of the objective and therefore reveal no useful information. It is possible to construct resisting oracles for deterministic algorithms since the algorithm's every move can be anticipated from the outset. 

For these reasons, there have been recent efforts to extend results on the oracle complexity of optimization to broader families of algorithms, up to and including the class of all randomized optimization algorithms \citep[e.g.][]{woodworth16tight,carmon2017lower1}. Indeed, in this thesis, we will mostly focus on the complexity of optimization for classes of randomized algorithms that are not necessarily span-restricted or zero-respecting. We note that proving lower bounds for such broad families of algorithms often requires considerably more sophisticated proofs than are needed for deterministic span-restricted or zero-respecting classes, as we discuss in \pref{subsec:high-level-lower-bound-approach}.

Orthogonal to these issues of randomization versus determinism is the question of how to formalize different types of distributed optimization algorithms. There are myriad distributed optimization settings---from parallelization across distant devices, to synchronous single-instruction-multiple-data parallelism, to asynchronous parallel processing---and capturing the complexity of optimization in any given setting requires carefully delineating what exactly what the algorithm is allowed to do. 

A key challenge is capturing the difference between the following two scenarios: (1) two machines query a stochastic gradient oracle $T$ times each, and may communicate whatever and whenever they would like, and (2) two machines query a stochastic gradient oracle $T$ times each, but they may not communicate at all. While it is clear that the first class of algorithms is more powerful, we note that the number of oracle accesses does nothing to distinguish between these two scenarios, which each involve $2T$ stochastic gradient oracle queries. We observe that the relevant distinction is the dependence structure between the queries: in case (1), the second oracle query on the first machine might depend both on the first query on the first machine and the first query on the second machine, whereas in case (2) all of the queries on the first machine are completely independent of all the queries on the second machine.

We therefore introduce a ``graph oracle framework'' which captures the nature of a distributed algorithm using a directed, acyclic graph and we define families of distributed optimization algorithms in terms of the associated graph. At a high level, each vertex in the graph corresponds to a single oracle access, and the result of each oracle access is only available in \emph{descendents} of the corresponding vertex in the graph. 

Let $\mc{G}$ be a directed, acyclic graph with vertices $\mc{V} = \crl{v_1,\dots,v_N}$ and define
\begin{equation}
\textrm{Ancestors}(v) = \crl*{v'\in \mc{V} : \exists \textrm{ a directed path from } v' \textrm{ to } v}
\end{equation}
We associate a query rule $\mc{Q}_v : \prn*{\mathfrak{Q} \times \mathfrak{A}}^{\abs*{\textrm{Ancestors}(v)}} \to \mathfrak{Q}$ and an oracle $\mc{O}_v : \mathfrak{Q} \to \mathfrak{A}$ with each vertex $v$ in the graph. The query rule at vertex $v$ is a mapping from all of the available information about the function---i.e.~the queries and oracle responses in the ancestors of $v$, which are in the set of possible queries, $\mathfrak{Q}$, and set of possible answers, $\mathfrak{A}$---in order to choose a new query $q_v \in \mathfrak{Q}$ to submit to the oracle $\mc{O}_v$:
\begin{equation}
q_v = \mc{Q}_v\prn*{\prn*{q_{v'}, \mc{O}_{v'}(q_{v'}) : v' \in \textrm{Ancestors}(v)}, \xi}
\end{equation}
where $\xi \in \crl{0,1}^*$ is a string of independent, random bits available to the algorithm, which allows us to capture randomized optimization algorithms that nevertheless have deterministic query rules. Finally, the algorithm has an output rule $\hat{X}$ which takes all of the queries and oracle responses and chooses the algorithm's output:
\begin{equation}
\hat{x} = \hat{X}\prn*{\prn*{q_v, \mc{O}_v(q_v) : v \in \mc{V}}, \xi}
\end{equation}
In this way, for a given graph structure $\mc{G}$ and set of associated oracles $\crl*{\mc{O}_v:v\in\mc{V}}$, an optimization algorithm is specified by the query rules $\crl*{Q_v:v\in\mc{V}}$ and the output rule $\hat{X}$. We therefore define the family $\mc{A}(\mc{G},\crl*{\mc{O}_v:v\in\mc{V}})$ as the set of all algorithms that can be implemented in this way. We can, of course, consider subclasses of $\mc{A}(\mc{G},\crl*{\mc{O}_v:v\in\mc{V}})$ consisting of, for example, only deterministic algorithms, or only span-restricted algorithms. However, we will mostly avoid such restrictions, and many of our lower bounds will apply to arbitrary randomized algorithms in $\mc{A}(\mc{G},\crl*{\mc{O}_v:v\in\mc{V}})$.

It will be helpful to consider several examples:

\subsubsection{Example: The Sequential Graph}\label{subsec:sequential-graph}

\begin{figure}
\centering
\includegraphics[width=0.8\textwidth]{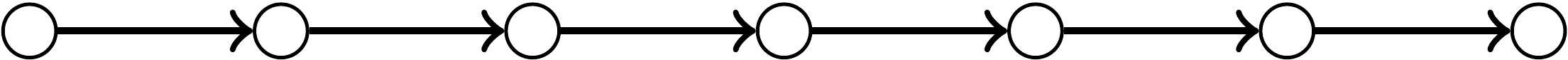}
\vspace{-3mm}\caption{The sequential graph.\label{fig:sequential-graph}}\vspace{4mm}
\includegraphics[width=0.8\textwidth]{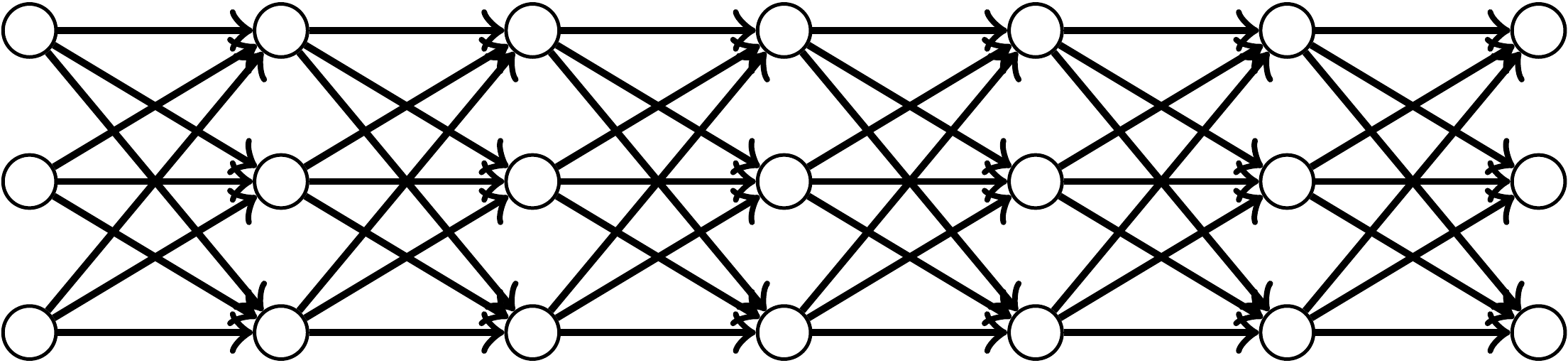}
\vspace{-3mm}\caption{The layer graph.\label{fig:layer-graph}}\vspace{4mm}
\includegraphics[width=0.8\textwidth]{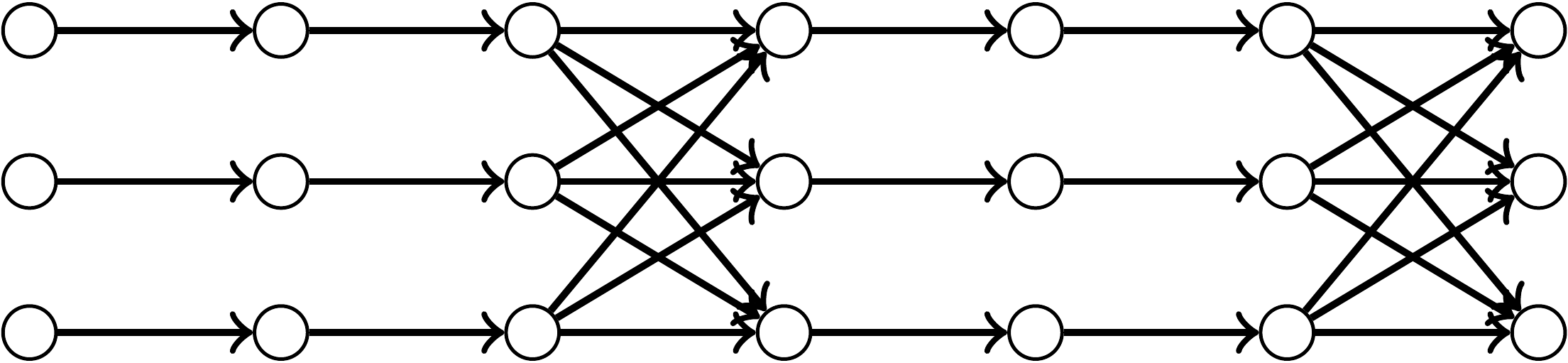}
\vspace{-3mm}\caption{The intermittent communication graph.\label{fig:intermittent-graph}}\vspace{4mm}
\end{figure}

The sequential graph $\mc{G}_{\textrm{seq}}$, depicted in \pref{fig:sequential-graph}, has $T$ vertices labelled $\crl{1,\dots,T}$ with an edge $t \to t+1$ for each $t$. This is the most basic non-trivial graph, and it corresponds to the standard serial optimization setting. In particular, for algorithms that correspond to the sequential graph, the $t\mathth$ oracle access is allowed to depend on all of the first $t-1$ oracle queries and oracle responses. 

This graph specifies the \emph{structure} of the oracle accesses allowed to the algorithm, but to pose useful questions about the complexity of optimization we also need to associate an oracle with each vertex in the graph. Some natural examples include:

\textbf{A Deterministic First-Order Oracle:}
When each vertex is associated with a single deterministic gradient oracle access, $\mc{O}_v = \mc{O} : x\mapsto \nabla F(x)$, the family $\mc{A}(\mc{G}_{\textrm{seq}},\mc{O})$ contains all deterministic first-order serial optimization algorithms that compute at most $T$ gradients. For instance, $T$ steps of Gradient Descent corresponds to query rules 
\begin{equation}
q_{t+1} = q_t - \eta_t \mc{O}(q_t)
\end{equation}
with output rule $\hat{X}((q_t,\mc{O}(q_t):t\in[T])) = q_T$. In a similar way, other query rules can be chosen that capture common algorithms like Accelerated Gradient Descent \citep{nesterov1983method}, Mirror Descent \citep{nemirovskyyudin1983}, and many more.

\textbf{A Stochastic First-Order Oracle:}
Each vertex could instead be associated with a stochastic gradient oracle access, $\mc{O}_v = \mc{O} : x\mapsto g_x$ such that $\E g_x = \nabla F(x)$. In this case, the family $\mc{A}(\mc{G}_{\textrm{seq}},\mc{O})$ contains all stochastic first-order serial optimization algorithms that compute at most $T$ stochastic gradients. With appropriately defined query rules, this allows us to capture a wide range of algorithms including stochastic gradient descent, stochastic mirror descent, accelerated variants of stochastic gradient descent, and more.

\textbf{Finite Sum Optimization:}
When the optimization objective has finite sum structure, i.e.~$F(x) = \frac{1}{n}\sum_{i=1}^n F_i(x)$, we can consider a component gradient oracle $\mc{O}_v = \mc{O} : (x,i) \mapsto \nabla F_i(x)$. With appropriately defined query rules this can capture many of the existing finite sum algorithms like SAG \citep{schmidt2017minimizing}, SAGA \citep{defazio2014saga}, SVRG \citep{johnson2013accelerating}, accelerated variants of these, and more. 

These are only a few examples and there are innumerable other oracles that could be paired with the sequential graph to specify families of serial optimization algorithms. Nevertheless, the graph oracle framework is not actually necessary to understand the complexity of serial optimization. Indeed, for all of the listed examples, the minimax oracle complexity was already well understood before the graph oracle framework was even proposed \citep{nemirovskyyudin1983,woodworth16tight}. For this reason, we will focus on the complexity of optimization in more interesting graphs.

\subsubsection{Example: The Layer Graph} 

The layer graph $\mc{G}_{\textrm{layer}}$, shown in \pref{fig:layer-graph}, has $MT$ vertices labelled $v^m_t$ for $m \in [M]$ and $t \in [T]$, with edges from $v^m_{t-1} \to v^{m'}_t$ for all $m,m'$. This corresponds to simple synchronous parallelism where algorithms can issue $M$ oracle queries in parallel. Such algorithms are natural when using multi-core processors or when multiple computing devices are available. As in the previous section, pairing the layer graph $\mc{G}_{\textrm{layer}}$ with a set of oracles allows us to capture natural families of distributed optimization algorithms, and to study the minimax complexity of optimization for these families of algorithms. As an example, we will discuss stochastic first-order algorithms with this graph:

\textbf{Stochastic First-Order Parallel Optimization}
The layer graph with stochastic gradient oracles $\mc{O}_v = \mc{O} : x \mapsto g_x$ with $\E g_x = \nabla F(x)$ specifies a family of stochastic first-order parallel algorithms $\mc{A}(\mc{G}_{\textrm{layer}},\mc{O})$. A natural algorithm in this family is minibatch stochastic gradient descent, which corresponds to query rules
\begin{equation}
q_{t+1}^m = q_{t+1} = q_t - \frac{\eta_t}{M}\sum_{m=1}^M \mc{O}_{v_t^m}(q_t)
\end{equation}
This family of algorithms also includes Accelerated Minibatch SGD \citep{cotter2011better,lan2012optimal}, Minibatch Stochastic Mirror Descent, and many others. As with Minibatch SGD, any of the stochastic first-order algorithms corresponding to the sequential graph can be naturally extended to the layer graph via minibatching, reducing the variance of the stochastic gradients and generally speeding convergence.

\subsubsection{Example: The Intermittent Communication Graph} \label{subsec:intermittent-graph}

The intermittent communication graph $\mc{G}_{\textrm{I.C.}}$, see \pref{fig:intermittent-graph}, has $MKR$ vertices labelled $v^m_{k,r}$ for $m \in [M]$, $k \in [K]$, and $r \in [R]$, with edges from $v^m_{k,r} \to v^m_{k+1,r}$ and from $v^m_{K,r} \to v^{m'}_{1,r+1}$ for each $m,m',k,r$. In this way, the intermittent communication graph most naturally corresponds to a setting in which $M$ devices work in parallel but where communication between the devices is limited. In contrast to the layer graph where oracle queries are issued in parallel but all of the responses from time $t$ are available for all of the queries at time $t+1$, in the intermittent communication graph the queries are broken up into $R$ ``rounds of communication,'' each of which corresponds to $K$ queries on each machine. So, $\textrm{Ancestors}(v^m_{k,r}) = \crl{v^m_{k',r} : k' < k} \cup \crl{v^{m'}_{k',r'} : r' < r}$, and only the oracle responses obtained on device $m$ and responses from previous rounds ($r' < r$) are available to choose the queries on device $m$.

This natural distributed optimization setting will be our main focus in \pref{sec:intermittent-communication-setting}. As with the previously discussed examples, we can pair the intermittent communication graph with many oracles in order to define families of optimization algorithms and for the most part, we will focus on stochastic first-order oracles. 

\textbf{The Homogeneous Setting:} 
In this setting, each vertex is associated with the same oracle, a stochastic gradient oracle $\mc{O}_v = \mc{O} : x\mapsto g_x$ such that $\E g_x = \nabla F(x)$. This family of algorithms $\mc{A}(\mc{G}_{\textrm{I.C.}},\mc{O})$ includes Minibatch SGD, Local SGD, and many more, which we will discuss in \pref{sec:local-sgd} and \pref{sec:intermittent-communication-setting}. In the context of supervised machine learning, this could correspond to a situation where each of the stochastic gradients is computed using an independent sample from the data distribution.

\textbf{The Heterogeneous Setting:}
However, unlike the previous examples, it is often interesting to associate different vertices in the graph with \emph{different} oracles. In the heterogeneous setting, we suppose that the objective has finite sum structure $F(x) = \frac{1}{M}\sum_{m=1}^M F_m(x)$, and that the oracle queries in vertices corresponding to the $m\mathth$ machine yield stochastic gradient estimates for $F_m$ specifically. That is, $\mc{O}_{v^m_{k,r}} = \mc{O}^m: x \mapsto g^m_x$ such that $\E g^m_x = \nabla F_m(x)$. This should be thought of as the $m\mathth$ machine having stochastic gradient access to $F_m$ and the goal is for the $M$ machines to achieve consensus by finding parameters $x$ that minimize the average of the local objectives. In a supervised learning setting, this corresponds to each machine computing stochastic gradients of the local loss based on a separate samples on each machine. Heterogeneity can arise, for example, when partitioning an i.i.d.~training dataset across the machines, which introduces some (probably ``small'') amount of heterogeneity to the stochastic gradients. Otherwise, when each machine uses data from genuinely different sources, for example from users on different continents, this can also introduce heterogeneity.

\textbf{The Federated Setting:}
We could also consider a stylized version of Federated Learning \citep{kairouz2019advances} that captures some, but not all, of the interesting features of the setting. This version of Federated Learning is similar to the heterogeneous setting, except that the components of the objective are not tied to any particular parallel worker. In particular, we suppose that the objective has the form $F(x) = \E_{i\sim\mc{D}} F_i(x)$ where $\mc{D}$ is an arbitrary distribution (whose support need not be finite or even countable). We then associate with each machine and each round of communication a stochastic gradient oracle for $F_i$ for a random $i \sim \mc{D}$, i.e.~$\mc{O}_{v^m_{k,r}} = \mc{O}^m_r: x \mapsto g^{i^m_r}_x$ such that $i^m_r \sim \mc{D}$ and $\E[g^{i^m_r}_x | i^m_r] = \nabla F_{i^m_r}(x)$. 

The Federated setting captures optimization in the intermittent communication setting when the stochastic gradients on each machine in each round are allowed to be correlated. This can arise, for example, when training a language model using data held on users' cell phones. In each round, $M$ of the available cell phones are randomly chosen and used to compute $K$ stochastic gradients, these gradients are then communicated back to a central coordinator and the process repeats. Since each user will have different language patterns, the stochastic gradients computed on each machine in each round will correspond to somewhat different objectives.

\subsection{The Minimax Complexity}

The combination of a function class, a graph and oracles that define the structure of an algorithm's interaction with the objective, and the associated class of optimization algorithms define ``an optimization problem.'' We proceed to define the minimax oracle complexity of an optimization problem, which asks what is the best guarantee that any algorithm can provide for every function in the class? For a given function class $\mc{F}$, oracle graph $\mc{G}$, assignment of oracles to vertices $\crl{\mc{O}_v}$, and family of algorithms $\mc{A}(\mc{G},\crl{\mc{O}_v})$, we define
\begin{equation}\label{eq:def-minimax-complexity}
\epsilon(\mc{F},\mc{A}(\mc{G},\crl{\mc{O}_v})) = \inf_{A \in \mc{A}(\mc{G},\crl{\mc{O}_v})} \sup_{F\in\mc{F}} \crl*{\E F(\hat{x}) - \min_x F(x)}
\end{equation}
Throughout this thesis, we will bound the quantity \eqref{eq:def-minimax-complexity} for various distributed optimization settings of interest. 

There are other similar but distinct ways that we could have defined the minimax complexity. One minor variation is to require a bound on the suboptimality $F(\hat{x}) - F^*$ with constant or high probability. We note that constant probability bounds are essentially equivalent to in-expectation bounds up to constant factors and, indeed, many of our lower bounds are shown to hold with constant probability. High probability bounds on the suboptimality are often impossible using merely bounded-variance stochastic oracles, and they typically require less standard assumptions like subgaussianity of the oracle. Obtaining high probability bounds is interesting and important, but here we focus on the more standard setting of in-expectation bounds.

It is also common to see the definition of the minimax complexity turned around---rather than asking what is the smallest achievable error with a certain number of oracle queries, asking instead how many oracle queries would be necessary to reach a given suboptimality $\epsilon$. In the context of sequential optimization (i.e.~the sequential graph), it is easy to see that these questions are two sides of the same coin: a bound on \eqref{eq:def-minimax-complexity} in terms of the number of queries, $T$, can be solved for $T$ to yield a bound on the number of queries needed to reach accuracy $\epsilon$ as a function of $\epsilon$. 

However, for more complex distributed optimization settings like the intermittent communication setting, there are multiple dimensions along which the graph could vary ($M$, the number of machines, $R$, the number of rounds of communication, and $K$, the number of queries per round), and it is therefore less obvious how to ``invert'' \eqref{eq:def-minimax-complexity} in a general-purpose way. For this reason, we prefer to think of the graph as fixed and to ask about the minimax complexity with respect to that graph specifically. Of course, by seeing how the minimax complexity depends on various properties of this graph, we can also answer questions about what sort of graph would allow us to reach a particular accuracy $\epsilon$.

\paragraph{Alternatives to the Graph Oracle Model}
Besides the graph oracle model, there are other possible formulations of minimax complexity for distributed optimization algorithms. One alternative is a communication complexity approach \citep{tsitsiklis1987communication,zhang2013information,garg2014communication,braverman2016communication}, where $M$ parallel workers have a local function---perhaps based on a locally held dataset---and the goal is to compute the minimizer of the average of the local functions. In the communication complexity formulation, each worker has unlimited computational power and can compute arbitrary information about its local objective (including, e.g.~its exact minimizer), but it is limited to transmit only a limited number of bits to the other machines. This approach is necessarily dimension-dependent because the number of bits needed just to represent the solution scales with the dimension, and beyond this issue, algorithms in this setting often explicitly rely on the dimension being bounded. Consequently, this approach is not as well suited to our settings.
Another alternative allows the machines to communicate 
real-valued vectors, but restricts the vectors that they are allowed to compute and transmit.
For instance, \citet{arjevani2015communication} presents communication complexity lower bounds for
algorithms that can only compute vectors that lie in a certain subspace, which includes
e.g.~linear combinations of gradients of their local function. \citet{lee2017distributed} 
impose a similar restriction, but allow the data defining the local functions
to be allocated to the different machines in a strategic manner.

Our framework applies to general stochastic optimization problems and does not impose any restrictions on what computation the
algorithm may perform or what it can communicate. Rather, we restrict the means by which the algorithm interacts with the objective and the structure of that interaction. In this way, our lower bounds can apply to very broad classes of algorithms, up to and including the family of all randomized algorithms that correspond to a given graph, whereas previous arguments are typically restricted to substantially smaller families of algorithms.

\section{Tools for Proving Lower Bounds}\label{sec:lower-bound-tools}

We will now introduce several tools that will be useful for analyzing the minimax complexity of optimization. 

First, we will introduce a conceptual approach to proving lower bounds for arbitrary randomized algorithms. As mentioned in \pref{subsec:algorithm-class-and-graph-framework}, optimization lower bounds can be quite simple for classes of zero-respecting algorithms and often require much more sophisticated constructions when dealing with broader families of randomized algorithms. Nevertheless, we will describe a minor modification to a lower bound construction which allows us to argue that \emph{any} randomized algorithm is nearly zero-respecting, which facilitates proving lower bounds. 

Second, we will prove a lower bound on the minimax complexity for classes of algorithms based on very simple and generic properties of the associated graph. These lower bounds are very general, and we argue that they are tight in a certain sense. However, for some graphs, including the intermittent communication graph, they are not tight and a more specialized analysis is required, which we will perform in later sections. Nevertheless, these basic lower bounds are broadly useful and we will refer to them frequently.

Finally, we will describe a method of corresponding algorithmic guarantees for convex objectives with better guarantees in the strongly convex setting. It is well-known that algorithms for strongly convex optimization can be applied to merely convex functions by adding a small regularization term to the objective. We show a reduction that goes in the opposite direction. Of course, since strongly convex objectives are also convex, a convex optimization algorithm will obviously succeed when applied to a strongly convex function. However, it is not at all obvious that the algorithm would obtain better guarantees with strong convexity; we show that it will indeed perform better and identify the better rate.

\subsection{A Technique for Proving Lower Bounds for Randomized Algorithms}\label{subsec:high-level-lower-bound-approach}

Before proceeding to the argument for randomized algorithms, it is worthwhile to describe the basic proof of lower bounds for zero-respecting algorithms. For simplicity, we focus on lower bounds for smooth, convex objectives in $\mc{F}_0(H,B)$, but a similar technique applies more broadly. The classic lower bound for functions in $\mc{F}_0(H,B)$ with a deterministic first-order oracle is based on the following hard instance due to \citet{nesterov2004introductory}:
\begin{equation}
F(x) = -x_1 + x_N^2 + \sum_{i=1}^{N-1} (x_{i+1} - x_i)^2
\end{equation}
The key property of this function $F$ is the ``chain-like'' nature of its gradient. Specifically, it is easy to see from the gradient,
\begin{equation}
\nabla F(x) = -e_1 + 2x_N e_N + 2\sum_{i=1}^{N-1} (x_{i+1} - x_i)(e_{i+1} - e_i)
\end{equation}
that if $x_i = x_{i+1} = \dots = x_N = 0$, then $[\nabla F(x)]_{i+1} = [\nabla F(x)]_{i+2} = \dots = [\nabla F(x)]_N = 0$ too. For this reason, a zero-respecting algorithm---whose queries have non-zero coordinates only where previously seen gradients had non-zero coordinates---can only increase the number of non-zero coordinates in its iterates by one for each gradient it computes. Therefore, any span-restricted or zero-respecting optimization algorithm that makes $T$ first-order oracle queries will have an output with $\hat{x}_{T+1} = \hat{x}_{T+2} = \dots = \hat{x}_{N} = 0$. From here, the rest of the lower bound proof is very simple; all that is necessary is to show that the suboptimality of such a point $\hat{x}$ is relatively large. For this particular function, it can be shown that the suboptimality scales with $T^{-2}$, which gives the classic and well-known lower bound for first-order optimization. This technique of arguing that the algorithm's output will lie in some restricted subspace, and then showing that any vector in that subspace will have high suboptimality is quite powerful, and will form the basis for many of our results.

However, the above argument relied \emph{crucially} on the algorithm being zero-respecting. In particular, there is a non-zero-respecting algorithm that immediately, exactly minimizes this objective without making even a single gradient oracle query, which is the algorithm that just returns $x^* \in \argmin_x F(x)$. Of course, this algorithm only works for this specific objective, but this just goes to show that proving lower bounds for non-zero-respecting algorithms is considerably more difficult. Indeed, such lower bounds cannot be proven using a single hard instance for this reason. Moreover, even if a non-zero-respecting algorithm isn't ``cheating'' by immediately returning the minimizer of $F$, we note that the chain-like property of the gradient is extremely delicate. If an algorithm simply queries the gradient oracle at any point plus almost any miniscule perturbation, then the gradient would be dense, and the algorithm could immediately ``find'' all of relevant coordinates. Fortunately, there is a relatively straightforward fix for this issue, which involves two pieces. 

The first piece is the observation that, in high dimensions, a vector's inner product with a random unit vector will be very small with high probability, specifically, on the order of $1/\sqrt{d}$. To capitalize on this, we introduce a random rotation $U \in \R^{d \times N}$ for large $d$ and take as our hard instance $F(U^\top x)$. Now, optimizing this rotated function requires obtaining a significant inner product with the columns of $U$, which is very unlikely to happen by merely ``guessing.''

Nevertheless, despite the fact that a vector's inner product with each columns of $U$ is likely to be very small, it won't be exactly zero, so a single gradient oracle query can reveal a lot of information about $U$, which can break the lower bound. To address this, the second idea is to ``flatten out'' the objective in such a way that the gradient maintains the important ``chain-like'' property even when the query has a slightly non-zero inner product with potentially all of the columns of $U$. To that end, we can modify the objective $F$ to be
\begin{equation}
\tilde{F}(x) = -x_1 + \psi(x_N) + \sum_{i=1}^{N-1} \psi(x_{i+1} - x_i)
\end{equation}
where
\begin{equation}
\psi(y) = \max\crl*{0, \abs{y} - \alpha}^2
\end{equation}
for some small parameter $\alpha \approx 1/\sqrt{d}$ (see \pref{fig:flattened-out-psi}).
\begin{wrapfigure}{R}{0.3\textwidth}
\centering
\includegraphics[width=0.3\textwidth]{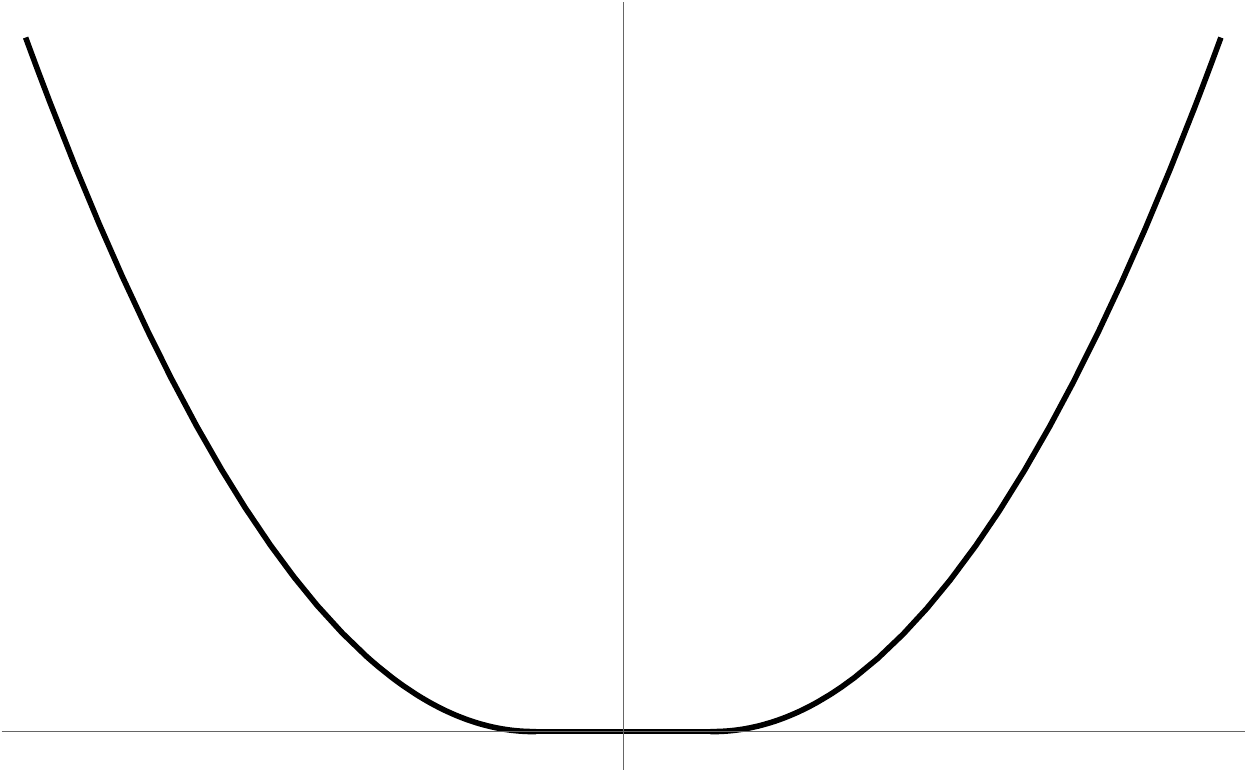}
\caption[The function $\psi(y)$.]{$\psi(y)$\label{fig:flattened-out-psi}}
\vspace{-5mm}
\end{wrapfigure}
Because $\psi'(y) = 0$ for $y \in [-\alpha,\alpha]$, the coordinates of $\nabla \tilde{F}(x)$ are zero until the corresponding coordinates of $x$ are \emph{substantially} non-zero, to an extent that would not happen by chance. Specifically, if $x$ is a vector such that $\abs{\inner{x}{U_j}} \leq \frac{\alpha}{2}$ for $j \geq i$, then the gradient
\begin{equation}
\nabla \tilde{F}(U^\top x) = -U_1 + \psi'(x_N)U_N + \sum_{i=1}^{N-1} \psi'(x_{i+1} - x_i)(U_{i+1} - U_i)
\end{equation}
will be a linear combination of $U_1,\dots,U_i$ only, and very little information about $U_{i+1},\dots,U_N$ is leaked beyond the fact that their inner products with $x$ are small. 

Using this approach, it is generally possible to extend the classic lower bound technique for span-restricted or zero-respecting algorithms to the class of all randomized algorithms. However, the first applications of the random rotation and flattening-out of the objective involved very long and delicate proofs \citep{woodworth16tight,woodworth2017lower,carmon2017lower1,woodworth2018graph,arjevani2019lower}, which hinged on carefully controlling the statistical dependencies between the yet ``unknown'' columns of $U$ and the previous oracle interactions. Since then, the argument has gradually been refined and simplified, culminating in the PhD thesis of \citet{yairthesis}, who shows in a very general manner that any algorithm is almost zero-respecting when optimizing functions like $\tilde{F}(U^\top x)$ above, and provides a concise and simple proof of this fact. 

This idea forms the basis of many of the lower bounds that we will prove. However, there is one additional technicality that requires some attention. In particular, the intuition that the inner product of a vector with a random unit vector is on the order $1/\sqrt{d}$ depends on that vector having bounded norm. Annoyingly, it is still the case that an algorithm can achieve a high inner product with a random vector---even in high dimensions---by simply querying the oracle with a vector that has a huge norm. Although it seems intutitive that, generally speaking, querying the oracle at a point with a super large norm---much larger than the norm of the function's minimizer---should not be an effective strategy, this must be addressed by our lower bound proofs.

The easiest way to deal with this is to further modify the objective such that its gradient at far away points is independent of $U$, and therefore querying there reveals no useable information. When proving lower bounds for non-convex optimization, this can be easily accomplished by introducing a soft projection to the objective so the algorthm must optimize $\tilde{F}(U^\top \rho(x))$ where $\rho(x) = (1 + \gamma^{-2}\nrm{x}^2)^{-1/2}x$ \citep{carmon2017lower1}. This way, the algorithm's queries are essentially bounded by $\gamma$, and the argument goes through. Unfortunately, when the objective is required to be convex, this approach does not work as readily. In the next section, we address the issue of bounding the queries by using a slightly different construction than we have so far described. Nevertheless, the proof follows the same idea: we introduce a random rotation, and we make the objective insensitive to small inner products with the columns of the rotation matrix.

\subsection{A Generic Graph Oracle Lower Bound}\label{subsec:generic-graph-oracle-lower-bound}

In this section, we prove a lower bound in the graph oracle model for any distributed, stochastic first-order optimization algorithm which depends on the associated graph.
Our generic lower bound for first-order algorithms in the graph oracle model is based on a hard instance with the following properties:
\begin{restatable}{lemma}{proxconstruction}\label{lem:prox-construction}
Let $U \in \R^{D\times d}$ for $D \geq d$ be orthogonal so that $U^\top U = I_{d\times d}$, and let $H,B>0$ and $d \in \mathbb{N}$ be given. Then there exists a function $F_U\in \mc{F}_0(H,B)$ such that for any $x$ with $\abs{\inner{U_d}{x}} \leq \frac{B}{12d^{3/2}}$
\[
F_U(x) - \min_x F_U(x) \geq \frac{HB^2}{16d^2}
\]
Furthermore, for each $i$, $\abs{\inner{\nabla F_U(x)}{U_i}} \leq \frac{HB}{8d^{3/2}} + H\abs{\inner{x}{U_i}}$; if $\nrm{x} \geq 5B$ then $\nabla F_U(x) = \frac{H}{2}x$ regardless of $U$; and if $\abs{\inner{U_i}{x}} \leq \frac{B}{12d^{3/2}}$ for all $i \geq j$, then $\nabla F_U(x) \in \textrm{Span}(x,U_1,\dots,U_j)$ and it does not depend on the columns $U_{j+1},\dots,U_d$.
\end{restatable}
The function $F_U$ is constructed as the Moreau envelope \citep{bauschke2011convex} of a function with the form
\begin{equation}
G_U(x) = \max_{1\leq i \leq d} \max\crl*{\frac{HB}{8d^{3/2}}\inner{U_i}{x} - c(i-1),\,-\frac{HB^2}{8d^{2}},\,\frac{H}{2}\prn*{\nrm{x}^2 - B^2}-\frac{HB^2}{8d^{2}}}
\end{equation}
for a small constant $c$. This resembles a classic construction for lower bounds for non-smooth objectives \citep{nemirovskyyudin1983}, and taking $F_U$ to be its Moreau envelope ``smoothes it out'' to also be $H$-smooth. Upon inspection, it is fairly clear that minimizing $G_U$ requires finding a point $x$ whose inner product with each column $U_i$ is substantially negative, on the order of $-B/\sqrt{d}$. The gradient of $F_U$ is related to the subgradients of $G_U$, and when $\nrm{x}^2$ is large, it is easy to see that $G_U(x) = \frac{H}{2}\prn*{\nrm{x}^2 - B^2}-\frac{HB^2}{8d^{2}}$ with $\nabla G_U(x) = Hx$ regardless of $U$. On the other hand, because of the terms $-c(i-1)$, if $\abs{\inner{U_i}{x}}$ is very small, then $i$ will not be in the $\argmax$ and therefore, $U_i$ will play no role in the subgradients of $G_U$. The proof is straightforward but technical, and we defer the details to \pref{app:generic-graph-oracle-lower-bound}. 

As discussed above, the idea of the lower bound is that any algorithm will be almost zero-respecting when optimizing $F_U$ for a uniformly random orthogonal matrix $U$ when the dimension $D$ is sufficiently large. Furthermore, the gradient of $F_U$ has the property that for approximately zero-respecting queries, each query only reveals a single new column of $U$. Consequently, the number of columns that can be learned by the algorithm is bounded by the \emph{depth} of the graph, i.e.~the length of the longest directed path in the graph. 
\begin{restatable}{theorem}{genericgraphlowerbound}\label{thm:generic-graph-lower-bound}
For any graph $\mc{G}$, let $\mc{O}_v$ be an exact gradient oracle for each $v$. For any $H$ and $B$ and any dimension $D \geq c\cdot\textrm{Depth}(\mc{G})^3\log(\abs{\mc{V}})$
\[
\epsilon(\mc{F}_0(H,B), \mc{A}(\mc{G},\crl{\mc{O}_v})) \geq \frac{HB^2}{128\textrm{Depth}(\mc{G})^2}
\]
\end{restatable}
\begin{proof}
Let $N = \textrm{Depth}(\mc{G})$ and for each $v \in \mc{V}$, define $\textrm{Depth}(v)$ to be the length of the longest directed path in $\mc{G}$ that ends at $v$. Let $U \in \mathbb{R}^{D\times d}$ be a uniformly random orthogonal matrix and let $F_U \in \mc{F}_0(H,B)$ be the objective described in \pref{lem:prox-construction}. Finally, consider an arbitrary algorithm in $\mc{A}(\mc{G},\nabla F_U)$. 

For $\alpha = \frac{B}{12d^{3/2}}$, we define the following ``good'' events
\begin{align}
G_v &= \crl*{\max\crl*{1\leq i\leq d:\abs{\inner{U_i}{q_v}} > \alpha} \leq \textrm{Depth}(v)} \\
\hat{G} &= \crl*{\max\crl*{1\leq i\leq d:\abs{\inner{U_i}{\hat{x}}} > \alpha} \leq N}
\end{align}
which indicates that the query $q_v$ does not have a large inner product with $U_i$ for $i$ greater than its depth, and similarly for the algorithm's output. We now proceed to lower bound $\P(\hat{G})$ conditioned on an abritrary realization of the algorithm's coins $\xi$:
\begin{align}
\P\prn*{\lnot \hat{G}\,\middle|\,\xi}
&\leq \P\prn*{\lnot \hat{G} \cup \bigcup_{v\in\mc{V}} \lnot G_v\,\middle|\,\xi} \\
&= \P\prn*{\lnot \hat{G}, \bigcap_{v\in\mc{V}} G_v\,\middle|\,\xi} + \sum_{v\in\mc{V}}\P\prn*{\lnot G_v, \bigcap_{v'\in\textrm{Ancestors}(v)} G_{v'}\,\middle|\,\xi}
\end{align}
Here, we rewrote the union as a disjoint union and applied the union bound, following the clever approach of \citet{diakonikolas2019lower}. Writing it this way is helpful for the following reason: by \pref{lem:prox-construction}, the event $G_{v'}$ implies that $\nabla F_U(q_{v'})$ is a measurable function of $q_{v'}$ and $U_1,\dots,U_{\textrm{Depth}(v')+1}$. Furthermore, for all $v'\in\textrm{Ancestors}(v)$, $\textrm{Depth}(v') + 1 \leq \textrm{Depth}(v)$. We also recall that in the graph oracle model, for each vertex $v$, the query, $q_v$, is generated according to a query rule $\mc{Q}_v$ as
\begin{equation}
q_v = \mc{Q}_v\prn*{\prn*{\prn*{q_{v'},\nabla F_U(q_{v'})}:v'\in\textrm{Ancestors}(v)},\xi}
\end{equation}
where $\xi$ are the random coins of the algorithm. Therefore, under the event $\bigcap_{v'\in\textrm{Ancestors}(v)} G_{v'}$, there there exists a function $\tilde{\mc{Q}}_v$ such that
\begin{equation}
q_v = \mc{Q}_v\prn*{\prn*{\prn*{q_{v'},\nabla F_U(q_{v'})}:v'\in\textrm{Ancestors}(v)},\xi} = \tilde{\mc{Q}}_v(U_1,\dots,U_{\textrm{Depth}(v)},\xi)
\end{equation}
Therefore, we just need to bound
\begin{align}
\P&\prn*{\lnot G_v, \bigcap_{v'\in\textrm{Ancestors}(v)} G_{v'}\,\middle|\,\xi}\nonumber\\
&= \P\prn*{\exists_{j > \textrm{Depth}(v)}\ \abs{\inner{U_j}{q_v}} > \alpha, \bigcap_{v'\in\textrm{Ancestors}(v)} G_{v'}\,\middle|\,\xi} \\
&= \P\prn*{\exists_{j > \textrm{Depth}(v)}\ \abs*{\inner{U_j}{\tilde{\mc{Q}}_v(U_1,\dots,U_{\textrm{Depth}(v)},\xi)}} > \alpha, \bigcap_{v'\in\textrm{Ancestors}(v)} G_{v'}\,\middle|\,\xi} \\
&\leq \sum_{j = \textrm{Depth}(v)+1}^N \P\prn*{\abs*{\inner{U_j}{\tilde{\mc{Q}}_v(U_1,\dots,U_{\textrm{Depth}(v)},\xi)}} > \alpha\,\middle|\,\xi}
\end{align}
Finally, we note that since $U$ is a uniformly random orthogonal matrix, $U_j$ is independent of $\xi$ and conditioned on $U_1,\dots,U_{\textrm{Depth}(v)}$, $U_j$ is uniformly distributed on the unit sphere in the $(D - \textrm{Depth}(v))$-dimensional subspace that is orthogonal to their span. Therefore, the above probability corresponds to the inner product between a fixed vector and a uniformly random unit vector being larger than $\alpha$. For now, assume that the queries $q_v$ have norm bounded by $\nrm{q_v} \leq 5B$. Then, standard results about concentration on the sphere \citep{ball1997elementary} proves that
\begin{align}
\P\prn*{\lnot G_v, \bigcap_{v'\in\textrm{Ancestors}(v)} G_{v'}\,\middle|\,\xi} 
&\leq \sum_{j=\textrm{Depth}(v)+1}^N \P\prn*{\abs*{\inner{U_j}{\tilde{\mc{Q}}_v(U_1,\dots,U_{\textrm{Depth}(v)},\xi)}} > \alpha\,\middle|\,\xi} \\
&\leq 2\abs{\mc{V}}\exp\prn*{-\frac{(D - \textrm{Depth}(v) + 1)\alpha^2}{50B^2}}
\end{align}
By the same argument,
\begin{equation}
\P\prn*{\lnot \hat{G}, \bigcap_{v\in\mc{V}} G_v\,\middle|\,\xi} \leq 2\abs{\mc{V}}\exp\prn*{-\frac{(D - N + 1)\alpha^2}{50B^2}}
\end{equation}
From this and the fact that $\xi$ was arbitrary, we conclude that if $D \geq N + \frac{50B^2}{\alpha^2}\log(8\abs{\mc{V}}^2)$ then
\begin{equation}
\P\prn*{\hat{G}}
\geq 1 - 4\abs{\mc{V}}^2\exp\prn*{-\frac{(D - N + 1)\alpha^2}{50B^2}}
\geq \frac{1}{2}
\end{equation}

Above, we assumed that the algorithm's queries have norm $\nrm{q_v} \leq 5B$. However, this assumption is without loss of generality because, by \pref{lem:prox-construction}, $F_U(q_v) = \frac{H}{2}q_v$ for any query with $\nrm{q_v} > 5B$. Therefore, for any algorithm that makes queries with norm greater than $5B$, there is another equally good algorithm that instead simply assumes that the gradient for such queries would be equal to $\frac{H}{2}q_v$. 

Finally, the event $\hat{G}$ implies that for all $j > N$, $\abs{\inner{U_j}{\hat{x}}} \leq \alpha$. We therefore take $d = N+1$ and conclude from \pref{lem:prox-construction} that under the event $\hat{G}$
\begin{equation}
F_U(\hat{x}) - \min_x F_U(x) \geq \frac{HB^2}{16(N+1)^2}
\end{equation}
This completes the proof.
\end{proof}
This lower bound is tight in the sense that it is the highest lower bound that applies for all graphs with a given depth. Momentarily, we will discuss several examples of graphs where we can identify algorithms whose guarantees match this lower bound, and we will also discuss one example where the lower bound does \emph{not} match known upper bounds, which therefore motivates further study.

It is also worth emphasizing that the lower bound \pref{thm:generic-graph-lower-bound} applies even for an exact gradient oracle, and does not rely at all upon making the problem difficult through stochasticity. Before we proceed, we also provide a simple lower bound on the ``statistical term,'' which corresponds to the information-theoretic difficulty of optimizing on the basis of only $\abs{\mc{V}}$ samples. Lower bounds very similar to this are well-known \citep[see e.g.][]{nemirovskyyudin1983}, but we include it to be self-contained:
\begin{lemma}\label{lem:statistical-term-lower-bound}
Let $H,B,\lambda,\Delta,\sigma^2 > 0$ and an arbitrary graph $\mc{G}$ be given, and let $\mc{O}_v$ be a stochastic gradient oracle with variance bounded by $\sigma^2$. Then in any dimension
\[
\epsilon(\mc{F}_0(H,B),\mc{A}(\mc{G},\crl{\mc{O}_v})) \geq c\cdot\min\crl*{\frac{\sigma B}{\sqrt{\abs{\mc{V}}}}, HB^2}
\]
and
\[
\epsilon(\mc{F}_\lambda(H,B),\mc{A}(\mc{G},\crl{\mc{O}_v})) \geq c\cdot\min\crl*{\frac{\sigma^2}{\lambda \abs{\mc{V}}}, \Delta}
\]
\end{lemma}
\begin{proof}
Consider the following pair of objectives:
\begin{equation}
\begin{aligned}
F_+(x) &= \frac{a}{2}x^2 + bx \\
F_-(x) &= \frac{a}{2}x^2 - bx
\end{aligned}
\end{equation}
with a stochatic gradient oracles
\begin{equation}
\begin{aligned}
z &\sim \mc{N}(0,\sigma^2) \\
g_+(x;z) &= \nabla F_+(x) + z \\
g_-(x;z) &= \nabla F_-(x) + z
\end{aligned}
\end{equation}
First, we note that for any $x$, 
\begin{equation}
F_+(x) - \min_x F_+(x) \leq \frac{b^2}{2a} \implies x \leq 0 \implies F_-(x) - \min_x F_-(x) \geq \frac{b^2}{2a}
\end{equation}
and vice versa. Therefore, any algorithm that succeeds in optimizing both $F_+$ and $F_-$ to accuracy better than $\frac{b^2}{2a}$ with probability at least $\frac{3}{4}$ needs to determine which of the two functions it is optimizing with probability at least $\frac{3}{4}$. However, by the Pinsker inequality, the total variation distance between $\abs{\mc{V}}$ queries to $g_+$ and $g_-$ is at most
\begin{align}
\nrm{\P_+ - \P_-}_{\textrm{TV}} 
&\leq \sqrt{\frac{1}{2}\textrm{D}_{\textrm{KL}}\prn*{\P_+\middle\|\P_-}} \\
&\leq \sqrt{\frac{\abs{\mc{V}}}{2}\textrm{D}_{\textrm{KL}}\prn*{\mc{N}(2b,\sigma^2)\middle\|\mc{N}(0,\sigma^2)}} \\
&= \frac{b}{\sigma}\sqrt{\abs{\mc{V}}}
\end{align}
Therefore, if $b \leq \frac{3\sigma}{4\sqrt{\abs{\mc{V}}}}$, no algorithm can optimize to accuracy better $\frac{b^2}{2a}$ with probability greater than $\frac{3}{4}$. 

Finally, we note that $F_+$ and $F_-$ are $a$-smooth, and have minimizers $\mp \frac{b}{a}$. Therefore, in the smooth convex case we take $b = \min\crl*{aB, \frac{3\sigma}{4\sqrt{\abs{\mc{V}}}}}$ and $a = \min\crl*{H,\ \frac{3\sigma}{4B\sqrt{\abs{\mc{V}}}}}$ so that the objectives are $H$-smooth and have solutions of norm $B$ and with probability at least $\frac{1}{4}$
\begin{equation}
\max_{* \in \crl{+,-}} F_*(\hat{x}) - \min_x F_*(x) 
\geq \min\crl*{\frac{aB^2}{2}, \frac{9\sigma^2}{32aMKR}} 
\geq \min\crl*{\frac{HB^2}{2}, \frac{3\sigma B}{8\sqrt{MKR}}}
\end{equation}

Similarly, $F_+$ and $F_-$ are $a$-smooth, $a$-strongly convex, and $F_*(0) - \min_x F_*(x) = \frac{b^2}{2a}$ for $*\in\crl{+,-}$. Therefore, we take $b = \min\crl*{\sqrt{2\lambda\Delta},\ \frac{3\sigma}{4\sqrt{MKR}}}$ and $a = \lambda$ so that the objectives are $\lambda$-smooth and $\lambda$-strongly convex and so that $F_*(0) - \min_x F_*(x) \leq \Delta$ for $* \in \crl{+,-}$, and so that with probability at least $\frac{1}{4}$
\begin{equation}
\max_{* \in \crl{+,-}} F_*(\hat{x}) - \min_x F_*(x) 
\geq \min\crl*{\Delta, \frac{9\sigma^2}{32\lambda MKR}} 
\end{equation}
This completes the proof.
\end{proof}

\subsubsection{Example: The Sequential Graph}\label{subsec:sequential-graph-minimax-complexity}\label{subsec:initially-introduce-ac-sa-convex-rate}

In the sequential graph, corresponding to $T$ sequential queries to the oracle, the minimax complexity of stochastic first-order optimization is already very well known and our lower bounds are redundant \citep{nemirovskyyudin1983,lan2012optimal}. Neverthless, it is the case that the pair of lower bounds \pref{thm:generic-graph-lower-bound} and \pref{lem:statistical-term-lower-bound} match the known minimax error, i.e.
\begin{equation}
\epsilon(\mc{F}_0(H,B),\mc{A}(\mc{G}_{seq},\mc{O}_g^\sigma)) = c\cdot\prn*{\frac{HB^2}{T^2} + \min\crl*{\frac{\sigma B}{\sqrt{T}},\, HB^2}}
\end{equation}
Our lower bounds are matched by the guarantee of an accelerated variant of SGD, AC-SA \citep{lan2012optimal}, which establishes this as the minimax optimal rate. Since we will be referring to AC-SA frequently, we will also take this opportunity to describe the algorithm in more detail. 
\begin{algorithm}[tb]
\caption{AC-SA}\label{alg:ac-sa}
\begin{algorithmic}
\STATE Initialize: $\xag_0 = x_0$
\FOR{$t = 0,1,\dots,T-1$}
\STATE $\xmd_{t+1} = \beta_t^{-1}x_t + (1-\beta_t^{-1})\xag_t$
\STATE $x_{t+1} = x_t - \gamma_t g(\xmd_t;z_t)$
\STATE $\xag_{t+1} = \beta_t^{-1}x_{t+1} + (1-\beta_t^{-1})\xag_t$
\ENDFOR
\STATE Return: $\xag_T$
\end{algorithmic}
\end{algorithm}

The algorithm is inspired by Nesterov's Accelerated Gradient Descent algorithm, which converges at the optimal $HB^2/T^2$ rate in the case of an exact gradient oracle \citep{nesterov1983method}. In the same way, AC-SA maintains multiple iterates, which are updated using a carefully tuned sequence of momentum parameters. When the momentum and stepsize parameters are optimally tuned, \citet{lan2012optimal} shows that for any $F_0 \in \mc{F}_0(H,B)$,
\begin{equation}
\E F_0 (\xag_T) - F_0^* \leq c\cdot\prn*{\frac{HB^2}{T^2} + \frac{\sigma B}{\sqrt{T}}}
\end{equation}
i.e.~the optimal error. We note that by the $H$-smoothness of $F$, it is always the case that $F(0) - F^* \leq \frac{HB^2}{2}$, so the final term in the lower bound can be attained for free.


\subsubsection{Example: The Layer Graph}\label{subsec:layer-graph-minimax-complexity}
In the layer graph, $\mc{G}_{\textrm{layer}}$, with a stochastic gradient oracle, corresponding to $T$ sequential batches of $M$ parallel queries to the oracle, the depth is $T$ and $\abs{\mc{V}} = TM$. Therefore, by combining \pref{thm:generic-graph-lower-bound} and \pref{lem:statistical-term-lower-bound}, we have the lower bound
\begin{equation}
\epsilon(\mc{F}_0(H,B), \mc{A}(\mc{G}_{\textrm{layer}}, \mc{O})) \geq c\cdot \prn*{\frac{HB^2}{T^2} + \min\crl*{\frac{\sigma B}{\sqrt{MT}},\,HB^2}}
\end{equation}
This lower bound is matched by the minibatch AC-SA algorithm---that is $T$ iterations of AC-SA using minibatches of size $M$, which reduces the stochastic gradient variance by a factor of $M$---so we conclude that the lower bounds are also tight in this setting.

\subsubsection{Example: The Delay Graph}\label{subsec:delay-graph-minimax-complexity}
\begin{figure}
\centering
\includegraphics[width=0.8\textwidth]{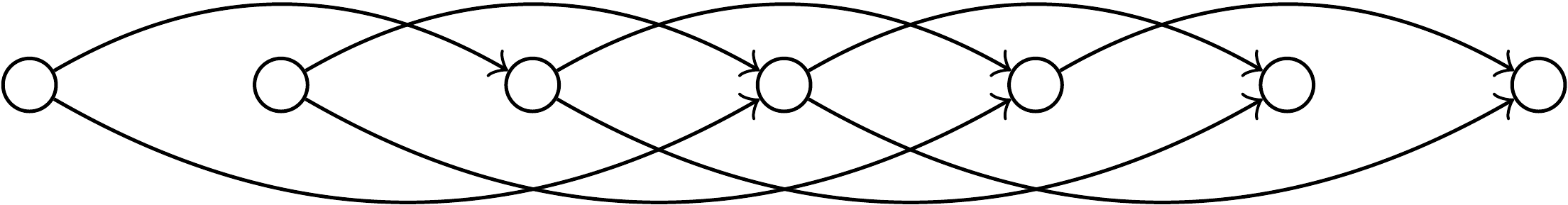}
\caption[The delay graph.]{The delay graph with $T=7$ and $\tau = 2$ \label{fig:delay-graph}}
\end{figure}
The delay graph $\mc{G}_{\textrm{delay}}$, shown in \pref{fig:delay-graph}, has two parameters: $T$, the number of vertices, and a delay $\tau$. The vertices are labelled $1,\dots,T$ and $\textrm{Ancestors}(t) = \crl{1,\dots,t-\tau}$. Therefore, the delay graph might correspond to an asynchronous setting in which the algorithm issues queries to an oracle, but does not receive a response for $\tau$ time steps. The depth of the delay graph is $T/\tau$, so the lower bound from \pref{thm:generic-graph-lower-bound} and \pref{lem:statistical-term-lower-bound} is
\begin{equation}\label{eq:delay-graph-lower-bound}
\epsilon(\mc{F}_0(H,B),\mc{A}(\mc{G}_{\textrm{delay}}, \mc{O})) \geq c \cdot\prn*{\frac{HB^2}{(T/\tau)^2} + \min\crl*{\frac{\sigma B}{\sqrt{T}},\, HB^2}}
\end{equation}
A natural algorithm for this setting is delayed-update SGD, which uses updates $x_{t+1} = x_t - \eta_{t-\tau}\mc{O}_{t-\tau}(x_{t-\tau})$. Early analysis of this algorithm only proved an optimization term scaling with the poor scaling $(T/\tau^2)^{-1}$ \citep{feyzmahdavian2016asynchronous}. In later work, \citet{arjevani2020tight} improved the optimization term to $(T/\tau)^{-1}$ in the special case of quadratic objectives, and most recently, \citet{stich2019error} showed that delayed-update SGD guarantees
\begin{equation}
\E F(\hat{x}) - \min_x F(x) \leq c\cdot\prn*{\frac{HB^2}{T/\tau} + \frac{\sigma B}{\sqrt{T}}}
\end{equation}
for any $F \in \mc{F}_0(H,B)$. However, to have any hope of matching the lower bound would require accelerating the algorithm, and we are not aware of any successful attempts to do this. 

Nevertheless, the following simple approach turns out to be optimal: first, we note that $T/2\tau$ minibatch gradients of size $\tau$ can be computed fully sequentially, and we can therefore implement minibatch AC-SA \citep{lan2012optimal} to guarantee
\begin{equation}
\E F(\hat{x}) - \min_x F(x) \leq c\cdot\prn*{\frac{HB^2}{(T/2\tau)^2} + \frac{\sigma B}{\sqrt{(T/2\tau)\tau}}}
\end{equation}
which matches the lower bound \eqref{eq:delay-graph-lower-bound}\footnote{In case $HB^2\leq \sigma B / \sqrt{T}$, we can always return 0 to match the lower bound since $F(0) - F^* \leq \frac{H}{2}\nrm{x^*}^2 \leq \frac{HB^2}{2}$.}.
We accomplish this by splitting the $T$ steps into $T/\tau$ windows of length $\tau$; during the first window, we query the stochastic gradient oracle $\tau$ times at the same point $x_0$, then during the second window we wait $\tau$ time steps until all of the $\tau$ queries are answered, we take one AC-SA update, and we repeat this. This algorithm is somewhat unnatural, since it wastes half of its allowed oracle queries waiting, and it would be nice if a prettier algorithm like an accelerated variant of delayed-update SGD could be shown to achieve this same rate. Nevertheless, this algorithm is optimal and confirms that \pref{thm:generic-graph-lower-bound} is also tight for the delay graph.

\subsubsection{A Gap: The Intermittent Communication Graph}\label{subsec:a-gap-the-intermittent-communication-graph}

The fact that \pref{thm:generic-graph-lower-bound} and \pref{lem:statistical-term-lower-bound} are tight for the sequential, layer, and delay graphs makes it tempting to hope that they might be tight for all graphs. However, it turns out to be loose for the intermittent communication graph, and much of the rest of this thesis will be dedicated to closing that gap.

We recall from \pref{subsec:intermittent-graph} that the intermittent communication graph, $\mc{G}_{I.C.}$, corresponds to a setting in which $M$ machines work in parallel over $R$ rounds of communication with $K$ oracle queries per round of communication. The number of vertices in the graph is $MKR$ and the depth is equal to $\textrm{Depth}(\mc{G}_{I.C.}) = KR$. The lower bounds \pref{thm:generic-graph-lower-bound} and \pref{lem:statistical-term-lower-bound} indicate
\begin{equation}
\epsilon(\mc{F}_0(H,B),\mc{A}(\mc{G}_{\textrm{I.C.}}, \mc{O}_g^\sigma)) \geq c \cdot\prn*{\frac{HB^2}{K^2R^2} + \min\crl*{\frac{\sigma B}{\sqrt{MKR}},\, HB^2}}
\end{equation} 
However, the depth of the graph already suggests a problem---since the lower bound \pref{thm:generic-graph-lower-bound} depends only on the depth $KR$, it does not distinguish between algorithms that communicate once and make $T$ queries per machine ($R=1$, $K=T$, and $KR=T$) and algorithms that communicate $T$ times and make one query per communication ($R=T$, $K=1$, and $KR=T$). It seems intuitively obvious that the latter family of algorithms can perform better, but the lower bound \pref{thm:generic-graph-lower-bound} fails to show this. 

Indeed, this shortcoming is not limited to our lower bound. The intermittent communication setting has been widely studied for over a decade, with many optimization algorithms proposed and analyzed \citep{zinkevich2010parallelized,cotter2011better,dekel2012optimal,zhang2013divide,zhang2013communication,shamir2014distributed}, and obtaining new methods and improved analysis is still a very active area of research \citep{stich2019error,wang2017memory,stich2018local,wang2018cooperative,khaled2020tighter,haddadpour2019local,woodworth2020local}.
However, despite these efforts, existing results cannot rule out the possibility that the optimal rate for fixed $T=KR$ can be achieved using only a single round of communication ($R=1$). Existing upper bounds do not beat the lower bound \pref{thm:generic-graph-lower-bound}, and existing lower bounds do not distinguish between $R$ and $K$. The possibility that the optimal rate is achievable with $R=1$ was originally suggested by \citet{zhang2013communication}, and indeed \citet{woodworth2020local} proved that an algorithm that communicates just once is optimal in the special case of quadratic objectives. While it seems unlikely that a single round of communication suffices in the general case, existing results cannot answer this extremely basic question. 

We will revisit this issue and establish matching upper and lower bounds which establish the minimax complexity in the intermittent communication setting in \pref{sec:intermittent-communication-setting}.

\subsection{Reductions Between Convex and Strongly Convex Optimization}\label{subsec:reduction-section-all}

Algorithms for convex optimization are typically analyzed in two distinct settings: the convex setting and the strongly convex setting. Since strongly convex functions are also convex, any algorithm that guarantees a particular rate of convergence for arbitrary convex objectives immediately guarantees that same rate when applied to a strongly convex objective. However, it is generally possible to attain a much \emph{better} convergence rate when the objective is strongly convex, and it is not clear that an algorithm designed for merely convex objectives would necessarily do this. For example, it is well known that gradient descent guarantees convergence at a rate $1/T$ for smooth and convex objectives, and it therefore guarantees the same $1/T$ rate for smooth and strongly convex objectives too. However, it is also well known that gradient descent actually converges at the much faster rate $\exp(-T)$ for strongly convex objectives. 

It is not clear if and when algorithmic guarantees in the convex setting can be parlayed into better guarantees for the strongly convex case. Indeed, it would not be surprising if convex guarantees did not, in and of themselves, imply a better strongly convex guarantee, and if guaranteeing faster rates required a direct analysis. Perhaps surprisingly, we show that this is not the case.

We present a generic ``reduction'' from strongly convex optimization to convex optimization which indicates that any algorithm for convex optimization will, with very minor modification, achieve a faster convergence when the objective has strongly convex structure. In fact, our reduction implies faster convergence for a broader class of objectives that increase sufficiently quickly away from their minima (see \pref{def:growth-condition} below). 

Our reduction shows, for example, that \emph{any} algorithm that converges as $1/T$ for convex objectives will also convergence as $\exp(-T)$ for strongly convex objectives. In other words, there is nothing particularly special about gradient descent that enables it to converge linearly in the strongly convex setting. In this way, our reduction directly connects convex guarantees to better strongly convex guarantees, without requiring any special analysis of the algorithm. This conversion appears to be optimal in the sense that algorithms with optimal convergence rates in the convex setting yield optimal strongly convex guarantees.

The reduction in the opposite direction---from convex optimization to strongly convex optimization---is a standard and well-known part of the optimization toolkit. This, combined with our result, identifies a certain equivalence between convex and strongly convex optimization that was not previously recognized. In many cases, it implies that algorithms with a certain rate of convergence for convex objectives exist \emph{if and only if} algorithms with a corresponding, faster rate exist for strongly convex objectives, and the relationship between these rates is given by the pair of reductions. 

\paragraph{Setting:}
We consider optimization problems of the form
\begin{equation}
\min_{x\in\mc{X}} F(x)
\end{equation}
in two cases, one where $F$ is convex, and the other where $F$ satisfies the following condition:
\begin{definition}\label{def:growth-condition}
For $d:\mc{X}\times\mc{X}\to\R$, we say that $F$ satisfies the $(\lambda,d)$-growth condition (hereafter abbreviated $(\lambda,d)$-GC) if there exists a mapping $x^*:\mc{X}\to\mc{X}$ such that $F(x^*(x)) = F^* = \min_x F(x)$ for all $x \in \mc{X}$, and for all $x \in \mc{X}$
\[
F(x) - F^* \geq \lambda d(x,x^*(x))
\]
\end{definition}
For the most part, we will be interested in Euclidean case where $d(x,y) = \frac{1}{2}\nrm*{x-y}^2$. In this case, it is easy to see that the $(\lambda,d)$-GC is implied by $\lambda$-strong convexity with respect to $\nrm{\cdot}$. However, the growth condition can be substantially weaker, in particular, it does not require that $F$ have a unique minimizer. Furthermore, the results are general and extend to any function $d$, including those that do not correspond to a norm such as Bregman divergences.

We consider optimization algorithms, $\mc{A}$, equipped with a convergence guarantee for convex objectives: that is, for a given class of objectives $\mc{F}$---e.g.~the set of all convex and smooth functions---the algorithm guaratees that it will find an $\epsilon$-suboptimal solution $F(\hat{x}) - F^* \leq \epsilon$ in a certain amount of time. We are deliberately vague about the precise meaning of ``time'' here; often, it correspond to the number of iterations of the algorithm, but it could also count the number of times the algorithm accesses a certain oracle, or even the wall-clock time of an implementation of the algorithm.

The amount of time that an algorithm needs to optimize a convex objective must depend in some way on the distance from the algorithm's ``initialization'' to the minimizer of the function. This can be measured in different ways, but in the convex case we focus on the standard setting where the algorithm is provided with a potentially random point $x_0$ such that $\E d(x_0,x^*(x_0)) \leq B$, and the convergence guarantee scales with $B$. When discussing convergence rates for $(\lambda,d)$-GC objectives, we suppose the algorithm is provided with a point $x_0$ for which $\E F(x_0) - F^* \leq \Delta$, and the rate will depend on $\Delta$.

For a given optimization algorithm, we therefore define two types of convergence rates, $\tm$ and $\tm_\lambda$ corresponding to the convex and $(\lambda,d)$-GC settings, respectively:
\begin{definition}\label{def:tm0-tmlambda}
Let $\mc{F}$ be a class of convex objectives, then $\tm(\epsilon,B,d,\mc{F},\mc{A})$ is the time needed by the algorithm $\mc{A}$ to find a point $\hat{x}$ with $\E F(\hat{x}) - F^* \leq \epsilon$ for any objective $F\in\mc{F}$, when provided with a point $x_0$ with $\E d(x_0,x^*(x_0)) \leq B$. Similarly, $\tm_\lambda(\epsilon,\Delta,d,\mc{F},\mc{A})$ is the time needed by the algorithm $\mc{A}$ to find a point $\hat{x}$ with $\E F(\hat{x}) - F^* \leq \epsilon$ for any $F\in\mc{F}$ that also satisfies the $(\lambda,d)$-GC, when provided with a point $x_0$ such that $\E F(x_0) - F^* \leq \Delta$.
\end{definition}

\subsubsection{The Reduction from Convex to Strongly Convex Optimization}\label{subsec:elads-reduction}

Before we present our result, it is helpful to present the well-known existing reduction from convex optimization to strongly convex optimization. Throughout this section, we focus specifically on $(\lambda,d_2)$-GC objectives where $d_2(x,y) = \frac{1}{2}\nrm{x-y}^2$ is the squared Euclidean distance. 

The simplest form of the reduction from convex to strongly convex optimization begins with an algorithm for optimizing $(\lambda,d_2)$-GC objectives and uses it to optimize convex objectives by applying it to the $(\lambda,d_2)$-GC surrogate $F_\lambda(x) = F(x) + \frac{\lambda}{2}\nrm{x}^2$. Since for $x^* \in \argmin_x F(x)$
\begin{equation}
F(x) - F(x^*) = F_\lambda(x) - F_\lambda(x^*) + \frac{\lambda}{2}\prn*{\nrm*{x^*}^2 - \nrm{x}^2} \leq F_\lambda(x) - \min_x F_\lambda(x) + \frac{\lambda}{2}\nrm*{x^*}^2
\end{equation}
when we choose $\lambda \leq \epsilon / \nrm{x^*}^2$, minimizing $F_\lambda$ to accuracy $\epsilon$ also implies minimizing $F$ to accuracy $O(\epsilon)$. This simple idea allows us to apply the algorithm for $(\lambda,d_2)$-GC objectives to convex functions. In some cases, this reduction can be suboptimal, and it is necessary to use the more sophisticated approach of \citet{zeyuan2016optimal}, which involves solving a sequence of regularized problems with an exponentially decreasing regularization parameter. For the purpose of discussion, we will briefly present their reduction:
\begin{algorithm}[tb]
\caption{$\eladreduction(\mc{A},\theta)$}\label{alg:elad-reduction}
\begin{algorithmic}
\STATE Given: $x_0$ s.t.~$\E F(x_0) - F^* \leq \Delta$
\FOR{$t = 1,2,\dots,T=\ceil{\log_\theta\frac{(1+\theta)\Delta}{\epsilon}}$}
\STATE Set $F_t(x) = F(x) + \frac{\Delta(B^2)}{2B^2}\theta^{-t}\nrm{x - x_0}^2$
\STATE Set $x_t$ to be the output of $\mc{A}$ on $F_t$ initialized with $x_{t-1}$ after \\$\qquad\tm_{\frac{\Delta(B^2)}{B^2}\theta^{1-t}}\prn*{\theta^{-t}\Delta(B^2),\theta^{2-t}\Delta(B^2),d_2,\mc{F},\mc{A}}$
\ENDFOR
\STATE Return $x_T$
\end{algorithmic}
\end{algorithm}

\begin{restatable}{theorem}{eladthm}[c.f.~Theorem 3.1 \citep{zeyuan2016optimal}]\label{thm:elad-reduction}
Let $\Delta(B^2)$ satisfy that for all $F \in \mc{F}$, $\E \nrm{x_0 - x^*}^2 \leq B^2$ implies $\E F(x_0) - F^* \leq \Delta (B^2)$. For any algorithm $\mc{A}$ and $\theta > 1$, then $\eladreduction(\mc{A},\theta)$ defined as in \pref{alg:elad-reduction} guarantees
\[
\tm\prn*{\epsilon,B^2,d_2,\mc{F},\eladreduction(\mc{A},\theta)}
\leq \sum_{t=1}^{\ceil{\log_\theta \frac{(1+\theta)\Delta(B^2)}{\epsilon}}} \tm_{\frac{\Delta(B^2)}{B^2}\theta^{1-t}}\prn*{\theta^{-t}\Delta(B^2),\theta^{2-t}\Delta(B^2),d_2,\mc{F},\mc{A}}
\]
\end{restatable}
\begin{proof}
Let $x^* \in \argmin_x F(x)$ and for each $t$ let $x^*_t = \argmin_x F_t(x)$ and $\lambda_t = \frac{\Delta(B^2)}{B^2}\theta^{1-t}$. First, we note that by the definition of $\tm_{\frac{\Delta(B^2)}{B^2}\theta^{1-t}}\prn*{\theta^{-t}\Delta(B^2),\theta^{2-t}\Delta(B^2),d_2,\mc{F},\mc{A}}$, if $\E F_t(x_{t-1}) - F_t(x^*_t) \leq \theta^{2-t}\Delta(B^2)$ for every $t$, then
\begin{equation}
F(x_T) - F^* \leq F_T(x_T) - F_T(x^*) + \frac{\lambda_T}{2}\nrm{x_0 - x^*}^2 \leq F_T(x_T) - F_T(x_T^*) + \frac{\lambda_T}{2} B^2 \leq \theta^{-T}\Delta(B^2) + \theta^{1-T}\Delta(B^2) \leq \epsilon
\end{equation}
We now prove that $\E F_t(x_{t-1}) - F_t(x^*_t) \leq \theta^{2-t}\Delta(B^2)$ for every $t$.
For any $x_0$,
\begin{equation}\label{eq:convex-to-sc-reduction-closeness}
0 \geq F_t(x^*_t) - F_t(x^*) = F(x^*_t) - F^* + \frac{\lambda_t}{2}\prn*{\nrm{x_0 - x^*_t}^2 - \nrm{x_0 - x^*}^2}
\end{equation}
Therefore, for each $t$, $\nrm{x_0 - x^*_t}^2 \leq \nrm{x_0 - x^*}^2$. We now bound
\begin{align}
\E&\brk*{ F_t(x_{t-1}) - F_t(x_t^*) } \nonumber\\
&= \E\brk*{ F_{t-1}(x_{t-1}) - \frac{\lambda_{t-1} - \lambda_t}{2}\nrm{x_0 - x_{t-1}}^2 - F_{t-1}(x_t^*) + \frac{\lambda_{t-1} - \lambda_t}{2}\nrm{x_0 - x^*_t}^2 } \\
&\leq \E\brk*{ F_{t-1}(x_{t-1}) - F_{t-1}(x_{t-1}^*) + \frac{\lambda_{t-1} - \lambda_t}{2}\nrm{x_0 - x^*_t}^2 } \\
&\leq \E\brk*{ F_{t-1}(x_{t-1}) - F_{t-1}(x_{t-1}^*) + \frac{\lambda_{t-1} - \lambda_t}{2}\nrm{x_0 - x^*}^2 }\label{eq:convex-to-sc-reduction-delta-bound}
\end{align}
For the first inequality, we used the optimality of $x_{t-1}^*$ and for the third inequality we used \eqref{eq:convex-to-sc-reduction-closeness}. 

As the base case, by the definition of $\Delta(B^2)$, $\E F_1(x_0) - F_1^* \leq \E F(x_0) - F^* \leq \Delta(B^2) \leq \theta\Delta(B^2)$. Therefore, by the definition of $\tm_{\frac{\Delta(B^2)}{B^2}}\prn*{\theta^{-1}\Delta(B^2),\theta^{1}\Delta(B^2),d_2,\mc{F},\mc{A}}$, we have $\E F(x_1) - F^* \leq \theta^{-1}\Delta(B^2)$. Furthermore, for any $t$, if $\E\brk*{ F_{t-1}(x_{t-1}) - F_{t-1}(x_{t-1}^*)} \leq \theta^{1-t}\Delta(B^2)$ then by \eqref{eq:convex-to-sc-reduction-delta-bound}
\begin{equation}
\E\brk*{ F_t(x_{t-1}) - F_t(x_t^*) } \leq \theta^{1-t}\Delta(B^2) + \frac{1}{2}(\theta - 1)\theta^{1-t}\Delta(B^2) \leq \theta^{2-t}\Delta(B^2)
\end{equation}
This proves the claim by induction.
\end{proof}

In this way, the algorithm for $(\lambda,d_2)$-GC objectives can be readily applied to merely convex functions. Our reduction, which we present in the next section is conceptually similar to $\eladreduction$ but it is even simpler.

\subsubsection{The Reduction from Strongly Convex to Convex Optimization}\label{subsec:the-reduction}

\begin{algorithm}[tb]
\caption{$\myreduction(\mc{A},\theta)$}\label{alg:my-reduction}
\begin{algorithmic}
\STATE Given: $x_0$ s.t.~$\E F(x_0) - F^* \leq \Delta$
\FOR{$t = 1,2,\dots,T=\ceil{\log_\theta\frac{\Delta}{\epsilon}}$}
\STATE Set $x_t$ to be the output of $\mc{A}$ initialized with $x_{t-1}$ after running for $\tm\prn*{\theta^{-t}\Delta,\theta^{1-t}\frac{\Delta}{\lambda},d,\mc{F},\mc{A}}$
\ENDFOR
\STATE Return $x_T$
\end{algorithmic}
\end{algorithm}

We start with an objective $F$ that satisfies the $(\lambda,d)$-growth condition and an algorithm $\mc{A}$ equipped with a guarantee for merely convex objectives $\tm(\epsilon, B, d, \mc{F},\mc{A})$. Our reduction, \pref{alg:my-reduction} simply applies $\mc{A}$ to the objective a logarithmic number of times, and we show that each application reduces the suboptimality by a factor of $\theta^{-1}$ for some $\theta > 1$. The key insight is just that decreasing the suboptimality for an objective satisfying the growth condition \emph{also} implies reducing the distance to the minimizer. So, 
\begin{equation}
\E\brk*{F(x) - F^*} \leq \theta^{-t} \E\brk*{F(x_0) - F^*} \implies \E d(x,x^*(x)) \leq \theta^{-t}\frac{1}{\lambda}\E\brk*{F(x_0) - F^*}
\end{equation}
Therefore, with each application of the algorithm, the distance to a minimizer is smaller, meaning that the time needed for the next call to $\mc{A}$ is correspondingly smaller. The following theorem uses this idea to generically upper bound $\tm_\lambda(\epsilon,\Delta,d,\mc{F},\myreduction(\mc{A}))$ in terms of $\tm(\epsilon,B,d,\mc{F},\mc{A})$:
\begin{restatable}{theorem}{stronglyconvextoconvexreduction}\label{thm:my-reduction}
For any algorithm $\mc{A}$ and $\theta > 1$, $\myreduction(\mc{A},\theta)$ as in \pref{alg:my-reduction} guarantees
\[
\tm_\lambda(\epsilon,\Delta,d,\mc{F},\myreduction(\mc{A},\theta))
\leq \sum_{t=1}^{\ceil{\log_\theta \frac{\Delta}{\epsilon}}} \tm\prn*{\theta^{-t}\Delta,\theta^{1-t}\frac{\Delta}{\lambda},d,\mc{F},\mc{A}}
\]
\end{restatable}
\begin{proof}
By the definition of $\tm\prn*{\theta^{-t}\Delta,\theta^{1-t}\frac{\Delta}{\lambda},d,\mc{F},\mc{A}}$, if $\E d(x_{t-1}, x^*(x_{t-1})) \leq \theta^{1-t}\frac{\Delta}{\lambda}$ at each iteration, then $\E F(x_t) - F^* \leq \theta^{-t}\Delta$ for each $t$, and $\E F(x_T) - F^* \leq \theta^{-T}\Delta \leq \epsilon$. We now prove by induction that the condition $\E d(x_{t-1}, x^*(x_{t-1})) \leq \theta^{1-t}\frac{\Delta}{\lambda}$ always holds.

As the base case, the $(\lambda,d)$-GC implies that 
\begin{equation}
\lambda d(x_0 - x^*(x_0)) \leq F(x_0) - F^* \implies \E d(x_0,x^*(x_0)) \leq \frac{\Delta}{\lambda}
\end{equation}
Now, suppose that for all $t' < t$, $\E d(x_{t'},x^*(x_{t'})) \leq \theta^{-t'}\frac{\Delta}{\lambda}$. Then, by the definition of $\tm\prn*{\theta^{-t}\Delta,\theta^{1-t}\frac{\Delta}{\lambda},d,\mc{F},\mc{A}}$, we have $\E F(x_t) - F^* \leq \theta^{-t}\Delta$ so by the $(\lambda,d)$-GC
\begin{equation}
\lambda d(x_t,x^*(x_t)) \leq F(x_t) - F^* \implies \E d(x_t, x^*(x_t)) \leq \frac{\E F(x_t) - F^*}{\lambda} \leq \theta^{-t}\frac{\Delta}{\lambda}
\end{equation}
This completes the proof.
\end{proof}
This idea has been applied before in specific cases \citep[see, e.g.][]{ghadimi2013optimal}, but not with this level of generality or simplicity. To understand the utility of the theorem, it is helpful to consider some examples.

\paragraph{Example: Gradient Descent for Smooth Objectives}
Let $\mc{F}_{H}$ be the set of all $H$-smooth (w.r.t.~the Euclidean norm $\nrm{\cdot}_2$), convex objectives, and let $d_2(x,y) = \frac{1}{2}\nrm*{x - y}_2^2$. It is well known that the gradient descent algorithm, which we denote $\mc{A}_{GD}$, requires
\begin{equation}
\tm(\epsilon,B^2,d_2,\mc{F}_{H},\mc{A}_{GD}) \leq c\cdot\frac{HB^2}{\epsilon}
\end{equation}
gradients to find an $\epsilon$-suboptimal point, where $c$ is a universal constant. Applying the reduction, \pref{thm:my-reduction} implies that
\begin{align}
\tm_\lambda(\epsilon,\Delta,d_2,\mc{F}_{H},\myreduction(\mc{A}_{GD},e))
&\leq \sum_{t=1}^{\ceil{\log \frac{\Delta}{\epsilon}}} \tm\prn*{e^{-t}\Delta,e^{1-t}\frac{\Delta}{\lambda},d_2,\mc{F}_{H},\mc{A}_{GD}} \\
&\leq cH\sum_{t=1}^{\ceil{\log \frac{\Delta}{\epsilon}}} \frac{e^{1-t}\frac{\Delta}{\lambda}}{e^{-t}\Delta} \\
&= ec\cdot\frac{H}{\lambda}\left\lceil\log \frac{\Delta}{\epsilon}\right\rceil\label{eq:reduced-to-sc-gd-rate}
\end{align}
Therefore, our reduction recovers (up to constant factors) the $\kappa \log(\Delta/\epsilon)$ convergence rate of gradient descent for $\lambda$-strongly convex (or, more broadly, $(\lambda,d_2)$-GC) objectives. We emphasize that this guarantee \eqref{eq:reduced-to-sc-gd-rate} has nothing to do with gradient descent specifically---for any algorithm $\mc{A}$ with 
\begin{equation}
\tm(\epsilon,B^2,d_2,\mc{F}_{H},\mc{A}) \leq c\cdot\frac{HB^2}{\epsilon},
\end{equation}
the modified algorithm $\myreduction(\mc{A},e)$ will enjoy the same linear convergence as gradient descent, \eqref{eq:reduced-to-sc-gd-rate}.

\paragraph{Example: Accelerated SGD for Smooth Objectives}
For $\mc{F}_H$, the class of convex and $H$-smooth (w.r.t.~the Euclidean norm $\nrm{\cdot}_2$) objectives, \citet{lan2012optimal} proposed an algorithm, AC-SA which, for $d_2(x,y) = \frac{1}{2}\nrm{x-y}^2$, requires
\begin{equation}
    \tm(\epsilon,B^2,d_2,\mc{F}_H,\mc{A}_{AC-SA}) = c\cdot\prn*{\sqrt{\frac{HB^2}{\epsilon}} + \frac{\sigma^2B^2}{\epsilon^2}}
\end{equation}
stochastic gradients with variance bounded by $\sigma^2$ to find an $\epsilon$-suboptimal point, which is optimal. Our reduction says that this algorithm needs
\begin{align}
\tm_\lambda(\epsilon,\Delta,d_2,\mc{F}_H,\myreduction(\mc{A}_{AC-SA},e))
&\leq \sum_{t=1}^{\left\lceil\log\frac{\Delta}{\epsilon}\right\rceil}\tm\prn*{e^{-t}\Delta,e^{1-t}\frac{\Delta}{\lambda},d_2,\mc{F}_{H},\mc{A}_{AC-SA}} \\
&= c\cdot\prn*{\sqrt{H}\sum_{t=1}^{\left\lceil\log\frac{\Delta}{\epsilon}\right\rceil}\sqrt{\frac{e^{1-t}\frac{\Delta}{\lambda}}{e^{-t}\Delta}} + \sigma^2 \sum_{t=1}^{\left\lceil\log\frac{\Delta}{\epsilon}\right\rceil}\frac{e^{1-t}\frac{\Delta}{\lambda}}{e^{-2t}\Delta^2} }  \\
&\leq ec\cdot\prn*{\sqrt{\frac{H}{\lambda}}\left\lceil\log\frac{\Delta}{\epsilon}\right\rceil + \frac{\sigma^2}{\lambda\Delta} \sum_{t=1}^{\left\lceil\log\frac{\Delta}{\epsilon}\right\rceil}e^t} \\
&\leq ec\cdot\prn*{\sqrt{\frac{H}{\lambda}}\left\lceil\log\frac{\Delta}{\epsilon}\right\rceil + \frac{e\sigma^2}{(e-1)\lambda\Delta}\exp\prn*{\left\lceil\log\frac{\Delta}{\epsilon}\right\rceil }} \\
&\leq c'\cdot\prn*{\sqrt{\frac{H}{\lambda}}\left\lceil\log\frac{\Delta}{\epsilon}\right\rceil + \frac{\sigma^2}{\lambda\epsilon}} \label{eq:ac-sa-strongly-convex}
\end{align}
stochastic gradients with variance bounded by $\sigma^2$ for a $\lambda$-strongly convex objective, where $c'$ is a universal constant. This is, up to constant factors, the optimal rate for strongly convex objectives.

In follow-up work to \citet{lan2012optimal}, \citet{ghadimi2013optimal} describe a ``multi-stage'' variant of AC-SA which they show achieves the same rate \eqref{eq:ac-sa-strongly-convex}. This algorithm closely resembles $\myreduction(\mc{A}_{AC-SA},e)$ with some small differences, and their analysis matches at a high level the proof of \pref{thm:my-reduction}. However, their analysis is considerably more complicated, and we feel that there is significant value in our generalized approach which, by reasoning at a higher level of abstraction, results in a much simpler proof.

\subsubsection{A New Minibatch Accelerated SGD Algorithm}

In the context of machine learning, stochastic optimization algorithms are often employed in an ``overparametrized'' regime, where there exist settings of the parameters (typically many of them) that achieve zero, or near zero, loss on the population. An important observation made by \citet{cotter2011better}, is that when the loss function is itself smooth and non-negative, this introduces a useful bound on the variance of the gradient of the instantaneous losses, which decreases as the parameters approach a minimizer. Specifically, for $F(x) = \E_z f(x;z)$ with $f(\cdot;z)$ being $H$-smooth and non-negative for all $z$, we have
\begin{equation}
\E\nrm{\nabla f(x;z)}^2 \leq 2H\E[f(x;z) - \min_x f(x;z)] \leq 2H F(x)
\end{equation}
Therefore, when $\min_x F(x) \approx 0$ and $F(x) - \min_x F(x)$ is small, then the stochastic gradient variance is also small. 

At the same time, given increases in the size of datasets and the availability of parallel computing resources, many training procedures utilize minibatch stochastic gradients which can easily be calculated in parallel. Using minibatch stochastic gradients of size $b$ reduces the variance of the updates by a factor of $b$, which naturally gives faster convergence for any given number of updates, $T$. However, given runtime and sample complexity costs to computing these minibatch gradients, using larger minibatches typically necessitates making fewer updates, and it is important to understand to what extent trading off $b$ and $T$ affects the performance of an algorithm. 

Let $\mc{F}_{H,+,F^*}$ denote the set of objectives, $F$ for which $\min_x F(x) = F^*$ and $F(x) = \E_z f(x;z)$ for some $H$-smooth, non-negative, and convex $f(\cdot;z)$; and let $d_2(x,y) = \frac{1}{2}\nrm{x- y}_2^2$. In their work, \citet{cotter2011better} propose an accelerated SGD variant, AG, which uses minibatch gradients of the form $\frac{1}{b}\sum_{i=1}^b \nabla f(x;z_i)$ for i.i.d.~$z_1,\dots,z_b \sim \mc{D}$ to guarantee
\begin{equation}
\tm\prn*{\epsilon, B^2, d_2, \mc{F}_{H,+,F^*}, \mc{A}_{AG,b}} \leq c \cdot \prn*{\sqrt{\frac{HB^2}{\epsilon}} + \frac{HB^2 F^*}{b\epsilon^2} + \frac{HB^2\prn*{1 + \frac{1}{\sqrt{b}}\sqrt{\log\frac{HB^2}{\epsilon}}}}{\sqrt{b}\epsilon}}
\end{equation}
As is discussed by \citet{cotter2011better}, this bound suggests that this algorithm has two or three regimes of convergence depending on the batchsize, $b$, and the relative scale of $F^*$, $\epsilon$, $H$, and $B$ (ignoring the logarithmic factor in the third term):

The first regime corresponds to convergence after $\frac{HB^2 F^*}{b\epsilon^2}$ iterations, which happens when
\begin{equation}\label{eq:ag-b-phase-1}
b \leq \min\crl*{\frac{F^*\sqrt{HB^2}}{\epsilon^{3/2}},\, \frac{{F^*}^2}{\epsilon^2}}
\end{equation}
When $b$ is this small, the time needed to reach accuracy $\epsilon$ improves linearly with an increase the batchsize, so it would generally be advantageous to take $b$ at least as large as the upper bound \eqref{eq:ag-b-phase-1} to exploit this. Convergence in this regime can be extremely fast when $F^* \approx 0$, as is common in many machine learning applications.

The second regime requires $\frac{HB^2}{\sqrt{b}\epsilon}$ iterations when
\begin{equation}
\frac{{F^*}^2}{\epsilon^2} \leq b \leq \frac{HB^2}{\epsilon}
\end{equation}
This intermediate regime shows that there are diminishing returns to increasing the batchsize beyond a certain point. Once $b$ is moderately large, increasing $b$ results in a sublinear reduction in the number of iterations needed to reach accuracy $\epsilon$ versus a linear reduction in the first case. 

Finally, the third regime has convergence governed by the term $\sqrt{\frac{HB^2}{\epsilon}}$, which occurs once the batchsize is sufficiently large:
\begin{equation}
b \geq \max\crl*{\frac{HB^2}{\epsilon},\, \frac{F^*\sqrt{HB^2}}{\epsilon^{3/2}}}
\end{equation}
Once the batchsize has passed this critical threshold, there is nothing to be gained by increasing it further. Indeed, the rate $\sqrt{\frac{HB^2}{\epsilon}}$ is optimal even for first-order methods with access to exact gradients of the objective \citep{nemirovskyyudin1983}, so this represents a setting where the batchsize is large enough that the noise in the stochastic gradients becomes negligible.

This algorithm achieves fast convergence in the convex setting, but it is not clear how well it would perform when the objective has more favorable properties such as strong convexity. Indeed, \citeauthor{cotter2011better}'s guarantee for the algorithm required many pages of analysis, and it is definitely not trivial to directly extend their proof to the strongly convex case. Luckily, to understand how much improvement is possible, we can simply apply our reduction. By \pref{thm:my-reduction}, for $F \in \mc{F}_{H,+,F^*}$ which also satisfies the $(\lambda,d_2)$-growth condition,
\begin{align}
\tm_\lambda&(\epsilon,\Delta,d_2,\mc{F}_{H,+,F^*},\myreduction(\mc{A}_{AG,b},e)) \nonumber\\
&\leq \sum_{t=1}^{\left\lceil\log\frac{\Delta}{\epsilon}\right\rceil}\tm\prn*{e^{-t}\Delta,e^{1-t}\frac{\Delta}{\lambda},d_2,\mc{F}_{H,+,F^*},\mc{A}_{AG,b}} \\
&\leq c\sum_{t=1}^{\left\lceil\log\frac{\Delta}{\epsilon}\right\rceil} \brk*{\sqrt{\frac{He^{1-t}\frac{\Delta}{\lambda}}{e^{-t}\Delta}} + \frac{HF^* e^{1-t}\frac{\Delta}{\lambda}}{be^{-2t}\Delta^2} + \frac{He^{1-t}\frac{\Delta}{\lambda}\prn*{1 + \frac{1}{\sqrt{b}}\sqrt{\log\frac{He^{1-t}\frac{\Delta}{\lambda}}{e^{-t}\Delta}}}}{\sqrt{b}e^{-t}\Delta}} \\
&\leq ec\sum_{t=1}^{\left\lceil\log\frac{\Delta}{\epsilon}\right\rceil} \brk*{\sqrt{\frac{H}{\lambda}} + \frac{HF^* e^t}{\lambda b\Delta} + \frac{H}{\lambda\sqrt{b}}\prn*{1 + \frac{1}{\sqrt{b}}\sqrt{\log\frac{eH}{\lambda}}}} \\
&\leq ec\cdot\prn*{\prn*{\sqrt{\frac{H}{\lambda}} + \frac{H}{\lambda\sqrt{b}}\prn*{1 + \frac{1}{\sqrt{b}}\sqrt{\log\frac{eH}{\lambda}}}}\left\lceil\log\frac{\Delta}{\epsilon}\right\rceil + \frac{HF^*}{\lambda b\Delta}\frac{e\prn*{\exp\prn*{\left\lceil\log\frac{\Delta}{\epsilon}\right\rceil} - 1}}{e-1}} \\
&\leq c'\cdot\prn*{\prn*{\sqrt{\frac{H}{\lambda}} + \frac{H}{\lambda\sqrt{b}}\prn*{1 + \frac{1}{\sqrt{b}}\sqrt{\log\frac{H}{\lambda}}}}\ceil*{\log\frac{\Delta}{\epsilon}} + \frac{HF^*}{\lambda b\epsilon}}\label{eq:ag-with-growth-condition}
\end{align}
As in the other examples, under the $(\lambda,d_2)$-GC, convergence can be substantially faster and can depend only logarithmically on the accuracy parameter $\epsilon$. We also note that while the $(\lambda,d_2)$-GC is implied by $\lambda$-strong convexity with respect to the L2 norm, the growth condition can apply more broadly. This is of particular importance in the context of training overparametrized machine learning models, for which there are generally many minimizers of the objective so strong convexity will not hold. Despite the fact that the objective may be constant along some directions around its minimizers, as long as it grows sufficiently quickly along the other directions the $(\lambda,d_2)$-GC can still hold. For example, for least squares regression in $\R^d$ with $n < d$ training examples, the sample covariance matrix is rank-deficient, and thus the training loss cannot be strongly convex. Nevertheless, the training loss will satisfy the $(\lambda,d_2)$-GC with $\lambda$ equal to the smallest \emph{non-zero} eigenvalue of the sample covariance matrix, and the modified algorithm can therefore minimize the training loss very efficiently.

Examining the fast rate \eqref{eq:ag-with-growth-condition} and ignoring the $\log (H/\lambda)^{1/2}$ term, we see that as in the convex setting there are three regimes of convergence depending on the batchsize:

In the small batchsize regime, the algorithm requires $\frac{HF^*}{\lambda b\epsilon}$ iterations for
\begin{equation}
b \leq \min\crl*{\frac{\sqrt{H}F^*}{\sqrt{\lambda}\epsilon\log\frac{\Delta}{\epsilon}},\, \frac{{F^*}^2}{\epsilon^2 \log^2 \frac{\Delta}{\epsilon}}}
\end{equation}
As in the convex case, for this regime, the number of iterations needed to reach accuracy $\epsilon$ decreases linearly with an increase in the batchsize, and the rate of convergence depends very favorably on minimal value of the objective when $F^* \approx 0$. 

In the intermediate batchsize regime, the algorithm requires $\frac{H}{\lambda\sqrt{b}}\log\frac{\Delta}{\epsilon}$ iterations when
\begin{equation}
\frac{{F^*}^2}{\epsilon^2 \log^2 \frac{\Delta}{\epsilon}} \leq b \leq \frac{H}{\lambda}
\end{equation}
Again, this shows that there are diminishing returns to increasing the batchsize beyond a certain point, and the iteration complexity decreases only sublinearly in this case. Nevertheless, this still represents rapid convergence to an approximate minimizer of $F$, depending only logarithmically on the accuracy parameter.

Finally, in the large batchsize regime, the algorithm requires $\sqrt{\frac{H}{\lambda}}\log\frac{\Delta}{\epsilon}$ iterations for
\begin{equation}
b \geq \max\crl*{\frac{{F^*}^2}{\epsilon^2 \log^2 \frac{\Delta}{\epsilon}},\, \frac{H}{\lambda}}
\end{equation}
After this point, there is nothing to be gained by further increasing the batchsize and this represents the optimal rate of convergence for first-order algorithms, even when they have access to exact gradients of the objective \citep{nemirovskyyudin1983}\footnote{The lower bound applies to $\lambda$-strongly convex (w.r.t.~L2) objectives, a subclass of $(\lambda,d_2)$-GC objectives which are therefore covered by the same lower bound.}.

\paragraph{Comparison with \citet{liu2018mass}:}
In recent work, \citet{liu2018mass} proposed another stochastic first-order algorithm, MaSS, for optimizing functions in $\mc{F}_{H,+,F^*}$ in the special case $F^* = 0$, which obtains qualitatively similar guarantees. There are, however, several key differences between our modification of \citeauthor{cotter2011better}'s algorithm and MaSS. 

In the special case of least squares problems, where $f(x;z) = \frac{1}{2}(\inner{z_1}{x} - z_2)^2$, and where $F(x) = \frac{1}{N}\sum_{n=1}^N f(x;z^{(n)})$ is the training loss, MaSS requires \citep[Theorem 2][]{liu2018mass}
\begin{equation}
\prn*{\sqrt{\frac{H}{\lambda}} + \frac{H}{\lambda\sqrt{b}}}\log\frac{\Delta}{\epsilon}
\end{equation}
iterations, where $\lambda$ is the smallest non-zero eigenvalue of the Hessian $\E [z_1z_1^\top]$. This matches the guarantee \eqref{eq:ag-with-growth-condition} for the modification of \citeauthor{cotter2011better}'s algorithm, $\myreduction(\mc{A}_{AG,b},e)$. However, this result is limited to least squares problems, to the case that $F^* = 0$, and only guarantees finding an $\epsilon$-suboptimal point with respect to the \emph{training} loss, whereas the guarantee \eqref{eq:ag-with-growth-condition} applies to the population loss. 

In addition to the least squares result, \citeauthor{liu2018mass} also analyze the MaSS algorithm for a more general class of $\lambda$-strongly convex objectives that satisfy a certain smoothness property \citep[see Theorem 3][]{liu2018mass}. However, it requires $\frac{H}{\lambda}\log\frac{\Delta}{\epsilon}$ iterations to reach an $\epsilon$-suboptimal point, i.e.~an ``unaccelerated'' rate. Furthermore, it does not show any benefit from minibatching like \eqref{eq:ag-with-growth-condition} does. Finally, this result is limited to the case of (1) strongly convex objectives and (2) $F^* = 0$, which implies very strong constraints on the objective. In particular, a common source of strong convexity is training a linear model with an L2 regularization penalty $\frac{\lambda}{2}\nrm{x}^2$, in which case these conditions are only satisfied in the trivial case where $x=0$ minimizes the objective. In contrast, our guarantee \eqref{eq:ag-with-growth-condition} applies even when $F^* > 0$, and also to the broader class of $(\lambda,d_2)$-GC objectives.

\subsubsection{Optimality of the Reductions}

The combination of \pref{thm:elad-reduction} and \pref{thm:my-reduction} suggests that there is a certain equivalence between convex and strongly convex optimization. Our reduction shows that the existence of an algorithm that converges like $1/T$ for convex objectives implies the existence of an algorithm that converges like $\exp(-\lambda T)$ for $\lambda$-strongly convex ones. Conversely, \citeauthor{zeyuan2016optimal}'s reduction in the other direction says that the existence of an algorithm that converges as $\exp(-\lambda T)$ for $\lambda$-strongly convex objectives implies the existence of an algorithm with rate $1/T$ in the convex setting. Therefore, these are actually equivalent statements: a $1/T$ algorithm for convex objectives exists \emph{if and only if} an $\exp(-\lambda T)$ algorithm exists for $\lambda$-strongly convex objectives. 

Indeed, we argue that these reductions are optimal in the following sense: if we take an algorithm $\mc{A}$ for optimizing convex objectives and use \pref{alg:my-reduction} to convert it into an algorithm for strongly convex optimization, $\myreduction(\mc{A},e)$, and then we use \pref{alg:elad-reduction} to convert it \emph{back} into an algorithm for convex optimization, $\eladreduction(\myreduction(\mc{A},e),e)$ then the guarantee generally degrades by just a constant factor. 

By \pref{thm:my-reduction} and \pref{thm:elad-reduction}, we have
\begin{align}
\tm&\prn*{\epsilon,B^2,d_2,\mc{F},\eladreduction(\myreduction(\mc{A},e),e)}\nonumber\\
&\leq \sum_{t=1}^{\ceil{\log \frac{(1+e)\Delta(B^2)}{\epsilon}}} \tm_{\frac{\Delta(B^2)}{B^2}e^{1-t}}\prn*{e^{-t}\Delta(B^2),e^{2-t}\Delta(B^2),d_2,\mc{F},\myreduction(\mc{A},e)} \\
&\leq \sum_{t=1}^{\ceil{\log \frac{(1+e)\Delta(B^2)}{\epsilon}}}\sum_{s=1}^{2} \tm\prn*{e^{2-t-s}\Delta(B^2),e^{2-s}B^2,d_2,\mc{F},\mc{A}}
\end{align}
Therefore, in the typical case where $\tm\prn*{\epsilon,B^2,d,\mc{F},\mc{A}} = c\cdot B^{2m}\epsilon^{-n}$ for some $m,n > 0$, depends polynomially on $B^2$ and $\epsilon$, then
\begin{align}
\tm&\prn*{\epsilon,B^2,d_2,\mc{F},\eladreduction(\myreduction(\mc{A},e),e)}\nonumber\\
&\leq c\cdot\sum_{t=1}^{\ceil{\log \frac{(1+e)\Delta(B^2)}{\epsilon}}} \sum_{s=1}^2 \frac{e^{(2-s)m}B^{2m}}{e^{(2-t-s)n}\Delta(B^2)^n} \\
&= c\cdot\frac{(1+e^{m-n}) B^{2m}}{\Delta(B^2)^n} \frac{e^n\prn*{e^{n\ceil{\log \frac{(1+e)\Delta(B^2)}{\epsilon}}} - 1}}{e^n - 1} \\
&\leq c\cdot\frac{(e^{2n}+e^{m+n})(1+e)^n}{e^n - 1} \frac{B^{2m}}{\epsilon^m}
\end{align}
So, applying both reductions maintains the same $B$ and $\epsilon$ dependence, and loses only a ``small constant'' factor assuming that $m$ and $n$ are relatively small, as they typically are. This is evidence that we should not expect there to be a significantly better general purpose reduction in either direction.

\subsubsection{Proving Lower Bounds by Reduction}\label{subsec:lower-bounds-by-reduction}

Here, we consider the implications of these reductions to lower bounds. Specifically, \pref{thm:elad-reduction} shows that algorithms for strongly convex optimization with a guarantee $\tm_\lambda(\epsilon,\Delta)$ implies the existence of an algorithm for convex optimization with a corresponding guarantee $\tm_0(\epsilon,B^2)$. Taking the contrapositive of this statement, we conclude that a lower bound that shows $\tm_0(\epsilon,B^2) \geq T_{0}(\epsilon,B^2)$ implies a certain lower bound on $\tm_\lambda(\epsilon,\Delta)$. However, this lower bound does not apply for all values of $\lambda$, $\epsilon$, and $\Delta$. In particular, we have that for any $\theta > 0$
\begin{equation}
T_{0}(\epsilon,B^2) \leq \tm_0(\epsilon,B^2)
\leq \sum_{t=1}^{\ceil{\log_\theta \frac{(1+\theta)\Delta(B^2)}{\epsilon}}} \tm_{\frac{\Delta(B^2)}{B^2}\theta^{1-t}}\prn*{\theta^{-t}\Delta(B^2),\theta^{2-t}\Delta(B^2)}
\end{equation}
Therefore, the lower bound only really applies to strongly convex functions for parameters $\lambda$, $\epsilon$, and $\Delta$ with a very particular relationship with each other. For instance, with $\Delta = \theta^2\epsilon$, $\lambda = \theta B^2 \epsilon$, etc. Nevertheless, this does tell us that an upper bound with a particular functional form cannot exist. We now show how to apply this idea to derive a lower bound in the strongly convex setting for the graph oracle model:

\begin{restatable}{theorem}{genericgraphoraclestronglyconvex}\label{thm:generic-graph-oracle-strongly-convex}
For any graph $\mc{G}$, let $\mc{O}_v$ be an exact gradient oracle for each $v$. Then for any $H$ and any dimension $D \geq c\cdot \textrm{Depth}(\mc{G})\log(\abs{\mc{V}})$, there exists $\lambda$ and $\Delta$ and a function $F_\lambda \in \mc{F}_\lambda(H,\Delta)$ in dimension $D$ such that the output of any algorithm in $\mc{A}(\mc{G},\mc{O}_v)$ will have suboptimality at least
\[
F_\lambda(\hat{x}) - F_\lambda^* \geq c\cdot\Delta\prn*{-\frac{c'\sqrt{\lambda}\textrm{Depth}(\mc{G})}{\sqrt{H}}}
\]
\end{restatable}
\begin{proof}
We prove this by contradiction. Let $\mc{F}_\lambda(H,\Delta,D)$ be the subset of $\mc{F}_\lambda(H,\Delta)$ in dimension $D \geq c\cdot \textrm{Depth}(\mc{G})\log(\abs{\mc{V}})$, and suppose there were an algorithm $\mc{A}$ with guarantee
\begin{equation}\label{eq:sc-reduction-proof-contradiction}
\tm_\lambda(\epsilon,\Delta,d_2,\mc{F}_\lambda(H,\Delta,D),\mc{A}) \leq c'\sqrt{\frac{H}{\lambda}}\log\frac{c\Delta}{\epsilon}
\end{equation}
where here $\tm$ refers to the minimum graph depth $\textrm{Depth}(\mc{G})$ needed to guarantee that $\E F(\hat{x}) - F^* \leq \epsilon$. We note that if such an algorithm existed, then its guarantee would imply a guarantee in terms of $\textrm{Depth}(\mc{G})$ that would contradict the claim of the theorem, which can be seen by solving this for $\epsilon$.

Then, \pref{thm:elad-reduction} implies
\begin{align}
\tm_0&(\epsilon,B^2,d_2,\mc{F}_0(H,B),\eladreduction(\mc{A},e)) \nonumber\\
&\leq \sum_{t=1}^{\ceil*{\log\prn*{\frac{(1+e)HB^2}{\epsilon}}}} \tm_{He^{1-t}}\prn*{HB^2e^{-t},HB^2e^{2-t},d_2,\mc{F}_\lambda(H,\Delta,D),\mc{A}} \\
&\leq c'\cdot \sum_{t=1}^{\ceil*{\log\prn*{\frac{(1+e)HB^2}{\epsilon}}}} \sqrt{\frac{1}{e^{1-t}}}\log(ce^2) \\
&\leq c'(2+\log(c)) \frac{1}{\sqrt{e}} \frac{\sqrt{e}}{\sqrt{e} - 1}\prn*{e^{\frac{1}{2}\ceil{\log\prn{\frac{(1+e)HB^2}{\epsilon}}}} - 1} \\
&\leq c'(2+\log(c)) \frac{\sqrt{e^2 + e}}{\sqrt{e} - 1}\sqrt{\frac{HB^2}{\epsilon}}
\end{align}
However, for sufficiently small $c$ and $c'$, this would contradict the lower bound \pref{thm:generic-graph-lower-bound}. We conclude that there are constants $c$ and $c'$ for which no algorithm can provide the guarantee \eqref{eq:sc-reduction-proof-contradiction}.
\end{proof}

While this lower bound is limited in that it only applies for some values $\lambda$ and $\Delta$, it is nevertheless suggestive of a more broad lower bound. For essentially all of the strongly convex optimization settings we are aware of, optimal algorithms provide an upper bound that is a simple, continuous function of $H$, $\lambda$, and $\Delta$, and these upper bounds apply for any value of these parameters. Therefore, although it is certainly possible that this lower bound does not apply for some values of $\lambda$ and $\Delta$, that appears quite unlikely.

\subsubsection{The Practicality of the Reduction}

Our reduction \pref{alg:my-reduction} takes an algorithm for convex optimization and lightly modifies it to create an algorithm for strongly convex (or GC) objectives, but it is important to be clear about what this modified algorithm actually looks like. Consider the case of gradient descent with a constant stepsize $\eta$. In this case, since each of $\mc{A}_{GD}$'s updates have exactly the same form, and since the $t\mathth$ call to $\mc{A}_{GD}$ in \pref{alg:my-reduction} picks up at the final iterate of the $(t-1)\mathth$ run of $\mc{A}_{GD}$, $\myreduction(\mc{A}_{GD})$ is actually exactly the same as gradient descent with the same constant stepsize $\eta$---it's just run for a different number of iterations. 

In contrast, consider an algorithm like SGD. Importantly, in order to achieve the rate $\frac{HB^2}{\epsilon} + \frac{\sigma^2 B^2}{\epsilon^2}$ using SGD, it is necessary to use a stepsize that depends on $\epsilon$ and $B$, and it is also necessary to return an average of the SGD iterates. For this reason, the modified algorithm $\myreduction(\mc{A}_{SGD})$, when viewed as a single unit, will appear somewhat strange. It will still resemble SGD, but the stepsize schedule is separated into several distinct phases and at the end of each phase, the next iterate will be changed to an average of some of the previous iterates. This is, of course, a perfectly valid first-order optimization algorithm, but it is fairly strange, and from a practical perspective, it is quite possible that $\myreduction(\mc{A}_{SGD})$ would not perform as well as $\mc{A}_{SGD}$ itself despite its enhanced guarantee. Nevertheless, our analysis of $\myreduction(\mc{A}_{SGD})$ proves that \emph{some} algorithm with that better guarantee exists, and suggests that a ``more natural'' probably exists too. On the other hand, if a ``more natural'' algorithm \emph{did not} exist, that would also be extremely interesting! For example, the only optimal algorithm for stochastic first-order optimization in the strongly convex setting that we are aware of is either $\myreduction(\mc{A}_{\textrm{AC-SA}},e)$ or \citet{ghadimi2013optimal}'s essentially identical method. Because of the reductions, both of these methods resemble SGD with momentum, but the momentum and stepsize parameters are somewhat crazy and non-monotonic.

\section{Analysis of Local SGD}\label{sec:local-sgd}

Local SGD is a very popular and natural algorithm for the intermittent communication setting with a stochastic first-order oracle. The idea is simple: during a round of communication, each of the $M$ machines independently takes $K$ SGD steps and at the end of the round, their $M$ iterates are averaged together to form the starting point for the next round. 

Local SGD has been widely used in a variety of applications, including in the homogeneous, heterogeneous, and federated settings, and for convex and non-convex objectives. However, while Local SGD is often quite successful in these applications, its theoretical properties were not well understood. Specifically, despite dozens of papers trying to prove convergence guarantees for Local SGD, none show substantial improvement over simple baseline algorithms that we will discuss shortly. Despite its empirical success, this raises the question of whether Local SGD truly has poorer theoretical performance, or if we have just not figured out how to analyze it properly.

We will partially resolve these theoretical questions about Local SGD for smooth and convex/strongly convex objectives in the homogeneous and heterogeneous settings. In \pref{subsec:local-sgd-and-baselines}, we will define Local SGD along with several baseline algorithms which will serve as a useful point of reference in evaluating guarantees for Local SGD. In \pref{subsec:local-sgd-homogeneous}, we analyze Local SGD in the homogeneous setting, where each machine relies on data drawn from the same distribution. For homogeneous objectives, we provide new upper and lower bounds on the accuracy of Local SGD and show how these are the first to significantly improve over the baselines in a certain regime. In \pref{subsec:local-sgd-heterogeneous}, we move on to the heterogeneous setting, where each machine has access to data drawn from different distributions. Here, we also provide new and better upper and lower bounds on the accuracy of Local SGD and we compare with the baseline algorithms.

\subsection{Local, Minibatch, and Single-Machine SGD}\label{subsec:local-sgd-and-baselines}

We will now define the Local SGD algorithm along with two natural baselines and one unnatural baseline. Each of these algorithms belongs to the class $\mc{A}(\mc{G}_{\textrm{I.C.}},\crl{\mc{O}_v})$ as defined in \pref{sec:formulating-the-complexity}. We recall that the intermittent communication graph captures a setting in which $M$ parallel machines collaborate to optimize the objective over the course of $R$ rounds of communication, and in each round of communication, each machine is allowed $K$ queries to a stochastic gradient oracle. 

In the homogeneous setting, the stochastic gradient oracle associated with each vertex is the same---$\mc{O}_v = \mc{O}_g^\sigma$ for all $v$---meaning that each oracle access on each machine provides an unbiased estimate of $\nabla F$ with variance bounded by $\sigma^2$. Conversely, in the heterogeneous setting, the stochastic gradient oracles are different depending on which machine is making a query. Specifically, the objective has the form $F(x) = \frac{1}{M}\sum_{m=1}^M F_m(x)$, and each of the vertices corresponding to the $m\mathth$ machine is associated with an oracle $\mc{O}_{g,m}^\sigma$ that provides an unbiased estimate of $\nabla F_m$ with variance less than $\sigma^2$. Throughout this section, we will use $g(\cdot;z^m_{k,r})$ to denote the stochastic gradient oracle for the $k\mathth$ query on the $m\mathth$ machine during the $r\mathth$ round of communication. 

We ask what is the best guarantee that can be provided for Local SGD in the intermittent communication setting, for a particular and fixed $M$, $K$, and $R$. In order to make a fair comparison, it is therefore important that the baselines algorithms also belong to the same class of algorithms $\mc{A}(\mc{G}_{\textrm{I.C.}},\crl{\mc{O}_v})$. After all, it would not be fair to compare Local SGD against an algorithm which is allowed to use $2M$ parallel workers, or one that is only allowed to communicate $R/2$ times.

\textbf{Local SGD:} 
We first define the Local SGD updates in \pref{alg:local-sgd}. At any given time, Local SGD maintains $M$ iterates---one for each machine---and we use $x^m_{k,r}$ to denote the $k\mathth$ iterate during the $r\mathth$ round of communication on machine $m$. Between communications, each  machine simply updates its iterate according to a standard stochastic gradient step. At the end of each round of communication, the local iterates are averaged and form the starting point for the next round.
\begin{algorithm}[tb]
\caption{Local SGD}\label{alg:local-sgd}
\begin{algorithmic}
\STATE For each $m$: $x^m_{0,0} = x_0$
\FOR{$r = 0,1,\dots,R-1$}
\FOR{$k = 0,1,\dots,K-1$}
\STATE For each $m$ in parallel: $x^m_{k+1,r} = x^m_{k,r} - \eta_{k,r} g(x^m_{k,r};z^m_{k,r})$
\ENDFOR
\STATE Communicate and for each $m$: $x^m_{0,r+1} = \frac{1}{M}\sum_{m=1}^M x^m_{K,r}$
\ENDFOR
\STATE Return: $\hat{x} = \frac{1}{M\sum_{k=1}^K\sum_{r=1}^R w_{k,r}} \sum_{m=1}^M\sum_{k=1}^K\sum_{r=1}^R w_{k,r} x^m_{k,r}$
\end{algorithmic}
\end{algorithm}

\textbf{Minibatch SGD:}
The first baseline algorithm is Minibatch SGD, meaning $R$ steps of SGD using minibatches of size $M\cdot K$. While ``Minibatch SGD'' could refer to many different algorithms (i.e.~many different combinations of a number of steps and a minibatch size), we focus on one specific version that belongs to the family $\mc{A}(\mc{G}_{\textrm{I.C.}},\crl{\mc{O}_v})$. In particular, \pref{alg:minibatch-sgd} can be implemented in the intermittent communication setting by having each of the $M$ machines compute $K$ stochastic gradients, all at the same point. At the end of the round of communication, all of these stochastic gradients can be combined into a single large minibatch of size $M \cdot K$ and a single SGD step can be taken, meaning that $R$ SGD steps can be taken in total. 
\begin{algorithm}[tb]
\caption{Minibatch SGD}\label{alg:minibatch-sgd}
\begin{algorithmic}
\STATE Initialize: $x_0$
\FOR{$r = 0,1,\dots,R-1$}
\FOR{$k = 0,1,\dots,K-1$}
\STATE For each $m$ in parallel: compute $g(x_r;z^m_{k,r})$
\ENDFOR
\STATE $\bar{g}_r = \frac{1}{MK}\sum_{m=1}^M\sum_{k=0}^{K-1} g(x_r;z^m_{k,r})$
\STATE $x_{r+1} = x_r - \eta_r \bar{g}_r$
\ENDFOR
\STATE Return: $\hat{x} = \frac{1}{\sum_{r=1}^R w_{r}} \sum_{r=1}^R w_{r} x_{r}$
\end{algorithmic}
\end{algorithm}

This is a very simple algorithm, and it seems intuitive that Minibatch SGD would be generally be \emph{worse} than Local SGD. After all, Minibatch SGD only takes $R$ SGD steps in total, $K$ times fewer than Local SGD takes. Furthermore, Minibatch SGD only computes stochastic gradients in $R$ locations versus Local SGD which computes stochastic gradients in $\sim MKR$ different locations so, in some sense, it seems that Local SGD should obtain more diverse and informative gradients of the objective.

\textbf{Single-Machine SGD:}
The second baseline algorithm is ``Single-Machine SGD,'' which is just $KR$ steps of SGD using minibatches of size $1$. \pref{alg:single-machine-sgd} is easily implemented in the intermittent communication setting by simply running SGD on one machine, and completely ignoring the remaining $M-1$ machines. 
\begin{algorithm}[tb]
\caption{Single-Machine SGD}\label{alg:single-machine-sgd}
\begin{algorithmic}
\STATE Initialize: $x^1_{0,0} = x_0$
\FOR{$r = 0,1,\dots,R-1$}
\FOR{$k = 0,1,\dots,K-1$}
\STATE On machine $1$: $x^1_{k+1,r} = x^1_{k,r} - \eta_{k,r} g(x^1_{k,r};z_{k,r})$
\STATE On machines $2,\dots,M$: do nothing.
\ENDFOR
\STATE $x^1_{0,r+1} = x^1_{K,r}$
\ENDFOR
\STATE Return: $\hat{x} = \frac{1}{\sum_{k=1}^K\sum_{r=1}^R w_{k,r}} \sum_{k=1}^K\sum_{r=1}^R w_{k,r} x_{k,r}$
\end{algorithmic}
\end{algorithm}

This algorithm seems a bit silly since it does not leverage the available parallelism at all, and for this reason, it seems clearly worse than Local SGD, which does in fact use all $M$ of the machines. Indeed, this intuition is correct and we will see later that Local SGD is never worse than Single-Machine SGD. Nevertheless, many previous Local SGD analyses appear worse than the Single-Machine SGD guarantee so this is an important baseline to keep in mind. 

Furthermore, Minibatch and Single-Machine SGD are useful points of reference because they constitute opposite poles of a spectrum. On the one hand, Minibatch SGD fully exploits the available parallelism but it hardly takes advantage of the availability of local computation on each of the machines between communications. On the other hand, Single-Machine SGD in some sense fully exploits the local computation but does not take advantage of the parallelism. The idea of Local SGD is to find a happy medium between these extremes---exploiting both the local computation and the parallelism at the same time in order to achieve better performance.

\textbf{Thumb-Twiddling SGD:}
The final baseline algorithm is something of a strawman, but one that is nevertheless useful as a point of comparison, and it corresponds simply to $R$ steps of SGD using minibatches of size $M$. \pref{alg:thumb-twiddling-sgd} belongs to the family of intermittent communication algorithms since it can be implemented by each machine computing just a single stochastic gradient during each round of communication and sitting and ``twiddling its thumbs'' rather than computing any of the other $K-1$ stochastic gradients.
\begin{algorithm}[tb]
\caption{Thumb-Twiddling SGD}\label{alg:thumb-twiddling-sgd}
\begin{algorithmic}
\STATE Initialize: $x_0$
\FOR{$r = 0,1,\dots,R-1$}
\FOR{$k = 0,1,\dots,K-1$}
\IF{$k = 0$}
\STATE For each $m$ in parallel: compute $g(x_r;z^m_{0,r})$
\ELSE
\STATE Twiddle thumbs.
\ENDIF
\ENDFOR
\STATE $\bar{g}_r = \frac{1}{M}\sum_{m=1}^M g(x_r;z^m_{0,r})$
\STATE $x_{r+1} = x_r - \eta_r \bar{g}_r$
\ENDFOR
\STATE Return: $\hat{x} = \frac{1}{\sum_{r=1}^R w_{r}} \sum_{r=1}^R w_{k,r} x_{r}$
\end{algorithmic}
\end{algorithm}

If Single-Machine SGD is a bit silly, this algorithm is completely ludicrous and it is obviously strictly worse than Minibatch SGD, which is the same algorithm except with bigger minibatches (and therefore lower-variance stochastic gradient steps). So, there is no reason to use Thumb-Twiddling SGD in the intermittent communication setting but it still serves as a useful---and surprisingly strong---baseline to compare against Local SGD. 

\paragraph{Accelerated Methods:}
It is worth pointing out that none of these methods are actually minimax optimal, and obtaining optimal algorithms requires acceleration. There are accelerated variants of Minibatch and Single-Machine SGD \citep{lan2012optimal,ghadimi2013optimal} that achieve faster rates of convergence than the ``unaccelerated'' methods we disscussed here, and \citet{yuan2020federated} recently proposed an accelerated variant of Local SGD and analyzed it in the homogeneous setting. However, our goal here is to understand the theoretical properties of the popular and natural algorithm Local SGD, and to compare this to other reasonable baseline algorithms. As we will show in \pref{sec:intermittent-communication-setting}, although accelerated Local SGD may be better than unaccelerated Local SGD, it cannot beat an analogous set of accelerated baselines, which are optimal.

\subsubsection{An Alternative Viewpoint: Reducing Communication}

Here, and in much of this thesis, we consider the intermittent graph to be fixed---that is, we think of $M$, $K$, and $R$ as parameters that are beyond our control, and we ask how well we can do with whatever we are given. However, this viewpoint is dual to another, which asks how we should set these parameters $M$, $K$, and $R$ in order to achieve a given level of accuracy. We are particularly interested in how small we can set $R$---i.e.~how little we can communicate---without ``paying for it'' by suffering worse error, since communication is typically expensive.

In this view, we can consider as a baseline the class of algorithms that process $T$ stochastic gradients on each machine, but can communicate at every step, i.e.~$T$ times (this corresponds to the layer graph \pref{fig:layer-graph}). In this setting, we can implement $T$ steps of SGD using minibatches of $M$ samples, and we know that this algorithm is essentially optimal up to acceleration (see \pref{subsec:layer-graph-minimax-complexity}). 

It then makes sense to ask whether we can achieve the same performance as this baseline while communicating less frequently?  Communicating only $R$ times instead of $T$ while maintaining the total number of gradients computed per machine brings us right back to the intermittent communication setting we are considering, and all the methods discussed above (Minibatch SGD, Local SGD, etc) reduce the communication.  Checking whether these algorithms' guarantees for $R<T$ rounds matches the same accuracy as the dense communication baseline is a starting point, but the question is how small can we push $R$ (while keeping $T=KR$ fixed) before accuracy degrades. Better error guarantees (in terms of $K,M$ and $R$) mean we can use a smaller $R$ with less degradation, and the smallest $R$ with no asymptotic degradation can be directly calculated from the error guarantee \citep[see, e.g., discussion in][]{cotter2011better}.

\subsection{Local SGD in the Homogeneous Setting}\label{subsec:local-sgd-homogeneous}

Let us now consider the theoretical performance of Local SGD in the homogeneous setting, that is, when each machine in each round of communication computes stochastic gradients from $\mc{O}_g^\sigma$ which give unbiased estimates of $\nabla F$. Over the years, dozens of papers have attempted to prove convergence guarantees for Local SGD but all of them fail to show improvement over the simple baselines of Minibatch and Single-Machine SGD. In fact, they often fail to show an advantage over Thumb-Twiddling SGD!

This really raises the question of whether the unfavorable comparison between Local SGD's guarantees versus the baselines' is an artifact of our analysis techniques, or if Local SGD really is worse. After surveying some of the existing work, we will now show that the answer depends on the details of the setting. First, we show that Local SGD performs very well in the special case of quadratic objectives, and dominates the baselines. In the case of general convex and smooth objectives, we show a new upper bound which is always at least as good as Single-Machine SGD and \emph{sometimes} improves over Minibatch SGD. Finally, we show a lower bound which indicates that Local SGD is indeed sometimes strictly worse than Minibatch SGD, and even than Thumb-Twiddling SGD.

\subsubsection{Prior Analysis of Local SGD in the Homogeneous Setting}\label{subsec:local-sgd-homogeneous-prior-analysis}

\begin{table}[tb]
\renewcommand{\arraystretch}{1.2}%
\centering
\begin{tabular}{ p{.51\linewidth} p{.3\linewidth}  } 
\toprule  
\textbf{Algorithm} & \textbf{Suboptimality Bound} \\
\midrule
Minibatch SGD & 
$\frac{HB^2}{R} + \frac{\sigma B}{\sqrt{MKR}}$ \\
\cmidrule{2-2}   
Single-Machine SGD & 
$\frac{HB^2}{KR} + \frac{\sigma B}{\sqrt{KR}}$ \\
\cmidrule{2-2}   
Thumb-twiddling SGD & 
$\frac{HB^2}{R} + \frac{\sigma B}{\sqrt{MR}}$ \\
\midrule 
Local SGD: \citet{stich2018local} &
$\frac{HB^2}{R^{2/3}} + \frac{HB^2}{(KR)^{3/5}} + \frac{\sigma B}{\sqrt{MKR}}$ \\
\cmidrule{2-2}  
Local SGD: \citet{stich2019error} & 
$\frac{HB^2 M}{R} + \frac{\sigma B}{\sqrt{MKR}}$ \\
\midrule
Local SGD: \pref{thm:local-sgd-homogeneous-upper-bound} & 
$\frac{HB^2}{KR} + \frac{\sigma B}{\sqrt{MKR}} + \frac{(H\sigma^2B^4)^{1/3}}{K^{1/3}R^{2/3}}$\\
\cmidrule{2-2} 
Local SGD Lower Bound: \pref{thm:local-sgd-homogeneous-lower-bound} & 
$\frac{HB^2}{KR} + \frac{\sigma B}{\sqrt{MKR}} + \frac{(H\sigma^2B^4)^{1/3}}{K^{2/3}R^{2/3}}$\\
\bottomrule
\end{tabular}
\caption[Upper bounds for homogeneous, convex, intermittent communication algorithms.]{A comparison of upper bounds on the suboptimality for the function class $\mc{F}_0(H,B)$, both for the baseline algorithms and Local SGD. Also included is our lower bound on the suboptimality of Local SGD. \label{tab:existing-local-sgd-analysis}}
\end{table}

There is a long history of analyses of Local SGD, going back almost 30 years to \citet{mangasarian1994backpropagation}, who first proposed the algorithm and showed asymptotic convergence to stationary points of arbitrary continuously differentiable objectives. Since then, numerous papers have analyzed Local SGD in a number of settings \citep{khaled2020tighter, stich2018local, stich2019error, haddadpour2019local,wang2018cooperative, dieuleveut2019communication, Zhou2018:Kaveraging, yu2019parallel, wang2017memory, haddadpour2019trading,zinkevich2010parallelized, zhang2012communication, li2014efficient, rosenblatt2016optimality, godichon2017rates, jain2017parallelizing}. The best existing guarantees for Local SGD specialized to the convex, smooth, and homogeneous setting, due to \citet{stich2018local} and \citet{stich2019error}, are included in \pref{tab:existing-local-sgd-analysis}.

It is worth emphasizing that the very best of the existing analysis for Local SGD are actually always strictly worse than Minibatch SGD. The guarantees of \citet{stich2018local} and \citet{stich2019error} do not necessarily even improve over Thumb-Twiddling SGD due to the $R^{-2/3}$ and $MR^{-1}$ dependence in their first terms, versus the $R^{-1}$ rate for Minibatch SGD and Thumb-Twiddling SGD. Finally, while they can sometimes improve over Single-Machine SGD, they can only do so when Single-Machine SGD is already worse than Minibatch SGD. 

While we focus on the function classes $\mc{F}_0(H,B)$ and $\mc{F}_\lambda(H,\Delta)$, there are many other analyses for Local SGD under different sets of assumptions, or with a more detailed dependence on the noise in the stochastic gradients. For instance, \citet{stich2019error,haddadpour2019local} analyze local SGD assuming two notions of not-quite-convexity; and \citet{wang2018cooperative,dieuleveut2019communication} derive guarantees under both multiplicative and additive bounds on the noise.   \citet{dieuleveut2019communication} analyze local SGD with the additional assumption of a bounded third derivative, but even with this assumption do not improve over Minibatch SGD. Numerous works study Local SGD in the non-convex setting \citep{Zhou2018:Kaveraging, yu2019parallel, wang2017memory, stich2019error, haddadpour2019trading}, and although their bounds would apply in our convex setting, they are understandably worse than Minibatch SGD due to the much weaker assumptions.  There is also a large body of work studying the special case $R=1$, i.e.~where the iterates are averaged just one time at the end \citep{zinkevich2010parallelized, zhang2012communication, li2014efficient, rosenblatt2016optimality, godichon2017rates, jain2017parallelizing}. However, these analyses do not easily extend to multiple rounds, and the $R=1$ constraint may harm performance  \citep{shamir2014communication}. Finally, there are many papers analyzing Local SGD in the heterogeneous setting, which we will discuss in \pref{subsec:local-sgd-heterogeneous}.

\subsubsection{Local SGD Analysis for Quadratic Homogeneous Objectives}\label{subsec:local-sgd-quadratic-upper-bound}

To better understand the theoretical performance of Local SGD, we begin by studying the special case of quadratic objectives, e.g.~least squares problems. Here, it turns out that Local SGD is always at least as good as Minibatch SGD, and can be much better. 
More generally, we show that a Local SGD analogue for a large family of serial first-order optimization algorithms enjoys an error guarantee which depends only on the product $KR$ and not on $K$ or $R$ individually. Therefore, such algorithms essentially perfectly parallelize, and a single round of communication allows for equally good performance as many would.

We consider the following family of ``linear update algorithms'':
\begin{definition}\label{def:linear-update-algorithm}
We say that a sequential, stochastic first-order optimization algorithm is a linear update algorithm if, for fixed linear functions $\mc{L}^{(t)}_1,\mc{L}^{(t)}_2$, the algorithm generates its $(t+1)\mathth$ iterate according to 
\[
x_{t+1} = \mc{L}^{(t)}_2\prn*{x_1,\dots,x_t,g\prn*{\mc{L}^{(t)}_1\prn*{x_1,\dots,x_t};z_t}}
\]
\end{definition}
This family captures many standard first-order methods including SGD, which corresponds to the linear mappings $\mc{L}^{(t)}_1\prn*{x_1,\dots,x_t} = x_t$ and $\mc{L}^{(t)}_2\prn*{x_1,\dots,x_t,g_t} = x_t - \eta_t g_t$. Another notable algorithm in this class is AC-SA \citep{lan2012optimal} (see \pref{alg:ac-sa}), an accelerated variant of SGD which also has linear updates. Some important non-examples, however, are adaptive gradient methods like AdaGrad \citep{mcmahan2010adaptive,duchi2011adaptive}---these methods use stepsizes that depend on previous gradients which leads to non-linear updates.

For a linear update algorithm $\mc{A}$, we will use Local-$\mc{A}$ to denote the Local SGD analogue with $\mc{A}$ replacing SGD. 
That is, during each round of communication, each machine independently executes $K$ iterations of $\mc{A}$ and then the $M$ resulting iterates are averaged. For quadratic objectives, we show that this approach inherits the guarantee of $\mc{A}$ with the additional benefit of variance reduction:
\begin{restatable}{theorem}{localsgdquadratics}\label{thm:linear-update-alg-quadratics}
Let $\mc{A}$ be a linear update algorithm which, when executed for $T$ iterations on any quadratic $F \in \mc{F}_0(H,B)$, guarantees $\E F(x_T) - F^* \leq \epsilon_0(T, \sigma^2)$ and for any quadratic $F \in \mc{F}_\lambda(H,\Delta)$, guarantees $\E F(x_T) - F^* \leq \epsilon_\lambda(T, \sigma^2)$. Then, Local-$\mc{A}$'s averaged final iterate $\bar{x}_{KR} = \frac{1}{M}\sum_{m=1}^M x_{KR}^m$ will satisfy $\E F(\bar{x}_{KR}) - F^* \leq \epsilon_0\prn{KR, \frac{\sigma^2}{M}}$ and $\E F(\bar{x}_{KR}) - F^* \leq \epsilon_\lambda\prn{KR, \frac{\sigma^2}{M}}$ in the convex and strongly convex cases, respectively.
\end{restatable}
We prove this in \pref{app:local-sgd-homogeneous-quadratics} by showing that the average iterate $\bar{x}_t$ is updated according to $\mc{A}$ using minibatch stochatic gradients of size $M$---even in the middle of rounds of communication when $\bar{x}_t$ is not explicitly computed. The key property that we exploit is that the gradient of a quadratic function is linear. Therefore, we write the updates on the average iterate as
\begin{equation}
\bar{x}_{t+1}
= \mc{L}_2^{(t)}\bigg(\bar{x}_1,\dots,\bar{x}_t,
\frac{1}{M}\sum_{m'=1}^M g\prn*{\mc{L}_1^{(t)}\prn*{x_1^{m'},\dots,x_t^{m'}};z_t^{m'}}\bigg)
\end{equation}
Then, by the linearity of $\nabla F$ and $\mc{L}_1^{(t)}$, we have
\begin{equation}
\E\brk*{\frac{1}{M}\sum_{m'=1}^M g\prn*{\mc{L}_1^{(t)}\prn*{x_1^{m'},\dots,x_t^{m'}};z_t^{m'}}} 
= \nabla F\prn*{\mc{L}_1^{(t)}\prn*{\bar{x}_1,\dots,\bar{x}_t}}
\end{equation}
and we show that the variance is reduced to $\frac{\sigma^2}{M}$. Therefore, $\mc{A}$'s guarantee applies but with the added benefit of smaller variance.

To rephrase \pref{thm:linear-update-alg-quadratics}, on quadratic objectives, Local-$\mc{A}$ is in some sense equivalent to $KR$ iterations of $\mc{A}$ using minibatch stochastic gradients of size $M$. Furthermore, this guarantee depends only on the product $KR$, and not on $K$ or $R$ individually. Thus, averaging the $T\mathth$ iterate of $M$ independent executions of $\mc{A}$, sometimes called ``one-shot averaging,'' enjoys the same error upper bound as $T$ iterations $\mc{A}$ using size-$M$ minibatches.

Nevertheless, it is important to highlight the boundaries of \pref{thm:linear-update-alg-quadratics}. Firstly, $\mc{A}$'s error guarantee $\epsilon(T,\sigma^2)$ must not rely on any particular structure of the stochastic gradients themselves, as this structure might not hold for the implicit updates of Local-$\mc{A}$. 
Furthermore, even if some structure of the stochastic gradients \emph{is} maintained for Local-$\mc{A}$, the particular iterates generated by Local-$\mc{A}$ will generally vary with $K$ and $R$ (even when $KR$ is held constant). Thus, Theorem \ref{thm:linear-update-alg-quadratics} does \emph{not} guarantee that Local-$\mc{A}$ with two different values of $K$ and $R$ would perform the same on any particular instance. We have merely proven matching upper bounds on their worst-case performance.

We apply \pref{thm:linear-update-alg-quadratics} to yield error upper bounds for Local SGD and Local AC-SA:
\begin{restatable}{corollary}{localacsaquadratic}\label{cor:acc-local-sgd-optimal-quadratics}
For any quadratic $F_0 \in \mc{F}_0(H,B)$ and quadratic $F_\lambda \in \mc{F}_\lambda(H,B)$, Local SGD guarantees
\begin{align*}
\E F_0(\hat{x}) - F_0^* &\leq c\cdot \prn*{\frac{HB^2}{KR} + \frac{\sigma B}{\sqrt{MKR}}} \\
\E F_\lambda(\hat{x}) - F_\lambda^* &\leq c\cdot \prn*{\Delta\exp\prn*{-\frac{c'\lambda KR}{H}} + \frac{\sigma^2}{\lambda MKR}}
\end{align*}
and Local AC-SA\footnote{The AC-SA algorithm must be slightly modified in order to achieve linear convergence in the strongly convex setting. This modification is provided by \citet{ghadimi2013optimal} and is also discussed in \pref{subsec:the-reduction}.} guarantees
\begin{align*}
\E F_0(\tilde{x}) - F_0^* &\leq c\cdot \prn*{\frac{HB^2}{K^2R^2} + \frac{\sigma B}{\sqrt{MKR}}} \\
\E F_\lambda(\tilde{x}) - F_\lambda^* &\leq c\cdot \prn*{\Delta\exp\prn*{-\frac{c'\sqrt{\lambda} KR}{\sqrt{H}}} + \frac{\sigma^2}{\lambda MKR}}
\end{align*}
\end{restatable}
This follows immediately from the guarantees of SGD and AC-SA in the sequential setting and \pref{thm:linear-update-alg-quadratics}. In the convex case, comparing \pref{cor:acc-local-sgd-optimal-quadratics}'s guarantee for Local SGD with the bound for Minibatch SGD in \pref{tab:existing-local-sgd-analysis}, we see that Local SGD's bound is strictly better, due to the first term scaling as $(KR)^{-1}$ versus $R^{-1}$. At the same time, the Local SGD guarantee is also strictly better than Single-Machine SGD due to the $(MKR)^{-1/2}$ scaling of the statistical term versus just $(KR)^{-1/2}$. We note that Minibatch and Single-Machine SGD can also be accelerated \citep{cotter2011better,lan2012optimal}, which improves their optimization terms, however these guarantees are still outmatched by Local AC-SA.

There is also reason to think that Local AC-SA might be minimax optimal for quadratic objectives. The statistical terms $\frac{\sigma B}{\sqrt{MKR}}$ and $\frac{\sigma^2}{\lambda MKR}$ are tight by \pref{lem:statistical-term-lower-bound}, which was proven using quadratic hard instances. In addition, the optimization terms match the lower bounds \pref{thm:generic-graph-lower-bound} and \pref{thm:generic-graph-oracle-strongly-convex}. These lower bounds were proven using non-quadratic hard instances, so it is possible that the minimax error for quadratic objectives might be lower. However, the generic graph lower bounds \emph{could} have been proven using quadratic hard instances if we restricted our attention to the class of span-restricted/zero-respection/deterministic optimization algorithms. Also, the work of \citet{simchowitz2018randomized} shows that the Local AC-SA guarantee is minimax optimal (up to a log factor) for quadratic objectives and the class of all randomized algorithms, but only in the special case $M=1$. It is certainly plausible, and perhaps even likely, that these results could be extended to show that Local AC-SA is optimal for quadratics for all $M$, although such lower bounds do not yet exist.

\subsubsection{Local SGD Analysis for General Homogeneous Objectives}\label{subsec:local-sgd-homogeneous-upper-bound}

As we just saw, Local SGD is extremely effective for the special case of quadratic objectives. In this section, we turn to general convex and strongly convex objectives. Our main result is the following theorem:
\begin{restatable}{theorem}{localsgdhomogeneousupperbound}\label{thm:local-sgd-homogeneous-upper-bound}
For any $F_0 \in \mc{F}_0(H,B)$ and $F_\lambda \in \mc{F}_\lambda(H,\Delta)$, Local SGD guarantees
\begin{align*}
\E F_0(\hat{x}) - F_0^* &\leq c \cdot \min\crl*{\frac{HB^2}{KR} + \frac{\sigma B}{\sqrt{MKR}} + \frac{\prn*{H\sigma^2B^4}^{1/3}}{K^{1/3}R^{2/3}},\, \frac{HB^2}{KR} + \frac{\sigma B}{\sqrt{KR}}} \\
\E F_\lambda(\hat{x}) - F_\lambda^* &\leq c \cdot \min\crl*{\frac{H\Delta}{\lambda}\exp\prn*{-\frac{\lambda KR}{4H}} + \frac{\sigma^2}{\lambda MKR} + \frac{H\sigma^2\log\prn*{9 + \frac{\lambda KR}{H}}}{\lambda^2 KR^2},\, \frac{H\Delta}{\lambda}\exp\prn*{-\frac{\lambda KR}{4H}} + \frac{\sigma^2}{\lambda KR}}
\end{align*}
\end{restatable}
This Theorem is proven in \pref{app:local-sgd-homogeneous-upper-bound}. We use a similar approach as \citet{stich2018local}, who analyzes the behavior of the averaged iterate $\bar{x}_t = \frac{1}{M}\sum_{m=1}^M x_t^m$, even when it is not explicitly computed. \citeauthor{stich2018local} shows, in particular, that the averaged iterate \emph{almost} evolves according to size-$M$-Minibatch SGD updates, up to a term proportional to the dispersion of the individual machines' iterates $\frac{1}{M}\sum_{m=1}^M \nrm{\bar{x}_t - x_t^m}^2$. \citeauthor{stich2018local} bounds this with $O(\eta_t^2 K^2 \sigma^2)$, but this bound is too pessimistic---in particular, it holds even if the gradients are replaced by arbitrary vectors of norm $\sigma$. In \pref{lem:local-sgd-homogeneous-upper-bound-distance-bound-between-local-iterates}, we improve this bound to $O(\eta_t^2 K \sigma^2)$ which allows for our improved guarantee.\footnote{In recent work, \citet{stich2019error} present a new analysis of Local SGD which, in the general convex case has the form $\frac{MHB^2}{R} + \frac{\sigma B}{\sqrt{MKR}}$. As stated, this is strictly worse than Minibatch SGD. However, we suspect that this bound should hold \emph{for any $1 \leq M' \leq M$} because, intuitively, having more machines should not hurt you. If this is true, then optimizing their bound over $M'$ yields a similar result as \pref{thm:local-sgd-homogeneous-upper-bound}.} Our approach resembles the concurrent work of \citet{khaled2020tighter}, however our analysis is more refined and we optimize more carefully over the stepsize to get a better rate.

\paragraph{Comparison with the Baselines} 
We now compare the upper bound from \pref{thm:local-sgd-homogeneous-upper-bound} with the guarantees for Minibatch and Single-Machine SGD. The second term in the $\min$s for Local SGD match the convergence rate for Single-Machine SGD, so we conclude that Local SGD is never worse. On the other hand, its relationship with Minibatch SGD is more complicated.

For clarity, and in order to highlight the role of $M$, $K$, and $R$ in the convergence rate, we will compare rates for general convex objectives when $H = B = \sigma^2 = 1$, and we will also ignore numerical constants. In this setting, the worst-case error of Minibatch SGD is \citep{nemirovskyyudin1983}:
\begin{equation}\label{eq:local-sgd-section-mbsgd-rate}
\epsilon_{\textrm{MB-SGD}} = c\cdot\prn*{\frac{1}{R} + \frac{1}{\sqrt{MKR}}}
\end{equation}
Our guarantee for Local SGD from \pref{thm:local-sgd-homogeneous-upper-bound} reduces to:
\begin{equation}
\epsilon_{\textrm{Local SGD}} \leq c\cdot\left(\frac{1}{KR} + \frac{1}{K^{\frac{1}{3}}R^{\frac{2}{3}}} + \frac{1}{\sqrt{MKR}}\right)
\end{equation}
These guarantees have matching statistical terms of $\frac{1}{\sqrt{MKR}}$, which cannot be improved by any first-order algorithm \citep{nemirovskyyudin1983}. Therefore, in the regime where the statistical term dominates both rates, i.e.~$M^3K\lesssim R$ and $MK\lesssim R$, both algorithms will have similar worst-case performance. When we leave this noise-dominated regime, we see that Local SGD's guarantee $K^{-\frac{1}{3}}R^{-\frac{2}{3}}$ is better than Minibatch SGD's $R^{-1}$ when $K \gtrsim R$ and is worse when $K \lesssim R$. 
This makes sense intuitively: Minibatch SGD takes advantage of very precise stochastic gradient estimates, but pays for it by taking fewer gradient steps; conversely, each Local SGD update is noisier, but Local SGD is able to make $K$ times more updates.

We therefore establish that for general convex objectives in the large-$M$ and large-$K$ regime, Local SGD will strictly outperform Minibatch SGD. However, in the large-$M$ and small-$K$ regime, our Local SGD guarantee does \emph{not} show improvement over Minibatch SGD. We are only comparing upper bounds, so it is not clear that Local SGD is in fact worse, yet it raises the question of whether this is the best we can hope for from Local SGD. Is Local SGD truly better than Minibatch SGD in some regimes but worse in others? Or, should we believe the intuitive argument suggesting that Local SGD is always at least as good as Minibatch SGD in the same way that it is always better than Single-Machine SGD?

\subsubsection{A Lower Bound for Local SGD for Homogeneous Objectives}\label{subsec:local-sgd-homogeneous-lower-bound}

\begin{figure*}[!tb]
\centering
\includegraphics[width=0.92\textwidth]{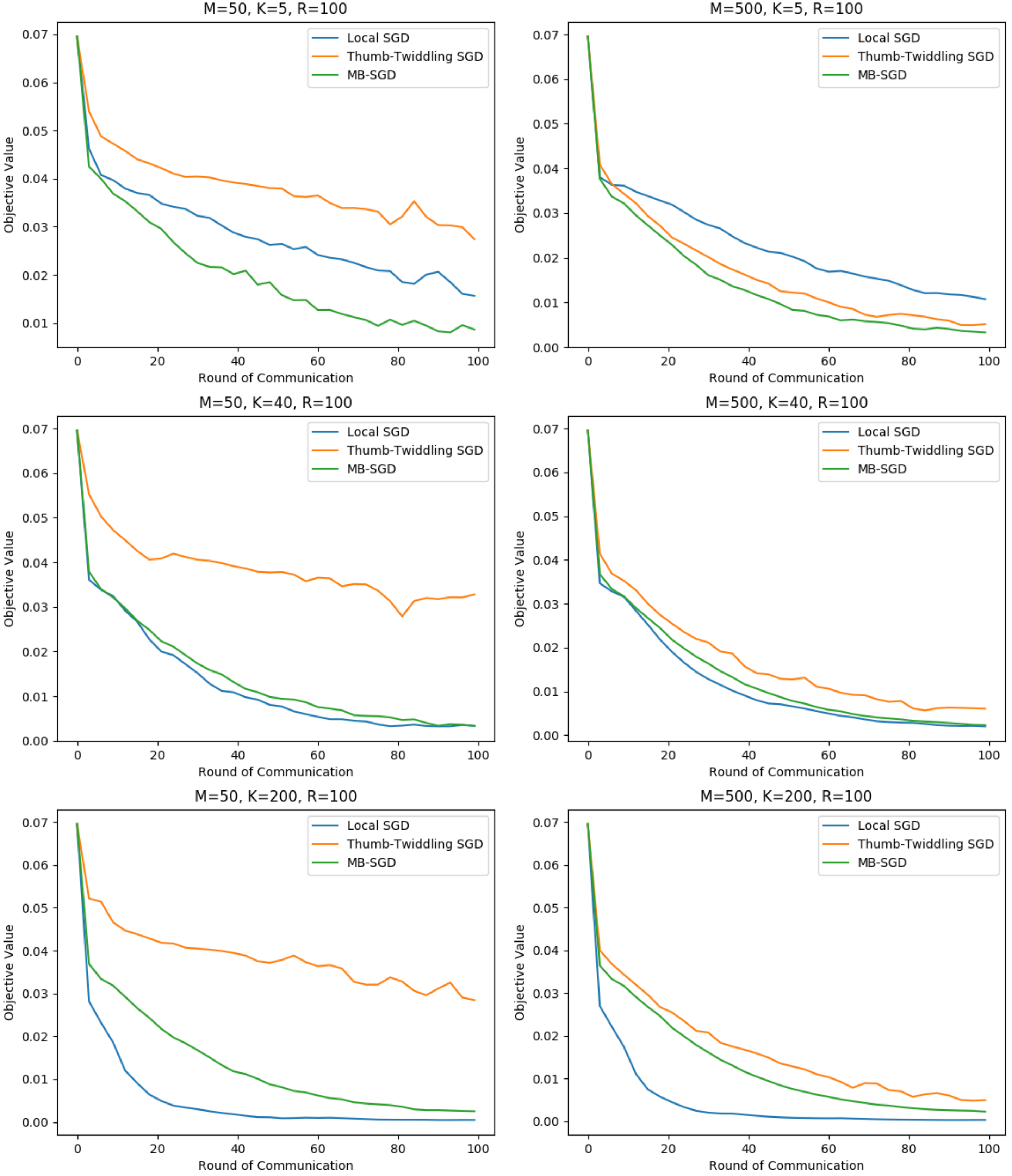}
\caption[Empirical performance of homogeneous intermittent communication algorithms applied to a logistic regression problem.]{\small We constructed a dataset of 50000 points in $\R^{25}$ with the $i$th coordinate of each point distributed independently according to a Gaussian distribution $\mc{N}(0, \frac{10}{i^2})$. The labels are generated via $\mathbb{P}\brk{y = 1\,|\, x} = \sigma(\min\crl{\inner{w_1^*}{x} + b_1^*, \inner{w_2^*}{x} + b_2^*})$ for $w_1^*, w_2^* \sim \mc{N}(0,I_{25\times 25})$ and $b_1^*, b_2^* \sim \mc{N}(0,1)$, where $\sigma(a) = 1/(1+\exp(-a))$ is the sigmoid function, i.e.~the labels correspond to an intersection of two halfspaces with label noise which increases as one approaches the decision boundary. We used each algorithm to train a linear model with a bias term to minimize the logistic loss over the 50000 points, i.e.~$f$ is the logistic loss on one sample and $\mc{D}$ is the empirical distribution over the 50000 samples. For each $M$, $K$, and algorithm, we tuned the constant stepsize to minimize the loss after $r$ rounds of communication individually for each $1 \leq r \leq R$. Let $x_{\mathsf{A},r,\eta}$ denote algorithm $\mathsf{A}$'s iterate after the $r$th round of communication when using constant stepsize $\eta$. The plotted lines are an approximation of $g_{\mathsf{A}}(r) = \min_{\eta} F(x_{\mathsf{A},r,\eta}) - F(x^*)$ for each $\mathsf{A}$ where the minimum is calculated using grid search on a log scale.} 
\label{fig:local-sgd-LR-experiments}
\end{figure*}

We will now show that in a certain regime, Local SGD really is inferior (in the worst-case) to Minibatch SGD, and even to Thumb-Twiddling SGD. We show this by constructing a simple, smooth piecewise-quadratic objective in three dimensions on which Local SGD performs poorly. We define this hard instance as $F(x) = \E_z f(x;z)$ where
\begin{equation}
f(x;z) = \frac{\lambda}{2}\prn*{x_1 - \frac{B}{\sqrt{3}}}^2 + \frac{H}{2}\prn*{x_2 - \frac{B}{\sqrt{3}}}^2 
+ \frac{H}{8}\prn*{\prn*{x_3 - \frac{B}{\sqrt{3}}}^2 + \pp{x_3 - \frac{B}{\sqrt{3}}}^2} + zx_3
\label{eq:local-sgd-homogeneous-lower-bound-construction}
\end{equation}
and where $\P\brk*{z=\sigma} = \P\brk*{z=-\sigma} = \frac{1}{2}$ and $\pp{y} := \max\crl{y,0}$.
\begin{restatable}{theorem}{localsgdhomogeneouslowerbound}\label{thm:local-sgd-homogeneous-lower-bound}
For any $K \geq 2$, $M,R \geq 1$, and dimension at least 3, there exists $F_0 \in \mc{F}_0(H,B)$ such that the final averaged iterate of Local SGD initialized at $0$ with any fixed stepsize will have suboptimality at least
\[
\E F_0(\hat{x}) - F_0^*
\geq c\cdot\prn*{\min\crl*{\frac{\prn*{H\sigma^2B^4}^{1/3}}{K^{2/3}R^{2/3}},\,HB^2} + \min\crl*{\frac{\sigma B}{\sqrt{MKR}},\,HB^2}}
\]
Similarly, for any $\lambda \leq \frac{H}{16}$, there exists $F_\lambda \in \mc{F}_\lambda(H,\Delta)$ such that the final averaged iterate of Local SGD initialized at $0$ with any fixed stepsize will have suboptimality at least
\[
\E F_\lambda(\hat{x}) - F_\lambda^*
\geq c\cdot\prn*{\min\crl*{\frac{H\sigma^2}{\lambda^2K^2R^2},\,\Delta} + \min\crl*{\frac{\sigma^2}{\lambda MKR},\Delta}}
\]
\end{restatable}

We defer a detailed proof of \pref{thm:local-sgd-homogeneous-lower-bound} to \pref{app:local-sgd-homogeneous-lower-bound}. Intuitively, it relies on the fact that for non-quadratic functions, the SGD updates are no longer linear as in \pref{subsec:local-sgd-quadratic-upper-bound}, and the Local SGD dynamics introduce an additional bias term which does not improve with $M$, and scales poorly with $K,R$. This phenomenon does not seem to be unique to our construction, and should be expected to exist for any ``sufficiently non-quadratic'' function. 

The proof shows specifically that the suboptimality is large unless $x_3 \approx \frac{B}{\sqrt{3}}$, but Local SGD introduces a bias which causes $x_3$ to drift in the negative direction by an amount proportional to the stepsize. On the other hand, optimizing the first term of the objective requires the stepsize to be relatively large. Combining these yields the first term of the lower bound. The second term is classical and holds even for first-order algorithms that compute $MKR$ stochastic gradients sequentially \citep{nemirovskyyudin1983}.

In order to compare this lower bound with \pref{thm:local-sgd-homogeneous-upper-bound} and with Minibatch SGD, we again consider the general convex setting with $H = B = \sigma^2 = 1$. Then, the upper bound reduces to $K^{-\frac{1}{3}}R^{-\frac{2}{3}} + \prn{MKR}^{-\frac{1}{2}}$. Comparing this to \pref{thm:local-sgd-homogeneous-lower-bound}, we see that our upper bound is tight up to a factor of $K^{-\frac{1}{3}}$ in the optimization term. Furthermore, comparing the lower bound to the worst-case error of Minibatch SGD \eqref{eq:local-sgd-section-mbsgd-rate}, we see that Local SGD is indeed worse than Minibatch SGD in the worst case when $K \lesssim \sqrt{R}$. Our lower bound is unable to identify the exact cross-over point, but it is some $K^* \in [\sqrt{R},R]$. For $K \leq K^*$, Minibatch SGD is better than Local SGD in the worst case, for $K \geq K^*$, Local SGD is better. Since the optimization terms of Minibatch SGD and Thumb-Twiddling SGD are identical, this further indicates that Local SGD can even be outperformed by Thumb-Twiddling SGD in the small $K$ and large $M$ regime.

For \pref{thm:local-sgd-homogeneous-lower-bound}, we constructed a hard instance that is convenient for analysis but is maybe somewhat ``artificial.'' We also conducted an experiment showing that the same qualitative picture holds also for a more ``natural'' logistic regression task. In \pref{fig:local-sgd-LR-experiments}, we plot the suboptimality of Local, Minibatch, and Thumb-Twiddling SGD iterates with optimally tuned stepsizes and, as is predicted by \pref{thm:local-sgd-homogeneous-lower-bound}, we see Local SGD performing worse than Minibatch in the small $K=5$ regime, but improving relative to the other algorithms as $K$ increases to $40$ and then $200$, when Local SGD is far superior to Minibatch. For each fixed $K$, increasing $M$ causes Thumb-Twiddling SGD to improve relative to Minibatch SGD, but does not have a significant effect on Local SGD, which is consistent with the method introducing a bias that depends on $K$ but not on $M$. This highlights that the ``problematic regime'' for Local SGD is one where there are few iterations per round.

\subsection{Local SGD in the Heterogeneous Setting}\label{subsec:local-sgd-heterogeneous}

We now move on to optimization in the intermittent communication setting with heterogeneous data, where the objective has the form $F(x) = \frac{1}{M}\sum_{m=1}^M F_m(x)$ and each of the $M$ machines has access to a stochastic gradient oracle for its corresponding objective $F_m$. Here, we focus on the problem of finding a single consensus solution for $F$, which achieves a low value on all of the local objectives $F_1,\dots,F_M$ on average \citep{bertsekas1989parallel,boyd2011distributed}. Because each machine only has access to a single component of the objective, the heterogeneous is substantially harder than the previously considered homogeneous case. 

Like in the homogeneous setting, a number of recent papers have analyzed the convergence properties of Local SGD in the heterogeneous data setting \citep{wang2018cooperative,karimireddy2019scaffold,khaled2020tighter,koloskova2020unified}. Also as in the homogeneous setting, we will show that none of these Local SGD guarantees show improvement over the baseline of Minibatch SGD, even without acceleration, and in many regimes their guarantees are much worse. This again raises the question of whether this a weakness of the analysis or the Local SGD method itself. Can the bounds be improved to show that Local SGD is actually better than Minibatch SGD in certain regimes, or is Local SGD always worse?

Recall that in the homogeneous setting, prior analysis had not been able to show that Local SGD improves over Minibatch SGD yet the combination of \pref{thm:local-sgd-homogeneous-upper-bound} and \pref{thm:local-sgd-homogeneous-lower-bound} showed that Local SGD is better than Minibatch SGD in some regimes and worse in others. Specifically, when communication was relatively infrequent Local SGD improves over Minibatch SGD. 

How does this situation play out in the more challenging, and perhaps more interesting, heterogeneous setting? In several discussions following the publication of \citet{woodworth2020local}, people had suggested that the more difficult heterogeneous setting is where we should expect Local SGD to really shine, and that Minibatch SGD is too na\"ive. 
So, how does heterogeneity affect Local SGD, Minibatch SGD, and the comparison between them? Is Local SGD better than Minibatch SGD when communication is rare as in the homogeneous case?  Does the difficulty introduced by heterogeneity perhaps necessitate the more sophisticated Local SGD approach, as some have suggested?

In recent work, \citet{karimireddy2019scaffold} showed heterogeneity can be problematic for Local SGD, proving a lower bound that indicates some degradation as degree of heterogeneity increases. As we will discuss, this lower bound implies that Local SGD is strictly worse than Minibatch SGD when the level of heterogeneity is very large, but it is not clear whether or not Local SGD can improve over Minibatch SGD for slightly or moderately heterogeneous objectives.

\paragraph{Boundedly Heterogeneous Objectives}
In addition to the usual assumptions of smoothness and convexity/strong convexity, we will introduce a new parameter that captures the extent to which the local objectives disagree about the minimizer. In particular, we will say that a heterogeneous objective $F(x) = \frac{1}{M}\sum_{m=1}^M F_m(x)$ is $\sdiff^2$-heterogeneous if, for some $x^* \in \argmin_x F(x)$
\begin{equation}\label{eq:def-sdiff}
\frac{1}{M}\sum_{m=1}^M \nrm*{\nabla F_m(x^*)}^2 \leq \sdiff^2
\end{equation}
Since $\nabla F(x^*) = 0$, this can be thought of as measuring the variance of the gradient at $x^*$ when selecting a random component. If the objective is $0$-heterogeneous, then all of the objectives share a minimizer, but we note that it does \emph{not} imply that all the local objectives are the same, since \eqref{eq:def-sdiff} is only a statement about the gradients at $x^*$. This assumption of $\sdiff^2$-heterogeneity is common in the Local SGD literature \citep{karimireddy2019scaffold,khaled2020tighter,koloskova2020unified} and it is used to prove most of the convergence guarantees for Local SGD that we are aware of. Nevertheless, in the consensus optimization literature, it is common to analyze distributed algorithms with no such bound on the heterogeneity, and it is certainly still possible to minimize $F$ without this assumption \citep{boyd2011distributed,nedic2009distributed,nedic2010constrained,ram2010distributed}.

We define the function class $\mc{F}_0(H,B,M,\sdiff^2)$ as the class of $\sdiff^2$-heterogeneous objectives of the form $F(x) = \frac{1}{M}\sum_{m=1}^M F_m(x)$ where $F_m$ is $H$-smooth for all $m$, and $F \in \mc{F}_0(H,B)$. Similarly, we define $\mc{F}_\lambda(H,\Delta,M,\sdiff^2)$ as the class of $\sdiff^2$-heterogeneous objectives of the form $F(x) = \frac{1}{M}\sum_{m=1}^M F_m(x)$ where $F_m$ is $H$-smooth for all $m$, and $F \in \mc{F}_\lambda(H,\Delta)$.

\subsubsection{Minibatch SGD in the Heterogeneous Setting}\label{subsec:minibatch-sgd-in-heterogeneous-local-sgd}

To begin, we analyze the worst-case error of our baseline Minibatch SGD and Minibatch AC-SA \citep{lan2012optimal,ghadimi2013optimal} (see \pref{alg:ac-sa}) in the heterogeneous setting. A simple but important observation is that the minibatch gradients used by these algorithms are unbiased estimates of $\nabla F$ despite the heterogeneity of the objective:
\begin{equation}
\E\brk*{\frac{1}{MK}\sum_{m=1}^{M}\sum_{k=1}^{K} g(x_{r}; z_{r,k}^m)} = \frac{1}{M}\sum_{m=1}^{M} \nabla F_m(x_{r}) = \nabla F(x_r)
\end{equation}
Since the minibatch stochastic gradients are unbiased estimates of $\nabla F$, we can simply appeal to the standard analysis for (accelerated) SGD. To do so, we calculate the variance of these estimates:
\begin{align}
\E\nrm*{\frac{1}{MK}\sum_{m=1}^{M}\sum_{k=1}^{K} g(x_{r}; z_{r,k}^m) - \nabla F(x_r)}^2 
&= \frac{1}{M^2K^2}\sum_{m=1}^{M}\sum_{k=1}^{K}\E\nrm*{g(x_{r}; z_{r,k}^m) - \nabla F_m(x_r)}^2 
\leq \frac{\sigma^2}{MK}
\end{align}
So, the variance is always reduced by $MK$ and, importantly, it is not affected by the level of heterogeneity $\sdiff$.  Plugging this calculation into the analysis of SGD and Accelerated SGD \citep{nemirovskyyudin1983,lan2012optimal} yields:
\begin{restatable}{theorem}{mbsgdheterogeneousupperbound}\label{thm:mbsgd-heterogeneous-upper-bound}
For any $F_0 \in \mc{F}_0(H,B,M,\sdiff^2)$ and any $F_\lambda \in \mc{F}_\lambda(H,\Delta,M,\sdiff^2)$, the output of Minibatch SGD guarantees 
\begin{align*}
\E F_0(\hat{x}) - F_0^* &\leq c\cdot\prn*{\frac{HB^2}{R} + \frac{\sigma B}{\sqrt{MKR}}}  \\
\E F_\lambda(\hat{x}) - F_\lambda^* &\leq c\cdot\prn*{\frac{H\Delta}{\lambda}\exp\prn*{-\frac{\lambda R}{2H}} + \frac{\sigma^2}{\lambda MKR}}
\end{align*}
and Minibatch AC-SA guarantees
\begin{align*}
\E F_0(\hat{x}) - F_0^* &\leq c\cdot\prn*{\frac{HB^2}{R^2} + \frac{\sigma B}{\sqrt{MKR}}} \\
\E F_\lambda(\hat{x}) - F_\lambda^* &\leq c\cdot\prn*{\Delta\exp\prn*{-\frac{\sqrt{\lambda}R}{c_3\sqrt{H}}} + \frac{\sigma^2}{\lambda MKR}}
\end{align*}
\end{restatable}
The theorem follows immediately from the observation about the unbiaseness and variance of the stochastic gradients above, and previously established convergence rates for SGD \citep{nemirovskyyudin1983,stich2019unified} and AC-SA \citep{lan2012optimal,ghadimi2013optimal}. The most important feature of these guarantees is that they are completely independent of $\sdiff^2$ because of the use of minibatch stochastic gradients. In the following sections, we will see how Local SGD compares.

\subsubsection{Prior Analysis of Local SGD in the Heterogeneous Setting}

\begin{table}
\renewcommand{\arraystretch}{1.2}
\centering
\begin{tabular}{ l l }
\toprule  
\textbf{Algorithm} & \textbf{Suboptimality Bound} \\
\midrule
Minibatch SGD: \pref{thm:mbsgd-heterogeneous-upper-bound}
& $\frac{HB^2}{R} + \frac{\sigma B}{\sqrt{MKR}}$ \\
\midrule
Local SGD: \citet{koloskova2020unified}
& $\frac{HB^2}{R} + \frac{\sigma B}{\sqrt{MKR}} + \frac{(H\sdiff^2B^4)^{1/3}}{R^{2/3}} + \frac{(H\sigma^2B^4)^{1/3}}{K^{1/3}R^{2/3}} $ \\
\cmidrule{2-2}
Local SGD: \citet{khaled2020tighter}
& $\frac{HB^2}{R} + \frac{B\sqrt{\sigma^2 + \sdiff^2} }{\sqrt{MKR}} + \frac{(H(\sigma^2 + \sdiff^2)B^4)^{1/3}}{R^{2/3}} $ \\
\cmidrule{2-2}
\textsc{SCAFFOLD}: \citet{karimireddy2019scaffold} 
& $\frac{HB^2}{R} + \frac{\sigma B}{\sqrt{MKR}} + \frac{\sigma^2}{HKR}$ \\
\midrule
Local SGD: \pref{thm:local-sgd-heterogeneous-uppper-bound}
& $\frac{HB^2}{KR} + \frac{\sigma B}{\sqrt{MKR}} + \frac{(H\bar{\zeta}^2B^4)^{1/3}}{R^{2/3}} + \frac{(H\sigma^2B^4)^{1/3}}{K^{1/3}R^{2/3}}$ \\
\cmidrule{2-2}
Local SGD Lower Bound: \pref{thm:local-sgd-heterogeneous-lower-bound}
& $\min\crl*{\frac{HB^2}{R},\, \frac{(H\sdiff^2 B^4)^{1/3}}{R^{2/3}}} + \frac{\sigma B}{\sqrt{MKR}} + \frac{(H\sigma^2B^4)^{1/3}}{K^{2/3}R^{2/3}}$ \\
\bottomrule
\end{tabular}
\caption[Upper bounds for heterogeneous, convex, intermittent communication algorithms.]{Guarantees for objectives in $\mc{F}_0(H,B,M,\sdiff^2)$. See the discussion around \pref{thm:local-sgd-heterogeneous-uppper-bound} for the definition of $\bar{\zeta}$.
\label{tab:prior-local-sgd-analysis-heterogeneous-convex}}
\end{table}

\begin{table}
\renewcommand{\arraystretch}{1.2}
\centering
\begin{tabular}{ l l }
\toprule
\textbf{Algorithm} & \textbf{Suboptimality Bound}  \\ 
\midrule
Minibatch SGD: \pref{thm:mbsgd-heterogeneous-upper-bound}
& $\frac{H\Delta}{\lambda}\exp\prn*{\frac{-\lambda R}{H}} + \frac{\sigma_*^2}{\lambda MKR}$
\\\midrule
Local SGD: \citet{koloskova2020unified}
& $\Delta\exp\prn*{\frac{-\lambda R}{H}} + \frac{\sigma_*^2}{\lambda MKR} + \frac{H\sdiff^2}{\lambda^2 R^2} + \frac{H\sigma_*^2}{\lambda^2 KR^2}$
\\\cmidrule{2-2}
\textsc{SCAFFOLD}: \citet{karimireddy2019scaffold}
& $\prn*{\Delta + \frac{\lambda\sigma^2}{H^2K}} \exp\prn*{\frac{-\lambda R}{H}} + \frac{\sigma^2}{\lambda MKR}$
\\\midrule
Local SGD: \pref{thm:local-sgd-heterogeneous-uppper-bound}
& $\frac{H\Delta}{\lambda}\exp\prn*{-\frac{c'\lambda KR}{H}} + \frac{\sigma^2}{\lambda MKR} + \frac{H\bar{\zeta}^2}{\lambda^2 R^2} + \frac{H\sigma^2}{\lambda^2 KR^2}$
\\\cmidrule{2-2}
Local SGD Lower Bound: \pref{thm:local-sgd-heterogeneous-lower-bound}
& $\min\crl*{\Delta\exp\prn*{\frac{-\lambda R}{H}},\, \frac{H\sdiff^2}{\lambda^2R^2}} + \frac{\sigma^2}{\lambda MKR} + \min\crl*{\Delta,\,\frac{H\sigma^2}{\lambda^2 K^2R^2}} $
\\
\bottomrule
\end{tabular}
\caption[Upper bounds for heterogeneous, strongly convex, intermittent communication algorithms.]{Guarantees for objectives in $\mc{F}_\lambda(H,\Delta,M,\sdiff^2)$, with some log factors omitted. See the discussion around \pref{thm:local-sgd-heterogeneous-uppper-bound} for the definition of $\bar{\zeta}$. \label{tab:prior-local-sgd-analysis-heterogeneous-strongly-convex}}
\end{table}

Recently, \citet{khaled2020tighter} and \citet{koloskova2020unified} analyzed Local SGD in the heterogeneous and convex setting, and \citeauthor{koloskova2020unified} did also in the strongly convex setting. Their guarantees are summarized in \pref{tab:prior-local-sgd-analysis-heterogeneous-convex} and \pref{tab:prior-local-sgd-analysis-heterogeneous-strongly-convex}.
Also included are guarantees for \textsc{SCAFFOLD}\footnote{\citeauthor{karimireddy2019scaffold} analyze \textsc{SCAFFOLD} in the Federated Learning setting where only a random subset of $S \leq M$ of the machines are available in each round. Here, we present the analysis as it applies to our setting where $S=M$.}, a related method for heterogeneous distributed optimization \citep{karimireddy2019scaffold}.

Upon inspection, \citeauthor{koloskova2020unified}'s guarantee is slightly better than \citeauthor{khaled2020tighter}, but even this guarantee is (up to logarithmic terms) the sum of the Minibatch SGD bound plus additional terms, and is thus always worse in every regime. The question is whether this just reflects a weakness of their analysis, or a true weakness of the Local SGD algorithm? 

Indeed, \pref{thm:local-sgd-homogeneous-upper-bound} shows that a tighter upper bound for Local SGD is possible in the homogeneous case, and this rate is better than \citet{koloskova2020unified}'s when $\sdiff^2=0$\footnote{Although, we remind the reader that $\sdiff^2= 0$ does not imply that the problem is homogeneous.}. Can we generalize \pref{thm:local-sgd-homogeneous-upper-bound} to the heterogeneous case and show improvement over Minibatch SGD?

Optimistically, we might hope that \citeauthor{koloskova2020unified}'s $(\sdiff / R)^{2/3}$ term, in particular, could be improved.
Unfortunately, it is already known that some dependence on $\sdiff$ is necessary, as \citet{karimireddy2019scaffold} shows a lower bound of $\sdiff^2 / (\lambda R^2)$ in the strongly convex case, which suggests a lower bound of $\sdiff B / R$ in the convex case. But perhaps the  \citeauthor{koloskova2020unified} analysis can be improved to match this bound? 

If the $\sdiff B / R$ dependence suggested by \citeauthor{karimireddy2019scaffold}'s lower bound were possible, it would be lower-order than $HB^2/R$ for $\sdiff < HB$, and we would see no slow down until the level of heterogeneity is fairly large. In particular, since the components $F_m$ are $H$-smooth, adding the assumption that the local objectives $F_m$ each have a minimizer of norm at most $O(B)$ would be enough to bound $\sdiff \leq O(HB)$. In this case, the dependence on the level of heterogeneity would be fairly mild, and could be ignored under reasonable circumstances. On the other hand, if the $(\sdiff / R)^{2/3}$ term from \citeauthor{koloskova2020unified} cannot be improved, then we see a slowdown as soon as $\sdiff = \Omega(HB/R)$, which corresponds to a quite low level of heterogeneity! So, what is the correct rate?

\subsubsection{Upper and Lower Bounds for Local SGD in the Heterogeneous Setting}

We now show that the poor dependence on $\sdiff$ from \citeauthor{koloskova2020unified}'s analysis cannot be improved. Consequently, for sufficiently heterogeneous data, Local SGD is strictly worse than Minibatch SGD, regardless of the frequency of communication, unless the level of heterogeneity is very small.
\begin{restatable}{theorem}{localSGDheterogeneouslowerbound}\label{thm:local-sgd-heterogeneous-lower-bound}
For $K \geq 2$ and any dimension at least 4, there exists an objective $F_0 \in \mc{F}_0(H,B,M,\sdiff^2)$ and for any $\lambda \leq \frac{H}{16}$, there exists $F_\lambda \in \mc{F}_\lambda(H,\Delta,M,\sdiff^2)$ such that the final averaged iterate of Local SGD initialized at zero and using any fixed stepsize $\eta$ will have suboptimality at least
\begin{align*}
\E F_0(\hat{x}) - F_0^* &\geq c\cdot\prn*{\min\crl*{\frac{HB^2}{R},\, \frac{\prn*{H\sdiff^2 B^4}^{1/3}}{R^{2/3}}} 
+ \frac{\prn*{H\sigma^2B^4}^{1/3}}{K^{2/3}R^{2/3}} + \frac{\sigma B}{\sqrt{MKR}}} \\
\E F_\lambda(\hat{x}) - F_\lambda^* &\geq c\cdot\prn*{\min\crl*{\Delta\exp\prn*{-\frac{c'\lambda R}{H}},\, \frac{H\sdiff^2}{\lambda^2R^2}} + \min\crl*{\Delta,\,\frac{H\sigma^2}{\lambda^2 K^2R^2}} + \frac{\sigma^2}{\lambda MKR}}
\end{align*}
\end{restatable}

This is proven in \pref{app:local-sgd-heterogeneous-lower-bound} using a similar approach as \pref{thm:local-sgd-homogeneous-lower-bound}, and it is conceptually similar to the lower bounds for heterogeneous objectives of \citet{karimireddy2019scaffold}.  \citet{koloskova2020unified} also prove a lower bound, but specifically for 1-strongly convex objectives, which obscures the important role of the strong convexity parameter. 

In the convex case, this lower bound closely resembles the upper bound of \citet{koloskova2020unified}. Focusing on the case $H = B = \sigma^2 = 1$ to emphasize the role of $\sdiff^2$, the only gaps are between \textbf{(i)} a term $1/(K^{1/3}R^{2/3})$ vs $1/(K^{2/3}R^{2/3})$---a gap which also exists in the homogeneous case (see \pref{thm:local-sgd-homogeneous-upper-bound} and \pref{thm:local-sgd-homogeneous-lower-bound})---and \textbf{(ii)} another term $1/R + (\sdiff / R)^{2/3}$ vs $\min\crl{1/R,\ (\sdiff / R)^{2/3}}$. 

For $\sdiff^2 \geq 1/R \Rightarrow (\sdiff / R)^{2/3} \geq 1/R$, the lower bound shows that Local SGD has error at least $1/R + 1/\sqrt{MKR}$ and thus performs strictly worse than Minibatch, regardless of $K$. This is quite surprising---Local SGD is often suggested as an improvement over Minibatch SGD for the heterogeneous setting, yet we see that even a small degree of heterogeneity can make it much worse. Furthermore, increasing the duration of each round, $K$, is often thought of as more beneficial for Local SGD than Minibatch SGD, but the lower bound indicates it does little to help Local SGD in the heterogeneous setting. 

Similarly, in the strongly convex case, the lower bound from \pref{thm:local-sgd-heterogeneous-lower-bound} nearly matches the upper bound of \citeauthor{koloskova2020unified}.
Focusing on the case $H=B=\sigma=1$ in order to emphasize the role of $\sdiff$, the only differences are between \textbf{(i)} a term $1/(KR^2)$ versus $1/(K^2R^2)$---a gap which also exists in the homogeneous case (see \pref{thm:local-sgd-homogeneous-upper-bound} and \pref{thm:local-sgd-homogeneous-lower-bound})---and \textbf{(ii)} between $\exp(- \lambda R) + \sdiff^2/(\lambda^2R^2)$ and $\min\crl*{\exp(- \lambda R),\,\sdiff^2/(\lambda^2R^2)}$. The latter gap is more substantial than the convex case, but nevertheless indicates that the $\sdiff^2 / (\lambda^2R^2)$ rate cannot be improved until the number of rounds of communication is at least the condition number or $\sdiff^2$ is very small.

Thus, \pref{thm:local-sgd-heterogeneous-lower-bound} indicates that it is not possible to radically improve over the \citeauthor{koloskova2020unified} analysis, and thus over Minibatch SGD for even moderate heterogeneity, without stronger assumptions.
In order to obtain an improvement over Minibatch SGD in a heterogeneous setting, at least with very low heterogeneity, we introduce a stronger version of the heterogeneity measure $\sdiff$ which bounds the difference between the local objectives' gradients everywhere, not just at $x^*$:
\begin{equation}\label{eq:def-zetabar}
\sup_x \frac{1}{M}\sum_{m=1}^M\nrm*{\nabla F_m(x) - \nabla F(x)}^2 \leq \bar{\zeta}^2
\end{equation}
This quantity precisely captures homogeneity since $\bar{\zeta}^2 = 0$ if and only if $F_m = F$ (up to an irrelevant additive constant). In terms of $\bar{\zeta}^2$, we are able to analyze Local SGD and see a smooth transition from the heterogeneous ($\bar{\zeta}^2$ large) to homogeneous ($\bar{\zeta}^2=0$) setting. 
\begin{restatable}{theorem}{localsgdheterogeneousupperbound}{\label{thm:local-sgd-heterogeneous-uppper-bound}}
For any $F_0 \in \mc{F}_0(H,B,M,\sdiff^2)$ and any $F_\lambda \in \mc{F}_\lambda(H,\Delta,M,\sdiff^2)$ with the additional property that $\sup_x \frac{1}{M}\sum_{m=1}^M\nrm*{\nabla F_m(x) - \nabla F(x)}^2 \leq \bar{\zeta}^2$, Local SGD guarantees
\begin{align*}
\E F_0(\hat{x}) - F_0^* &\leq 
c\cdot\prn*{\frac{HB^2}{KR} + \frac{\prn*{H\bar{\zeta}^2B^4}^{1/3}}{R^{2/3}} + \frac{\prn*{H\sigma^2B^4}^{1/3}}{K^{1/3}R^{2/3}} + \frac{\sigma B}{\sqrt{MKR}}}, \\
\E F_\lambda(\hat{x}) - F_\lambda^* &\leq 
c\cdot\prn*{\frac{H\Delta}{\lambda}\exp\prn*{-\frac{c'\lambda KR}{H}} + \prn*{\frac{H\bar{\zeta}^2}{\lambda^2 R^2} + \frac{H\sigma^2}{\lambda^2 KR^2}} \log\prn*{e + \frac{\lambda KR}{H}} + \frac{\sigma^2}{\lambda MKR}}.
\end{align*}
\end{restatable}
We prove this in \pref{app:local-sgd-heterogeneous-upper-bound}.  This is the first analysis of Local SGD, or any other method for heterogeneous distributed optimization, which shows any improvement over Minibatch SGD in any heterogeneous regime. When $\bar{\zeta}=0$, \pref{thm:local-sgd-heterogeneous-uppper-bound} reduces to the homogeneous analysis of Local SGD given by \pref{thm:local-sgd-homogeneous-upper-bound}, which already showed that in that case, we see improvement when $K \gtrsim R$. \pref{thm:local-sgd-heterogeneous-uppper-bound} degrades smoothly when $\bar{\zeta}$ increases, and shows improvement for Local SGD over Minibatch SGD also when $\bar{\zeta}^2 \lesssim 1/R$ in the convex case, i.e.~with low, yet positive, heterogeneity. It is yet unclear whether this rate of convergence can be ensured in terms of $\sdiff^2$ rather than $\bar{\zeta}^2$.

\paragraph{Experimental evidence}
Finally, while \pref{thm:local-sgd-heterogeneous-lower-bound} proves that Local SGD is worse than Minibatch SGD unless $\sdiff$ is very small \emph{in the worst case}, one might hope that for ``normal'' heterogeneous problems, Local SGD might perform better than its worst case error suggests. However, a simple binary logistic regression experiment on MNIST indicates that this behavior likely extends significantly beyond the worst case. The results, depicted in \pref{fig:heterogeneous-local-sgd-experiments}, show that Local SGD performs worse than Minibatch SGD unless both $\sdiff$ is very small and $K$ is large. Finally, we also observe that Minibatch SGD's performance is essentially unaffected by $\sdiff$ empirically as predicted by theory.

\begin{figure}
\centering
\includegraphics[width=0.8\textwidth]{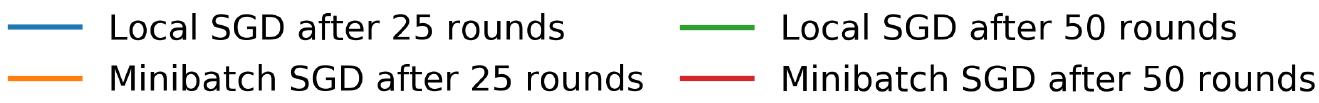}
\includegraphics[width=\textwidth]{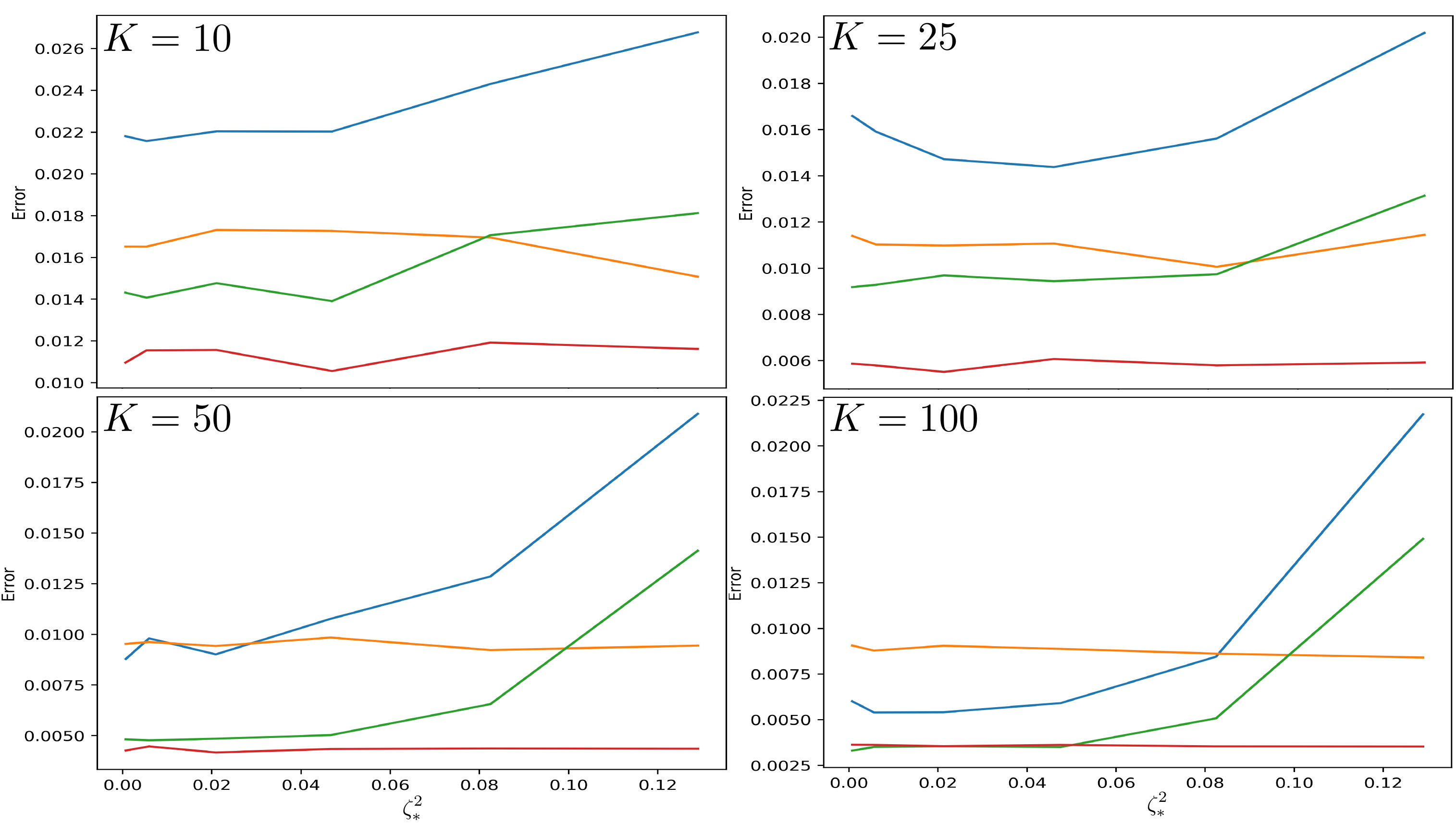}
\small \caption[Empirical performance of heterogeneous intermittent communication algorithms applied to a logistic regression problem.]{Binary logistic regression between even vs odd digits of MNIST. Twenty-five ``tasks'' were constructed, one for each combination of $i$ vs $j$ for even $i$ and odd $j$. For $p \in \{0,20,40,60,80,100\}$, we assigned to each of $M=25$ machines $p\%$ data from task $m$, and $(100-p)\%$ data from a mixture of all tasks. For several choices of $R$ and $K$, we plot the error (averaged over four runs) versus the value of $\sdiff^2$ resulting from each choice of $p$. For both algorithms, we used the best fixed stepsize for each choice of $K$, $R$, and $\sdiff$ individually. Complete details are provided in \pref{app:heterogeneous-experiment-details}.
\label{fig:heterogeneous-local-sgd-experiments}}
\end{figure}

\subsection{Conclusion}\label{subsec:local-sgd-conclusion}\label{subsec:inner-outer}

In the homogeneous setting we showed that: (1) Local SGD attains very low error for quadratic objectives, and strictly dominates the baselines of Single-Machine SGD and Minibatch SGD; (2) for general objectives, Local SGD is always at least as good as Single-Machine SGD, and it is strictly better than Minibatch SGD when communication is relatively infrequent ($K \gtrsim R$); and (3) Local SGD is strictly worse than Minibatch SGD when communication is relatively frequent ($K \lesssim \sqrt{R}$), and it can even be worse than Thumb-Twiddling SGD. 

In the heterogeneous setting, Local SGD compares relatively less favorably against Minibatch SGD. For $\sdiff^2$-heterogeneous functions, no existing analysis shows any improvement over Minibatch SGD in any regime, and our lower bound shows that no such improvement is possible as soon as $\sdiff^2 \gtrsim 1/R$. On the other hand, when the heterogeneity is bounded everywhere by $\bar{\zeta}^2$, then we show that Local SGD can improve over Minibatch SGD, at least when $\bar{\zeta}^2 \lesssim 1/R$.

To better understand the relationship between Minibatch SGD and Local SGD, and for thinking about how to improve over them, it is useful to consider a unified algorithm that interpolates between them.  
This involves taking SGD steps locally with one stepsize, and then when the machines communicate, they take a step in the resulting direction with a second, different stepsize.  Such a dual-stepsize approach was already presented and analyzed as \textsc{FedAvg} by \citet{karimireddy2019scaffold}.
We will refer to these two stepsizes as ``inner'' and ``outer'' stepsizes, respectively, and consider
\begin{equation}\label{eq:inner-outer-updates}
\begin{aligned}
x_{r,k}^m &= x_{r,k-1}^m - \eta_{\textrm{inner}}g(x_{r,k-1}^m; z_{r,k-1}^m) &\quad\forall_{m\in [M], k \in [K]}\\
x_{r+1,0}^m &= x_{r,0}^m - \eta_{\textrm{outer}}\frac{1}{M}\sum_{n=1}^M \sum_{k=1}^{K}g(x_{r,k-1}^m; z_{r,k-1}^m)  &\quad\forall_{m\in [M], r \in [R]}
\end{aligned}
\end{equation}
Choosing $\eta_{\textrm{inner}} = 0$, this is equivalent to Minibatch SGD with stepsize $\eta_{\textrm{outer}}$, and choosing $\eta_{\textrm{inner}} = \eta_{\textrm{outer}}$ recovers Local SGD. Therefore, when the stepsizes are chosen optimally, this algorithm is always at least as good as both Minibatch and Local SGD. Therefore, using the inner-outer algorithm \eqref{eq:inner-outer-updates} with optimal stepsizes guarantees for $F_0 \in \mc{F}_0(H,B,M,\sdiff^2)$
\begin{equation}\label{eq:inner-outer-bound}
\E F_0(\hat{x}) - F_0^* \leq c\cdot\min \bigg\{ \frac{HB^2}{R} + \frac{\sigma B}{\sqrt{MKR}},  
\frac{HB^2}{KR} + \frac{\prn*{H\bar{\zeta}^2B^4}^{1/3}}{R^{2/3}} + \frac{\prn*{H\sigma^2B^4}^{1/3}}{K^{1/3}R^{2/3}} + \frac{\sigma_*B}{\sqrt{MKR}}  \; \bigg\}
\end{equation}
where the first option in the $\min$ is obtained by choosing $\eta_{\textrm{inner}}=0$ and the second by choosing $\eta_{\textrm{inner}}=\eta_{\textrm{outer}}$. We can also get a similar minimum of the Minibatch SGD and Local SGD rates in the strongly convex case and also in the homogeneous setting.

\section{The Intermittent Communication Setting}\label{sec:intermittent-communication-setting}

In the intermittent communication setting, $M$ parallel workers are used to optimize a single objective over the course of $R$ rounds. During each round, each machine sequentially and locally computes $K$ independent unbiased stochastic gradients of the global objective, and then all the machines communicate with each other. This captures, for example, the natural setting where multiple parallel workers  are available, and computation on each worker is much faster than communication between workers. It includes applications ranging from optimization using multiple cores or GPUs, to using a cluster of servers, to Federated Learning where workers are edge devices. 

As a concrete example, \citet{goyal2017accurate} were able to train a large ResNet-50 neural network on Imagenet in under an hour using SGD with very large minibatches, which was implemented in exactly the intermittent communication setting. Specifically, they used $M=256$ GPUs in parallel, computed $K$ stochastic gradients per communication between the GPUs, and performed roughly $R=10000$ updates/communications. This is, of course, just one example, but it is indicative of many practical uses of distributed optimization for training large machine learning models.

The theoretical properties of optimization algorithms in the intermittent communication setting have been widely studied for over a decade, with many approaches proposed and analyzed \citep{zinkevich2010parallelized,cotter2011better,dekel2012optimal,zhang2013divide,zhang2013communication,shamir2014distributed}, and obtaining new methods and improved analysis is still a very active area of research \citep{wang2017memory,stich2018local,wang2018cooperative,khaled2019better,haddadpour2019local,woodworth2020local}.
However, despite these efforts, we do not yet know which methods are optimal, what the minimax complexity is, and what methodological or analytical improvements might allow us to make further progress. 

A key issue in the existing literature is that known lower bounds for the intermittent communication setting depend only on the product $KR$ (i.e.~the total number of gradients computed on each machine over the course of optimization), and not on the number of rounds, $R$, and the number of gradients per round, $K$, separately. 
Thus, existing results cannot rule out the possibility that the optimal rate for fixed $T=KR$ can be achieved using only a single round of communication ($R=1$), since they do not distinguish between methods that communicate very frequently ($R = T$, $K = 1$) and methods that communicate just once ($R = 1$, $K = T$). The possibility that the optimal rate is achievable with $R=1$ was suggested by \citet{zhang2013communication}, and indeed we showed in \pref{subsec:local-sgd-homogeneous} that in the special case of quadratic objectives, Local SGD and Local AC-SA perform just as well with a single communication as they do with many rounds of communication. While it seems unlikely that a single round of communication suffices in the general case, none of our existing lower bounds are able to answer this extremely basic question.

In \pref{subsec:homogeneous-intermittent-minimax}, we resolve (up to a logarithmic factor) the minimax complexity of smooth, convex stochastic optimization in the homogeneous intermittent communication setting and we show that, generally speaking, a single round of communication does not suffice to achieve the min-max optimal rate. In \pref{subsec:homogeneous-convex-intermittent-minimax}, we prove lower bounds on the optimal rate of convergence with matching upper bounds for convex and strongly convex objectives, and in \pref{subsec:homogeneous-non-convex-intermittent-minimax}, we prove matching upper and lower bounds on the optimal rate for finding approximate stationary points of non-convex objectives.

Interestingly, in all of these cases we show that the combination of two extremely simple and na\"ive methods are optimal. In the convex setting, the methods are based on an accelerated SGD variant AC-SA \citep{lan2012optimal}. Specifically, we show that the better of the following methods is optimal: Minibatch AC-SA which executes $R$ steps of AC-SA using minibatch gradients of size $MK$, and Single-Machine AC-SA which executes $KR$ steps of AC-SA on just one of the machines, completely ignoring the other $M-1$. Similarly, in the non-convex setting, the better of Minibatch SGD and Single-Machine SGD are optimal.

These methods might appear suboptimal: the Minibatch methods only perform one update per round of communication, and the Single-Machine methods only use one of the available workers! This perceived inefficiency has prompted many attempts at developing improved methods which take multiple steps on each machine locally in parallel including, in particular, numerous analyses of Local SGD \citep{zinkevich2010parallelized,dekel2012optimal,stich2018local,haddadpour2019local,khaled2019better,woodworth2020local,woodworth2020minibatch}, which we already discussed in \pref{sec:local-sgd}.  Nevertheless, we establish that one or the other is optimal in every regime, so more sophisticated methods cannot yield improved guarantees for arbitrary smooth objectives. Our results therefore highlight an apparent dichotomy between exploiting the available parallelism but not the local computation (Minibatch) and exploiting the local computation but not the parallelism (Single-Machine).

In addition to the homogeneous setting, in \pref{subsec:heterogeneous-convex-intermittent-minimax}, we also study the heterogeneous setting. Here, we also prove matching upper and lower bounds for convex and strongly convex objectives which establishes that Minibatch AC-SA is minimax optimal.

Our lower bounds apply quite broadly, including to the settings covered by the bulk of the existing work on stochastic first-order optimization in the intermittent communication setting. However, like many lower bounds, we should not interpret them to mean that progress is impossible, and that we are stuck with na\"ive algorithms like Minibatch SGD. Instead, these results indicate that we need to modify our assumptions in order to develop better methods.
In \pref{sec:breaking-the-lower-bounds} we explore several additional assumptions that allow---or might plausibly allow---for circumventing the lower bounds in various ways. These include when the third derivative of the objective is bounded (as in recent work by \citet{yuan2020federated}), when the objective has a certain statistical learning-like structure, or when the algorithm has access to a more powerful oracle.

\begin{table}
\renewcommand{\arraystretch}{1.2}
\centering
\begin{tabular}{ l l l }
\toprule
\textbf{Setting} & \textbf{Function Class} & \textbf{Minimax Error}  \\ 
\midrule
\multirow{3}{*}{Homogeneous} & $\mc{F}_0(H,B)$ & $\frac{HB^2}{K^2R^2} + \frac{\sigma B}{\sqrt{MKR}} + \min\crl*{\frac{HB^2}{R^2},\, \frac{\sigma B}{\sqrt{KR}}}$  \\
\cmidrule{2-3}
& $\mc{F}_\lambda(H,\Delta)$ & $\Delta\exp\prn*{-\frac{c'\sqrt{\lambda}KR}{\sqrt{H}}} + \frac{\sigma^2}{\lambda MKR} + \min\crl*{\Delta\exp\prn*{-\frac{c'\sqrt{\lambda}R}{\sqrt{H}}},\, \frac{\sigma^2}{\lambda KR}}$ \\
\cmidrule{2-3}
& $\mc{F}_{-H}(H,\Delta)$ & $\min\crl*{\frac{\sqrt{H\Delta}}{\sqrt{KR}} + \frac{\sqrt{\sigma} (H\Delta)^{1/4}}{(KR)^{1/4}},\ \frac{\sqrt{H\Delta}}{\sqrt{R}} + \frac{\sqrt{\sigma}(H\Delta)^{1/4}}{(MKR)^{1/4}}}$ \\
\midrule
\multirow{2}{*}{Heterogeneous} & $\mc{F}_0(H,B)$ & $\frac{HB^2}{R^2} + \frac{\sigma B}{\sqrt{MKR}}$ \\
\cmidrule{2-3}
& $\mc{F}_\lambda(H,\Delta)$ & $\Delta\exp\prn*{-\frac{c'\sqrt{\lambda}R}{\sqrt{H}}} + \frac{\sigma^2}{\lambda MKR}$ \\
\bottomrule
\end{tabular}
\caption{A summary of the results in \pref{sec:intermittent-communication-setting}, with constant and logarithmic factors omitted. \label{tab:intermittent-communication-setting-results}}
\end{table}

\subsection{The Homogeneous Setting}\label{subsec:homogeneous-intermittent-minimax}

We begin with the homogeneous setting, where all of the algorithm's queries are to the same stochastic gradient oracle which gives an unbiased estimate of $\nabla F$. More precisely, each vertex of the intermittent communication graph corresponds to an oracle $\mc{O}_v = \mc{O}_g^\sigma$, which is an ``independent-noise'' oracle (see \pref{subsec:the-oracle}) which just gives an unbiased estimate of $\nabla F$ with variance bounded by $\sigma^2$ that is independent of all other oracle queries. In \pref{sec:breaking-the-lower-bounds}, we will discuss other types of first-order oracles that have additional structure.

The proofs of our lower bounds generally follow the approach outlined in \pref{subsec:high-level-lower-bound-approach}. As was discussed there, part of the argument hinges on the norm of the algorithm's queries being bounded so that the algorithm cannot ``cheat'' and get a large inner product with the unknown columns of $U$ by simply guessing a random vector with huge norm. In the proof of \pref{thm:generic-graph-lower-bound}, we were able to avoid this issue by constructing a function for which querying the gradient oracle at a point with norm larger than $5\nrm{x^*}$ gives essentially no information. However, for the constructions used in this section this is more difficult, and we will instead rely on an explicit bound on the norm of the algorithm's queries. Specifically, we define $\mc{A}^\gamma(\mc{G}_{\textrm{I.C.}},\mc{O}_g^\sigma)$ to be the class of optimization algorithms in the intermittent communication setting with stochastic gradient oracles $\mc{O}_g^\sigma$ for which all queries are bounded in norm by $\nrm{x^m_{k,r}} \leq \gamma$. The bound $\gamma$ is arbitrary in the sense that our lower bounds apply for any $\gamma$ in a sufficiently large dimension of at least $\Omega(\gamma^2)$. However, our lower bounds do not apply to algorithms which query the oracle at unboundedly large points, or which query the oracle at points with norm that depend on the dimension. The restriction that the norm of the algorithm's queries is bounded can also be removed if the algorithm is deterministic or span-restricted/zero-respecting.

\subsubsection{Convex Objectives}\label{subsec:homogeneous-convex-intermittent-minimax}

We begin with our lower bound in the convex, smooth, and homogeneous intermittent communication setting:
\begin{restatable}{theorem}{homogeneousconvexlowerbound}\label{thm:homogeneous-convex-lower-bound}
For any $H,B,\sigma^2,\gamma$, there exists a function $F_0 \in \mc{F}_0(H,B)$ in any dimension 
\[
D \geq c\cdot\prn*{KR + \prn*{\frac{\gamma^2KR}{B^2} + \frac{H^2\gamma^2KR\prn*{\frac{\sqrt{\sigma}KR}{\sqrt{HB}} + MKR}}{\sigma^2}}\log(MK^2R^2)}
\]
such that the output of any algorithm in $\mc{A}^\gamma(\mc{G}_{\textrm{I.C.}},\mc{O}_g^\sigma)$ will have suboptimality at least
\[
\E F_0(\hat{x}) - F_0^* \geq c\cdot\prn*{\frac{HB^2}{K^2R^2} + \min\crl*{\frac{\sigma B}{\sqrt{MKR}},\ HB^2} + \min\crl*{\frac{\sigma B}{\sqrt{KR}},\ \frac{HB^2}{R^2(1 + \log M)^2}}}
\]
\end{restatable}
\begin{proofsketch}
The first two terms of this lower bound follow directly from \pref{thm:generic-graph-lower-bound} and \pref{lem:statistical-term-lower-bound}; the $\frac{HB^2}{K^2R^2}$ term corresponds to the error when optimizing a function using a deterministic gradient oracle, and the $\frac{\sigma B}{\sqrt{MKR}}$ term is a very well-known statistical limit \citep{nemirovskyyudin1983}. The distinguishing feature of our lower bound is the third term, which depends differently on $K$ than on $R$. For quadratics, Local AC-SA attains the rate given by just the first two terms, and actually does depend only on the product $KR$, as shown in \pref{cor:acc-local-sgd-optimal-quadratics}. Consequently, proving our lower bound necessitates going beyond quadratics. In contrast, all or at least most of the lower bounds for sequential smooth convex optimization apply even for quadratic objectives.

We start by describing the proof of the theorem for zero-respecting algorithms, and we will discuss how it is extended to arbitrary algorithms at the end. The proof uses the following non-quadratic hard instance:
\begin{equation}\label{eq:def-F-main}
F(x) = \psi'(-\zeta)x_1 + \psi(x_N) + \sum_{i=1}^{N-1}\psi(x_{i+1} - x_i)
\end{equation}
where $\psi:\R\to\R$ is defined as
\begin{equation}
\psi(x) := \frac{\sqrt{H}x}{2\beta}\arctan\prn*{\frac{\sqrt{H}\beta x}{2}} - \frac{1}{2\beta^2}\log\prn*{1+\frac{H\beta^2x^2}{4}}
\end{equation}
\begin{wrapfigure}{r}{0.23\textwidth}
\begin{center}
\vspace{-9mm}
{\includegraphics[width=\linewidth]{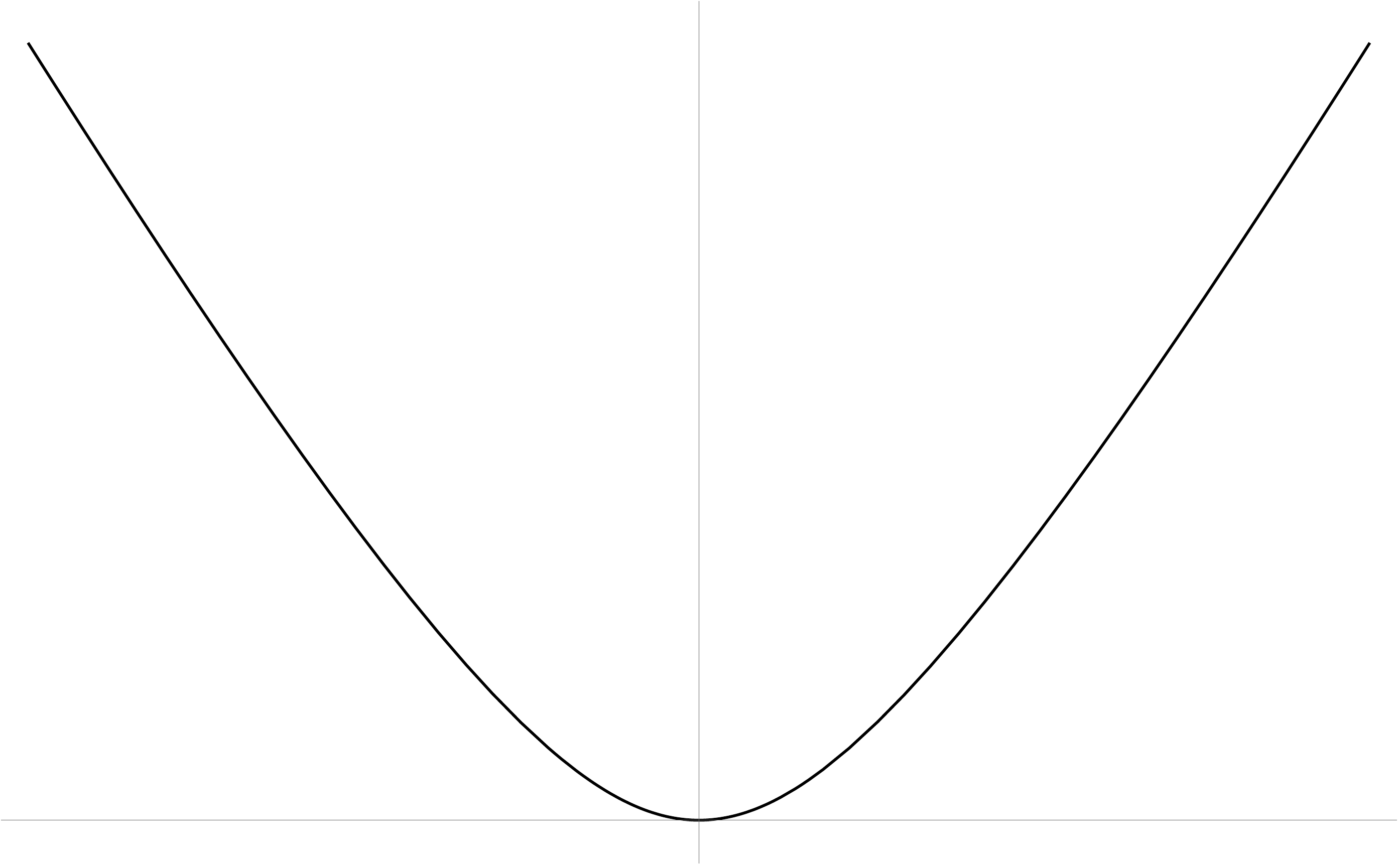}}
\small The function $\psi(x)$
\end{center}
\vspace{-7mm}
\end{wrapfigure}
and where $\beta$, $\zeta$, and $N$ are hyperparameters that are chosen depending on $H,B,\sigma,M,K,R$ so that $F$ satisfies the necessary conditions. This construction closely resembles the classic lower bound for deterministic first-order optimization of \citet{nesterov2004introductory}, which essentially uses $\psi(x) = x^2$. 
To describe our stochastic gradient oracle, we will use $\prog{0}(x) := \max\crl*{j\,:\, x_j \neq 0}$, which denotes the highest index of a non-zero coordinate of $x$. We also define $F^-$ to be equal to the objective with the $\prog{0}(x)^{\textrm{th}}$ term removed:
\begin{equation}\label{eq:thm10-gradient-oracle}
F^-(x) = \psi'(-\zeta) x_1 + \psi(x_N) + \sum_{i=1}^{\prog{0}(x)-1} \psi(x_{i+1} - x_i) + \sum_{i=\prog{0}(x)+1}^{N-1} \psi(x_{i+1} - x_i)
\end{equation}
The stochastic gradient oracle for $F$ is then given by 
\begin{equation}\label{eq:def-g-main}
g(x) = \begin{cases}
\nabla F^-(x) & \textrm{with probability } 1-p \\
\nabla F(x) + \frac{1-p}{p}\prn*{\nabla F(x) - \nabla F^-(x)} & \textrm{with probability } p 
\end{cases}
\end{equation}
This stochastic gradient oracle resembles the one used by \citet{arjevani2019lower} to prove lower bounds for non-convex optimization, and its key property is that $\P\brk*{\prog{0}(g(x)) \leq \prog{0}(x)} = 1 - p$. Therefore, for zero-respecting algorithms, each oracle access only allows the algorithm to increase its progress with probability $p$. 
The rest of the proof revolves around bounding the total progress of the algorithm and showing that if $\prog{0}(x) \leq \frac{N}{2}$, then $x$ has high suboptimality.

Since each machine makes $KR$ sequential queries and only makes progress with probability $p$, the total progress scales like $KRp$. By taking $p$ smaller, we decrease the amount of progress made by the algorithm, and therefore increase the lower bound. Indeed, when $p \approx 1/K$, the algorithm only increases its progress by about $\log M$ per round, which gives rise to the key $(HB^2)/(R^2\log^2 M)$ term in the lower bound. However, we are constrained in how small we can take $p$ since our stochastic gradient oracle has variance 
\begin{equation}
\sup_x \E\nrm*{g(x) - \nabla F(x)}^2 = \frac{2(1-p)}{p}\sup_x \psi'(x)^2
\end{equation}
This is where our choice of $\psi$ comes in. Specifically, we chose the function $\psi$ to be convex and smooth so that $F$ is, but it is also Lipschitz:
\begin{equation}
\psi'(x) = \frac{\sqrt{H}}{2\beta}\arctan\prn*{\frac{\sqrt{H}\beta x}{2}} 
\in \brk*{-\frac{\pi\sqrt{H}}{4\beta},\,\frac{\pi\sqrt{H}}{4\beta}} 
\end{equation}
Notably, this Lipschitz bound on $\psi$, which implies a bound on $\nrm{\nabla F(x)}_{\infty}$, is the key non-quadratic property that allows for our lower bound. Since $\psi'$ is bounded, we are able to able to choose $p \approx \sigma^{-2}\beta^{-2}$ without violating the variance constraint on the stochastic gradient oracle. Carefully balancing $\beta$ completes the argument. 

To extend this argument to randomized algorithms that may not be zero-respecting, we follow the approach described in \pref{subsec:high-level-lower-bound-approach} by introducing a random rotation $U$ and ``flattening out'' the $\psi$ functions around the origin. We defer the remaining details to \pref{app:intermittent-communication-homogeneous-convex-lower-bound}.
\end{proofsketch}

To complement the lower bound \pref{thm:homogeneous-convex-lower-bound} and to establish the minimax error for the convex, smooth, and homogeneous intermittent communication setting, we prove a nearly matching upper bound. This upper bound is attained by either Minibatch AC-SA or Single-Machine AC-SA. We recall from \pref{subsec:initially-introduce-ac-sa-convex-rate} that AC-SA (see \pref{alg:ac-sa}) is an accelerated variant of SGD \citep{lan2012optimal}. 

The Minibatch AC-SA algorithm corresponds to taking $R$ steps of AC-SA using minibatch stochastic gradients of size $MK$. This can be implemented in the intermittent communication setting by having all $M$ machines calculate $K$ stochastic gradients at the same point during each round of communication. When the machine do communicate, they can combine all $MK$ of these gradients into one large minibatch and compute a single AC-SA update.

The Single-Machine AC-SA algorithm corresponds to simply taking $KR$ steps of AC-SA using minibatch stochastic gradients of size $1$. This can be implemented in the intermittent communication setting by simply implementing the algorithm on a single machine, and ignoring the remaining $M-1$ workers altogether. 

While these approaches may seem simple, the following theorem shows that the better of the two is optimal:
\begin{restatable}{theorem}{homogeneousconvexupperbound}\label{thm:homogeneous-convex-upper-bound}
For any $H,B,\sigma^2$, either Minibatch AC-SA or Single-Machine AC-SA guarantees that for any $F_0 \in \mc{F}_0(H,B)$
\[
\E F_0(\hat{x}) - F_0^* \leq c\cdot\min\crl*{\frac{HB^2}{(KR)^2} + \frac{\sigma B}{\sqrt{KR}},\ \frac{HB^2}{R^2} + \frac{\sigma B}{\sqrt{MKR}},\ HB^2}
\]
\end{restatable}
A simple proof is given in \pref{app:intermittent-communication-homogeneous-convex-upper-bound}, and requires simply plugging the number of updates and bound on the variance of the minibatch stochastic gradients into the existing guarantee for AC-SA.

\subsubsection{Strongly Convex Objectives}\label{subsec:homogeneous-strongly-convex-intermittent-minimax}

We also show nearly matching upper and lower bounds in the strongly convex setting:
\begin{restatable}{theorem}{homogeneousstronglyconvexlowerbound}\label{thm:homogeneous-strongly-convex-lower-bound}
For any $H,B,\sigma^2,\gamma$, and dimension 
\[
D \geq c\cdot\prn*{KR + \prn*{\frac{\gamma^2KR}{B^2} + \frac{H^2\gamma^2KR\prn*{\frac{\sqrt{\sigma}KR}{\sqrt{HB}} + MKR}}{\sigma^2}}\log(MK^2R^2)}
\]
there exist $\lambda$ and $\Delta$, and an objective $F_\lambda \in \mc{F}_\lambda(H,\Delta)$ in dimension $D$ such that the output of any algorithm in $\mc{A}^\gamma(\mc{G}_{\textrm{I.C.}},\mc{O}_g^\sigma)$ will have suboptimality at least
\[
\E F_\lambda(\hat{x}) - F_\lambda^* \geq c\cdot\prn*{\Delta\exp\prn*{-\frac{c'\sqrt{\lambda}KR}{\sqrt{H}}} + \min\crl*{\frac{\sigma^2}{\lambda MKR},\,\Delta} + \min\crl*{\frac{\sigma^2}{\lambda KR},\, \Delta\exp\prn*{-\frac{c'\sqrt{\lambda}R\log M}{\sqrt{H}}}}}
\]
\end{restatable}
This is proven in \pref{app:intermittent-communication-homogeneous-strongly-convex-lower-bound} using the reduction between convex and strongly-convex objectives described in \pref{subsec:lower-bounds-by-reduction}, which explains the weaker statement. Nevertheless, as in the convex case, this lower bound is matched by the better of Minibatch and Single-Machine AC-SA\footnote{The AC-SA algorithm requires a slight modification in order to achieve the optimal rate for strongly convex objectives, see \pref{subsec:the-reduction}.}.
\begin{restatable}{theorem}{homogeneousstronglyconvexupperbound}\label{thm:homogeneous-strongly-convex-upper-bound}
For any $H,\Delta,\sigma^2$, either Minibatch AC-SA or Single-Machine AC-SA guarantees that for any $F_\lambda \in \mc{F}_\lambda(H,\Delta)$
\[
\E F_\lambda(\hat{x}) - F_\lambda^* \leq c\cdot\prn*{\Delta\exp\prn*{-\frac{c'\sqrt{\lambda}KR}{\sqrt{H}}} + \min\crl*{\frac{\sigma^2}{\lambda MKR},\,\Delta} + \min\crl*{\frac{\sigma^2}{\lambda KR},\, \Delta\exp\prn*{-\frac{c'\sqrt{\lambda}R}{\sqrt{H}}}}}
\]
\end{restatable}
This is proven in \pref{app:intermittent-communication-homogeneous-strongly-convex-upper-bound} by simply plugging the number of steps of AC-SA and the variance of the minibatch stochastic gradients into the existing guarantees for AC-SA.

\subsubsection{Non-Convex Objectives}\label{subsec:homogeneous-non-convex-intermittent-minimax}

We now consider the non-convex setting. Without assuming convexity, it is generally intractable to ensure convergence to a global minimizer of the objective \citep{nemirovskyyudin1983}. For this reason, it is common to analyze algorithms in terms of their ability to find approximate stationary points of the objective \citep{vavasis1993black,nocedal2006numerical,nesterov2006cubic,ghadimi2013stochastic,carmon2017convex,lei2017non,fang2018spider,zhou2018stochastic,fang2019sharp}, i.e.~a point $\hat{x}$ such that
\begin{equation}
\E\nrm{\nabla F(\hat{x})} \leq \epsilon
\end{equation}
It is important to understand optimal algorithms for non-convex since most modern machine learning applications, like training neural networks, involve solving non-convex optimization. Furthermore, given the large scale of these non-convex problems, it is often critical to leverage parallelism in order to speed training.

In the homogeneous intermittent communication setting, we can pose the same questions of minimax optimality for non-convex optimization as we did for the convex case, with this new success criterion of finding approximate stationary points replacing approximate minimization. In order to state our results, we define $\mc{F}_{-H}(H,\Delta)$ to be the class of all $H$-smooth, possibly non-convex objectives such that $F(0) - \min_x F(x) \leq \Delta$. The following theorem proves a lower bound on the optimal rate of convergence to an approximate stationary point for any intermittent communication algorithm:
\begin{restatable}{theorem}{homogeneousnonconvexlowerbound}\label{thm:homogeneous-non-convex-lower-bound}
For any $H,\Delta,\sigma^2$, there exists a function $F \in \mc{F}_{-H}(H,\Delta)$ in a sufficiently large dimension $D \geq c\cdot KR\log(MKR)$ such that for any algorithm in $\mc{A}^{\infty}(\mc{G}_{\textrm{I.C}},\mc{O}_g^\sigma)$
\[
\E\nrm{\nabla F(\hat{x})} \geq c\cdot\min\crl*{\frac{\sqrt{H\Delta}}{\sqrt{KR}} + \frac{\sqrt{\sigma} (H\Delta)^{1/4}}{(KR)^{1/4}},\ \frac{\sqrt{H\Delta}}{\sqrt{R(1 + \log M)}} + \frac{\sqrt{\sigma}(H\Delta)^{1/4}}{(MKR)^{1/4}},\ \sqrt{H\Delta}}
\]
\end{restatable}
The construction is based on one used by \citet{carmon2017lower1} to show a lower bound of $H\Delta/\sqrt{T}$ for sequential non-convex optimization using an exact gradient oracle. Beyond extending the argument to the intermittent communication setting, we also augment the construction with a stochastic gradient oracle like was used for the proof of \pref{thm:homogeneous-convex-lower-bound}. The stochastic gradient oracle ``zeros out'' the next relevant component of the gradient with probability $1-p$, which reduces the amount of progress made by the algorithm by a factor of $p$. Setting $p$ as small as possible without violating the gradient variance constraint completes the arguement. We defer additional details of the proof to \pref{app:intermittent-communication-homogeneous-non-convex-lower-bound}.

In this case, we can also identify a pair of algorithms whose combined guarantee matches the lower bound. Indeed, the same pattern holds in the non-convex case and, again, the better of Minibatch SGD and Single-Machine SGD is optimal. While it is perhaps unsurprising that the convex and strongly convex settings have the same punchline, it was less clear that these observations would apply in the non-convex setting, yet the only difference is that for non-convex objectives, there is no need for acceleration, and regular SGD is able to attain the optimal rate.
\begin{restatable}{theorem}{homogeneousnonconvexupperbound}\label{thm:homogeneous-non-convex-upper-bound}
For any $H,\Delta,\sigma^2$, either Minibatch SGD or Single-Machine SGD guarantees that for any objective $F \in \mc{F}_{-H}(H,\Delta)$, 
\[
\E\nrm{\nabla F(\hat{x})} \leq c\cdot\min\crl*{\frac{\sqrt{H\Delta}}{\sqrt{KR}} + \frac{\sqrt{\sigma} (H\Delta)^{1/4}}{(KR)^{1/4}},\ \frac{\sqrt{H\Delta}}{\sqrt{R}} + \frac{\sqrt{\sigma}(H\Delta)^{1/4}}{(MKR)^{1/4}},\ \sqrt{H\Delta}}
\]
\end{restatable}
This is proven in \pref{app:intermittent-communication-homogeneous-non-convex-upper-bound} by appealing to existing analysis for SGD for smooth non-convex objectives \cite{ghadimi2013stochastic}.

\subsubsection{Conclusions}

In this section, we identified the minimax error and optimal algorithms in the homogeneous intermittent communication setting with convex, strongly convex, and non-convex objectives up to logarithmic factors. In doing so, we have highlighted several interesting characteristics of the intermittent communication setting:

\paragraph{A Tradeoff Between Parallelism and Local Computation}
In light of \pref{thm:homogeneous-convex-upper-bound} and the third term in the lower bound \pref{thm:homogeneous-convex-lower-bound} (and the analogous terms in the other settings), we see that algorithms are offered the following dilemma: they may either attain the optimal statistical rate $\sigma B / \sqrt{MKR}$ for the convex setting but suffer an optimization rate $HB^2 / (R^2 \log^2 M)$ that does not improve with $K$, or they may attain the optimal optimization rate of $HB^2 / (K^2R^2)$ but suffer a statistical term $\sigma B / \sqrt{KR}$ that does not improve with $M$. In this way, there is a very real dichotomy between exploiting the availability of parallel computation (e.g.~using Minibatch AC-SA) an exploiting the availability of sequential local computation (e.g.~using Single-Machine AC-SA). Importantly, under the conditions we study, it is impossible to exploit both simultaneously, and one of the extremes of this spectrum is always optimal. This is quite surprising, and for a long time we thought that it should be possible to design an algorithm which gets the best of both worlds. 

\paragraph{Mixed Statistical and Optimization Terms}
The structure of the optimal error in the convex and strongly convex settings have a notably different structure than in many other stochastic optimization settings. Typically in convex optimization, the minimax error for stochastic first-order optimization is the sum of two terms, an ``optimization term''---which is equal to the minimax error when using an \emph{exact} first-order oracle---and a ``statistical term''---which is equal is the optimal error of any method using that number of samples, e.g.~of the empirical risk minimizing solution. This holds, for example, in the sequential graph, the layer graph, and the delay graph (see \pref{subsec:sequential-graph-minimax-complexity}, \pref{subsec:layer-graph-minimax-complexity}, and \pref{subsec:delay-graph-minimax-complexity}). For instance, the minimax error for smooth, convex stochastic first-order optimization in the sequential settings is
\begin{equation}
\epsilon(\mc{F}_0(H,B),\mc{A}(\mc{G}_{\textrm{seq}},\mc{O}_g^\sigma)) = c\cdot\prn*{\frac{HB^2}{T^2} + \frac{\sigma B}{\sqrt{T}}}
\end{equation}
which is the sum of the deterministic first-order minimax error plus the error of (regularized) ERM. In contrast, the optimal error in the intermittent communication setting is different. In the convex case, it is (ignoring log factors)
\begin{equation}
\epsilon(\mc{F}_0(H,B),\mc{A}(\mc{G}_{\textrm{I.C.}},\mc{O}_g^\sigma)) = c\cdot\prn*{\frac{HB^2}{K^2R^2} + \frac{\sigma B}{\sqrt{MKR}} +
\min\crl*{\frac{\sigma B}{\sqrt{KR}},\ \frac{HB^2}{R^2}}}
\end{equation}
The first term is the deterministic first-order minimax error (see \pref{thm:generic-graph-lower-bound}) and the second term is the statistical limit (see \pref{lem:statistical-term-lower-bound}), but the third term is something different. In particular, it mixes optimization and statistical terms in a different and interesting manner. 

This suggests that we should think about optimization in the sequential setting as qualitatively different from optimization in the intermittent communication setting. In the former case, optimal deterministic first-order algorithms (e.g.~accelerated gradient descent) are often, perhaps with slight modifications, also optimal stochastic first-order algorithms (e.g.~AC-SA). Furthermore, algorithms can be understood via a bias-variance decomposition---the bias is basically the algorithm's performance if the gradients were exact, and the variance is how much the algorithm's output is affected by the noisy gradients. In contrast, in the intermittent communication setting, the interplay between optimization and statistics appears to be more complex. This shows up also in the proof of our lower bound, where the ``progress'' of the algorithm---which is relevant only to the optimization term in the sequential setting---is hindered by the noise in the stochastic gradient oracle.

\paragraph{Computational Efficiency}
The optimal algorithms in each setting---the better of Minibatch or Single-Machine AC-SA/SGD---are computationally efficient and require no significant overhead. Each machine only needs to store a constant number of vectors, performs only a constant number of vector additions for each stochastic gradient oracle access, and communicates just one vector per round. Therefore, the total storage complexity is just $O(d)$ per machine, the sequential runtime complexity (excluding the oracle computation) is $O(KR\cdot d)$, and the total communication complexity is at most $O(MR\cdot d)$. In fact, the communication complexity is exactly $0$ for the Single-Machine methods. Therefore, we should not expect a substantially better algorithm from the standpoint of computational efficiency either.

\paragraph{Aesthetics}
The optimal algorithms are somewhat ``ugly'' because of the hard switch between the Minibatch and Single-Machine approach. It would be nice, if only aesthetically, to have an algorithm that more naturally transitions between the Minibatch to the Single-Machine rate. Accelerated Local SGD \citep{yuan2020federated} or something similar is a contender for such an algorithm, although it is unclear whether or not this method can match the optimal rate in all regimes. Local SGD methods can also be augmented by using two stepsizes---a smaller, conservative stepsize for the local updates between communications, and a larger, aggressive stepsize when the local updates are aggregated---this two-stepsize approach allows for interpolation between Minibatch-like and Single-Machine-like behavior, and could be used to design a more ``natural'' optimal algorithm (see \pref{subsec:inner-outer}).

\subsection{The Heterogeneous Setting}\label{subsec:heterogeneous-convex-intermittent-minimax}

We now consider the complexity of optimization in the heterogeneous intermittent communication setting. The objective is the average of $M$ components, one for each machine,
\begin{equation}\label{eq:hetero-ic-objective}
F(x) = \frac{1}{M}\sum_{m=1}^M F_m(x)
\end{equation}
We focus on the goal of finding a single consensus solution that achieves low value on all of the machines on average \cite{bertsekas1989parallel,boyd2011distributed}. There are, of course, many other sensible formulations, including ``personalized'' approaches where separate minimizers are computed for each objective while still leveraging relevant information from the others \citep{hanzely2021personalized}. For this section, we consider the class $\mc{F}_0(H,B,M)$ consisting of all objectives \eqref{eq:hetero-ic-objective} where $F \in \mc{F}_0(H,B)$ and $\mc{F}_\lambda(H,\Delta,M)$ where $F \in \mc{F}_\lambda(H,\Delta)$. Notably, we make no assumptions about the individual components $F_1,\dots,F_M$, and we only require that their average is smooth, convex, etc.

In the heterogeneous intermittent communication setting, each machine only has access to information about its corresponding objective. So, the $m\mathth$ machine has access to a stochatic gradient oracle $\mc{O}_{g,m}^\sigma$ to each machine, which provides an unbiased estimate of the $m\mathth$ component $\E \mc{O}_{g,m}^\sigma(x) = \nabla F_m(x)$, with variance bounded by $\sigma^2$. In the language of the graph oracle model, we consider algorithms in $\mc{A}^\gamma(\mc{G}_{\textrm{I.C.}},\crl*{\mc{O}_{g,m}^\sigma})$ where each vertex $v = v^m_{k,r}$ is associated with the oracle $\mc{O}_{v^m_{k,r}} = \mc{O}_{g,m}^\sigma$. As in the homogeneous setting, for technical reasons we also restrict our attention to algorithms whose queries are bounded in norm by $\gamma$, although this restriction can be eliminated by instead considering deterministic or span-restricted/zero-respecting algorithms. 

\subsubsection{Convex Objectives}

We begin with our lower bound for convex objectives:
\begin{restatable}{theorem}{heterogeneousconvexlowerbound}\label{thm:heterogeneous-convex-lower-bound}
For any $M\geq 2$ and $H,B,\sigma^2,\gamma$, there exists a quadratic objective $F \in \mc{F}_0(H,B,M)$ in any dimension
\[
D \geq c\cdot\gamma^2\prn*{\frac{R^3}{B^2} + \frac{H^2MKR^2}{\sigma^2}}\log(MKR)
\]
such that the output of any algorithm in $\mc{A}^\gamma(\mc{G}_{\textrm{I.C.}},\crl*{\mc{O}_{g,m}^\sigma})$ will have suboptimality at least
\[
\E F(\hat{x}) - F^* \geq c\cdot\prn*{\frac{HB^2}{R^2} + \min\crl*{\frac{\sigma B}{\sqrt{MKR}},\ HB^2}}
\]
\end{restatable}
\begin{proofsketch}
To describe the idea of the proof, we will first focus our attention on the class of span-restricted/zero-respecting algorithms.  
The proof is based on the following two quadratic functions:
\begin{equation}
\begin{aligned}
\tilde{F}_1(x) &= -\zeta x_1 + Cx_d^2 + \sum_{i=1}^{d/2-1}\prn*{x_{2i+1} - x_{2i}}^2 \\
\tilde{F}_2(x) &= \sum_{i=1}^{d/2}\prn*{x_{2i} - x_{2i-1}}^2
\end{aligned}
\end{equation}
These functions are identical to a construction from \citet{arjevani2015communication}, who show similar lower bounds for a different formulation of distributed optimization with an exact gradient oracle. These objectives essentially partition the classic lower bound construction of \citet{nesterov2004introductory} across two functions. Importantly, if $x_i = 0$ for all $i > j$, and $j$ is even, then
\begin{equation}\label{eq:heterogeneous-lower-bound-zero-chain}
\forall_{i>j+1}\ \ [\nabla \tilde{F}_1(x)]_i = 0 \quad\textrm{and}\quad \forall_{i>j}\ \ [\nabla \tilde{F}_2(x)]_i = 0
\end{equation}
and the vice versa when $j$ is odd. In other words, querying the gradient of $\tilde{F}_1$ only reveals an additional coordinate when the current progress is even, and the gradient of $\tilde{F}_2$ only gives a new coordinate when the current progress is odd. This means that making progress requires alternately querying $\tilde{F}_1$ and $\tilde{F}_2$, which can only be done once per round of communication. This argument paired with assigning $\tilde{F}_1$ to the first $M/2$ machines and $\tilde{F}_2$ to the remaining $M/2$ machines essentially completes the proof for the span-restricted/zero-respecting case. We note that this part of the argument applies even when the algorithm has access to exact gradients of the objective, as in the setting of \citeauthor{arjevani2015communication}.

To extend the argument to the case of general randomized algorithms, we again introduce a random rotation matrix $U$ and consider the functions $\tilde{F}_1(U^\top x)$ and $\tilde{F}_2(U^\top x)$. If we followed the ideas in \pref{subsec:high-level-lower-bound-approach}, we would, in addition, ``flatten out'' the objective so that small inner products with the columns of $U$ do not reveal any information about those columns in the gradient. However, in this case we do something different in order to preserve the quadratic nature of the objective. Specifically, we do this ``flattening out'' \emph{only using the stochatic gradient oracle}. Specifically, we construct stochastic gradient oracles for $\tilde{F}_1$ and $\tilde{F}_2$ such that the property \eqref{eq:heterogeneous-lower-bound-zero-chain} is maintained even when $x_i$ is slightly non-zero for $i > j$. We show that it is possible to construct such stochastic gradient oracles without introducing too much variance, and this allows us to prove the result even for quadratic objectives. The remaining details of the proof can be found in \pref{app:heterogeneous-convex-lower-bound}.
\end{proofsketch}

In the homogeneous intermittent communication setting, all of our lower bounds were matched by the better of two algorithms: Minibatch AC-SA and Single-Machine AC-SA. Notably, in the heterogeneous setting, Single-Machine AC-SA is not a sensible algorithm because it will only succeed in optimizing a single component of the objective, which does not imply anything about optimizing $F$. Of course, other approaches like averaging the outputs of Single-Machine AC-SA run on each machine individually, or an accelerated variant of Local SGD \citep{yuan2020federated} are perfectly reasonable methods which could plausibly outperform Minibatch AC-SA in certain regimes. However, we now show that actually Minibatch AC-SA is minimax optimal in the heterogeneous case and no other algorithm can improve over it in any regime
\begin{restatable}{theorem}{heterogeneousconvexupperbound}\label{thm:heterogeneous-convex-upper-bound}
For any $H,B,\sigma^2$ and any $F \in \mc{F}_0(H,B,M)$ the output of Minibatch AC-SA will have suboptimality at most
\[
\E F(\hat{x}) - F^* \leq c\cdot\prn*{\frac{HB^2}{R^2} + \min\crl*{\frac{\sigma B}{\sqrt{MKR}},\ HB^2}}
\]
\end{restatable}
\begin{proof}
The result follows from the observation that was made in \pref{subsec:minibatch-sgd-in-heterogeneous-local-sgd}: that the minibatch stochastic gradients used by Minibatch AC-SA are actually unbiased estimates of $\nabla F$ despite the heterogeneity of the problem, and their variance is also reduced by a factor of $MK$. The result then follows immediately from the guarantee for $R$ steps of AC-SA with stochastic gradients of variance $\sigma^2/MK$ \citep{lan2012optimal}.
\end{proof}

\subsubsection{Strongly Convex Objectives}

Unsurprisingly, the picture is qualitatively very similar in the strongly convex setting. We have an analogous lower bound
\begin{restatable}{theorem}{heterogeneousstronglyconvexlowerbound}\label{thm:heterogeneous-strongly-convex-lower-bound}
For any $M\geq 2$ and $H \geq 9\lambda,\Delta,\sigma^2$, there exists an objective $F \in \mc{F}_\lambda(H,\Delta,M)$ in any dimension
\[
D \geq c\cdot\gamma^2\prn*{\frac{\sqrt{H\lambda}}{\Delta}\prn*{1-\sqrt{\frac{\lambda}{H}}}^{-2R} + \frac{H^2MKR^2}{\sigma^2}}\log(MKR)
\]
such that the output of any algorithm in $\mc{A}^\gamma(\mc{G}_{\textrm{I.C.}},\crl*{\mc{O}_{g,m}^\sigma})$ will have suboptimality at least
\[
\E F(\hat{x}) - F^* \geq c\cdot\prn*{\Delta\exp\prn*{\frac{-c'\sqrt{\lambda}R}{\sqrt{H}}} + \min\crl*{\frac{\sigma^2}{\lambda MKR},\ \Delta}}
\]
\end{restatable}
This algorithm is proven in \pref{app:heterogeneous-convex-lower-bound} with essentially the same argument as was used for \pref{thm:heterogeneous-convex-lower-bound}. The only difference is that $\frac{\lambda}{2}\nrm{x}^2$ was added to the hard instance to make the function strongly convex. Similarly, we have a matching upper bound from the guarantee of Minibatch AC-SA:
\begin{restatable}{theorem}{heterogeneousstronglyconvexupperbound}\label{thm:heterogeneous-strongly-convex-upper-bound}
For any $H \geq 9\lambda,\Delta,\sigma^2$, and any $F \in \mc{F}_\lambda(H,\Delta,M)$, the output of Minibatch will have suboptimality at most
\[
\E F(\hat{x}) - F^* \leq c\cdot\prn*{\Delta\exp\prn*{\frac{-c'\sqrt{\lambda}R}{\sqrt{H}}} + \min\crl*{\frac{\sigma^2}{\lambda MKR},\ \Delta}}
\]
\end{restatable}
As for \pref{thm:heterogeneous-convex-upper-bound}, we observe that the minibatch stochastic gradients are unbiased estimates of $\nabla F$ with variance $\sigma^2 /MK$ and plug this into the AC-SA guarantee (see \pref{subsec:the-reduction}).

\subsubsection{Conclusions}

These results tightly bound the minimax error for optimization in the heterogeneous intermittent communication setting with convex and strongly convex objectives, and establish that Minibatch AC-SA is an optimal algorithm. By analogy, we speculate that Minibatch SGD would be minimax optimal in the non-convex setting too. These results are particularly interesting in comparison with the homogeneous intermittent communication setting:

\paragraph{Sequential Local Computation Does Not Help}
In the homogeneous setting, the optimal algorithm was the better of Minibatch AC-SA and Single-Machine AC-SA. This demonstrated a tradeoff between exploiting the parallelism (using Minibatch AC-SA) and exploiting the available sequential local computation (Single-Machine AC-SA). In contrast, in the heterogeneous regime, where the optimal algorithm is just Minibatch AC-SA, we see that there is no such tradeoff. In particular, it is actually impossible to exploit the availability of sequential local computation at all! To see this, we note that if each machine could only query its stochastic gradient oracle once, but with variance $\sigma^2/K$, the Minibatch AC-SA guarantee would be unchanged. In this way, the sequential nature of the $K$ local stochastic gradient queries is useless beyond its ability to reduce the variance. 

\paragraph{Quadratic Structure Does Not Help}
In a similar vein, in the homogeneous convex setting, the optimal rate could be substantially different when the objective is quadratic versus when it is an arbitrary convex function. Indeed, our lower bounds \pref{thm:homogeneous-convex-lower-bound} and \pref{thm:homogeneous-strongly-convex-lower-bound} were based on decidedly non-quadratic constructions, and their proofs relied critically on their non-quadratic nature. In contrast, the lower bounds \pref{thm:heterogeneous-convex-lower-bound} and \pref{thm:heterogeneous-strongly-convex-lower-bound} apply even for quadratic objectives, so this additional structure does not help the algorithm at all.

\section{Better than Optimal: Breaking The Lower Bounds}\label{sec:breaking-the-lower-bounds}

Many of the results so far presented have been lower bounds on the minimax error in various optimization settings. In a certain sense, these are positive results since they are used to identify optimal algorithms, but, they also have a negative interpretation, as they show fundamental limits on optimization algorithms for these settings. For the intermittent communication setting in particular, the optimal algorithms---a combination of Minibatch AC-SA and Single-Machine AC-SA---are disappointingly na\"ive, and their guarantees unable to simultaneously exploit the availability of local computation and parallelism. 

However, one of the most important uses of optimization lower bounds is to identify how to break them. Rather than viewing them as impossibility results that show when we should throw up our hands because we can do no better than what we have, we should view them as a hint about which of our assumptions should be modified or strengthened. One way to break a lower bound---by designing an algorithm that achieves a better guarantee---is to impose additional structure on the problem that an algorithm might exploit. Lower bounds  identify particular hard objectives and show us \emph{why} those objectives, specifically, are hard to optimize, which in turn motivates new assumptions that (1) obviate the need to deal with those particular hard functions and (2) allow for more effective algorithms more broadly.

Of course, it is trivial to introduce new assumptions that make optimization easy---for instance, we could assume that the function has a minimizer $x^*=0$, which would make optimization very easy indeed!---but these assumptions are too strong. Therefore, it is important to identify new assumptions that allow for the circumvention of the lower bounds by making the problem ``easier'' while simultaneously applying to the objectives that we are actually interested in optimizing. We should therefore think of this moreso as modelling rather than assuming---we want to distill the relevant properties of the objectives we are interested in down to a short list that are amenable for analyzing algorithms. In most of the results so far, we have relied on a very short list of properties: smoothness, convexity, a bound on the gradient variance, and a bound on $\nrm{x^*}$. In this section, we will explore some possible additions to this list that will allow us to break the lower bounds.

\subsection{Intermittent Communication and Near-Quadratic Objectives}\label{subsec:homogeneous-nearly-quadratic-intermittent-minimax}

In the homogeneous intermittent communication setting, \pref{thm:homogeneous-convex-lower-bound} and \pref{thm:homogeneous-strongly-convex-lower-bound} show that the better of Minibatch and Single-Machine AC-SA is optimal. However, we had already seen from \pref{cor:acc-local-sgd-optimal-quadratics} that Local AC-SA could achieve substantially lower error when the objective is quadratic. Naturally, this suggests that if the objective were ``nearly'' quadratic in some way, then Local AC-SA or some other similar method should be able to defeat the lower bounds \pref{thm:homogeneous-convex-lower-bound} and \pref{thm:homogeneous-strongly-convex-lower-bound}. Indeed, the constructions used to prove those lower bounds were not quadratic, and their proofs relied crucially on this fact.

One means of quantifying ``near quadratic'' is to impose a bound on the third derivative of the objective---a quadratic objective has a constant Hessian, and therefore its third derivative is uniformly zero. We therefore introduce the class of $\beta$-second order smooth convex objectives $\mc{F}_0(H,B,\beta)$ and strongly convex objectives $\mc{F}_\lambda(H,\Delta,\beta)$, which contain all twice-differentiable $F_0 \in \mc{F}_0(H,B)$ and $F_\lambda \in \mc{F}_\lambda(H,\Delta)$ for which the Hessian is $\beta$-Lipschitz. 
The assumption of a bounded third derivative is often reasonable and, for example, it holds for training generalized linear models with sufficiently smooth link functions, such as logistic regression. 

In recent work, \citet{yuan2020federated} proposed an accelerated Local SGD variant \textrm{FedAc} and analyzed its error for objectives in $\mc{F}_0(H,B,\beta)$ and $\mc{F}_\lambda(H,\Delta,\beta)$. They showed that
\begin{theorem}[c.f.~Theorems C.1 and E.3 \citep{yuan2020federated}]
For any $H,B,\Delta,\sigma,\beta$, and any $F_0 \in \mc{F}_0(H,B,\beta)$ and $F_\lambda \in \mc{F}_\lambda(H,\Delta,\beta)$, \textsc{FedAc}, using a stochastic gradient oracle with bounded $4\mathth$ moment $\E\nrm{\mc{O}_g(x) - \nabla F(x)}^4 \leq \sigma^4$ guarantees (omitting logarithmic factors)
\begin{align*}
\E F_0(\hat{x}) - F_0^*
&\leq c\cdot\prn*{\frac{HB^2}{KR^2} + \frac{\sigma B}{\sqrt{MKR}} + \frac{\prn*{H\sigma^2B^4}^{1/3}}{M^{1/3}K^{1/3}R} + \frac{\prn*{\beta\sigma^2B^5}^{1/3}}{K^{1/3}R^{4/3}}} \\
\E F_\lambda(\hat{x}) - F^*
&\leq c\cdot\prn*{\Delta\exp\prn*{-\frac{c'\lambda KR}{H}-\frac{c'\sqrt{\lambda K}R}{\sqrt{H}}} + \frac{\sigma^2}{\lambda MKR} + \frac{H\sigma^2}{\lambda^2 MKR^3} + \frac{\beta^2\sigma^4}{\lambda^5K^2R^8}}
\end{align*}
\end{theorem}
The first two terms of both guarantees come close to matching the guarantee of Local AC-SA for quadratic objectives, except with the optimization terms scaling with $K^{-1}$ and $\exp(-\sqrt{K})$ rather than $K^{-2}$ and $\exp(-K)$ as for Local AC-SA. The third term in each rate is reminiscent of the (unaccelerated) Local SGD rates (see \pref{thm:local-sgd-homogeneous-upper-bound}), but with a better dependence on $M$ and $R$. The final terms of these rates are where $\beta$ comes in, and they vanish as $\beta \to 0$. 

It is perhaps not immediately obvious, but these guarantees can be better than the lower bounds in certain parameter regimes. For example, in the convex case, when $H=B=\sigma=1$ and $1 \ll K \leq R^3 \ll MK$, the lower bound reduces to (ignoring logarithmic factors) 
\begin{equation}
\E F(\hat{x}) - F^* \geq \Omega\prn*{\frac{1}{K^2R^2} + \frac{1}{\sqrt{MKR}} + \min\crl*{\frac{1}{R^2},\frac{1}{\sqrt{KR}}}} = \Omega\prn*{\frac{1}{R^2}}
\end{equation}
On the other hand, the \textsc{FedAc} guarantee is
\begin{equation}
\E F_0(\hat{x}) - F_0^* \leq O\prn*{\frac{1}{KR^2} + \frac{1}{\sqrt{MKR}} + \frac{1}{M^{1/3}K^{1/3}R} + \frac{\beta^{1/3}}{K^{1/3}R^{4/3}}} = o\prn*{\frac{1}{R^2}} + O\prn*{\frac{\beta^{1/3}}{K^{1/3}R^{4/3}}}
\end{equation}
Therefore, whenever $\beta \ll \frac{K}{R^2}$, the \textsc{FedAc} guarantee is strictly better than the lower bound. This is just one example, but it illustrates that improvement over the lower bounds is possible when the second-order smoothness parameter is sufficiently small. 

A relevent question is how tight \citeauthor{yuan2020federated}'s guarantee is, and whether their algorithm might be optimal. It is actually clear that their method is \emph{not} optimal in all cases, because the guarantee does not match Local AC-SA's (\pref{cor:acc-local-sgd-optimal-quadratics}) for $\beta = 0$, i.e.~when the objective is quadratic, but in other regimes is it less clear. In order to understand to what extent a bounded third derivative might help, we provide the following lower bound:
\begin{restatable}{theorem}{homogeneousconvexlowerboundthirdorder}\label{thm:homogeneous-convex-lower-bound-third-order-smooth}
For any $H,B,\sigma,\beta,\gamma$, there is an objective $F_0 \in \mc{F}_0(H,B,\beta)$ in any dimension
\[
D \geq c\cdot\gamma^2\prn*{\frac{KR}{B^2} + \frac{H^2MK^2R^2}{\sigma^2} + \frac{H^2KR}{\sigma^2}\min\crl*{\frac{\sqrt{\sigma}(KR)^{3/4}}{\sqrt{HB}},\, \frac{\sigma^2\beta^2}{H^4}}}\log(MKR)
\]
such that the output of any algorithm in $\mc{A}^\gamma(\mc{G}_{\textrm{I.C.}},\mc{O}_g^\sigma)$ will have suboptimality at least
\[
\E F_0(\hat{x}) - F_0^* \geq c\cdot\prn*{\frac{HB^2}{K^2R^2} + \min\crl*{\frac{\sigma B}{\sqrt{MKR}},HB^2} + \min\crl*{\frac{\sigma B}{\sqrt{KR}},\ \frac{HB^2}{R^2(1+\log M)^2},\ \frac{\sqrt{\sigma \beta}B^2}{K^{1/4}R^2(1 + \log M)^{7/4}}}}
\]
\end{restatable}
This lower bound is identical to \pref{thm:homogeneous-convex-lower-bound} plus the addition of the final term in the $\min$, and indeed we prove this using the same proof as for \pref{thm:homogeneous-convex-lower-bound}. In fact, \pref{thm:homogeneous-convex-lower-bound} is proven as a corollary to this by taking $\beta$ sufficiently large that the corresponding term is irrelevant. The details of the proof can be found in \pref{app:intermittent-communication-homogeneous-convex-lower-bound-third-order-smooth}. There is a significant gap between the lower bound and existing upper bounds like \textsc{FedAc}, so this lower bound is not the final word, and there is additional work to be done.

\subsection{Boundedly-Heterogeneous Objectives}\label{subsec:bounded-heterogeneous-convex-intermittent-minimax}

In the heterogeneous intermittent communication setting, we proved that Minibatch AC-SA is an optimal algorithm. We also showed that it is essentially impossible to leverage the $K$ sequential stochastic gradients that each machine is allowed to compute in each round beyond simply computing (non-adaptively) a minibatch stochastic gradient with lower variance. However, we already saw in \pref{subsec:local-sgd-heterogeneous} that when the heterogeneous objective is not \emph{arbitrarily} heterogeneous, then can be opportunities for improvement. 

To that end, we will now consider heterogeneous intermittent communication optimization under a bounded-heterogeneity assumption. In particular, we will say that a heterogeneous objective $F(x) = \frac{1}{M}\sum_{m=1}^M F_m(x)$ is $\sdiff^2$-heterogeneous if for some $x^* \in \argmin_x F(x)$
\begin{equation}
\frac{1}{M}\sum_{m=1}^M \nrm{\nabla F_m(x^*)}^2 \leq \sdiff^2
\end{equation} 
and we will say that it is $\bar{\zeta}^2$-uniformly heterogeneous if
\begin{equation}
\sup_x \frac{1}{M}\sum_{m=1}^M \nrm{\nabla F_m(x)}^2 \leq \bar{\zeta}^2
\end{equation}
We then define the function classes $\mc{F}_0(H,B,M,\sdiff^2)$ and $\mc{F}_0(H,B,M,\bar{\zeta}^2)$ as the class of $\sdiff^2$-heterogeneous and $\bar{\zeta}^2$-uniformly heterogeneous objectives, respectively, where $F_m$ is $H$-smooth for all $m$, and $F \in \mc{F}_0(H,B)$. 

We already saw in \pref{thm:local-sgd-heterogeneous-uppper-bound} that Local SGD guarantees for any $F \in \mc{F}_0(H,B,M,\bar{\zeta}^2)$ that
\begin{equation}
\E F(\hat{x}) - F^* \leq c\cdot\prn*{\frac{HB^2}{KR} + \frac{\sigma B}{\sqrt{MKR}} + \frac{\prn*{H\sigma^2B^4}^{1/3}}{K^{1/3}R^{2/3}} + \frac{\prn*{H\bar{\zeta}^2B^4}^{1/3}}{R^{2/3}}}
\end{equation}
This algorithm is not accelerated and it is definitely not going to be optimal. However, even this algorithm's guarantee can improve over the lower bound \pref{thm:heterogeneous-convex-lower-bound} in certain regimes. For example, if $H=B=\sigma=1$, and $K \gg R^4$, then the lower bound reduces to
\begin{equation}
\E F(\hat{x}) - F^* \geq \Omega\prn*{\frac{1}{R^2} + \frac{\sigma}{\sqrt{MKR}}} = \Omega\prn*{\frac{1}{R^2}}
\end{equation}
and Local SGD's guarantee to
\begin{equation}
\E F(\hat{x}) - F^* \leq O\prn*{\frac{1}{KR} + \frac{\sigma}{\sqrt{MKR}} + \frac{1}{K^{1/3}R^{2/3}} + \frac{\bar{\zeta}^{2/3}}{R^{2/3}}} = o\prn*{\frac{1}{R^2}} + O\prn*{\frac{\bar{\zeta}^{2/3}}{R^{2/3}}}
\end{equation}
Therefore, when $\bar{\zeta}^2 \ll \frac{1}{R^4}$, then the upper bound breaks the lower bound. Of course, this constraint on $\bar{\zeta}^2$ is very tight and it requires the problem be very nearly homogeneous. Nevertheless, it indicates that improvement \emph{is} possible when the problem is not arbitrarily heterogeneous. Indeed, this raises the possibility that the assumption of $\bar{\zeta}^2$-uniform heterogeneity or the weaker constraint of $\sdiff^2$-heterogeneity might be sufficient to develop better algorithms that circumvent the pessimistic lower bound \pref{thm:heterogeneous-convex-lower-bound}. 

To test the limits of how far this could take us, we prove the following lower bound for optimizing $\sdiff^2$-heterogeneous objectives:
\begin{restatable}{theorem}{boundedheterogeneousconvexlowerbound}\label{thm:bounded-heterogeneous-convex-lower-bound}
For any $H,B,\sigma,\sdiff,\gamma$, there exists an objective $F_0 \in \mc{F}_0(H,B,M,\sdiff^2)$ in any dimension
\[
D \geq R + c\cdot\gamma^2\max\crl*{\frac{R^3}{B^2},\, \frac{H^2R}{\sdiff},\, \frac{H^2MKR^2}{\sigma^2}}\log(MKR)
\]
and an objective $F_\lambda \in \mc{F}_\lambda(H,B,M,\sdiff^2)$ in any dimension
\[
D \geq R + c\cdot\gamma^2\max\crl*{\max\crl*{\frac{\sqrt{H\lambda}}{\Delta},\,\frac{H\sqrt{H\lambda}}{\sdiff^2}}\prn*{1-\sqrt{\frac{\lambda}{H}}}^{-2R},\,\frac{H^2MKR^2}{\sigma^2}}\log(MKR)
\]
such that the output of any algorithm in $\mc{A}(\mc{G}_{I.C.},\mc{O}_{g,m}^\sigma)$ will have suboptimality at least
\begin{align*}
\E F_0(\hat{x}) - F_0^* &\geq c\cdot\prn*{\min\crl*{\frac{HB^2}{R^2},\ \frac{\sdiff B}{R}} + \min\crl*{\frac{\sigma B}{\sqrt{MKR}},\ HB^2}} \\
\E F_\lambda(\hat{x}) - F_\lambda^* &\geq c\cdot\prn*{\min\crl*{\Delta,\ H\sdiff^2}\exp\prn*{\frac{-18R\sqrt{\lambda}}{\sqrt{H}}} + \min\crl*{\frac{\sigma^2}{\lambda MKR},\ \Delta}}
\end{align*}
\end{restatable}
This is proven in \pref{app:heterogeneous-convex-lower-bound} using the same approach as for \pref{thm:heterogeneous-convex-lower-bound} and \pref{thm:heterogeneous-strongly-convex-lower-bound}. In fact, those theorems are proven as a corollary to this one by taking $\sdiff$ large enough that the $\sdiff$-dependent terms drop out. 

We see from the lower bounds that once $\sdiff$ becomes sufficiently small---smaller than $\frac{HB}{R}$ in the convex case and smaller than $\sqrt{\Delta/H}$ in the strongly convex case---there is potential for improvement over the lower bounds \pref{thm:heterogeneous-convex-lower-bound} and \pref{thm:heterogeneous-strongly-convex-lower-bound}. Of course, the lower bounds \pref{thm:homogeneous-convex-lower-bound} and \pref{thm:homogeneous-strongly-convex-lower-bound} also apply here since homogeneous objectives are a special case of heterogeneous ones, so the possibility of improvement has limits. 

This result raises the question of how large or small we should expect $\sdiff^2$ to be ``typically.'' The answer to this depends significantly on how the heterogeneity arises. We will focus on three cases in the context of machine learning training: different data sources, underdetermined problems, and randomly partitioned data.

\paragraph{Different Data Sources}
The first and probably most obvious way that heterogeneity can arise is when each parallel worker is computing stochastic gradients using data drawn from genuinely different sources. For instance, when an English language model is being trained in parallel on servers in various Anglophone countries. In this case, while the data sources are presumably somewhat related (or else why try to minimize the average of the local objectives?), there is no reason to think that it would be particularly small. Nevertheless in the convex case, if we make the assumption that the local objectives have minimizers $x^*_m$ with norm $\nrm{x^*_m} \approx \nrm{x^*}$ then by the $H$-Lipschitzness of $\nabla F_m$, we can bound
\begin{equation}
\nrm{\nabla F_m(x^*)} = \nrm{\nabla F_m(x^*) - \nabla F_m(x^*_m)} \leq H\nrm{x^*_m - x^*} \lesssim HB
\end{equation}
Therefore, it is reasonable to expect $\sdiff^2 \lesssim H^2B^2$ in the convex case, which is substantially larger than would allow for improvement over the original lower bound \pref{thm:heterogeneous-convex-lower-bound} by \pref{thm:bounded-heterogeneous-convex-lower-bound}.

\paragraph{Underdetermined Problems}
If we consider just the task of minimizing the training loss over $NM$ samples, $N$ per machine, then the local objectives will naturally be heterogeneous since they are based on different data. However, even if each machine's data comes from a completely different source, when the problem is underdetermined---meaning that there are many solutions which minimize the training loss---then $\sdiff^2 = 0$ because there will be at least one minimizer that is shared amongst all of the local objectives. However, when training machine learning models in this underdetermined regime, it is typically necessary to introduce a regularizer, often an L2 regularizer of the form $\lambda\nrm{x}^2$, to allow for better generalization performance, and the optimal regularization parameter typically scales with $\lambda \approx H/\sqrt{NM}$. In this case, since $x^*$ minimizes the unregularized local objectives, we would have
\begin{equation}
\nrm{\nabla F_m(x^*)} = \nrm*{\frac{H}{\sqrt{NM}}x^*} = \frac{HB}{\sqrt{NM}}
\end{equation}
Therefore, we can expect $\sdiff^2 \lesssim \frac{H^2B^2}{NM}$ in this regime, which would generally be small enough to hope for some improvement over the original lower bound \pref{thm:heterogeneous-convex-lower-bound}. 

\paragraph{Randomly Partitioned Data}
The final example of how heterogeneity might arise is when a large training set, all from the same source, is randomly partitioned across the $M$ machines, with $N$ samples per machine. Even when the problem is not underdetermined as in the previous example, we can again expect the level of heterogeneity to be small. In particular, for each individual sample, the expectation of the gradient of the loss of that sample at $x^*$ is zero, and $\nabla F_m(x^*)$ is the average of $N$ independent samples' gradients. Therefore, when the sample gradients have bounded variance $\sigma^2$, we would have 
\begin{equation}
\E\nrm*{\nabla F_m(x^*)}^2 = \E\nrm*{\frac{1}{N}\sum_{n=1}^N\nabla f(x^*;z^m_n)}^2 \leq \frac{\sigma^2}{N}
\end{equation}
where the expectation is over the draw of the $N$ i.i.d.~samples. Therefore, the level of heterogeneity would be bounded by $\sdiff^2 \lesssim \frac{\sigma^2}{N}$. Whether or not this is small enough for improvement over the original lower bound, of course, depends on $\sigma$ and the number of samples per machine, but it would certainly not require an unreasonably large number of samples.

\subsection{The Statistical Learning Setting: Assumptions on Components}\label{subsec:breaking-assumptions-on-components}

Stochastic optimization commonly arises in the context of statistical learning, where the goal is to minimize the expected loss with respect to a model's parameters. In this case, the objective can be written $F(x) = \E_{z\sim\mc{D}} f(x;z)$, where $z\sim\mc{D}$ represents data drawn i.i.d.~from an unknown distribution, and the ``components'' $f(x;z)$ represent the loss of the model parametrized by $x$ on the example $z$. 


For most of the results that have been presented so far, 
we only placed restrictions on the objective $F$ itself, and on the first and second moments of $g$. However, in the statistical learning setting, it is often natural to assume that the loss function $f(\cdot;z)$ itself satisfies particular properties \emph{for each $z$ individually}. For instance, for many machine learning problems, the loss $f$ is convex and smooth and furthermore, the most natural implementation of a gradient oracle is to compute $\nabla f(x;z)$ for an i.i.d.~$z\sim\mc{D}$. This is a non-trivial restriction on the stochastic gradient oracle, and it is conceivable that this property could be leveraged to design and analyze methods that converges faster than lower bounds like, for example, \pref{thm:homogeneous-convex-lower-bound} would allow. 

The specific stochastic gradient oracle \eqref{eq:thm10-gradient-oracle} used to prove \pref{thm:homogeneous-convex-lower-bound}, which zeroed out particular coordinates of the gradient depending on the query point, \emph{cannot} be written as the gradient of a random smooth function. Similarly, the gradient oracles used for several of the other lower bounds are also not expressible as the gradient of a smooth function. In this sense, these lower bound constructions are somewhat ``unnatural.'' However, we are not aware of any analysis that meaningfully exploits the fact that the gradient is given by $\nabla f(\cdot;z)$ for a smooth $f$. There are numerous papers that make this exact assumption: that $F(x) = \E_{z\sim\mc{D}}f(x;z)$ and that the stochastic gradients are given by $g = \nabla f(\cdot;z)$ for some smooth, convex $f$ \citep[e.g.][]{bottou2018optimization,nguyen2019new,koloskova2020unified,woodworth2020minibatch}. However, the \emph{purpose} of this assumption is just to bound quantities like $\E\nrm{g(x)}^2$ or $\E\nrm{g(x) - \nabla F(x)}^2$ in terms of $\sigma_*^2 = \E\nrm{g(x^*)}^2$, i.e.~the variance of the gradients at the optimum. It is, of course, useful to provide guarantees in terms just of $\sigma_*$, but we point out that the components do not necessarily have to be smooth to attain such bounds. For example, the stochatic gradients satisfying $\E\nrm{g(x)}^2 \leq \sigma_*^2 + c\cdot\nrm{x-x^*}^2$ is enough to obtain guarantees in terms just of the variance at the minimizer, and this is only a condition on the second moment of the gradient, not the components per se. Furthermore, in all of our lower bound constructions, the variance of the stochastic gradient oracles is bounded uniformly by $\sigma^2$, so $\sigma$ can always be replaced by $\sigma_*$ in our theorems. 

A very interesting question is what sorts of assumptions about the components can be leveraged to obtain better rates in the various settings we have considered, and under what conditions. Alternatively, it would also be interesting to find situations where properties like smooth components do not allow for any improvement. For example, perhaps it is possible to prove the same result as \pref{thm:homogeneous-convex-lower-bound} using a smooth gradient oracle?

\subsection{The Statistical Learning Setting: Repeated Access to Components}\label{subsec:repeated-accesses}

In the statistical learning setting, it is also natural to consider algorithms that can evaluate the gradient at multiple points for the same datum $z$. Specifically, allowing the algorithm access to a pool of samples $z_1,\dots,z_N$ drawn i.i.d.~from $\mc{D}$ and to compute $\nabla f(x;z)$ for any chosen $x$ and $z_n$ opens up additional possibilities. Indeed, \citet{arjevani2019lower} showed that multiple---even just two---accesses to each component enables substantially faster convergence ($T^{-1/3}$ vs.~$T^{-1/4}$) in sequential stochastic non-convex optimization. Similar results have been shown for zeroth-order and bandit convex optimization \citep{agarwal2010optimal,duchi2015optimal,shamir2017optimal,nesterov2017random}, where accessing each component twice allows for a quadratic improvement in the dimension-dependence. 

In sequential smooth convex optimization, if $F$ has ``finite-sum'' structure (i.e.~$\mc{D}$ is the uniform distribution on $\crl{1,\dots,N}$), then allowing the algorithm to pick a component and access it multiple times opens the door to variance-reduction techniques like SVRG \citep{johnson2013accelerating}. These methods have updates of the form:
\begin{equation}
x_{t+1} = x_t - \eta_t \prn*{\nabla f(x_t;z_t) - \nabla f(\tilde{x};z_t) + \nabla F(\tilde{x})}
\end{equation}
Computing this update therefore requires evaluating the gradient of $f(\cdot;z_t)$ at two different points, which necessitates multiple accesses to a chosen component. For finite sums, this stronger oracle access allows faster rates compared with a single-access oracle \citep[see discussion in, e.g.,][]{arjevani2020complexity}. 

Most relevantly, in the intermittent communication setting, distributed variants of SVRG are able to improve over the lower bound in \pref{thm:homogeneous-convex-lower-bound} \citep{wang2017memory,lee2017distributed,shamir2016without,woodworth2018graph}. Specifically, when the components are $H$-smooth and $L$-Lipschitz, and when the algorithm can make multiple stochastic gradient queries the same $z$, \citeauthor{woodworth2018graph} show that using distributed SVRG to optimize an empirical objective composed of suitably many samples is able to achieve convergence at the rate
\begin{equation}
\E F(\hat{x}) - F^* \leq c\cdot\prn*{\prn*{\frac{HB^2}{RK} + \frac{LB}{\sqrt{MKR}}}\log\frac{MKR}{LB}}
\end{equation}
While this guarantee (necessarily!) holds in a different setting than \pref{thm:homogeneous-convex-lower-bound}, the Lipschitz bound $L$ is generally analogous to the standard deviation of the stochastic gradient variance, $\sigma$ (indeed, $L$ is an upper bound on $\sigma$). With this in mind, this distributed SVRG algorithm can beat the lower bound in \pref{thm:homogeneous-convex-lower-bound} when $\sigma$, $L$, and $K$ are sufficiently large.

\subsection{Non-Convex Optimization with Mean Squared Smoothness}\label{subsec:non-convex-with-MSS}

We will now revisit the homogeneous intermittent communication setting with non-convex objectives. We recall that \pref{thm:homogeneous-non-convex-lower-bound} proved a lower bound on how small any intermittent communication algorithm can make the gradient of
\begin{equation}\label{eq:original-nc-lowerbound}
\E\nrm{\nabla F(\hat{x})} \geq c\cdot\min\crl*{\frac{\sqrt{H\Delta}}{\sqrt{KR}} + \frac{\sqrt{\sigma} (H\Delta)^{1/4}}{(KR)^{1/4}},\ \frac{\sqrt{H\Delta}}{\sqrt{R(1 + \log M)}} + \frac{\sqrt{\sigma}(H\Delta)^{1/4}}{(MKR)^{1/4}},\ \sqrt{H\Delta}}
\end{equation}
Our proof, which followed the general scheme described in \pref{subsec:high-level-lower-bound-approach}, involved constructing an objective whose argument is rotated by some unknown matrix $U$, and showing that any algorithm that finds a point where the gradient is small must essentially be able to  identify all of the columns of $U$. To make this more difficult, a stochastic gradient oracle was constructed such that the influence of the ``yet-unknown'' columns of $U$ is erased from the gradient with probability $1-p$, slowing progress by a factor of $p$. 

However, the responses of this stochastic gradient oracle are very sensitive to their input, because the columns of $U$ that were determined to be ``unknown'' based on a query $x$---specifically, those columns for which $\abs{\inner{U_i}{x}} \leq \alpha$---can change sharply with $x$. Consequently, the stochastic gradient oracle used in the proof of \pref{thm:homogeneous-non-convex-lower-bound} was highly non-smooth---discontinuous actually---as a function of $x$. Of course, this is allowed the context of ``independent noise'' oracles (see \pref{subsec:the-oracle}), and a reasonable algorithm (Minibatch/Single-Machine SGD) was able to match the lower bound, so there is nothing wrong with this setting. 

Nevertheless, in the statistical learning setting, we can identify a setting in which it is possible to improve over the lower bound \eqref{eq:original-nc-lowerbound} by using an algorithm which exploits a certain smoothness property of the stochastic gradients in addition to multiple queries for the same $z$. Specifically, 
we will consider the complexity of non-convex optimization in the homogeneous intermittent communication setting under the condition that the stochastic gradient oracle available to the algorithm is smooth. To quantify this, we use the notion of ``mean squared smoothness'' which has been previously considered in the non-convex optimization literature \citep{fang2018spider,lei2017non}. 
\begin{definition}\label{def:MSS}
For $F(x) = \E_{z\sim\mc{D}}f(x;z)$ equipped with a statistical learning first-order oracle which returns $\nabla f(x;z)$ for an i.i.d.~$z\sim\mc{D}$, we say that the oracle $\mc{O}_{\nabla f}^\sigma$ is $L^2$-mean squared smooth (MSS) if for all $x,y$
\[
\E_{z\sim\mc{D}}\nrm*{\nabla f(x;z) - \nabla f(y;z)}^2 \leq L^2\nrm{x-y}^2
\]
\end{definition}
We will use $\mc{O}_{\nabla f}^{\sigma,L}$ to denote an arbitrary $L^2$-MSS statistical learning first-order oracle for $F$ with variance bounded by $\sigma^2$, and we define $\mc{F}_{-L}(L,\Delta)$ to be the class of all $L$-smooth, possibly non-convex objectives with $F(0) - \min_x F(x) \leq \Delta$. We note that $L^2$-MSS is implied by $f(\cdot;z)$ being $L$-smooth, but can apply more broadly. We also note that by Jensen's inequality, $L^2$-MSS implies that $F$ is $L$-smooth. 

We also consider algorithms that may access the stochastic gradient oracle for the same $z$ multiple times. Specifically, algorithm has access to an oracle $\mc{O}_{\nabla f(\cdot;z)}^{\sigma,L}$ which, when queried with a vector $x$ returns $\nabla f(x;z)$ for an i.i.d.~$z\sim\mc{D}$, and when queried with $(x,z)$ for any previously seen $z$ returns $\nabla f(x;z)$ for the chosen $z$. 

In prior work, \citet{fang2018spider} analyzed an algorithm, \textsc{Spider}, which uses $T$ sequential queries to an oracle $\mc{O}_{\nabla f(\cdot;z)}^{\sigma,L}$ to find an approximate stationary point for any $F \in \mc{F}_{-L}(L,\Delta)$ of norm
\begin{equation}
\E\nrm*{\nabla F(\hat{x})} \leq c\cdot\min\crl*{\sqrt{L\Delta},\,\frac{\sqrt{L\Delta} + \sigma}{\sqrt{T}} + \prn*{\frac{L\sigma\Delta}{T}}^{1/3}}
\end{equation}
In the sequential seting, \citet{arjevani2019lower} show that this rate is essentially optimal and cannot be improved.

In the intermittent communication setting, as before, we consider two variants of this algorithm: Minibatch \textsc{Spider} and Single-Machine \textsc{Spider}. Minibatch \textsc{Spider} corresponds to $R$ steps of \textsc{Spider} using minibatches of size $MK$, and Single-Machine \textsc{Spider} corresponds to $KR$ steps of \textsc{Spider} using minibatches of size just $1$. Plugging the number of steps and the variance reduction implied by minibatching, we can guarantee using the better of these methods that
 \begin{equation}\label{eq:mss-upper-bound}
\E\nrm*{\nabla F(\hat{x})} \leq c\cdot\min\crl*{\frac{\sqrt{L\Delta} + \sigma}{\sqrt{KR}} + \prn*{\frac{L\sigma\Delta}{KR}}^{1/3},\, \frac{\sqrt{L\Delta}}{\sqrt{R}} + \frac{\sigma}{\sqrt{MKR}} + \prn*{\frac{L\sigma\Delta}{R\sqrt{MK}}}^{1/3},\,\sqrt{L\Delta}}
\end{equation}
The first term corresponds in the $\min$ to Single-Machine \textsc{Spider}; the second term corresponds to Minibatch \textsc{Spider}; and the last term corresponds to simply returning $0$ which, by the $L$-smoothness of $F$, has gradient norm at most $\sqrt{L\Delta}$.

Comparing this to \eqref{eq:original-nc-lowerbound}, we can see that these methods, which leverage the mean squared smoothness of the stochastic gradient oracle, are sometimes able to break the lower bound. Specifically, the upper bound \eqref{eq:mss-upper-bound} avoids dependence on any terms that scale with $(KR)^{-1/4}$ and replace them with potentially better $(KR)^{-1/3}$ or $(R\sqrt{MK})^{-1/3}$ terms instead. 

It is interesting to ask whether the combination of Minibatch and Single-Machine \textsc{Spider} might be optimal in the mean squared smooth intermittent communication setting, as \textsc{Spider} is in the sequential setting. To try to answer this question, we prove the following lower bound
\begin{restatable}{theorem}{homogeneousnonconvexMSSlowerbound}\label{thm:homogeneous-non-convex-MSS-lower-bound}
For any $L,\Delta,\sigma^2$, there exists a function $F \in \mc{F}_{-L}(L,\Delta)$ in a sufficiently large dimension $D \geq c\cdot KR\log(MKR)$ such that for any algorithm in $\mc{A}^{\infty}(\mc{G}_{\textrm{I.C}},\mc{O}_{\nabla f}^{\sigma,L})$
\[
\E\nrm*{\nabla \hat{F}_{T,U}(\hat{x})} 
\geq c\cdot\min\crl*{\frac{\sqrt{L\Delta}}{\sqrt{KR}} + \prn*{\frac{L\sigma\Delta}{KR}}^{1/3},\, \frac{\sqrt{L\Delta}}{K^{1/4}\sqrt{R}(1+\log M)^{1/4}}} + c\cdot \min\crl*{\frac{\prn*{L\sigma\Delta}^{1/3}}{\prn*{MKR}^{1/3}},\, \frac{\sqrt{L\Delta}}{(MKR)^{1/4}}}
\]
\end{restatable}
The proof of this lower bound is similar to the proof of \pref{thm:homogeneous-non-convex-lower-bound}, and is also very similar to the proof we used in the sequential setting to show the optimality of \textsc{Spider} \citep{arjevani2019lower}. In the proof of \pref{thm:homogeneous-non-convex-lower-bound} we used the stochastic gradient oracle to zero out the next relevant direction that the gradient might reveal using something like a non-smooth, discontinuous indicator function, which led to that oracle being highly non-mean squared smooth. This time, we instead use a smoothed out indicator function, which makes the oracle mean squared smooth but, of course, it makes the lower bound lower. Details of the proof can be found in \pref{app:intermittent-communication-homogeneous-non-convex-MSS-lower-bound}.
While a very similar argument sufficed to prove a lower bound that precisely matched the \textsc{Spider} guarantee in the sequential setting, there are some gaps between the upper bound \eqref{eq:mss-upper-bound} and the lower bound \pref{thm:homogeneous-non-convex-MSS-lower-bound}.

\section{Conclusion}\label{sec:conclusion}

This thesis addresses a number of theoretical questions in distributed stochastic optimization, with particular emphasis on understanding the minimax oracle complexity of distributed optimization. Answers to these theoretical questions are quite useful---they can be used to identify optimal algorithms; to identify gaps in our understanding which can prompt further study; and even when optimal algorithms are known, to shed light on additional problem structure that can be introduced and exploited to develop better, more specialized methods. 

Nevertheless, there are limits to how much we can learn from pure theory and from the concept of minimax oracle complexity in particular. In fact, there are frequently mismatches between theoretical prescriptions and practical observations. As an example, accelerated variants of common optimization algorithms like Accelerated Gradient Descent, Accelerated Stochastic Gradient Descent, Accelerated SVRG, etc.~require very carefully chosen momentum parameters in order for their convergence guarantees to hold. However, any practitioner will tell you that you should just set the momentum to some cross-validated constant value. 

Another example that is perhaps more consequential is the case of Local SGD in the intermittent communication setting. As discussed in \pref{sec:local-sgd}, the theoretical guarantees for Local SGD are not particularly impressive. In certain cases, Local SGD can fail to improve over very simple and naive baselines, and in \pref{sec:intermittent-communication-setting} we show that accelerated variants of these baselines will always dominate Local SGD or any of its accelerated variants. \emph{However}, all sorts of people use Local SGD to solve all sorts of optimization problems all the time, and it often works very well and better than the available alternatives \citep{lin2018don,zhang2016parallel,ijcai2018-447}. This suggests that there is more to Local SGD than just its worst case convergence guarantees under the particular set of assumptions that we consider. 

Moving forward, there are a number of interesting questions about the relationship between theoretical and practical properties of optimization algorithms. It is apparent that optimization algorithms are very often deployed outside of the worst-case, how should we think about studying and understanding their performance in the ``average case''? Proving theorems about optimization algorithms often requires choosing stepsizes/momentum parameters/etc.~very carefully, but how important is this, really? Does the algorithm not work with simpler parameter choices? When and why? Also, as we alluded to in \pref{sec:breaking-the-lower-bounds}, the particular details of the assumptions about the objective and oracle can have a substantial impact on the minimax oracle complexity, and on which algorithms are or are not optimal. This raises questions about which assumptions we should make and empirical questions about which choices best correspond with ``typical'' applications.

\newpage
\bibliographystyle{plainnat}
\bibliography{bibliography}
\appendix

\section{Proofs from \pref{sec:lower-bound-tools}}

\subsection{Proof of \pref{lem:prox-construction}}\label{app:generic-graph-oracle-lower-bound}

\proxconstruction*
\begin{proof}
Consider the function $G_U:\R^D\to\R$
\begin{equation}
G_U(x) = \max_{1\leq i \leq d} \max\crl*{\ell\prn*{\inner{U_i}{x} - \alpha\beta(i-1)},\,-\frac{\ell B}{\sqrt{d}},\,\frac{H}{2}\prn*{\nrm{x}^2 - B^2}-\frac{\ell B}{\sqrt{d}}}
\end{equation}
where the parameters $\beta$, $\alpha$, and $\ell$ satisfy
\begin{align}
\beta &= 3 + \frac{2\ell}{H\alpha} \\
\alpha &= \frac{B}{12d^{3/2}} \\
\ell &\leq \frac{HB}{8d^{3/2}}
\end{align}
We then define $F_U:\R^D\to\R$ to be the $H$-Moreau envelope of $G_U$:
\begin{equation}
F_U(x) = \inf_{y\in\R^D}\brk*{G_U(y) + \frac{H}{2}\nrm{x-y}^2}
\end{equation}
A key property of the $H$-Moreau envelope is that $F_U$ is $H$-smooth \citep{bauschke2011convex}. Furthermore, $G_U$ is the maximum of maxima of convex functions, so $G_U$ is convex, and therefore $F_U$ is too \citep{bauschke2011convex}. It is also easy to see that $\min_x F_U(x) = \min_x G_U(x) \geq -\frac{\ell B}{\sqrt{d}}$. Furthermore, 
\begin{equation}
G_U\prn*{-\frac{B}{\sqrt{d}}\sum_{i=1}^d U_i} = \max\crl*{-\frac{\ell B}{\sqrt{d}},\, -\frac{\ell B}{\sqrt{d}},\,\frac{H}{2}\prn*{B^2 - B^2}-\frac{\ell B}{\sqrt{d}}} = -\frac{\ell B}{\sqrt{d}}
\end{equation}
Therefore, $\min_x F_U(x) = \min_x G_U(x) = -\frac{\ell B}{\sqrt{d}}$ and
\begin{equation}
x^* = -\frac{B}{\sqrt{d}}\sum_{i=1}^d U_i \in \argmin_x G_U(x) \subseteq \argmin_x F_U(x)
\end{equation}
with $\nrm{x^*} = B$, so $F_U \in \mc{F}_0(H,B)$. 
Finally, if $\inner{U_d}{x} \geq -\alpha$ then 
\begin{equation}
F_U(x) \geq G_U(x) \geq \max\crl*{-\ell\alpha\beta d,\,-\frac{\ell B}{\sqrt{d}},\,\frac{H}{2}\prn*{\nrm{x}^2 - B^2}-\frac{\ell B}{\sqrt{d}}} \geq -\ell\alpha\beta d
\end{equation} 
so
\begin{equation}
F_U(x) - \min_x F_U(x) \geq \frac{\ell B}{\sqrt{d}} - \ell\alpha\beta d \geq \frac{\ell B}{2\sqrt{d}}
\end{equation}

We now show that the gradient mostly just depends on columns of $U$ for which the inner product with $x$ is already large. Let $x$ be any point such that $\abs{\inner{U_i}{x}} \leq \alpha$ for all $i \geq j$. The gradient $\nabla F_U(x)$ is given by \citep{bauschke2011convex}
\begin{equation}
\nabla F_U(x) = H\prn*{x - y^*}
\end{equation}
where
\begin{equation}
y^* = \argmin_y\crl*{G_U(y) + \frac{H}{2}\nrm{x-y}^2}
\end{equation}
The first order optimality condition for $y^*$ is
\begin{equation}
H(x-y^*) \in \partial G_U(y^*)
\end{equation}
Therefore, we observe that
\begin{equation}
\nabla F_U(x) \in \partial G_U(y^*) \subseteq \textrm{conv}\prn*{\ell U_1,\dots,\ell U_d,Hx}
\end{equation}
Therefore, since $U_1,\dots,U_d$ are orthogonal, $\abs{\inner{\nabla F_U(x)}{U_i}} \leq \ell + H\abs{\inner{x}{U_i}}$.

We now consider three cases:

Case 1:
\begin{equation}
G_U(y^*) = \frac{H}{2}\prn*{\nrm{y^*}^2 - B^2}-\frac{\ell B}{\sqrt{d}} \implies H y^* \in \partial G_U(y^*) \implies y^* = \frac{1}{2}x \implies \nabla F_U(x) = \frac{H}{2} x
\end{equation}
Case 2:
\begin{equation}
G_U(y^*) = -\frac{\ell B}{\sqrt{d}} \implies 0 \in \partial G_U(y^*) \implies y^* = x \implies \nabla F_U(x) = 0
\end{equation}
Case 3:
\begin{equation}
G_U(y^*) = \max_{1\leq i \leq d} \ell\prn*{\inner{U_i}{y^*} - \alpha\beta(i-1)} \implies \sup_{v \in \partial G_U(y^*)} \nrm{v} \leq \ell
\end{equation}
therefore, $\nrm{x - y^*} \leq \frac{\ell}{H}$, so $\abs{\inner{U_i}{y^*} - \inner{U_i}{x}} \leq \frac{\ell}{H}$ for all $i$, and thus
\begin{equation}
\inner{U_j}{y^*} - \alpha\beta(j-1) \geq \inner{U_i}{x} - \frac{\ell}{H} - \alpha\beta(j-1) \geq -\alpha - \frac{\ell}{H} - \alpha\beta(j-1) 
\end{equation}
and for $i > j$
\begin{equation}
\inner{U_i}{y^*} - \alpha\beta(i-1) \leq \inner{U_i}{x} + \frac{\ell}{H} - \alpha\beta(i-1) \geq \alpha + \frac{\ell}{H} - \alpha\beta(i-1)
\end{equation}
Since $\beta > 2 + \frac{2\ell}{H\alpha}$, this implies 
\begin{equation}
\inner{U_i}{y^*} - \alpha\beta(i-1) \leq \alpha + \frac{\ell}{H} - \alpha\beta(j-1) - \alpha\beta < -\alpha - \frac{\ell}{H} - \alpha\beta(j-1) \leq \inner{U_j}{y^*} - \alpha\beta(j-1)
\end{equation}
Therefore, for any $i > j$, $i \not\in \argmax_{1\leq i' \leq d} \ell\prn*{\inner{U_{i'}}{y^*} - \alpha\beta(i'-1)}$. It follows that
\begin{equation}
\nabla F_U(x) = H(x - y^*) \in \partial G_U(y^*) \subseteq \textrm{Conv}\prn*{\ell U_1,\dots,\ell U_j}
\end{equation}
Combining these three cases, we see that in any case, $\abs{\inner{U_i}{x}} \leq \alpha$ for all $i \geq j$ implies that $\nabla F_U(x)$ is a function of $x$ and $U_1,\dots,U_j$ only, and it does not depend at all on the columns $U_{j+1},\dots,U_d$. 

Furthermore, let $x$ be any point with
\begin{equation}
\nrm{x} \geq 5B \geq \frac{B}{d^{3/2}} + 4B \geq 2\prn*{\frac{2\ell}{H} + B + \frac{\sqrt{2\ell B}}{\sqrt{H}d^{1/4}}} 
\end{equation}
and let $y = \frac{1}{2}x$. Then, 
\begin{align}
\frac{H}{2}\prn*{\nrm{y}^2 - B^2}-\frac{\ell B}{\sqrt{d}} &- \max_{1\leq i \leq d} \max\crl*{\ell\prn*{\inner{U_i}{y} - \alpha\beta(i-1)},\,-\frac{\ell B}{\sqrt{d}}}\nonumber\\
&\geq \frac{H}{2}\nrm{y}^2 - \ell\nrm{y} - \frac{H}{2} B^2-\frac{\ell B}{\sqrt{d}} \\
&= \nrm{y}\prn*{\frac{H}{2}\nrm{y} - \ell} - \frac{H}{2} B^2-\frac{\ell B}{\sqrt{d}} \\
&\geq \prn*{\frac{2\ell}{H} + B + \frac{\sqrt{2\ell B}}{\sqrt{H}d^{1/4}}}\prn*{\frac{H}{2} B + \frac{\sqrt{H \ell B}}{\sqrt{2}d^{1/4}}} - \frac{H}{2} B^2-\frac{\ell B}{\sqrt{d}}
\geq 0
\end{align}
Therefore, $G_U(y) = \frac{H}{2}\prn*{\nrm{y}^2 - B^2}-\frac{\ell B}{\sqrt{d}}$ and $\partial G_U(y) = \crl{H y}$, and
\begin{equation}
H\prn*{x - y} = \frac{H}{2}x = H y \in \partial G_U(y) \implies y = \prox_{G_U}(x) \implies \nabla F_U(x) = \frac{H}{2}x
\end{equation}
Therefore, for $x$ with norm $\nrm{x} \geq \gamma$, the gradient of $F_U$ is a function of $x$ only and is independent of $U$.
\end{proof}

\section{Proofs from \pref{sec:local-sgd}}

\subsection{Proofs of \pref{thm:linear-update-alg-quadratics}} \label{app:local-sgd-homogeneous-quadratics}

\localsgdquadratics*
\begin{proof}
We will show that the average of the iterates at any particular time $\bar{x}_t = \frac{1}{M}\sum_{m=1}^M x_t^m$ evolves according to $\mc{A}$ with a lower variance stochastic gradient, even though this average iterate is not explicitly computed by the algorithm at every step. It is easily confirmed from \pref{def:linear-update-algorithm} that
\begin{align}
\bar{x}_{t+1} &= 
\frac{1}{M}\sum_{m'=1}^M\mc{L}^{(t)}_2\prn*{x_1^{m'},\dots,x_t^{m'},g\prn*{\mc{L}^{(t)}_1\prn*{x_1^{m'},\dots,x_t^{m'}};z_t^{m'}}} \\
&= \mc{L}^{(t)}_2\prn*{\bar{x}_1,\dots,\bar{x}_t,\frac{1}{M}\sum_{m'=1}^Mg\prn*{\mc{L}^{(t)}_1\prn*{x_1^{m'},\dots,x_t^{m'}};z_t^{m'}}}
\end{align}
where we used that $\mc{L}^{(t)}_2$ is linear. We will now show that $\frac{1}{M}\sum_{m'=1}^Mg\prn*{\mc{L}^{(t)}_1\prn*{x_1^{m'},\dots,x_t^{m'}};z_t^{m'}}$ is an unbiased estimate of $\nabla F\prn*{\mc{L}^{(t)}_1\prn*{\bar{x}_1,\dots,\bar{x}_t}}$ with variance bounded by $\frac{\sigma^2}{M}$. 

By the linearity of $\mc{L}^{(t)}_1$ and $\nabla F$ 
\begin{equation}
\E\brk*{\frac{1}{M}\sum_{m'=1}^Mg\prn*{\mc{L}^{(t)}_1\prn*{x_1^{m'},\dots,x_t^{m'}};z_t^{m'}}} 
= \frac{1}{M}\sum_{m'=1}^M\nabla F\prn*{\mc{L}^{(t)}_1\prn*{x_1^{m'},\dots,x_t^{m'}}} 
= \nabla F\prn*{\mc{L}^{(t)}_1\prn*{\bar{x}_1,\dots,\bar{x}_t}}
\end{equation}
Furthermore, since the stochastic gradients on each machine are independent with variance less than $\sigma^2$,
\begin{multline}
\E\nrm*{\frac{1}{M}\sum_{m'=1}^Mg\prn*{\mc{L}^{(t)}_1\prn*{x_1^{m'},\dots,x_t^{m'}};z_t^{m'}} - \E\brk*{\frac{1}{M}\sum_{m'=1}^Mg\prn*{\mc{L}^{(t)}_1\prn*{x_1^{m'},\dots,x_t^{m'}};z_t^{m'}}}}^2\\
= \frac{1}{M^2}\sum_{m=1}^M \E\nrm*{g\prn*{\mc{L}^{(t)}_1\prn*{x_1^{m},\dots,x_t^{m}};z_t^{m}} - \nabla F\prn*{\mc{L}^{(t)}_1\prn*{x_1^{m},\dots,x_t^{m}}}}^2 \leq \frac{\sigma^2}{M}
\end{multline}
Therefore, $\bar{x}_{t+1}$ is updated exactly according to $\mc{A}$ with a lower-variance stochastic gradient, and it therefore inherits the same guarantee. 
\end{proof}


\subsection{Proof of \pref{thm:local-sgd-homogeneous-upper-bound}} \label{app:local-sgd-homogeneous-upper-bound}

To prove \pref{thm:local-sgd-homogeneous-upper-bound}, we introduce some notation. Recall that the objective is of the form $F(x) := \E_{z\sim\mc{D}}\brk*{f(x;z)}$. Let $\eta_t$ denote the stepsize used for the $t\mathth$ overall iteration (i.e.~$t = k + (r-1)K$). Let $x_t^m$ denote the $t\mathth$ iterate on the $m\mathth$ machine, and let $\bxt = \frac{1}{M}\sum_{m=1}^M x_t^m$ denote the averaged $t\mathth$ iterate. The vector $\bxt$ may not actually be computed by the algorithm, but it will be central to our analysis. We will use $g(x_t^m; z_t^m)$ to denote the stochastic gradient computed at $x_t^m$ by the $m\mathth$ machine at iteration $t$, and $g_t = \frac{1}{M}\sum_{m=1}^M g(x_t^m; z_t^m)$ will denote the average of the stochastic gradients computed at time $t$. Finally, let $\bgt = \frac{1}{M}\sum_{m=1}^M \nabla F(x_t^m)$ denote the average of the exact gradients computed at the individual iterates. 

\begin{lemma}[c.f.~Lemma 3.1 \citep{stich2018local}]\label{lem:local-sgd-homogeneous-upper-bound-ourlemma31}
Let $F$ be $H$-smooth and $\lambda$-strongly convex, let\\ $\sup_x \E\nrm*{g(x;z) - \nabla F(x)}^2 \leq \sigma^2$, and let $\eta_t \leq \frac{1}{4H}$, then the iterates of Local SGD satisfy
\[
\E\brk*{F(\bxt) - F^*} \leq \prn*{\frac{2}{\eta_t} - 2\lambda}\E\nrm*{\bxt - x^*}^2 - \frac{2}{\eta_t}\E\nrm*{\bx_{t+1} -x^*}^2 + \frac{2\eta_t\sigma^2}{M} + \frac{4H}{M}\sum_{m=1}^M \E\nrm*{\bxt - x_t^m}^2
\]
\end{lemma}
\begin{proof}
This proof is nearly identical to the proof of Lemma 3.1 due to \citet{stich2018local}, but we include it in order to be self-contained. We begin by analyzing the distance of $\bx_{t+1}$ from the optimum. Below, expectations are taken over the all of the random variables $\crl*{z_t^m}$ which determine the iterates $\crl*{x_t^m}$.
\begin{align}
&\E\nrm*{\bx_{t+1} - x^*}^2 \nonumber\\
&= \E\nrm*{\bxt - \eta_t g_t - x^*}^2 \\
&= \E\nrm*{\bxt - x^*} + \eta_t^2\E\nrm*{\bgt}^2 + \eta_t^2\E\nrm*{g_t - \bgt}^2 - 2\eta_t\E\inner{\bxt - x^*}{\bgt} \\
&\leq \E\nrm*{\bxt - x^*} + \eta_t^2\E\nrm*{\bgt}^2 + \frac{\eta_t^2\sigma^2}{M} - \frac{2\eta_t}{M}\sum_{m=1}^M\E\inner{\bxt - x^*}{g(x_t^m;z_t^m)} \\
&= \E\nrm*{\bxt - x^*} + \eta_t^2\E\nrm*{\bgt}^2 + \frac{\eta_t^2\sigma^2}{M} - \frac{2\eta_t}{M}\sum_{m=1}^M\brk*{\E\inner{x_t^m - x^*}{\nabla F(x_t^m)} + \E\inner{\bxt - x_t^m}{\nabla F(x_t^m)}}\label{eq:lemma31-initial-bound}
\end{align}
For the second equality, we used that $\E\brk*{g_t - \bgt} = 0$; for the first inequality, we used that $\E\nrm*{g_t - \bgt}^2 = \E\nrm*{\frac{1}{M}\sum_{m=1}^M g(x_t^m;z_t^m) - \nabla F(x_t^m)}^2 \leq \frac{\sigma^2}{M}$ since the individual stochastic gradient estimates are independent; and for the final equality, we used that $z_t^m$ is independent of $\bxt$.

For any vectors $v_m$, $\nrm*{\sum_{m=1}^M v_m}^2 \leq M\sum_{m=1}^M \nrm*{v_m}^2$. In addition, for any point $x$ and $H$-smooth $F$, $\nrm*{\nabla F(x)}^2 \leq 2H(F(x) - F(x^*))$, thus
\begin{equation}
\eta_t^2\E\nrm*{\bgt}^2 \leq \eta_t^2M\sum_{m=1}^M \nrm*{\frac{1}{M}\nabla F(x_t^m)}^2 \leq \frac{2H\eta_t^2}{M}\sum_{m=1}^M F(x_t^m) - F(x^*)
\end{equation}
By the $\lambda$-strong convexity of $F$, we have that
\begin{multline}
    -\frac{2\eta_t}{M}\sum_{m=1}^M\inner{x_t^m - x^*}{\nabla F(x_t^m)} \leq -\frac{2\eta_t}{M}\sum_{m=1}^M \brk*{F(x_t^m) - F(x^*) + \frac{\lambda}{2}\nrm*{x_t^m - x^*}^2} \\
    \leq -\frac{2\eta_t}{M}\sum_{m=1}^M \brk*{F(x_t^m) - F(x^*)} - \lambda\eta_t\nrm*{\bxt - x^*}^2
\end{multline}
Finally, using the fact that for any vectors $a,b$ and any $\gamma > 0$, $2\inner{a}{b} \leq \gamma \nrm{a}^2 + \gamma^{-1}\nrm{b}^2$ we have
\begin{equation}
-2\eta_t\inner{\bxt - x_t^m}{\nabla F(x_t^m)} \leq \eta_t\gamma\nrm*{\bxt - x_t^m}^2 + \frac{\eta_t}{\gamma}\nrm*{\nabla F(x_t^m)}^2
\leq \eta_t\gamma\nrm*{\bxt - x_t^m}^2 + \frac{2H\eta_t}{\gamma}[F(x_t^m) - F(x^*)]
\end{equation}
Combining these with \eqref{eq:lemma31-initial-bound}, we conclude that for $\gamma = 2H$
\begin{align}
\E\nrm*{\bx_{t+1} - x^*}^2 
&\leq \prn*{1-\lambda\eta_t}\E\nrm*{\bxt - x^*} - \frac{2\eta_t\prn*{1-H\eta_t}}{M}\sum_{m=1}^M \E\brk*{F(x_t^m) - F(x^*)} + \frac{\eta_t^2\sigma^2}{M} \nonumber\\
&\qquad\qquad+ \frac{\eta_t}{M}\sum_{m=1}^M\brk*{2H\E\nrm*{\bxt - x_t^m}^2 + \E\brk*{F(x_t^m) - F(x^*)}} \\
&= \prn*{1-\lambda\eta_t}\E\nrm*{\bxt - x^*} - \frac{\eta_t\prn*{1-2H\eta_t}}{M}\sum_{m=1}^M \E\brk*{F(x_t^m) - F(x^*)} \nonumber\\
&\qquad\qquad+ \frac{\eta_t^2\sigma^2}{M} + \frac{2H\eta_t}{M}\sum_{m=1}^M\E\nrm*{\bxt - x_t^m}^2
\end{align}
By the convexity of $F$ and the fact that $\eta_t \leq \frac{1}{4H}$, this implies
\begin{align}
\E\nrm*{\bx_{t+1} - x^*}^2 
&\leq \prn*{1-\lambda\eta_t}\E\nrm*{\bxt - x^*} - \frac{\eta_t}{2}\E\brk*{F(\bxt) - F(x^*)} + \frac{\eta_t^2\sigma^2}{M} + \frac{2H\eta_t}{M}\sum_{m=1}^M\E\nrm*{\bxt - x_t^m}^2
\end{align}
Rearranging completes the proof.
\end{proof}

We will proceed to bound the final term in \pref{lem:local-sgd-homogeneous-upper-bound-ourlemma31} more tightly than was done by \citet{stich2018local}, which allows us to improve on their upper bound. To do so, we will use the following technical lemmas:
\begin{lemma}[Co-Coercivity of the Gradient]\label{lem:local-sgd-homogeneous-upper-bound-co-coercivity}
For any $H$-smooth and convex $F$, and any $x,y$
\begin{align*}
\nrm*{\nabla F(x) - \nabla F(y)}^2 &\leq H\inner{\nabla F(x) - \nabla F(y)}{x - y}\\
\nrm*{\nabla F(x) - \nabla F(y)}^2 &\leq 2H\prn*{F(x) - F(y) - \inner{\nabla F(y)}{x-y}}
\end{align*}
\end{lemma}
\begin{proof}
This proof follows closely from \citet{vandenbergheLecture}. Define the $H$-smooth, convex function
\begin{equation}
F_x(z) = F(z) - \inner{\nabla F(x)}{z}
\end{equation}
By setting its gradient equal to zero, it is clear that $x$ minimizes $F_x$ and $y$ minimizes $F_y$. For any $H$-smooth and convex $F$, for any $z$, $\nrm*{\nabla F(z)}^2 \leq 2H(F(z) - \min_x F(x))$, therefore,
\begin{equation}
F(y) - F(x) - \inner{\nabla F(x)}{y-x} 
= F_x(y) - F_x(x) 
\geq \frac{1}{2H}\nrm*{\nabla F_x(y)}^2 
= \frac{1}{2H}\nrm*{\nabla F(y) - \nabla F(x)}^2
\end{equation}
This establishes the second claim. Reversing the roles of $x$ and $y$, we also have 
\begin{equation}
F(x) - F(y) - \inner{\nabla F(y)}{x-y} 
\geq \frac{1}{2H}\nrm*{\nabla F(y) - \nabla F(x)}^2
\end{equation}
Adding these inequalities proves the first claim.
\end{proof}

\begin{lemma}[c.f.~Lemma 6 \citep{karimireddy2019scaffold}]\label{lem:local-sgd-homogeneous-upper-bound-contraction-map}
Let $F$ be any $H$-smooth and $\lambda$-strongly convex function, and let $\eta \leq \frac{1}{H}$. Then for any $x,y$
\[
\nrm*{x - \eta \nabla F(x) - y + \eta\nabla F(y)}^2 \leq \prn*{1-\lambda\eta}\nrm*{x - y}^2
\]
\end{lemma}
\begin{proof}
This Lemma and its proof are essentially identical to \citep[Lemma 6][]{karimireddy2019scaffold}, we include it here in order to keep our results self-contained, and we are more explicit about the steps used. First,
\begin{multline}
\nrm*{x - \eta \nabla F(x) - y + \eta\nabla F(y)}^2
= \nrm*{x-y}^2 + \eta^2 \nrm*{\nabla F(x) - \nabla F(y)}^2 - 2\eta\inner{\nabla F(x) - \nabla F(y)}{x-y} \\
\leq \nrm*{x-y}^2 + \eta^2H\inner{\nabla F(x) - \nabla F(y)}{x - y}  - 2\eta\inner{\nabla F(x) - \nabla F(y)}{x-y}
\end{multline}
where the inequality follows from \pref{lem:local-sgd-homogeneous-upper-bound-co-coercivity}. Since $\eta H \leq 1$, we further conclude that
\begin{equation}
\nrm*{x - \eta \nabla F(x) - y + \eta\nabla F(y)}^2
\leq \nrm*{x-y}^2 - \eta\inner{\nabla F(x) - \nabla F(y)}{x-y}
\end{equation}
Finally, by the $\lambda$-strong convexity of $F$
\begin{gather}
\inner{\nabla F(x)}{x-y} \geq F(x) - F(y) + \frac{\lambda}{2}\nrm*{x-y}^2 \\
-\inner{\nabla F(y)}{x-y} \geq F(y) - F(x) + \frac{\lambda}{2}\nrm*{x-y}^2
\end{gather}
Combining these, we conclude
\begin{align}
\nrm*{x - \eta \nabla F(x) - y + \eta\nabla F(y)}^2
&\leq \nrm*{x-y}^2 - \eta\inner{\nabla F(x) - \nabla F(y)}{x-y} \\
&\leq \nrm*{x-y}^2 - \eta\lambda \nrm*{x-y}^2
\end{align}
which completes the proof.
\end{proof}

\begin{lemma}\label{lem:local-sgd-homogeneous-upper-bound-distance-to-average-distance-to-other}
For any $t$ and $m \neq m'$
\[
\E\nrm*{x_t^m - \bxt}^2 \leq \frac{M-1}{M}\E\nrm*{x_t^m - x_t^{m'}}^2
\]
\end{lemma}
\begin{proof}
First, we note that $x_t^1,\dots,x_t^M$ are identically distributed. Therefore,
\begin{align}
\E\nrm*{x_t^m - \bxt}^2
&= \E\nrm*{x_t^m - \frac{1}{M}\sum_{m'=1}^M x_t^{m'}}^2 \\
&= \frac{1}{M^2}\E\nrm*{\sum_{m'=1}^M x_t^m - x_t^{m'}}^2 \\
&= \frac{1}{M^2}\brk*{\sum_{m'=1}^M\E\nrm*{x_t^m - x_t^{m'}}^2 + \sum_{m'\neq m''} \E\inner{x_t^m - x_t^{m'}}{x_t^m - x_t^{m''}}} \\
&\leq \frac{1}{M^2}\brk*{(M-1)\E\nrm*{x_t^m - x_t^{m'}}^2 + \sum_{m'\neq m''} \sqrt{\E\nrm*{x_t^m - x_t^{m'}}^2\E\nrm*{x_t^m - x_t^{m''}}^2}} \\
&= \frac{1}{M^2}\brk*{(M-1)\E\nrm*{x_t^m - x_t^{m'}}^2 + 2\binom{M-1}{2} \E\nrm*{x_t^m - x_t^{m'}}^2} \\
&= \frac{(M-1)^2}{M^2}\E\nrm*{x_t^m - x_t^{m'}}^2 \\
&\leq \frac{M-1}{M}\E\nrm*{x_t^m - x_t^{m'}}^2
\end{align}
\end{proof}

\begin{lemma}\label{lem:local-sgd-homogeneous-upper-bound-distance-bound-between-local-iterates}
Under the conditions of \pref{lem:local-sgd-homogeneous-upper-bound-ourlemma31}, with the additional condition that the sequence of stepsizes $\eta_1,\eta_2,\dots$ is non-increasing and $\eta_t \leq \frac{1}{H}$ for all $t$, for any $t$ and any $m$
\[
\E\nrm*{x_t^m - \bxt}^2 \leq \frac{2(M-1)(K-1)\eta_{t-K+1 \land 0}^2\sigma^2}{M}
\]
If $\eta_t = \frac{2}{\lambda(a+t+1)}$ for $a \geq \frac{2H}{\lambda}$, then it further satisfies
\[
\E\nrm*{x_t^m - \bxt}^2 \leq \frac{2(M-1)(K-1)\eta_{t-1}^2\sigma^2}{M}
\]
\end{lemma}
\begin{proof}
By \pref{lem:local-sgd-homogeneous-upper-bound-distance-to-average-distance-to-other}, we can upper bound
\begin{equation}
    \E\nrm*{x_t^m - \bxt}^2 \leq \frac{M-1}{M}\E\nrm*{x_t^m - x_t^{m'}}^2
\end{equation}
for all $t$ and $m\neq m'$. In addition,
\begin{align}
\E\nrm*{x_t^m - x_t^{m'}}^2
&= \E\nrm*{x_{t-1}^m - \eta_{t-1} g(x_{t-1}^m;z_{t-1}^m) - x_{t-1}^{m'} + \eta_{t-1} g(x_{t-1}^{m'};z_{t-1}^{m'})}^2 \\
&\leq \E\nrm*{x_{t-1}^m - \eta_{t-1} \nabla F(x_{t-1}^m) - x_{t-1}^{m'} + \eta_{t-1} \nabla F(x_{t-1}^{m'})}^2 + 2\eta_{t-1}^2\sigma^2 \\
&\leq \prn*{1-\lambda\eta_{t-1}}\E\nrm*{x_{t-1}^m - x_{t-1}^{m'}}^2 + 2\eta_{t-1}^2\sigma^2
\end{align}
where for the final inequality we used \pref{lem:local-sgd-homogeneous-upper-bound-contraction-map} and the fact that the stepsizes are less than $\frac{1}{H}$. Since the iterates are averaged every $K$ iterations, for each $t$, there must be a $t_0$ with $0 \leq t-t_0 \leq K-1$ such that $x_{t_0}^m = x_{t_0}^{m'}$. Therefore, we can unroll the recurrence above to conclude that
\begin{equation}
\E\nrm*{x_t^m - x_t^{m'}}^2
\leq \sum_{i=t_0}^{t-1}2\eta_i^2\sigma^2\prod_{j=i+1}^{t-1}\prn*{1-\lambda\eta_{j}} \leq 2\sigma^2\sum_{i=t_0}^{t-1}\eta_i^2
\end{equation}
where we define $\sum_{i=a}^b c_i = 0$ and $\prod_{i=a}^b c_i = 1$ for all $a > b$ and all $\crl{c_i}_{i\in\mathbb{N}}$. Therefore, for any non-increasing stepsizes, we conclude
\begin{equation}
\E\nrm*{x_t^m - \bxt}^2 \leq \frac{2\eta_{t-K+1 \land 0}^2\sigma^2(M-1)(K-1)}{M}
\end{equation}
This implies the first claim.

In the special case $\eta_t = \frac{2}{\lambda \prn*{a + t + 1}}$, we have 
\begin{align}
\E\nrm*{x_t^m - x^{m'}}^2 
&\leq 2\sigma^2 \sum_{i = t_0}^{t-1} \eta_i^2 \prod_{j=i+1}^{t-1}\prn*{1-\lambda\eta_j} \\
&= 2\sigma^2 \sum_{i = t_0}^{t-1} \eta_i^2 \prod_{j=i+1}^{t-1}\prn*{\frac{a + j - 1}{a + j + 1}} \\
&= 2\sigma^2\eta_{t-1}^2 + \frac{2\sigma^2\eta_{t-2}^2(a+t-2)}{a+t} + 2\sigma^2 \sum_{i = t_0}^{t-3} \eta_i^2 \frac{(a + i)(a + i + 1)}{(a+t-1)(a+t)} \\
&= 2\sigma^2\eta_{t-1}^2\prn*{1 + \frac{(a+t)(a+t-2)}{(a+t-1)^2} + \sum_{i = t_0}^{t-3}\frac{(a + i)(a + t)}{(a+t-1)(a+i+1)}} \\
&\leq 2\sigma^2\eta_{t-1}^2\prn*{t-t_0} \\
&\leq 2(K-1)\sigma^2\eta_{t-1}^2
\end{align}
This implies the second claim.
\end{proof}

Next, we will show that Local SGD is always at least as good as $KR$ steps of sequential SGD. To do so, we use the following result from \citet{stich2019unified}:
\begin{lemma}[Lemma 3 \citep{stich2019unified}]\label{lem:local-sgd-homogeneous-upper-bound-stich-recurrence}
For any recurrence of the form
\[
r_{t+1} \leq (1-a\gamma_t)r_t - b\gamma_t s_t + c\gamma_t^2
\]
with $a, b > 0$, there exists a sequence $0 < \gamma_t \leq \frac{1}{d}$ and weights $w_t > 0$ such that
\[
\frac{b}{W_T}\sum_{t=0}^T \brk*{s_t w_t + a r_{t+1}} \leq 32d r_0 \exp\prn*{-\frac{aT}{2d}} + \frac{36c}{aT}
\] 
where $W_T := \sum_{t=0}^T w_t$.
\end{lemma}

We now argue that Local SGD is never worse than $KR$ steps of sequential SGD:
\begin{lemma}\label{lem:local-sgd-homogeneous-upper-bound-local-not-worse-than-single}
Let $(f,\mc{D}) \in \mc{F}(H,\lambda,B,\sigma^2)$. When $\lambda=0$, an appropriate average of the iterates of Local SGD with an optimally tuned constant stepsize satisfies for a universal constant $c$
\[
F(\hat{x}) - F^* \leq c\cdot \frac{HB^2}{KR} + c\cdot\frac{\sigma B}{\sqrt{KR}}
\]
In the case $\lambda > 0$, then an appropriate average of the iterates of Local SGD with decreasing stepsize $\eta_t \asymp (\lambda t)^{-1}$ satisfies for a universal constant $c$
\[
F(\hat{x}) - F^* \leq c\cdot HB^2 \exp\prn*{-\frac{\lambda KR}{4H}} + c\cdot\frac{\sigma^2}{\lambda KR}
\]
\end{lemma}
\begin{proof}
Define $T := KR$ and consider the $(t+1)$st iterate on some machine $m$, $x_{t+1}^m$. If $t+1 \mod K \neq 0$, then $x_{t+1}^m = x_t^m - \eta_tg(x_t^m;z_t^m)$. In this case, for $\eta_t \leq \frac{1}{2H}$
\begin{align}
\E\nrm*{x_{t+1}^m - x^*}^2
&= \E\nrm*{x_t^m - \eta_tg(x_t^m;z_t^m) - x^*}^2 \\
&= \E\nrm*{x_t^m - x^*}^2 + \eta_t^2 \E\nrm*{g(x_t^m;z_t^m)}^2 - 2\eta_t\E\inner{g(x_t^m;z_t^m)}{x_t^m - x^*} \\
&\leq \E\nrm*{x_t^m - x^*}^2 + \eta_t^2\sigma^2 + \eta_t^2 \E\nrm*{\nabla F(x_t^m)}^2 - 2\eta_t\E\inner{\nabla F(x_t^m)}{x_t^m - x^*} \\
&\leq \E\nrm*{x_t^m - x^*}^2 + \eta_t^2\sigma^2 + 2H\eta_t^2 \E\brk*{F(x_t^m) - F^*} - 2\eta_t\E\brk*{F(x_t^m) - F^* + \frac{\lambda}{2}\nrm*{x_t^m - x^*}^2} \\
&= (1-\lambda\eta_t)\E\nrm*{x_t^m - x^*}^2 + \eta_t^2\sigma^2 - 2\eta_t(1-H\eta_t)\E\brk*{F(x_t^m) - F^*} 
\end{align}
Therefore, 
\begin{equation}
\E\brk*{F(x_t^m) - F^*} 
\leq \prn*{\frac{1}{\eta_t}-\lambda}\E\nrm*{x_t^m - x^*}^2 - \frac{1}{\eta_t}\E\nrm*{x_{t+1}^m - x^*}^2 + \eta_t\sigma^2
\end{equation}
For the first inequality above, we used the variance bound on the stochastic gradients; for the second inequality we used the $H$-smoothness and $\lambda$-strong convexity of $F$; and for the final inequality we used that $H\eta_t \leq \frac{1}{2}$ and rearranged.

If, on the other hand, $t+1 \mod K = 0$, then $x_{t+1}^m = \frac{1}{M}\sum_{m'=1}^M x_t^{m'} - \eta_tg(x_t^{m'};z_t^{m'})$. Since the local iterates on the different machines are \emph{identically distributed},
\begin{align}
\E\nrm*{x_{t+1}^m - x^*}^2 
&= \E\nrm*{\frac{1}{M}\sum_{m'=1}^M x_t^{m'} - \eta_tg(x_t^{m'};z_t^{m'}) - x^*}^2 \\
&\leq \frac{1}{M}\sum_{m'=1}^M\E\nrm*{x_t^{m'} - \eta_tg(x_t^{m'};z_t^{m'}) - x^*}^2 \\
&= \E\nrm*{x_t^m - \eta_tg(x_t^m;z_t^m) - x^*}^2
\end{align}
Where for the first inequality we used Jensen's inequality, and for the final equality we used that the local iterates are identically distributed. From here, using the same computation as above, we conclude that in either case
\begin{equation}\label{eq:local-sgd-homogeneous-upper-bound-serial-sgd-recurrence}
\E\brk*{F(x_t^m) - F^*} 
\leq \prn*{\frac{1}{\eta_t}-\lambda}\E\nrm*{x_t^m - x^*}^2 - \frac{1}{\eta_t}\E\nrm*{x_{t+1}^m - x^*}^2 + \eta_t\sigma^2
\end{equation}

\paragraph{Weakly Convex Case $\lambda = 0$:}
Choose a constant learning rate $\eta_t = \eta = \min\crl*{\frac{1}{2H}, \frac{B}{\sigma\sqrt{T}}}$ and define the averaged iterate
\begin{equation}
    \hat{x} = \frac{1}{MT}\sum_{m=1}^M\sum_{t=1}^T x_t^m
\end{equation}
Then, by the convexity of $F$:
\begin{align}
\E F(\hat{x}) - F^*
&\leq \frac{1}{MT}\sum_{m=1}^M\sum_{t=1}^T\E\brk*{F(x_t^m) - F^*} \\
&\leq \frac{1}{MT}\sum_{m=1}^M\sum_{t=1}^T\frac{1}{\eta}\E\nrm*{x_t^m - x^*}^2 - \frac{1}{\eta}\E\nrm*{x_{t+1}^m - x^*}^2 + \eta\sigma^2 \\
&= \frac{\nrm*{x_0 - x^*}^2}{T\eta} + \eta\sigma^2 \\
&= \max\crl*{\frac{2H\nrm*{x_0 - x^*}^2}{T}, \frac{\sigma \nrm*{x_0 - x^*}}{\sqrt{T}}} + \frac{\sigma \nrm*{x_0 - x^*}}{\sqrt{T}} \\
&\leq \frac{2H\nrm*{x_0 - x^*}^2}{T} + \frac{2\sigma \nrm*{x_0 - x^*}}{\sqrt{T}}
\end{align}

\paragraph{Strongly Convex Case $\lambda > 0$:}
Rearranging \eqref{eq:local-sgd-homogeneous-upper-bound-serial-sgd-recurrence}, we see that it has the same form as the recurrence analyzed in \pref{lem:local-sgd-homogeneous-upper-bound-stich-recurrence} with $r_t = \E\nrm*{x_t^m - x^*}^2$, $s_t = \E\brk*{F(x_t^m) - F^*}$, $a = \lambda$, $c = \sigma^2$, and $\gamma_t = \eta_t$ with the requirement that $\eta_t \leq \frac{1}{2H}$, i.e.~$d=2H$. Consequently, by \pref{lem:local-sgd-homogeneous-upper-bound-stich-recurrence}, we conclude that there is a sequence of stepsizes and weights $w_t$ such that
\begin{align}
\E\brk*{F\prn*{\frac{1}{M\sum_{t=0}^{KR} w_t} \sum_{m=1}^M\sum_{t=0}^{KR} w_t x_t^m} - F^*}
&\leq \frac{1}{M\sum_{t=0}^{KR} w_t} \sum_{m=1}^M\sum_{t=0}^{KR} \E\brk*{F\prn*{w_t x_t^m} - F^*} \\
&\leq 64H\E\nrm*{x_0 - x^*}^2 \exp\prn*{-\frac{\lambda KR}{4H}} + \frac{36\sigma^2}{\lambda KR}
\end{align}
The stepsizes and weights are chosen as follows:
If $KR \leq \frac{2H}{\lambda}$, then $\eta_t = \frac{1}{2H}$ and $w_t = (1-\lambda\eta)^{-t-1}$.
If $KR > \frac{2H}{\lambda}$ and $t < KR/2$, then $\eta_t = \frac{1}{2H}$ and $w_t = 0$.
If $KR > \frac{2H}{\lambda}$ and $t \geq KR/2$, then $\eta_t = \frac{2}{4H + \lambda(t - KR/2)}$ and $w_t = (4H/\lambda + t - KR/2)^2$.
This completes the proof.
\end{proof}

Finally, we prove our main analysis of Local SGD. Portions of the analysis of the strongly convex case follow closely the proof of \citep[Lemma 3][]{stich2019unified}.
\localsgdhomogeneousupperbound*
\begin{proof}
We will prove the first terms in the $\min$'s in Theorem in two parts, first for the convex case $\lambda = 0$, then for the strongly convex case $\lambda > 0$. Then, we conclude by invoking \pref{lem:local-sgd-homogeneous-upper-bound-local-not-worse-than-single} showing that Local SGD is never worse than $KR$ steps of SGD on a single machine, which corresponds to the second terms in the $\min$'s in the Theorem statement.

\paragraph{Convex Case $\lambda = 0$:}
By \pref{lem:local-sgd-homogeneous-upper-bound-ourlemma31} and the first claim of \pref{lem:local-sgd-homogeneous-upper-bound-distance-bound-between-local-iterates}, the mean iterate satisfies
\begin{equation}
\E\brk*{F(\bxt) - F^*} \leq \frac{2}{\eta_t}\E\nrm*{\bxt - x^*}^2 - \frac{2}{\eta_t}\E\nrm*{\bx_{t+1} -x^*}^2 + \frac{2\eta_t\sigma^2}{M} + \frac{8H(M-1)(K-1)\eta_{t-K+2 \land 0}^2\sigma^2}{M}
\end{equation}
Consider a fixed stepsize $\eta_t = \eta$ which will be chosen later, and consider the average of the iterates
\begin{equation}
    \hat{x} = \frac{1}{KR}\sum_{t=1}^{KR} \bxt
\end{equation}
By the convexity of $F$,
\begin{align}
\E \brk*{F(\hat{x}) - F^*}
&\leq \frac{1}{KR}\sum_{t=1}^{KR} \E\brk*{F(\bxt) - F^*} \\
&\leq \frac{1}{KR}\sum_{t=1}^{KR} \brk*{\frac{2}{\eta}\E\nrm*{\bxt - x^*}^2 - \frac{2}{\eta}\E\nrm*{\bx_{t+1} -x^*}^2 + \frac{2\eta\sigma^2}{M} + \frac{8H(M-1)(K-1)\eta^2\sigma^2}{M}} \\
&\leq \frac{2B^2}{\eta KR} + \frac{2\eta\sigma^2}{M} + \frac{8H(M-1)(K-1)\eta^2\sigma^2}{M}
\end{align}
Choose as a stepsize 
\begin{equation}
\eta = \begin{cases}
\min\crl*{\frac{1}{4H},\ \frac{B\sqrt{M}}{\sigma\sqrt{KR}}} & K=1\textrm{ or }M=1 \\
\min\crl*{\frac{1}{4H},\ \frac{B\sqrt{M}}{\sigma\sqrt{KR}},\ \prn*{\frac{B^2}{H\sigma^2K^2 R}}^{\frac{1}{3}}} & \textrm{Otherwise }
\end{cases}
\end{equation}
Then,
\begin{align}
\E \brk*{F(\hat{x}) - F^*}
&\leq \frac{2B^2}{\eta KR} + \frac{2\eta\sigma^2}{M} + \frac{8H(M-1)(K-1)\eta^2\sigma^2}{M} \\
&\leq \max\crl*{\frac{8HB^2}{KR},\ \frac{2\sigma B}{\sqrt{MKR}},\ \frac{2\prn*{H\sigma^2B^4}^{\frac{1}{3}}}{K^{1/3}R^{2/3}}} + \frac{2\sigma B}{\sqrt{MKR}} + \frac{8\prn*{H\sigma^2B^4}^{\frac{1}{3}}}{K^{1/3}R^{2/3}} \\
&\leq \frac{8HB^2}{KR} + \frac{4\sigma B}{\sqrt{MKR}} + \frac{10\prn*{H\sigma^2B^4}^{\frac{1}{3}}}{K^{1/3}R^{2/3}}
\end{align}

\paragraph{Strongly Convex Case $\lambda > 0$:}
For the strongly convex case, following \citet{stich2019unified}'s proof of \pref{lem:local-sgd-homogeneous-upper-bound-stich-recurrence}, we choose stepsizes according to the following set of cases:
If $KR \leq \frac{2H}{\lambda}$, then $\eta_t = \frac{1}{4H}$ and $w_t = (1-\lambda\eta)^{-t-1}$.
If $KR > \frac{2H}{\lambda}$ and $t \leq KR/2$, then $\eta_t = \frac{1}{4H}$ and $w_t = 0$.
If $KR > \frac{2H}{\lambda}$ and $t > KR/2$, then $\eta_t = \frac{2}{8H + \lambda(t - KR/2)}$ and $w_t = (8H/\lambda + t - KR/2)$.
We note that in the second and third cases, the stepsize is either constant or equal to $\eta_t = \frac{2}{\lambda(a + t - KR/2)}$ (for $a = \frac{8H}{\lambda}$) within each individual round of communication. 

By \pref{lem:local-sgd-homogeneous-upper-bound-ourlemma31} and the first claim of \pref{lem:local-sgd-homogeneous-upper-bound-distance-bound-between-local-iterates}, during the rounds of communication for which the stepsize is constant, we have the recurrence:
\begin{equation}\label{eq:local-sgd-homogeneous-upper-bound-recurrent-small-t}
\E\nrm*{\bx_{t+1} - x^*}^2 \leq \prn*{1 - \lambda\eta_t}\E\nrm*{\bx_t - x^*}^2 - \frac{\eta_t}{2}\E\brk*{F(\bx_t) - F^*} + \frac{\eta_t^2\sigma^2}{M} + 4HK\eta_t^3\sigma^2
\end{equation}
On the other hand, during the rounds of communication in which the stepsize is decreasing, we have by \pref{lem:local-sgd-homogeneous-upper-bound-ourlemma31} and the second claim of \pref{lem:local-sgd-homogeneous-upper-bound-distance-bound-between-local-iterates} that:
\begin{equation}
\E\nrm*{\bx_{t+1} - x^*}^2 \leq \prn*{1 - \lambda\eta_t}\E\nrm*{\bx_t - x^*}^2 - \frac{\eta_t}{2}\E\brk*{F(\bx_t) - F^*} + \frac{\eta_t^2\sigma^2}{M} + 4HK\eta_t\eta_{t-1}^2\sigma^2
\end{equation}
Furthermore, during the rounds (i.e.~when $t > KR$) where the stepsize is decreasing, 
\begin{equation}
\eta_{t-1}^2 = \eta_t^2 \frac{\prn*{a + t - KR/2}^2}{\prn*{a - 1 + t - KR/2}^2} \leq 4\eta_t^2
\end{equation}
So, for every $t$ we conclude
\begin{equation}
\E\nrm*{\bx_{t+1} - x^*}^2 \leq \prn*{1 - \lambda\eta_t}\E\nrm*{\bx_t - x^*}^2 - \frac{\eta_t}{2}\E\brk*{F(\bx_t) - F^*} + \frac{\eta_t^2\sigma^2}{M} + 16HK\eta_t^3\sigma^2
\end{equation}

First, suppose $KR > \frac{2H}{\lambda}$, and consider the steps during which $\eta_t = \frac{1}{4H}$:
\begin{align}
\E\nrm*{\bx_{KR/2} - x^*}^2 
&\leq \prn*{1 - \frac{\lambda}{4H}}\E\nrm*{\bx_t - x^*}^2 - \frac{1}{8H}\E\brk*{F(\bx_t) - F^*} + \frac{\sigma^2}{16H^2 M} + \frac{K\sigma^2}{4H^2} \\
&\leq \prn*{1 - \frac{\lambda}{4H}}\E\nrm*{\bx_t - x^*}^2 + \frac{\sigma^2}{16H^2 M} + \frac{K\sigma^2}{4H^2} \\
&\leq \prn*{1 - \frac{\lambda}{4H}}^{KR/2}\E\nrm*{\bx_0 - x^*}^2 + \prn*{\frac{\sigma^2}{16H^2 M} + \frac{K\sigma^2}{4H^2}}\sum_{t=0}^{KR/2 - 1} \prn*{1 - \frac{\lambda}{4H}}^{t} \\
&\leq \prn*{1 - \frac{\lambda}{4H}}^{KR/2}\E\nrm*{\bx_0 - x^*}^2 + \frac{4H}{\lambda}\prn*{\frac{\sigma^2}{16H^2 M} + \frac{K\sigma^2}{4H^2}} \\
&\leq \E\nrm*{\bx_0 - x^*}^2\exp\prn*{-\frac{\lambda KR}{8H}} + \frac{\sigma^2}{4H\lambda M} + \frac{K\sigma^2}{H\lambda}\label{eq:local-sgd-homogeneous-upper-bound-recurrence-mess}
\end{align}

Now, consider the remaining steps. Rearranging, we have
\begin{align}
\E\brk*{F(\bx_t) - F^*}
&\leq \prn*{\frac{2}{\eta_t} - \frac{\lambda}{2}}\E\nrm*{\bx_t - x^*}^2 - \frac{2}{\eta_t}\E\nrm*{\bx_{t+1} - x^*}^2 + \frac{2\eta_t\sigma^2}{M} + 32HK\eta_t^2\sigma^2
\end{align}
So, since $\eta_t = \frac{2}{\lambda(a+t)}$ where $a = \frac{8H}{\lambda} - \frac{KR}{2}$ and $w_t = (a+t)$, we have
\begin{align}
&\frac{1}{W_T}\sum_{t=KR/2}^{KR} w_t \E\brk*{F(\bx_t) - F^*} \nonumber\\
&\leq \frac{1}{W_T}\sum_{t=KR/2}^{KR}w_t\brk*{\prn*{\frac{2}{\eta_t} - 2\lambda}\E\nrm*{\bx_t - x^*}^2 - \frac{2}{\eta_t}\E\nrm*{\bx_{t+1} - x^*}^2 + \frac{2\eta_t\sigma^2}{M} + 32HK\eta_t^2\sigma^2} \\
&= \frac{1}{W_T}\sum_{t=KR/2}^{KR}\lambda(a+t)(a+t-2)\E\nrm*{\bx_t - x^*}^2 - \lambda(a+t)^2\E\nrm*{\bx_{t+1} - x^*}^2 + \frac{2\sigma^2}{\lambda M} + \frac{32HK\eta_t\sigma^2}{\lambda} \\
&\leq \frac{1}{W_T}\sum_{t=KR/2}^{KR}\lambda(a+t-1)^2\E\nrm*{\bx_t - x^*}^2 - \lambda(a+t)^2\E\nrm*{\bx_{t+1} - x^*}^2 + \frac{2\sigma^2}{\lambda M} + \frac{32HK\eta_t\sigma^2}{\lambda} \\
&\leq \frac{\lambda(a+KR/2-1)^2}{W_T}\E\nrm*{\bx_{KR/2} - x^*}^2 + \frac{2\sigma^2 (KR/2)}{W_T\lambda M} + \frac{64HK\sigma^2}{W_T\lambda^2}\sum_{t=KR/2}^{KR}\frac{1}{a+t} \\
&= \frac{\lambda\prn*{\frac{8H}{\lambda}-1}^2}{W_T}\E\nrm*{\bx_{KR/2} - x^*}^2 + \frac{2\sigma^2 (KR/2)}{W_T\lambda M} + \frac{64HK\sigma^2}{W_T\lambda^2}\sum_{t'=1}^{KR/2}\frac{1}{\frac{8H}{\lambda}+t'} \\
&\leq \frac{64H^2}{W_T \lambda}\E\nrm*{\bx_{KR/2} - x^*}^2 + \frac{2\sigma^2 (KR/2)}{W_T\lambda M} + \frac{64HK\sigma^2}{W_T\lambda^2}\log\prn*{e + \frac{\lambda KR}{4H}}
\end{align}
Finally, we recall \eqref{eq:local-sgd-homogeneous-upper-bound-recurrence-mess}, $KR > \frac{2H}{\lambda}$, and note that $W_T = \sum_{t=KR/2}^{KR} a+t \geq \frac{3K^2R^2}{8} + \frac{aKR}{2} = \frac{K^2R^2}{8} + \frac{4HKR}{\lambda} \geq \frac{8H^2}{\lambda^2}$ thus
\begin{align}
&\frac{1}{W_T}\sum_{t=KR/2}^{KR} w_t \E\brk*{F(\bx_t) - F^*} \nonumber\\
&\leq \frac{64H^2}{W_T \lambda}\prn*{\E\nrm*{\bx_0 - x^*}^2\exp\prn*{-\frac{\lambda KR}{8H}} + \frac{\sigma^2}{4H\lambda M} + \frac{K\sigma^2}{H\lambda}} + \frac{2\sigma^2 (KR/2)}{W_T\lambda M} + \frac{64HK\sigma^2}{W_T\lambda^2}\log\prn*{e + \frac{\lambda KR}{4H}}\\
&\leq \frac{64H^2}{W_T \lambda}\E\nrm*{\bx_0 - x^*}^2\exp\prn*{-\frac{\lambda KR}{8H}} + \frac{16H\sigma^2}{\lambda^2 M W_T} + \frac{64HK\sigma^2}{\lambda^2 W_T} + \frac{8\sigma^2}{\lambda MKR} + \frac{512H\sigma^2}{\lambda^2KR^2}\log\prn*{e + \frac{\lambda KR}{4H}} \\
&\leq 8\lambda\E\nrm*{\bx_0 - x^*}^2\exp\prn*{-\frac{\lambda KR}{8H}} + \frac{4\sigma^2}{\lambda M KR} + \frac{512H\sigma^2}{\lambda^2KR^2} + \frac{8\sigma^2}{\lambda MKR} + \frac{512H\sigma^2}{\lambda^2KR^2}\log\prn*{e + \frac{\lambda KR}{4H}} \\
&\leq 8\lambda\E\nrm*{\bx_0 - x^*}^2\exp\prn*{-\frac{\lambda KR}{8H}} + \frac{12\sigma^2}{\lambda M KR} + \frac{512H\sigma^2}{\lambda^2KR^2}\log\prn*{9 + \frac{\lambda KR}{H}}
\end{align}
This concludes the proof for the case $KR > \frac{2H}{\lambda}$. 

If $KR \leq \frac{2H}{\lambda}$, we use the constant stepsize $\eta_t = \eta$ and weights $w_t = (1-\lambda\eta)^{-t-1}$. Rearranging \eqref{eq:local-sgd-homogeneous-upper-bound-recurrent-small-t} therefore gives
\begin{align}
\E\brk*{F(\bx_t) - F^*} 
&\leq \frac{2}{\eta}\prn*{1 - \lambda\eta}\E\nrm*{\bx_t - x^*}^2 - \frac{2}{\eta}\E\nrm*{\bx_{t+1} - x^*}^2 + \frac{2\eta\sigma^2}{M} + 8HK\eta^2\sigma^2
\end{align}
so
\begin{align}
&\frac{1}{W_T}\sum_{t=1}^{KR}w_t\E\brk*{F(\bx_t) - F^*} \nonumber\\
&\leq \frac{1}{W_T}\sum_{t=1}^{KR}w_t\brk*{\frac{2}{\eta}\prn*{1 - \lambda\eta}\E\nrm*{\bx_t - x^*}^2 - \frac{2}{\eta}\E\nrm*{\bx_{t+1} - x^*}^2 + \frac{2\eta\sigma^2}{M} + 8HK\eta^2\sigma^2} \\
&= \frac{1}{W_T}\sum_{t=1}^{KR}\brk*{\frac{2}{\eta}\prn*{1 - \lambda\eta}^{-t}\E\nrm*{\bx_t - x^*}^2 - \frac{2}{\eta}\prn*{1 - \lambda\eta}^{-(t+1)}\E\nrm*{\bx_{t+1} - x^*}^2} + \frac{2\eta\sigma^2}{M} + 8HK\eta^2\sigma^2 \\
&\leq \frac{2\E\nrm*{\bx_0 - x^*}^2}{\eta W_T} + \frac{2\eta\sigma^2}{M} + 8HK\eta^2\sigma^2
\end{align}
Finally, we note that $W_T \geq (1-\lambda\eta)^{-KR-1}$ so
\begin{align}
\frac{1}{W_T}\sum_{t=1}^{KR}w_t\E\brk*{F(\bx_t) - F^*}
&\leq \frac{2\E\nrm*{\bx_0 - x^*}^2}{\eta}\exp\prn*{-\lambda\eta(KR+1)} + \frac{2\eta\sigma^2}{M} + 8HK\eta^2\sigma^2
\end{align}
We also observe that $2H \geq \lambda KR$ so with $\eta = \frac{1}{4H} \leq \frac{1}{2\lambda KR}$ we have
\begin{align}
\frac{1}{W_T}\sum_{t=1}^{KR}w_t\E\brk*{F(\bx_t) - F^*} 
&\leq 8H\E\nrm*{\bx_0 - x^*}^2\exp\prn*{-\frac{\lambda KR}{4H}}+ \frac{\sigma^2}{\lambda MKR} + \frac{2H\sigma^2}{\lambda^2 KR^2}
\end{align}
This completes the proof.
\end{proof}

\subsection{Proof of \pref{thm:local-sgd-homogeneous-lower-bound}} \label{app:local-sgd-homogeneous-lower-bound}

Here, we will prove the lower bound in \pref{thm:local-sgd-homogeneous-lower-bound}. Recall the objective and stochastic gradient estimator for the hard instance are defined by 
\begin{equation}\label{eq:local-sgd-homogeneous-lower-bound-construction-appendix}
F(x) = \frac{\mu}{2}\prn*{x_1 - b}^2 + \frac{H}{2}\prn*{x_2 - b}^2 + \frac{L}{2}\prn*{\prn*{x_3 - c}^2 + \pp{x_3 - c}^2}
\end{equation}
and
\begin{equation}\label{eq:local-sgd-homogeneous-lower-bound-stochastic-gradient}
g(x;z) = \nabla F(x) + \begin{bmatrix}0\\0\\z\end{bmatrix} \qquad\textrm{where}\qquad \P\brk*{z=\sigma} = \P\brk*{z=-\sigma} = \frac{1}{2}
\end{equation}
Due to the structure of the objective \eqref{eq:local-sgd-homogeneous-lower-bound-construction-appendix}, which decomposes as a sum over three terms which each depend only on a single coordinate, the Local SGD dynamics on each coordinate of the optimization variable are independent of each other. For this reason, we are able to analyze Local SGD on each coordinate separately.

Define the $2L$-smooth and $L$-strongly convex function
\begin{equation}\label{eq:local-sgd-homogeneous-lower-bound-def-gl}
g_L(x) = \frac{L}{2}x^2 + \frac{L}{2}\pp{x}^2
\end{equation}
Define a stochastic gradient estimator for $g_L$ via 
\begin{equation}\label{eq:local-sgd-homogeneous-lower-bound-def-glprime}
    g_L'(x,z) = g_L'(x) + z
\end{equation}
for $z \sim \textrm{Uniform}(\pm\sigma)$. Observe that the third coordinate of Local SGD on $F$ evolves exactly the same as Local SGD on the univariate function $g_L$. In the next three lemmas, we analyze the behavior of Local SGD on $g_L$:

\begin{lemma}\label{lem:local-sgd-homogeneous-lower-bound-sgd-two-step}
Fix $L,\eta,\sigma > 0$ such that $L\eta \leq \frac{1}{2}$. Let $x_0$ denote a random initial point with $\E x_0 \leq 0$, and let $x_2 = x_0 - \eta g_L'(x_0,z_0) - \eta g_L'(x_0 - \eta g_L'(x_0,z_0), z_1)$ be the second iterate of stochastic gradient descent with fixed stepsize $\eta$ intialized at $x_0$, and let $x_3 = x_2 - \eta g_L'(x_2, z_2)$ be the third iterate. Then
\begin{gather*}
\E x_2 \leq \begin{cases}
\frac{-\eta\sigma}{48} & \E x_0 \leq \frac{-\eta\sigma}{48} \\
\frac{-\eta\sigma}{4} + \prn*{1-L\eta}\prn*{\E x_0 + \frac{\eta\sigma}{4}} & \E x_0 \in \left(\frac{-\eta\sigma}{48}, 0\right]
\end{cases} \\
\E x_3 \leq \begin{cases}
\frac{-\eta\sigma}{48} & \E x_0 \leq \frac{-\eta\sigma}{48} \\
\frac{-\eta\sigma}{4} + \prn*{1-L\eta}^2\prn*{\E x_0 + \frac{\eta\sigma}{4}} & \E x_0 \in \left(\frac{-\eta\sigma}{48}, 0\right]
\end{cases}
\end{gather*}
\end{lemma}
\begin{proof}
Consider the $2$nd iterate of SGD with fixed stepsize $\eta$:
\begin{align}
x_2 
&= x_{1} - \eta g_L'(x_{1},z_1) \\
&= (1-L\eta)x_{1} -L\eta\pp{x_{1}} - \eta z_{1} \\
&= (1-L\eta)\prn*{x_{0} - \eta g_L'(x_{0},z_0)} - L\eta\pp{x_{0} - \eta g_L'(x_{0},z_0)} - \eta z_{1} \\
&= (1-L\eta)^2x_{0} - L\eta(1-L\eta)\pp{x_{0}}  - L\eta\pp{(1-L\eta)x_{0} - L\eta\pp{x_{0}} - \eta z_{0}} - \eta(1-\eta)z_{0} - \eta z_{1}
\end{align}
Thus,
\begin{equation}
\E x_2 
= (1-L\eta)^2\E x_{0} - L\eta(1-L\eta)\E \pp{x_{0}}
- L\eta\E \pp{(1-L\eta)x_{0} - L\eta\pp{x_{0}} - \eta z_{0}} 
\end{equation}
Define $y := (1-L\eta)x_{0} - L\eta\pp{x_{0}}$, then
\begin{align}
\E \pp{(1-L\eta)x_{0} - L\eta\pp{x_{0}} - \eta z_{0}}
&= \E\pp{y - \eta z_{0}} \\
&= \frac{1}{2}\E\pp{y - \eta\sigma} + \frac{1}{2}\E\pp{y + \eta\sigma} \\
&= \E\begin{cases}
y & y > \eta\sigma \\
\frac{y + \eta\sigma}{2} & \abs{y} \leq \eta\sigma \\
0 & y < -\eta\sigma
\end{cases}
\end{align}
The function 
\begin{equation}
z \mapsto \begin{cases}
z & z > \eta\sigma \\
\frac{z + \eta\sigma}{2} & \abs{z} \leq \eta\sigma \\
0 & z < -\eta\sigma
\end{cases}
\end{equation}
is convex, so by Jensen's inequality
\begin{align}
\E x_2 
&= (1-L\eta)\E y -  L\eta\E\begin{cases}
y & y > \eta\sigma \\
\frac{y + \eta\sigma}{2} & \abs{y} \leq \eta\sigma \\
0 & y < -\eta\sigma
\end{cases} \\
&\leq (1-L\eta)\E y -  L\eta\begin{cases}
\E y & \E y > \eta\sigma \\
\frac{\E y + \eta\sigma}{2} & \abs{\E y} \leq \eta\sigma \\
0 & \E y < -\eta\sigma
\end{cases} \\
&= \begin{cases}
(1-2L\eta)\E y & \E y > \eta\sigma \\
\prn*{1-\frac{3}{2}L\eta}\E y - \frac{L\eta^2\sigma}{2} & \abs{\E y} \leq \eta\sigma \\
(1-L\eta)\E y & \E y < -\eta\sigma
\end{cases} \\
&\leq \begin{cases}
(1-2L\eta)\E y & \E y > \eta\sigma \\
\prn*{1-\frac{3}{2}L\eta}\E y - \frac{L\eta^2\sigma}{2} & \abs{\E y} \leq \eta\sigma \\
\frac{-\eta\sigma}{2} & \E y < -\eta\sigma
\end{cases}\label{eq:local-sgd-homogeneous-lower-bound-one-step-lemma-eq1}
\end{align} 
where we used that $L\eta \leq \frac{1}{2}$ for the final inequality. Suppose $\E x_0 \leq \frac{-\eta\sigma}{48}$ which implies $\E y \leq \frac{-(1-L\eta)\eta\sigma}{48}$. Then we are in either the second or third case of \eqref{eq:local-sgd-homogeneous-lower-bound-one-step-lemma-eq1}. If we are in the third case then
\begin{equation}
\E x_2 \leq \frac{-\eta\sigma}{2} \leq \frac{-\eta\sigma}{48}
\end{equation}
If we are in the second case, then
\begin{align}
\E x_2 
&\leq \prn*{1-\frac{3}{2}L\eta}\E y - \frac{L\eta^2\sigma}{2} \\
&\leq \prn*{1-\frac{3}{2}L\eta}\frac{-(1-L\eta)\eta\sigma}{48} - \frac{L\eta^2\sigma}{2} \\
&= \frac{-\eta\sigma}{48} + \frac{3(1-L\eta)L\eta^2\sigma}{96} + \frac{L\eta^2\sigma}{48} - \frac{L\eta^2\sigma}{2} \\
&\leq \frac{-\eta\sigma}{48}
\end{align}
Either way, $\E x_2 \leq \frac{-\eta\sigma}{48}$. 

Suppose instead that $\E x_0 \in \left(\frac{-\eta\sigma}{48}, 0\right]$. Then,
\begin{align}
\E x_2 
&\leq \prn*{1-\frac{3}{2}L\eta}\E y - \frac{L\eta^2\sigma}{2} \\
&\leq (1-L\eta)\E x_0 - \frac{3L\eta(1-L\eta)}{2}\E x_0 - \frac{L\eta^2\sigma}{2} \\
&\leq (1-L\eta)\E x_0 + \frac{3L\eta}{2}\cdot\frac{\eta\sigma}{48} - \frac{L\eta^2\sigma}{2} \\
&\leq (1-L\eta)\E x_0 - \frac{L\eta^2\sigma}{4} \\
&= -\frac{\eta\sigma}{4} + \prn*{1-L\eta}\prn*{\E x_0 + \frac{\eta\sigma}{4}}
\end{align}
We conclude that
\begin{equation}\label{eq:local-sgd-homogeneous-lower-bound-one-step-lemma-x2}
\E x_2 \leq \begin{cases}
\frac{-\eta\sigma}{48} & \E x_0 \leq \frac{-\eta\sigma}{48} \\
\frac{-\eta\sigma}{4} + \prn*{1-L\eta}\prn*{\E x_0 + \frac{\eta\sigma}{4}} & \E x_0 \in \left( \frac{-\eta\sigma}{48}, 0 \right]
\end{cases}
\end{equation}

Now, consider the third iterate of SGD, $x_3$:
\begin{align}
\E x_3 
&= \E x_2 - \eta \E g'_L(x_2, z_2) \\
&= (1-L\eta)\E x_2 - L\eta\E \pp{x_2} \\
&= (1-L\eta)\E x_2 - L\eta\E \pp{\E[x_2\,|\,x_1] - \eta z_1} \\
&\leq (1-L\eta)\E x_2 - \frac{L\eta}{2}\E \pp{\E[x_2\,|\,x_1] + \eta \sigma}
\end{align}
Since $z \mapsto \pp{z}$ is convex, by Jensen's inequality
\begin{align}
\E x_3 
&\leq (1-L\eta)\E x_2 - \frac{L\eta}{2}\pp{\E x_2 + \eta\sigma} \\
&\leq \begin{cases}
\prn*{1-\frac{3L\eta}{2}}\E x_2 - \frac{L\eta^2\sigma}{2} & \E x_2 > -\eta\sigma \\
(1-L\eta)\E x_2 & \E x_2 \leq -\eta\sigma 
\end{cases} \\
&\leq \begin{cases}
\prn*{1-\frac{3L\eta}{2}}\E x_2 - \frac{L\eta^2\sigma}{2} & \E x_2 > -\eta\sigma \\
\frac{-\eta\sigma}{2} & \E x_2 \leq -\eta\sigma
\end{cases}\label{eq:local-sgd-homogeneous-lower-bound-one-step-lemma-eq2}
\end{align}
To complete the proof, we must show that 
\begin{equation}\label{eq:local-sgd-homogeneous-lower-bound-one-step-lemma-x3}
\E x_3 \leq \begin{cases}
\frac{-\eta\sigma}{48} & \E x_0 \leq \frac{-\eta\sigma}{48} \\
\frac{-\eta\sigma}{4} + \prn*{1-L\eta}^2\prn*{\E x_0 + \frac{\eta\sigma}{4}} & \E x_0 \in \left(\frac{-\eta\sigma}{48}, 0\right]
\end{cases}
\end{equation}
Returning to \eqref{eq:local-sgd-homogeneous-lower-bound-one-step-lemma-eq2}, note that if $\E x_2 \leq -\eta \sigma$ then $\E x_3 \leq \frac{-\eta\sigma}{2}$ implies \eqref{eq:local-sgd-homogeneous-lower-bound-one-step-lemma-x3}. Therefore, we only need to consider the first case of \eqref{eq:local-sgd-homogeneous-lower-bound-one-step-lemma-eq2}.

Suppose first that $\E x_0 \leq \frac{-\eta\sigma}{48}$, then by \eqref{eq:local-sgd-homogeneous-lower-bound-one-step-lemma-x2} we have $\E x_2 \leq \frac{-\eta\sigma}{48}$, thus
\begin{align}
\E x_3 
&\leq \prn*{1-\frac{3L\eta}{2}}\E x_2 - \frac{L\eta^2\sigma}{2} \\
&\leq \prn*{1-\frac{3L\eta}{2}}\frac{-\eta\sigma}{48} - \frac{L\eta^2\sigma}{2} \\
&\leq \frac{-\eta\sigma}{48}
\end{align}
If instead $\E x_0 \in \left(\frac{-\eta\sigma}{48}, 0\right]$, then by \eqref{eq:local-sgd-homogeneous-lower-bound-one-step-lemma-x2} we have $\E x_2 \leq \frac{-\eta\sigma}{4} + \prn*{1-L\eta}\prn*{\E x_0 + \frac{\eta\sigma}{4}}$, thus
\begin{align}
\E x_3 
&\leq \prn*{1-\frac{3L\eta}{2}}\E x_2 - \frac{L\eta^2\sigma}{2} \\
&\leq \prn*{1-\frac{3L\eta}{2}}\frac{-\eta\sigma}{4} + \prn*{1-\frac{3L\eta}{2}}\prn*{1-L\eta}\prn*{\E x_0 + \frac{\eta\sigma}{4}} - \frac{L\eta^2\sigma}{2}  \\
&\leq \frac{-\eta\sigma}{4} + \frac{3L\eta^2\sigma}{8} - \frac{L\eta^2\sigma}{2} + \prn*{1-L\eta}^2\prn*{\E x_0 + \frac{\eta\sigma}{4}} \\
&\leq \frac{-\eta\sigma}{4} + \prn*{1-L\eta}^2\prn*{\E x_0 + \frac{\eta\sigma}{4}}
\end{align}
This completes both cases of \eqref{eq:local-sgd-homogeneous-lower-bound-one-step-lemma-x3}.
\end{proof}

\begin{lemma}\label{lem:local-sgd-homogeneous-lower-bound-sgd-many-steps}
Fix $L,\eta,\sigma > 0$ such that $L\eta \leq \frac{1}{2}$ and let $k \geq 2$. Let $x_0$ denote a random initial point with $\E x_0 \leq 0$ and let $x_k$ denote the $k$th iterate of stochastic gradient descent on $g_L$ with fixed stepsize $\eta$ intialized at $x_0$. Then
\begin{gather*}
\E x_k \leq \begin{cases}
\frac{-\eta\sigma}{48} & \E x_0 \leq \frac{-\eta\sigma}{48} \\
\frac{-\eta\sigma}{4} + \prn*{1-L\eta}^{k/2}\prn*{\E x_0 + \frac{\eta\sigma}{4}} & \E x_0 \in \left(\frac{-\eta\sigma}{48}, 0\right]
\end{cases}
\end{gather*}
\end{lemma}
\begin{proof}
The idea of this proof is simple: $k$ steps of SGD initialized at some point $x_0$ is equivalent to doing two steps of SGD initialized at $x_0$ to get $x_2$, then doing two more steps initialized at $x_2$ to get $x_4$, and so forth until $k$ steps have been completed. The only minor complication is if $k$ is odd, in which case we start by doing three steps initialized at $x_{0}$ to get $x_3$ and continue in steps of two.

We will consider two cases, either $\E x_0 \leq \frac{-\eta\sigma}{48}$ or $\E x_0 \in \left(\frac{-\eta\sigma}{48}, 0\right]$. 
In the first case, $\E x_0 \leq \frac{-\eta\sigma}{48}$, if $k$ is even then by \pref{lem:local-sgd-homogeneous-lower-bound-sgd-two-step}
\begin{equation}
\E x_0 \leq \frac{-\eta\sigma}{48} \implies \E x_2 \leq \frac{-\eta\sigma}{48} \implies \E x_4 \leq \frac{-\eta\sigma}{48} \implies \dots \implies \E x_k \leq \frac{-\eta\sigma}{48}
\end{equation}
If $k$ is odd then
\begin{equation}
\E x_0 \leq \frac{-\eta\sigma}{48} \implies \E x_3 \leq \frac{-\eta\sigma}{48} \implies \E x_5 \leq \frac{-\eta\sigma}{48} \implies \dots \implies \E x_k \leq \frac{-\eta\sigma}{48}
\end{equation}

In the second case, $\E x_0 \in \left(\frac{-\eta\sigma}{48}, 0\right]$. Then, when $k$ is even, by repeatedly invoking \pref{lem:local-sgd-homogeneous-lower-bound-sgd-two-step} we get
\begin{align}
\E x_2 &\leq \frac{-\eta\sigma}{4} + \prn*{1-L\eta}\prn*{\E x_0 + \frac{\eta\sigma}{4}} \\
\E x_4 &\leq \frac{-\eta\sigma}{4} + \prn*{1-L\eta}\prn*{\E x_2 + \frac{\eta\sigma}{4}} \leq \frac{-\eta\sigma}{4} + \prn*{1-L\eta}^2\prn*{\E x_0 + \frac{\eta\sigma}{4}} \\
\E x_6 &\leq \frac{-\eta\sigma}{4} + \prn*{1-L\eta}\prn*{\E x_4 + \frac{\eta\sigma}{4}} \leq \frac{-\eta\sigma}{4} + \prn*{1-L\eta}^3\prn*{\E x_0 + \frac{\eta\sigma}{4}} \\
&\ \ \vdots \\
\E x_k &\leq \frac{-\eta\sigma}{4} + \prn*{1-L\eta}^{k/2}\prn*{\E x_0 + \frac{\eta\sigma}{4}}
\end{align}
The same argument applies when $k$ is odd (using the bound on $\E x_3$) to prove 
\begin{equation}
\E x_k \leq \frac{-\eta\sigma}{4} + \prn*{1-L\eta}^{(k+1)/2}\prn*{\E x_0 + \frac{\eta\sigma}{4}} \leq \frac{-\eta\sigma}{4} + \prn*{1-L\eta}^{k/2}\prn*{\E x_0 + \frac{\eta\sigma}{4}}
\end{equation}
\end{proof}

\begin{lemma}\label{lem:local-sgd-homogeneous-lower-bound-local-sgd-on-stochastic-coordinate}
Let $K \geq 2$ and let $\hat{x}$ be the output of Local SGD on $F$ using a fixed stepsize $\eta \leq \frac{1}{2L}$ and initialized at zero. Then
\[
\E\brk*{\frac{L}{2}\prn*{\prn*{\hat{x}_3 - c}^2 + \pp{\hat{x}_3 - c}^2}} 
\geq \frac{L\eta^2\sigma^2}{4608} \indicator{\eta \leq \frac{1}{2L}}\indicator{c \geq \frac{\eta\sigma}{48} \lor \eta \geq \frac{2}{LRK}}
\]
\end{lemma}
\begin{proof}
Since each coordinate evolves independently when optimizing $F$ using Local SGD, we can ignore the first two coordinates and focus only on the third. Observe that using Local SGD on $F$ with a fixed stepsize $\eta$ and initialized at zero to obtain $\hat{x}_3$ is exactly equivalent to using Local SGD on $g_{L}$ with the same fixed stepsize $\eta$ and initialized at $-c$. The different initialization is due to the fact that the Local SGD dynamics do not change with the change of variables $x - c \rightarrow x$. Let $\bar{x}_r$ denote the averaged iterate of Local SGD initialized at $-c$ with stepsize $\eta$ after the $r$th round of communication and let $x_{r,k,m}$ denote its $k$th iterate during the $r$th round of communication on the $m$th machine. We will start by proving that when $\eta \leq \frac{1}{2L}$ and \emph{either} $c \geq \frac{\eta\sigma}{8}$ \emph{or} $\eta \geq \frac{2}{LRK}$ then
\begin{equation}
    \E \hat{x}_3 - c = \E\bar{x}_R \leq \frac{-\eta\sigma}{48}
\end{equation}

Consider first the case $\E x_0 = -c \leq \frac{-\eta\sigma}{48}$. Then by \pref{lem:local-sgd-homogeneous-lower-bound-sgd-many-steps}
\begin{equation}
\E x_0 = -c \leq \frac{-\eta\sigma}{48} \implies \E x_{1,K,m} \leq \frac{-\eta\sigma}{48} \ \ \forall m
\end{equation}
therefore
\begin{equation}
\E \bar{x}_1 = \E\brk*{\frac{1}{M}\sum_{m=1}^M x_{1,K,m}} \leq \frac{-\eta\sigma}{48}
\end{equation}
Repeatedly applying \pref{lem:local-sgd-homogeneous-lower-bound-sgd-many-steps} shows that for each $r$
\begin{equation}
\E \bar{x}_r \leq \frac{-\eta\sigma}{48} \implies \E x_{r+1,K,m} \leq \frac{-\eta\sigma}{48} \implies \E \bar{x}_{r+1} = \E\brk*{\frac{1}{M}\sum_{m=1}^M x_{r+1,K,m}} \leq \frac{-\eta\sigma}{48}
\end{equation}
We conclude $\E \bar{x}_R \leq \frac{-\eta\sigma}{48}$.

Consider instead the case that $\E x_0 = -c \in \left( \frac{-\eta\sigma}{48}, 0 \right]$  and $\eta \geq \frac{2}{LRK}$. Then, by \pref{lem:local-sgd-homogeneous-lower-bound-sgd-many-steps}
\begin{equation}
\E x_0 = -c \in \left( \frac{-\eta\sigma}{48}, 0 \right] \implies \E x_{1,K,m} \leq \frac{-\eta\sigma}{4} + \prn*{1-L\eta}^{K/2}\prn*{\frac{\eta\sigma}{4} - c} \ \ \forall m
\end{equation}
and so
\begin{equation}
\E \bar{x}_1 = \E\brk*{\frac{1}{M}\sum_{m=1}^M x_{1,K,m}} \leq \frac{-\eta\sigma}{4} + \prn*{1-L\eta}^{K/2}\prn*{\frac{\eta\sigma}{4} - c}
\end{equation}
Again, we can repeatedly apply \pref{lem:local-sgd-homogeneous-lower-bound-sgd-many-steps} to show 
\begin{align}
\E \bar{x}_2 &\leq \frac{-\eta\sigma}{4} + \prn*{1-L\eta}^{K/2}\prn*{\E \bar{x}_1 + \frac{\eta\sigma}{4}} 
\leq \frac{-\eta\sigma}{4} + \prn*{1-L\eta}^{2K/2}\prn*{\frac{\eta\sigma}{4} - c} \\
\E \bar{x}_3 &\leq \frac{-\eta\sigma}{4} + \prn*{1-L\eta}^{K/2}\prn*{\E \bar{x}_2 + \frac{\eta\sigma}{4}} 
\leq \frac{-\eta\sigma}{4} + \prn*{1-L\eta}^{3K/2}\prn*{\frac{\eta\sigma}{4} - c} \\
&\ \ \vdots\\
\E \bar{x}_R 
&\leq \frac{-\eta\sigma}{4} + \prn*{1-L\eta}^{RK/2}\prn*{\frac{\eta\sigma}{4} - c} \\
&\leq -\prn*{1 - \prn*{1-L\eta}^{RK/2}}\frac{\eta\sigma}{4} \\
&\leq -\prn*{1 - \prn*{1-\frac{2}{RK}}^{RK/2}}\frac{\eta\sigma}{4} \\
&\leq \frac{\eta\sigma}{48}
\end{align}
These inequalities hold only as long as $\E \bar{x}_r > \frac{-\eta\sigma}{48}$. But, if for some $r$, $\E \bar{x}_r \leq \frac{-\eta\sigma}{48}$ then $\E \bar{x}_R \leq \frac{-\eta\sigma}{48}$ by the same argument as above.
We conclude that
\begin{equation}
\E \bar{x}_R \leq
\frac{-\eta\sigma}{48}\indicator{\eta \leq \frac{1}{2L}}\indicator{c \geq \frac{\eta\sigma}{48} \lor \eta \geq \frac{2}{LRK}}
\end{equation}
Since $\E \hat{x}_3 - c = \E \bar{x}_R$, by Jensen's inequality
\begin{align}
\E\brk*{\frac{L}{2}\prn*{\prn*{\hat{x}_3 - c}^2 + \pp{\hat{x}_3 - c}^2}} 
&\geq \frac{L}{2}\prn*{\prn*{\E \bar{x}_R}^2 + \pp{\E \bar{x}_R}^2} \\
&\geq \frac{L\eta^2\sigma^2}{4608} \indicator{\eta \leq \frac{1}{2L}}\indicator{c \geq \frac{\eta\sigma}{48} \lor \eta \geq \frac{2}{LRK}}
\end{align}
\end{proof}

We now analyze the progress of SGD on the first two coordinates of $F$ in the following lemma:
\begin{lemma}\label{lem:local-sgd-homogeneous-lower-bound-sgd-on-deterministic-coordinates}
Let $\hat{x}$ be the output of Local SGD on $F$ using a fixed stepsize $\eta$ and initialized at zero. Then with probability 1,
\[
\frac{\mu}{2}\prn*{\hat{x}_1 - b}^2 \geq \frac{\mu b^2}{8}\indicator{\eta < \frac{1}{2\mu KR}}
\]
and
\[
\frac{H}{2}\prn*{\hat{x}_2 - b}^2 \geq \frac{H b^2}{2}\indicator{\eta > \frac{2}{H}}.
\]
\end{lemma}
\begin{proof}
Since the stochastic gradient estimator has no noise along the first and second coordinates, and since the separate coordinates evolve independently, $\hat{x}_1$ is exactly the output of $KR$ steps of deterministic gradient descent with fixed stepsize $\eta$ on the univariate function $x \mapsto \frac{\mu}{2}\prn*{x - b}^2$. Similarly, $\hat{x}_2$ is the output of $KR$ steps of deterministic gradient descent with fixed stepsize $\eta$ on $x \mapsto \frac{H}{2}\prn*{x - b}^2$. Thus,
\begin{multline}
x_1\ind{t+1} - b = x_1\ind{t} - b - \eta\mu\prn*{x_1\ind{t} - b} 
\implies \hat{x}_1 = b + \prn*{1-\eta\mu}^{KR}\prn*{x_1\ind{0} - b} = b\prn*{1 - \prn*{1-\eta\mu}^{KR}}
\end{multline}
Thus, if $\eta < \frac{1}{2\mu KR}$, then 
\begin{equation}
\hat{x}_1 \leq b\eta\mu KR< \frac{b}{2} \implies \frac{\mu}{2}\prn*{\hat{x}_1 - b}^2 \geq \frac{\mu b^2}{8}\indicator{\eta < \frac{1}{2\mu KR}}
\end{equation}
Similarly,
\begin{multline}
x_2\ind{t+1} - b = x_2\ind{t} - b - \eta H\prn*{x_2\ind{t} - b} 
\implies \hat{x}_2 - b = \prn*{1-\eta H}^{KR}\prn*{x_2\ind{0} - b} = - b\prn*{1-\eta H}^{KR}
\end{multline}
Thus, if $\eta > \frac{2}{H}$, then 
\begin{equation}
\abs{\hat{x}_2 - b} \geq b \implies \frac{H}{2}\prn*{\hat{x}_2 - b}^2 \geq \frac{H b^2}{2}\indicator{\eta > \frac{2}{H}}
\end{equation}
\end{proof}

Combining \pref{lem:local-sgd-homogeneous-lower-bound-local-sgd-on-stochastic-coordinate} and \pref{lem:local-sgd-homogeneous-lower-bound-sgd-on-deterministic-coordinates}, we are ready to prove the theorem:
\localsgdhomogeneouslowerbound*
\begin{proof}
Consider optimizing the objective $F$ defined in \eqref{eq:local-sgd-homogeneous-lower-bound-construction-appendix} using the stochastic gradient oracle \eqref{eq:local-sgd-homogeneous-lower-bound-stochastic-gradient} initialized at zero and using a fixed stepsize $\eta$. The variance of the stochastic gradient oracle is equal to $\sigma^2$. This function is $\max\crl*{\mu, H, 2L}$-smooth, and $\min\crl*{\mu, H, L}$-strongly convex. We will be choosing $L = \frac{H}{4}$ and $\mu \in \left[\lambda, \frac{H}{16}\right]$ so that $F$ is $H$-smooth and $\lambda$-strongly convex.
Finally, the objective $F$ is minimized at the point $x^* = [b,b,c]^\top$ and $F(x^*) = 0$. This point has norm $\nrm*{x^*} = \sqrt{2b^2 + c^2}$ we will choose $b = c = \frac{B}{\sqrt{3}}$ so that $\nrm*{x^*} = B$.

By \pref{lem:local-sgd-homogeneous-lower-bound-local-sgd-on-stochastic-coordinate}, the output of Local SGD, $\hat{x}$ satisfies
\begin{equation}\label{eq:needed-for-heterogeneous-lower-bound-1}
\E\brk*{\frac{L}{2}\prn*{\prn*{\hat{x}_3 - c}^2 + \pp{\hat{x}_3 - c}^2}} 
\geq \frac{L\eta^2\sigma^2}{4608} \indicator{\eta \leq \frac{1}{2L}}\indicator{c \geq \frac{\eta\sigma}{48} \lor \eta \geq \frac{2}{LRK}}
\end{equation}
By \pref{lem:local-sgd-homogeneous-lower-bound-sgd-on-deterministic-coordinates}, the output of Local SGD, $\hat{x}$ satisfies
\begin{equation}\label{eq:needed-for-heterogeneous-lower-bound-2}
\frac{\mu}{2}\prn*{\hat{x}_1 - b}^2 + \frac{H}{2}\prn*{\hat{x}_2 - b}^2 
\geq \frac{\mu b^2}{8}\indicator{\eta < \frac{1}{2\mu KR}} + \frac{H b^2}{2}\indicator{\eta > \frac{2}{H}}
\end{equation}
Combining these, we have
\begin{equation}
\E F(\hat{x}) - \min_x F(x) 
\geq \frac{\mu b^2}{8}\indicator{\eta < \frac{1}{2\mu KR}} + \frac{H b^2}{2}\indicator{\eta > \frac{2}{H}} + \frac{L\eta^2\sigma^2}{4608} \indicator{\eta \leq \frac{1}{2L}}\indicator{\eta \leq \frac{48c}{\sigma} \lor \eta \geq \frac{2}{LRK}}
\end{equation}
Consider two cases: first, suppose that $\eta \not\in \brk*{\frac{1}{2\mu KR}, \frac{2}{H}}$. Then,
\begin{equation}
\E F(\hat{x}) - \min_x F(x)
\geq \min\crl*{\frac{\mu b^2}{8}, \frac{H b^2}{2}} 
= \frac{\mu b^2}{8}\label{eq:local-sgd-homogeneous-lower-bound-eq1}
\end{equation}
Suppose instead that $\eta \in \brk*{\frac{1}{2\mu KR}, \frac{2}{H}}$. Since $L = \frac{H}{4}$, $\eta \leq \frac{2}{H} \leq \frac{1}{2L}$. Similarly, since $\mu \leq \frac{H}{16} = \frac{L}{4}$, $\eta \geq \frac{1}{2\mu KR} \geq \frac{2}{LRK}$. Therefore, $\eta \in \brk*{\frac{1}{2\mu KR}, \frac{2}{H}}$ implies
\begin{align}
\E F(\hat{x}) - \min_x F(x) 
&\geq \min_{\eta \in \brk*{\frac{1}{2\mu KR}, \frac{2}{H}}} \frac{L\eta^2\sigma^2}{4608} \indicator{\eta \leq \frac{1}{2L}}\indicator{\eta \leq \frac{48c}{\sigma} \lor \eta \geq \frac{2}{LRK}} \\
&= \min_{\eta \in \brk*{\frac{1}{2\mu KR}, \frac{2}{H}}} \frac{L\eta^2\sigma^2}{4608} \\
&= \frac{L\sigma^2}{18432\mu^2K^2R^2}\label{eq:local-sgd-homogeneous-lower-bound-eq2}
\end{align}
Combining \eqref{eq:local-sgd-homogeneous-lower-bound-eq1} and \eqref{eq:local-sgd-homogeneous-lower-bound-eq2} yields
\begin{equation}
\E F(\hat{x}) - \min_x F(x)
\geq \min\crl*{\frac{\mu B^2}{24}, \frac{H\sigma^2}{73728\mu^2K^2R^2}}
\end{equation}
This statement holds for any $\mu \in \brk*{\lambda, \frac{H}{16}}$. Consider three cases: first, suppose $\mu = \prn*{\frac{H\sigma^2}{3072B^2K^2R^2}}^{1/3} \in \brk*{\lambda, \frac{H}{16}}$. Then
\begin{equation}
\E F(\hat{x}) - \min_x F(x)
\geq \frac{H^{1/3}\sigma^{2/3}B^{4/3}}{350K^{2/3}R^{2/3}}
\end{equation}
Consider next the case that $\prn*{\frac{H\sigma^2}{3072B^2K^2R^2}}^{1/3} > \frac{H}{16}$ $\implies$ $\frac{\sigma^2}{192B^2K^2R^2} > \frac{H^2}{256}$ and choose $\mu = \frac{H}{16}$. Then
\begin{equation}
\E F(\hat{x}) - \min_x F(x)
\geq \min\crl*{\frac{H B^2}{384}, \frac{H\sigma^2}{73728K^2R^2\cdot\frac{H^2}{256}}} 
= \frac{H B^2}{384}
\end{equation}
Finally, consider the case that $\prn*{\frac{H\sigma^2}{3072B^2K^2R^2}}^{1/3} < \lambda$ and choose $\mu = \lambda$. Then,
\begin{equation}
\E F(\hat{x}) - \min_x F(x)
\geq \min\crl*{\frac{\lambda B^2}{24}, \frac{H\sigma^2}{73728\lambda^2K^2R^2}} = \frac{H\sigma^2}{73728\lambda^2K^2R^2}
\end{equation}
Combining these cases proves that for a universal constant $c$,
\begin{equation}\label{eq:local-sgd-lower-bound-intermed-1}
\E F(\hat{x}) - \min_x F(x)
\geq c\cdot\min\crl*{\frac{\prn*{H\sigma^2B^4}^{1/3}}{K^{2/3}R^{2/3}},\,\frac{H\sigma^2}{\lambda^2K^2R^2},\,HB^2}
\end{equation}
Therefore, in the convex setting, where we are free to take $\lambda$ small enough that the second term of the $\min$ is never the minimizing term, we have
\begin{equation}
\E F(\hat{x}) - \min_x F(x)
\geq c\cdot\min\crl*{\frac{\prn*{H\sigma^2B^4}^{1/3}}{K^{2/3}R^{2/3}},\,HB^2}
\end{equation}
Finally, combining this with \pref{lem:statistical-term-lower-bound} we conclude that for some $F_0 \in \mc{F}_0(H,B)$
\begin{equation}
\E F_0(\hat{x}) - F_0^*
\geq c\cdot\prn*{\min\crl*{\frac{\prn*{H\sigma^2B^4}^{1/3}}{K^{2/3}R^{2/3}},\,HB^2} + \min\crl*{\frac{\sigma B}{\sqrt{MKR}},\,HB^2}}
\end{equation}

Returning to \eqref{eq:local-sgd-lower-bound-intermed-1}, we can also consider lower bounds in the strongly convex setting, where we require the objective to be in $\mc{F}_\lambda(H,\Delta)$. In this case, we can set $B^2 = \frac{\Delta}{H}$ so that since $\mu,L\leq H$
\begin{equation}
F(0) = F(0) - F^* = \frac{\mu}{2}\frac{B^2}{3} + \frac{H}{2}\frac{B^2}{3} + \frac{\mu}{2}\frac{B^2}{3} \leq HB^2 = \Delta
\end{equation}
Therefore, $F \in \mc{F}_\lambda(H,\Delta)$ and the lower bound \eqref{eq:local-sgd-lower-bound-intermed-1} becomes
\begin{equation}\label{eq:local-sgd-lower-bound-intermed-2}
\E F(\hat{x}) - \min_x F(x)
\geq c\cdot\min\crl*{\frac{\prn*{\sigma^2\Delta^2}^{1/3}}{H^{1/3}K^{2/3}R^{2/3}},\,\frac{H\sigma^2}{\lambda^2K^2R^2},\,\Delta}
\end{equation}
Furthermore,
\begin{equation}
\frac{\prn*{\sigma^2\Delta^2}^{1/3}}{H^{1/3}K^{2/3}R^{2/3}} \leq \frac{H\sigma^2}{\lambda^2K^2R^2} 
\implies
\Delta \leq \frac{H\sigma^2}{\lambda^2K^2R^2}
\end{equation}
so the first term of the $\min$ is irrelevant.
Combining this with \pref{lem:statistical-term-lower-bound}, we conclude that for some $F_\lambda \in \mc{F}_\lambda(H,\Delta)$
\begin{equation}
\E F_\lambda(\hat{x}) - F_\lambda^*
\geq c\cdot\prn*{\min\crl*{\frac{H\sigma^2}{\lambda^2K^2R^2},\,\Delta} + \min\crl*{\frac{\sigma^2}{\lambda MKR},\Delta}}
\end{equation}
This completes the proof.
\end{proof}

\subsection{Proof of \pref{thm:local-sgd-heterogeneous-lower-bound}}\label{app:local-sgd-heterogeneous-lower-bound}

Consider the following function $F:\R^4\rightarrow\R$:
\begin{equation}
F(x) = \frac{1}{2}\prn*{F_1(x) + F_2(x)} = \frac{1}{2}\prn*{\E_{z^1\sim\mc{D}^1}f(x;z^1) + \E_{z^2\sim\mc{D}^2}f(x;z^2)}
\end{equation}
The distribution $z^1 \sim \mc{D}^1$ is described by $z^1 = (1, z)$ for $z \sim \mc{N}(0,\sigma^2)$. Similarly, $z^2 \sim \mc{D}^2$ is specified by $z^2 = (2, z)$ for $z \sim \mc{N}(0,\sigma^2)$. The lower bound construction will be based on just two functions. For $M > 2$ machines, we simply assign the first $\lfloor M/2 \rfloor$ machines $F_1$ and the next $\lfloor M/2 \rfloor$ machines $F_2$. This diminishes the lower bound by at most a $(M-1)/M$ factor. Therefore, we continue with the case $M=2$.

Similar to the proof of \pref{thm:local-sgd-homogeneous-lower-bound}, we define the local functions $F_1$ and $F_2$ via the auxiliary function
\begin{gather}
g(x_1,x_2,x_3,z) = \frac{\mu}{2}\prn*{x_1 - c}^2 + \frac{H}{2}\prn*{x_2-\frac{\sqrt{\mu} c}{\sqrt{H}}}^2 + \frac{H}{8}\prn*{x_3^2 + \pp{x_3}^2} + z^\top x_3 \\
G(x_1,x_2,x_3) = \E_z g(x_1,x_2,x_3,z)
\end{gather}
where $c > 0$ and $\mu \in \brk*{\lambda, \frac{H}{16}}$ are parameters to be determined later, and where $\pp{x} := \max\crl{x,0}$.
Then, we define
\begin{gather}
f(x;(1,z)) = g(x_1,x_2,x_3,z)  + \frac{Lx_4^2}{2} + \sdiff x_4  \\
f(x;(2,z)) = g(x_1,x_2,x_3,z)  + \frac{\lambda x_4^2}{2} - \sdiff x_4 
\end{gather}
for a parameter $L \in [\lambda, H]$ to be determined later. Therefore,
\begin{gather}
F_1(x) = \E_{z^1\sim\mc{D}^1}f(x;z^1) = G(x_1,x_2,x_3)  + \frac{Lx_4^2}{2} + \sdiff x_4 \\
F_2(x) = \E_{z^2\sim\mc{D}^2}f(x;z^2) = G(x_1,x_2,x_3)  + \frac{\lambda x_4^2}{2} - \sdiff x_4 
\end{gather}
It is clear from inspection that both $F_1$ and $F_2$, and consequently $F$, are $H$-smooth and $\lambda$-strongly convex. Furthermore, the variance of the gradients is bounded by $\sigma^2$ for both $\mc{D}^1$ and $\mc{D}^2$.

The function $G$ attains its minimum of zero at $\brk*{c,\frac{\sqrt{\mu} c}{\sqrt{H}},0}$ so $\nabla G\prn*{c,\frac{\sqrt{\mu} c}{\sqrt{H}},0} = 0$, and thus
\begin{equation}
\nabla F\prn*{c,\frac{\sqrt{\mu} c}{\sqrt{H}},0,0} = \nabla G\prn*{c,\frac{\sqrt{\mu} c}{\sqrt{H}},0} + \prn*{\frac{\sdiff}{2} - \frac{\sdiff}{2}}e_4 = 0
\end{equation}
From now on, we use $x^* = \brk*{c,\frac{\sqrt{\mu} c}{\sqrt{H}},0,0}$ to denote the minimizer of $F$, which has norm
\begin{equation}
\nrm{x^*}^2 = \prn*{1 + \frac{\mu}{H}}c^2 \leq 2c^2
\end{equation}
We can therefore ensure $\nrm{x^*}^2 \leq B^2$ by choosing $c^2 \leq \frac{B^2}{2}$.
Furthermore, the initial suboptimality
\begin{align}
F(0,0,0,0) - F^* = \mu c^2
\end{align}
Therefore, we can ensure $F(0,0,0,0) - F^* \leq \Delta$ by choosing $c^2 \leq \frac{\Delta}{\mu}$.
We conclude by showing that for this objective, $\sdiff^2$ bounded by
\begin{align}
\frac{1}{2}\sum_{m=1}^2 \nrm*{\nabla F_m(x^*)}^2
= \nrm*{\nabla F_2(x^*)}^2 
= \nrm*{\nabla F_1(x^*)}^2 
= \sdiff^2
\end{align}
Therefore, this objective has the desired level of heterogeneity. We have shown that the objective satisfies all of the necessary conditions for the lower bound. All that remains is to lower bound the error of Local SGD with a constant stepsize $\eta$ applied to this function.

\begin{lemma}\label{lem:fourth-coordinate}
For $\mu \leq 2L$, Local SGD with any constant stepsize $\eta \leq \frac{1}{L}$ applied to $F_1$ and $F_2$ after being initialized at zero results in $\hat{x}_4$ such that
\[
\frac{(L+\mu)\hat{x}_4^2}{4} \geq 
\frac{\sdiff^2(L+\mu)}{16\mu^2}\prn*{\frac{L-\mu}{L} - \prn*{1-\mu\eta}^{K}}^2\indicator{\eta \leq \frac{1}{L}}\indicator{\prn*{1-\mu\eta}^{K} \leq \frac{L-\mu}{L}}
\] 
\end{lemma}
\begin{proof}
Since the coordinates of $F_1$ and $F_2$ are completely decoupled, the behavior of the fourth coordinate of the iterates can be analyzed separately from the others. 

Let $x_{k,r}^{(1)}$ denote the fourth coordinate of machine 1's iterate at the $k$th iteration of round $r$, and similarly for $x_{k,r}^{(2)}$. The local SGD dynamics give
\begin{gather}
x_{k+1,r}^{(1)} = x_{k,r}^{(1)} - \eta\prn*{L x_{k,r}^{(1)} + \sdiff} = - \frac{\sdiff}{L} + (1-L\eta)\prn*{x_{k,r}^{(1)} + \frac{\sdiff}{L}} \\
x_{k,r}^{(2)} = x_{k,r}^{(2)} - \eta\prn*{-\sdiff + \mu x_{k,r}^{(2)}} = \frac{\sdiff}{\mu} + \prn*{1-\mu\eta}\prn*{x_{k,r}^{(2)} - \frac{\sdiff}{\mu}}
\end{gather}
and $\hat{x}_4 = \frac{1}{2}\prn*{x_{K,R}^{(1)} +  x_{K,R}^{(2)}} = x_{0,R+1}$.
Unravelling this recursion, we have that 
\begin{equation}
x_{0,r+1} = x_{0,r+1}^{(1)} = x_{0,r+1}^{(2)} = \frac{1}{2}\prn*{\frac{\sdiff}{\mu} - \frac{\sdiff}{L} + \prn*{1-\mu\eta}^K\prn*{x_{0,r} - \frac{\sdiff}{\mu}} + \prn*{1-L\eta}^K \prn*{x_{0,r} + \frac{\sdiff}{L}}}
\end{equation}
Furthermore, if $\eta \leq \frac{1}{L}$ then $(1-L\eta) \geq 0$, so if $x_{0,r} \geq 0$ then
\begin{equation}
x_{0,r+1} \geq \frac{\sdiff}{2\mu} - \frac{\sdiff}{2L} + \prn*{1-\mu\eta}^K\prn*{\frac{x_{0,r}}{2} - \frac{\sdiff}{2\mu}} \geq \frac{\sdiff}{2\mu}\prn*{\frac{L-\mu}{L} - \prn*{1-\mu\eta}^{K}}
\end{equation}
Finally, since $x_{0,0} = 0 \geq 0$, the condition $x_{0,r} \geq 0$ will hold throughout optimization, so
\begin{equation}
\hat{x}_4 \geq \frac{\sdiff}{2\mu}\prn*{\frac{L-\mu}{L} - \prn*{1-\mu\eta}^{K}}
\end{equation}
Therefore, if $\eta \leq \frac{1}{L}$ and $\prn*{1-\mu\eta}^{K} \leq \frac{L-\mu}{L}$ then
\begin{align}
\frac{(L+\mu)\hat{x}_4^2}{4}
&\geq \frac{\sdiff^2(L+\mu)}{16\mu^2}\prn*{\frac{L-\mu}{L} - \prn*{1-\mu\eta}^{K}}^2 \label{eq:fourth-coord-lower-bound}
\end{align}
This completes the proof.
\end{proof}

\localSGDheterogeneouslowerbound*
\begin{proof}
Since the four different coordinates are completely decoupled from each other, it suffices to analyze each coordinate separately. 

From the proof of \pref{thm:local-sgd-homogeneous-lower-bound}, \eqref{eq:needed-for-heterogeneous-lower-bound-1} and \eqref{eq:needed-for-heterogeneous-lower-bound-2} (with $L = H/4$) imply that
\begin{equation}
\E G(\hat{x}_1, \hat{x}_2, \hat{x}_3) - G\prn*{c,\frac{\sqrt{\mu}c}{\sqrt{H}},0}
\geq 
\frac{\mu c^2\prn*{1-\mu\eta}^{KR}}{2} + \frac{\mu c^2}{2}\indicator{\eta > \frac{2}{H}} + \frac{H\eta^2\sigma^2}{18432}\indicator{\eta \leq \frac{2}{H}}\indicator{\eta \geq \frac{8}{HKR}}
\end{equation}

Furthermore, by \pref{lem:fourth-coordinate}
\begin{equation}
\frac{(L+\lambda)\hat{x}_4^2}{4} \geq 
\frac{\sdiff^2(L+\mu)}{16\mu^2}\prn*{\frac{L-\mu}{L} - \prn*{1-\mu\eta}^{K}}^2\indicator{\eta \leq \frac{1}{L}}\indicator{\prn*{1-\mu\eta}^{K} \leq \frac{L-\mu}{L}}
\end{equation}
Therefore, choosing $L = \frac{H}{2}$
\begin{align}
\E F(\hat{x}) - F^* 
&= \E G(\hat{x}_1, \hat{x}_2, \hat{x}_3) - G\prn*{c,\frac{\sqrt{\mu}c}{\sqrt{H}},0} + \frac{H+2\lambda}{8}\hat{x}_4^2 \\
&\geq \frac{\mu c^2\prn*{1-\mu\eta}^{KR}}{2} + \frac{\mu c^2}{2}\indicator{\eta > \frac{2}{H}} + \frac{H\eta^2\sigma^2}{18432}\indicator{\eta \leq \frac{2}{H}}\indicator{\eta \geq \frac{8}{HKR}} \nonumber\\
&\qquad+ \frac{\sdiff^2(H+2\mu)}{32\mu^2}\prn*{\frac{H-2\mu}{H} - \prn*{1-\mu\eta}^{K}}^2\indicator{\eta \leq \frac{2}{H}}\indicator{\prn*{1-\mu\eta}^{K} \leq \frac{H-2\mu}{H}}\label{eq:raw-lower-bound}
\end{align}

\subsubsection*{Stochastic terms}
First, we will show a lower bound in terms of $\sigma^2$ using solely the first three terms of \eqref{eq:raw-lower-bound}. Consider three cases:

\paragraph{Case 1 $\eta \geq \frac{2}{H}$:}
In this case, from the second term of \eqref{eq:raw-lower-bound} we see that 
\begin{equation}
\E F(\hat{x}) - F^* \geq \frac{\mu c^2}{2}
\end{equation}

\paragraph{Case 2 $\frac{1}{2\mu KR} \leq \eta \leq \frac{2}{H}$:}
In this case, the third term of \eqref{eq:raw-lower-bound} shows
\begin{align}
\E F(\hat{x}) - F^*
\geq  \frac{H\eta^2\sigma^2}{18432}
\end{align}
where we recalled that $\mu \leq \frac{H}{16}$, so $\eta \geq \frac{1}{2\mu KR} \geq \frac{8}{HKR}$. This is non-decreasing in $\eta$, so for any $\eta$
\begin{align}
\E F(\hat{x}) - F^* 
\geq \frac{H\sigma^2}{73728 \mu^2 K^2R^2}
\end{align}

\paragraph{Case 3 $\eta \leq \frac{2}{H}$ and $\eta \leq \frac{1}{2\mu KR}$:}
In this case, from the first term of \eqref{eq:raw-lower-bound},
\begin{align}
\E F(\hat{x}) - F^* 
\geq \frac{\mu c^2\prn*{1-\mu\eta}^{KR}}{2} 
\geq \frac{\mu c^2\prn*{1-\frac{1}{2KR}}^{KR}}{2} 
\geq \frac{\mu c^2}{4}
\end{align}

\paragraph{Combination:}
Combining these three cases, we conclude that for any $\eta$
\begin{align}
\E F(\hat{x}) - F^* 
\geq \min\crl*{\frac{\mu c^2}{2},\, \frac{H\sigma^2}{73728 \mu^2 K^2R^2},\, \frac{\mu c^2}{4}}
= \min\crl*{\frac{\mu c^2}{3},\, \frac{H\sigma^2}{73728 \mu^2 K^2R^2}}
\end{align}
This lower bound holds for any stepsize, and any $\mu \in \brk*{\lambda, \frac{H}{16}}$ and regardless of $\sdiff$. In the strongly convex case, we recall that $F(0) - F(x^*) = \mu c^2$, therefore, we choose $\mu = \lambda$, and $c^2 = \frac{\Delta}{\lambda}$ so the lower bound reads (for a universal constant $\beta$)
\begin{align}
\E F(\hat{x}) - F^* 
\geq \beta\cdot \min\crl*{\Delta,\, \frac{H\sigma^2}{\lambda^2 K^2R^2}}
\end{align}
To conclude, it is well known that any first-order method which accesses at most $MKR$ stochastic gradients with variance $\sigma^2$ for a $\lambda$-strongly convex objective will suffer error at least $\beta\frac{\sigma^2}{\lambda MKR}$ in the worst case \cite{nemirovskyyudin1983} for a universal constant $\beta$.
Therefore, the strongly convex lower bound is
\begin{align}
\E F(\hat{x}) - F^* 
\geq \beta\cdot \min\crl*{\Delta,\, \frac{H\sigma^2}{\lambda^2 K^2R^2}} + \beta\cdot\frac{\sigma^2}{\lambda MKR}
\end{align}

In the convex case, we recall that $\nrm*{x^*}^2 \leq 2c^2$, so we choose $c^2 = \frac{B^2}{2}$, and set $\mu = \prn*{\frac{H\sigma^2}{B^2K^2R^2}}^{1/3}$ so the lower bound reads
\begin{align}
\E F(\hat{x}) - F^* 
\geq \beta\cdot \frac{\prn*{H\sigma^2B^4}}{K^{2/3}R^{2/3}}
\end{align}
To conclude, it is well known that any first-order method which accesses at most $MKR$ stochastic gradients with variance $\sigma^2$ for a convex objective with $\nrm{x^*}\leq B$ will suffer error at least $\beta\frac{\sigma B}{\sqrt{MKR}}$ in the worst case \cite{nemirovskyyudin1983}. Therefore, the convex lower bound is
\begin{align}
\E F(\hat{x}) - F^* 
\geq \beta\cdot \frac{\prn*{H\sigma^2B^4}}{K^{2/3}R^{2/3}} + \beta\cdot \frac{\sigma B}{\sqrt{MKR}}
\end{align}

\subsubsection*{Heterogeneity terms}
Next, we consider solely the first, second, and fourth terms of \eqref{eq:raw-lower-bound} in order to show a lower bound with respect to $\sdiff$. Again, we consider three cases:

\paragraph{Case 1 $\eta \geq \frac{2}{H}$:}
Again, in this case, from the second term of \eqref{eq:raw-lower-bound} we see that 
\begin{equation}
\E F(\hat{x}) - F^* \geq \frac{\mu c^2}{2} 
\end{equation}

\paragraph{Case 2 $\eta \leq \frac{2}{H}$ and $\prn*{1-\mu\eta}^{K} > \frac{H-2\mu}{H}$:}
In this case, from the first term of \eqref{eq:raw-lower-bound}, we have
\begin{align}
\E F(\hat{x}) - F^* 
&\geq \frac{\mu c^2\prn*{1-\mu\eta}^{KR}}{2} \\
&\geq \frac{\mu c^2}{2}\prn*{1-\frac{2\mu}{H}}^R \\
&\geq \frac{\mu c^2}{2}\prn*{\prn*{1-\frac{4\mu}{H}\prn*{1-\frac{1}{e}}}^{\frac{H}{4\mu}}}^{\frac{4\mu R}{H}} \\
&\geq \frac{\mu c^2}{2}\exp\prn*{-\frac{4\mu R}{H}}
\end{align}

\paragraph{Case 3 $\eta \leq \frac{2}{H}$ and $\prn*{1-\mu\eta}^{K} \leq \frac{H-2\mu}{H}$:}
In this case, from the first and fourth terms of \eqref{eq:raw-lower-bound}, we have
\begin{align}
\E F(\hat{x}) - F^* 
\geq \frac{\mu c^2}{2}\prn*{1-\mu\eta}^{KR} + \frac{\sdiff^2(H+2\mu)}{32\mu^2}\prn*{\frac{H-2\mu}{H} - \prn*{1-\mu\eta}^{K}}^2
\end{align}
Suppose that $\prn*{1-\mu\eta}^{K} \geq \frac{H-2\mu}{H} - \frac{1}{4R}$, then 
\begin{align}
\frac{\mu c^2}{2}\prn*{1-\mu\eta}^{KR}
\geq \frac{\mu c^2}{2}\prn*{1-\frac{2\mu}{H} - \frac{1}{4R}}^R
\end{align}
Then, if $R \geq \frac{H}{4\mu}$, then
\begin{align}
\frac{\mu c^2}{2}\prn*{1-\mu\eta}^{KR}
\geq \frac{\mu c^2}{2}\prn*{1-\frac{3\mu}{H}}^R 
\geq \frac{\mu c^2}{2}\prn*{\prn*{1-\frac{6\mu}{H}\prn*{1-\frac{1}{e}}}^{\frac{H}{6\mu}}}^{\frac{6\mu R}{H}} 
\geq \frac{\mu c^2}{2}\exp\prn*{-\frac{6\mu R}{H}}
\end{align}
Otherwise, if $R \leq \frac{H}{4\mu}$, then
\begin{align}
\frac{\mu c^2}{2}\prn*{1-\mu\eta}^{KR}
\geq \frac{\mu c^2}{2}\prn*{1-\frac{1}{2R}}^R \geq \frac{\mu c^2}{4} \geq \frac{\mu c^2}{4}\exp\prn*{-\frac{6\mu R}{H}}
\end{align}
Therefore, when $\prn*{1-\mu\eta}^{K} \geq \frac{H-2\mu}{H} - \frac{1}{4R}$,
\begin{equation}
\E F(\hat{x}) - F^* \geq \frac{\mu c^2}{4}\exp\prn*{-\frac{6\mu R}{H}}
\end{equation}

On the other hand, if $\prn*{1-\mu\eta}^{K} \leq \frac{H-2\mu}{H} - \frac{1}{4R}$, then
\begin{align}
\E F(\hat{x}) - F^* 
&\geq \frac{\sdiff^2(H+2\mu)}{32\mu^2}\prn*{\frac{H-2\mu}{H} - \prn*{1-\mu\eta}^{K}}^2 \\
&\geq \frac{\sdiff^2(H+2\mu)}{32\mu^2}\prn*{\frac{1}{4R}}^2 \\
&\geq\frac{H\sdiff^2}{512\mu^2R^2}
\end{align}

\paragraph{Combination:}
Combining these three cases, we conclude that
\begin{align}
\E F(\hat{x}) - F^* 
&\geq \min\crl*{\frac{\mu c^2}{4}\exp\prn*{-\frac{6\mu R}{H}},\, \frac{H\sdiff^2}{512\mu^2R^2}}
\end{align}

In the strongly convex case, we recall that $F(0) - F(x^*) = \mu c^2$, so we choose $\mu = \lambda$ and $c^2 = \frac{\Delta}{\lambda}$ so that the objective satisfies the strongly convex assumptions. Now, the lower bound reads (for a universal constant $\beta$)
\begin{align}
\E F(\hat{x}) - F^* 
&\geq \beta\cdot\min\crl*{\Delta\exp\prn*{-\frac{6\lambda R}{H}},\, \frac{H\sdiff^2}{512\lambda^2R^2}}
\end{align}

In the convex case, we recall that $\nrm{x^*}^2 \leq 2c^2$, so we choose $c^2 = \frac{B}{2}$ so that the convex assumptions are satisfied. We now have two options, if $R \leq \frac{H^2B^2}{\sdiff^2}$, then we pick $\mu = \prn*{\frac{H\sdiff^2}{B^2R^2}}^{1/3}$ so that the lower bound reads
\begin{align}
\E F(\hat{x}) - F^* 
&\geq \beta\cdot \frac{\prn*{H\sdiff^2B^4}^{1/3}}{R^{2/3}}\exp\prn*{-\frac{6\sdiff^{2/3}R^{1/3}}{H^{2/3}B^{2/3}}} \\
&\geq \beta\cdot \frac{\prn*{H\sdiff^2B^4}^{1/3}}{R^{2/3}}\exp\prn*{-6} \\
&\geq \beta'\cdot \frac{\prn*{H\sdiff^2B^4}^{1/3}}{R^{2/3}}
\end{align}
On the other hand, if $R \geq \frac{H^2B^2}{\sdiff^2}$, then we pick $\mu = \frac{H}{6R}$ so the lower bound reads
\begin{align}
\E F(\hat{x}) - F^* 
\geq \beta\cdot\min\crl*{\frac{HB^2}{R},\, \frac{\sdiff^2}{H}} 
= \beta\cdot\frac{HB^2}{R}
\end{align}
Consequently, 
\begin{align}
\E F(\hat{x}) - F^* 
\geq \beta\cdot\min\crl*{\frac{HB^2}{R},\, \frac{\prn*{H\sdiff^2 B^4}^{1/3}}{R^{2/3}}}
\end{align}
Combining these with the stochastic terms completes the proof.
\end{proof}

\subsection{Proof of \pref{thm:local-sgd-heterogeneous-uppper-bound}}\label{app:local-sgd-heterogeneous-upper-bound}

We prove the theorem with the help of several technical lemmas.
\begin{lemma}\label{lem:zeta-everywhere-progress}
For any stepsize $\eta_t \leq \frac{1}{4H}$
\[
\E\brk*{F(\bx_t) - F^*} \leq \prn*{\frac{1}{\eta_t} - \lambda}\E\nrm*{\bx_t - x^*}^2 - \frac{1}{\eta_t}\E\nrm*{\bx_{t+1} - x^*}^2 + \frac{\eta_t\sigma^2}{M} + \frac{2H}{M}\sum_{m=1}^M\E\nrm*{\bx_t - x_t^m}^2
\]
\end{lemma}
\begin{proof}
This lemma and its proof are quite similar \cite[Lemma 8][]{koloskova2020unified}. Let $\bx_{t+1} = \frac{1}{M}\sum_{m=1}^Mx_t^m$ be the average of the machines' local iterates at time $t$. Then,
\begin{align}
\E\nrm*{\bx_{t+1} - x^*}^2
&= \E\nrm*{\bx_t - \frac{\eta_t}{M}\sum_{m=1}^M\nabla F_m(x_t^m) - x^*}^2 + \eta_t^2\E\nrm*{\frac{1}{M}\sum_{m=1}^M\nabla f(x_t^m;z_t^m) - \nabla F_m(x_t^m)}^2 \label{eq:koloskova8-1} \\
&\leq \E\nrm*{\bx_t - \frac{\eta_t}{M}\sum_{m=1}^M\nabla F_m(x_t^m) - x^*}^2 + \frac{\eta_t^2\sigma^2}{M}
\end{align}
Focusing on the first term of \eqref{eq:koloskova8-1}:
\begin{align}
\E&\nrm*{\bx_t - \frac{\eta_t}{M}\sum_{m=1}^M\nabla F_m(x_t^m) - x^*}^2\nonumber\\
&= \E\nrm*{\bx_t - x^*}^2 + \eta_t^2\E\nrm*{\frac{1}{M}\sum_{m=1}^M\nabla F_m(x_t^m)}^2 - \frac{2\eta_t}{M}\sum_{m=1}^M\E\inner{\bx_t - x^*}{\nabla F_m(x_t^m)} \label{eq:koloskova8-2}
\end{align}
We can bound the second term of \eqref{eq:koloskova8-2} with:
\begin{align}
\eta_t^2&\E\nrm*{\frac{1}{M}\sum_{m=1}^M\nabla F_m(x_t^m)}^2 \nonumber\\
&\leq 2\eta_t^2\E\nrm*{\frac{1}{M}\sum_{m=1}^M\nabla F_m(x_t^m) - \nabla F_m(\bx_t)}^2 + 2\eta_t^2\E\nrm*{\frac{1}{M}\sum_{m=1}^M\nabla F_m(\bx_t) - \nabla F_m(x^*)}^2 \\
&\leq \frac{2\eta_t^2}{M}\sum_{m=1}^M\E\nrm*{\nabla F_m(x_t^m) - \nabla F_m(\bx_t)}^2 + 2\eta_t^2\E\nrm*{\nabla F(\bx_t) - \nabla F(x^*)}^2 \\
&\leq \frac{2H^2\eta_t^2}{M}\sum_{m=1}^M\E\nrm*{x_t^m - \bx_t}^2 + 4H\eta_t^2\E\brk*{F(\bx_t) - F(x^*)}
\end{align}
For the third term of \eqref{eq:koloskova8-2}:
\begin{align}
-&\frac{2\eta_t}{M}\sum_{m=1}^M\E\inner{\bx_t - x^*}{\nabla F_m(x_t^m)} \nonumber\\
&= -\frac{2\eta_t}{M}\sum_{m=1}^M\E\inner{x_t^m - x^*}{\nabla F_m(x_t^m)} + \frac{2\eta_t}{M}\sum_{m=1}^M\E\inner{x_t^m - \bx_t}{\nabla F_m(x_t^m)} \\
&\leq -\frac{2\eta_t}{M}\sum_{m=1}^M\E\brk*{F_m(x_t^m) - F_m(x^*) + \frac{\lambda}{2}\nrm*{x_t^m - x^*}^2} \nonumber\\
&\qquad+ \frac{2\eta_t}{M}\sum_{m=1}^M\E\brk*{F_m(x_t^m) - F_m(\bx_t) + \frac{H}{2}\nrm*{x_t^m - \bx_t}^2} \\
&\leq -2\eta_t\E\brk*{F(\bx_t) - F(x^*) + \frac{\lambda}{2}\nrm*{\bx_t - x^*}^2} +  \frac{H\eta_t}{M}\sum_{m=1}^M\nrm*{x_t^m - \bx_t}^2
\end{align}
Combining all these results back into \eqref{eq:koloskova8-1}, we have
\begin{align}
\E\nrm*{\bx_{t+1} - x^*}^2
&\leq \prn*{1-\lambda\eta_t}\E\nrm*{\bx_t - x^*}^2 + \frac{H\eta_t + 2H^2\eta_t^2}{M} \sum_{m=1}^M \E\nrm*{x_t^m - \bx_t}^2 \nonumber\\
&\quad+ (4H\eta_t^2 - 2\eta_t)\E\brk*{F(\bx_t) - F(x^*)} + \frac{\eta_t^2\sigma^2}{M} \\
&\leq \prn*{1-\lambda\eta_t}\E\nrm*{\bx_t - x^*}^2 + \frac{2H\eta_t}{M} \sum_{m=1}^M \E\nrm*{x_t^m - \bx_t}^2 \nonumber\\
&\quad - \eta_t\E\brk*{F(\bx_t) - F(x^*)} + \frac{\eta_t^2\sigma^2}{M}
\end{align}
where for the final line we used that $\eta_t \leq \frac{1}{4H}$. Rearranging completes the proof.
\end{proof}

\begin{lemma}\label{lem:zeta-everywhere-divergence}
If $\sup_{x,m}\nrm*{\nabla F_m(x) - \nabla F(x)}^2 \leq \bar{\zeta}^2$, then for any fixed stepsize $\eta$
\[
\frac{1}{M}\sum_{m=1}^M \E\nrm*{x_t^m - \bx_t}^2 \leq 6K\sigma^2\eta^2 + 6K^2\eta^2\bar{\zeta}^2
\]
Similarly, the decreasing stepsize $\eta_t = \frac{2}{\lambda(a + t + 1)}$ for any $a$
\[
\frac{1}{M}\sum_{m=1}^M \E\nrm*{x_t^m - \bx_t}^2 \leq 6K\sigma^2\eta_{t-1}^2 + 6K^2\bar{\zeta}^2\eta_{t-1}^2
\]
\end{lemma}
\begin{proof}
By Jensen's inequality
\begin{equation}
\E\nrm*{x_t^m - \bx_t}^2 \leq \frac{1}{M}\sum_{n=1}^M\E\nrm*{x_t^m - x_t^n}^2
\end{equation}
Therefore, it suffices to bound $\E\nrm*{x_t^m - x_t^n}^2$, which we do now:
\begin{align}
&\E\nrm*{x_t^m - x_t^n}^2 \nonumber\\
&\leq \E\left\|x_{t-1}^m - x_{t-1}^n - \eta_{t-1}\prn*{\nabla F(x_{t-1}^m) - \nabla F(x_{t-1}^n)} \right.\nonumber\\
&+ \left.\eta_{t-1}\prn*{\nabla F(x_{t-1}^m) - \nabla F_m(x_{t-1}^m) - \nabla F(x_{t-1}^n) + \nabla F_n(x_{t-1}^n)}\right\|^2 + 2\eta_{t-1}^2\sigma^2 \\
&\leq \inf_{\gamma > 0} \prn*{1 + \frac{1}{\gamma}}\E\nrm*{x_{t-1}^m - x_{t-1}^n - \eta_{t-1}\prn*{\nabla F(x_{t-1}^m) - \nabla F(x_{t-1}^n)}}^2 \nonumber\\
&\qquad+ \prn*{1+\gamma}\eta_{t-1}^2\E\nrm*{\nabla F(x_{t-1}^m) - \nabla F_m(x_{t-1}^m) - \nabla F(x_{t-1}^n) + \nabla F_n(x_{t-1}^n)}^2 + 2\eta_{t-1}^2\sigma^2 \\
&\leq  \inf_{\gamma > 0} \prn*{1 + \frac{1}{\gamma}}\prn*{1 - \lambda\eta_{t-1}}\E\nrm*{x_{t-1}^m - x_{t-1}^n}^2 + 2\eta_{t-1}^2\sigma^2 \nonumber\\
&+ \prn*{1+\gamma}\eta_{t-1}^2\E\nrm*{\nabla F(x_{t-1}^m) - \nabla F_m(x_{t-1}^m)}^2 \nonumber\\
&+ \prn*{1+\gamma}\eta_{t-1}^2\E\nrm*{\nabla F(x_{t-1}^n) - \nabla F_n(x_{t-1}^n)}^2 \nonumber\\
&- 2\prn*{1+\gamma}\eta_{t-1}^2\E\inner{\nabla F(x_{t-1}^m) - \nabla F_m(x_{t-1}^m)}{\nabla F(x_{t-1}^n) - \nabla F_n(x_{t-1}^n)} 
\end{align}
For the third inequality we used \pref{lem:local-sgd-homogeneous-upper-bound-co-coercivity}. Therefore,
\begin{align}
\frac{1}{M^2}&\sum_{m=1}^M\sum_{n=1}^M \E\nrm*{x_t^m - x_t^n}^2 \nonumber\\
&\leq \frac{1}{M^2}\sum_{m=1}^M\inf_{\gamma > 0} \prn*{1 + \frac{1}{\gamma}}\prn*{1 - \lambda\eta_{t-1}}\E\nrm*{x_{t-1}^m - x_{t-1}^n}^2 + 2\eta_{t-1}^2\sigma^2 + 2\prn*{1+\gamma}\eta_{t-1}^2\bar{\zeta}^2
\end{align}
We will unroll this recurrence, using that $x_{t_0}^m = x_{t_0}^n$ for all $m,n$ where $t_0$ is the most recent time that the iterates were synchronized, so $t - t_0 \leq K-1$. 
Taking $\gamma = K-1$, we have
\begin{align}
\frac{1}{M^2}\sum_{m=1}^M\sum_{n=1}^M \E\nrm*{x_t^m - x_t^n}^2
&= \sum_{i=t_0}^{t-1} \prn*{2\eta_i^2\sigma^2 + 2(1+\gamma)\eta_i^2\bar{\zeta}^2}\prod_{j=i+1}^{t-1}\prn*{1+\frac{1}{\gamma}}\prn*{1-\lambda\eta_j}\\
&\leq \sum_{i=t_0}^{t-1} 2\prn*{\eta_i^2\sigma^2 + K\eta_i^2\bar{\zeta}^2}\prod_{j=i+1}^{t-1}\prn*{1+\frac{1}{K-1}}\prn*{1-\lambda\eta_j} \\
&\leq \sum_{i=t_0}^{t-1} 2\prn*{\eta_i^2\sigma^2 + K\eta_i^2\bar{\zeta}^2}\prn*{1+\frac{1}{K-1}}^{K-1} \prod_{j=i+1}^{t-1}\prn*{1-\lambda\eta_j} \\
&\leq 6\prn*{\sigma^2 + K\bar{\zeta}^2}\sum_{i=t_0}^{t-1} \eta_i^2\prod_{j=i+1}^{t-1}\prn*{1-\lambda\eta_j}
\end{align}
For a constant stepsize $\eta$, 
\begin{equation}
\frac{1}{M^2}\sum_{m=1}^M\sum_{n=1}^M \E\nrm*{x_t^m - x_t^n}^2
\leq 6\prn*{\sigma^2 + K\bar{\zeta}^2}\sum_{i=t_0}^{t-1} \eta^2 
\leq 6K\prn*{\sigma^2 + K\bar{\zeta}^2}\eta^2
\end{equation}
For decreasing stepsize $\eta_t = \frac{2}{\lambda(a+t+1)}$
\begin{align}
\frac{1}{M^2}\sum_{m=1}^M\sum_{n=1}^M \E\nrm*{x_t^m - x_t^n}^2
&\leq 6\prn*{\sigma^2 + K\bar{\zeta}^2}\sum_{i=t_0}^{t-1} \eta_i^2\prod_{j=i+1}^{t-1}\frac{a+j-1}{a+j+1} \\
&= 6\prn*{\sigma^2 + K\bar{\zeta}^2}\sum_{i=t_0}^{t-1} \eta_i^2\frac{(a+i)(a+i+1)}{(a+t)(a+t+1)} \\
&= 6\prn*{\sigma^2 + K\bar{\zeta}^2}\sum_{i=t_0}^{t-1} \eta_i^2\frac{\eta_{t-1}\eta_t}{\eta_{i-1}\eta_i} \\
&\leq 6\prn*{\sigma^2 + K\bar{\zeta}^2}\sum_{i=t_0}^{t-1} \eta_i^2\frac{\eta_{t-1}^2}{\eta_i^2} \\
&\leq 6K\prn*{\sigma^2 + K\bar{\zeta}^2}\eta_{t-1}^2
\end{align}
This completes the proof.
\end{proof}

\localsgdheterogeneousupperbound*
\begin{proof}
By \pref{lem:zeta-everywhere-progress}, for any $\eta_t \leq \frac{1}{4H}$
\begin{equation}
\E\brk*{F(\bx_t) - F^*} \leq \prn*{\frac{1}{\eta_t} - \lambda}\E\nrm*{\bx_t - x^*}^2 - \frac{1}{\eta_t}\E\nrm*{\bx_{t+1} - x^*}^2 + \frac{\eta_t\sigma^2}{M} + \frac{2H}{M}\sum_{m=1}^M\E\nrm*{\bx_t - x_t^m}^2
\end{equation}
By \pref{lem:zeta-everywhere-divergence}, when $\eta_t=\eta$ is constant then
\begin{equation}
\frac{1}{M}\sum_{m=1}^M \E\nrm*{x_t^m - \bx_t}^2 \leq 6K\sigma^2\eta^2 + 6K^2\eta^2\bar{\zeta}^2
\end{equation}
and when $\eta_t = \frac{2}{\lambda(a + t + 1)}$
\begin{equation}
\frac{1}{M}\sum_{m=1}^M \E\nrm*{x_t^m - \bx_t}^2 \leq 6K\sigma^2\eta_{t-1}^2 + 6K^2\bar{\zeta}^2\eta_{t-1}^2
\end{equation}
We now consider the convex and strongly convex cases separately:

\paragraph{Convex Case:}
In the convex case, we use a constant stepsize $\eta$, so 
\begin{align}
\E\brk*{F(\bx_t) - F^*} 
&\leq \prn*{\frac{1}{\eta_t} - \lambda}\E\nrm*{\bx_t - x^*}^2 - \frac{1}{\eta_t}\E\nrm*{\bx_{t+1} - x^*}^2 + \frac{\eta_t\sigma^2}{M} + \frac{2H}{M}\sum_{m=1}^M\E\nrm*{\bx_t - x_t^m}^2 \\
&\leq \frac{1}{\eta}\E\nrm*{\bx_t - x^*}^2 - \frac{1}{\eta}\E\nrm*{\bx_{t+1} - x^*}^2 + \frac{\eta\sigma^2}{M} + 12HK\sigma^2\eta^2 + 12HK^2\eta^2\bar{\zeta}^2
\end{align}
Therefore, by the convexity of $F$
\begin{align}
\E\brk*{F\prn*{\frac{1}{KR}\sum_{t=1}^{KR} \bx_t} - F^*} 
&\leq \frac{1}{KR}\sum_{t=1}^{KR}\E\brk*{F(\bx_t) - F^*} \\
&\leq \frac{B^2}{\eta KR} + \frac{\eta\sigma^2}{M} + 12HK\sigma^2\eta^2 + 12HK^2\eta^2\bar{\zeta}^2
\end{align}
Choosing
\begin{equation}
\eta = \min\crl*{\frac{1}{4H},\, \frac{B\sqrt{M}}{\sigma\sqrt{KR}},\, \prn*{\frac{B^2}{HK^2\sigma^2}}^{1/3},\, \prn*{\frac{B^2}{HK^2\bar{\zeta}^2}}^{1/3}}
\end{equation}
then ensures
\begin{align}
\E\brk*{F\prn*{\frac{1}{KR}\sum_{t=1}^{KR} \bx_t} - F^*} 
&\leq c\cdot\prn*{\frac{HB^2}{KR} + \frac{\prn*{H\bar{\zeta}^2B^4}^{1/3}}{R^{2/3}} + \frac{\prn*{H\sigma^2B^4}^{1/3}}{K^{1/3}R^{2/3}} + \frac{\sigma B}{\sqrt{MKR}}}
\end{align}

\paragraph{Strongly Convex Case:}
Following the approach of \citet{stich2019unified}, we consider three cases:

If $KR \leq \frac{2H}{\lambda}$, then we use a constant stepsize $\eta_t = \eta = \frac{1}{4H}$ and weights $w_t = (1-\lambda\eta)^{-t-1}$.

If $KR > \frac{2H}{\lambda}$ and $t \leq KR/2$, then we take $\eta_t = \eta = \frac{1}{4H}$ and weights $w_t = 0$.

If $KR > \frac{2H}{\lambda}$ and $t > KR/2$, then we take $\eta_t = \frac{2}{\lambda(a + t + 1)}$ for $a = \frac{8H}{\lambda} - \frac{KR}{2} - 1$ so that $\eta_t \leq \frac{1}{4H}$ and we use weights $w_t = a+t$.

From above, during iterations $t$ in which the stepsize is constant, we have the recurrence
\begin{align}
\E\nrm*{\bx_{t+1} - x^*}^2 
&\leq \prn*{1 - \eta\lambda}\E\nrm*{\bx_t - x^*}^2 - \eta\E\brk*{F(\bx_t) - F^*} + \frac{\eta^2\sigma^2}{M} + 12HK\sigma^2\eta^3 + 12HK^2\eta^3\bar{\zeta}^2 \\
&\leq \prn*{1 - \eta\lambda}\E\nrm*{\bx_t - x^*}^2 + \frac{\eta^2\sigma^2}{M} + 12HK\sigma^2\eta^3 + 12HK^2\eta^3\bar{\zeta}^2
\end{align}
and for the steps when the stepsize is decreasing like $\eta_t = \frac{2}{\lambda(a + t + 1)}$ we have
\begin{align}
\E\brk*{F(\bx_t) - F^*} 
&\leq \prn*{\frac{1}{\eta_t} - \lambda}\E\nrm*{\bx_t - x^*}^2 - \frac{1}{\eta_t}\E\nrm*{\bx_{t+1} - x^*}^2 + \frac{\eta_t\sigma^2}{M} + 12HK\sigma^2\eta_{t-1}^2 + 12HK^2\bar{\zeta}^2\eta_{t-1}^2 \\
&\leq \prn*{\frac{1}{\eta_t} - \lambda}\E\nrm*{\bx_t - x^*}^2 - \frac{1}{\eta_t}\E\nrm*{\bx_{t+1} - x^*}^2 + \frac{\eta_t\sigma^2}{M} + 24HK\sigma^2\eta_t^2 + 24HK^2\bar{\zeta}^2\eta_t^2 \label{eq:hetero-local-sgd-eq1}
\end{align}
where we used that
\begin{equation}
\frac{\eta_{t-1}}{\eta_t} = \frac{a+t+1}{a+t} \leq \frac{8}{7} \implies \eta_{t-1}^2 \leq 2\eta_t^2
\end{equation}

First, consider the case that $KR > \frac{2H}{\lambda}$, and consider the first half of the steps when $\eta_t = \eta = \frac{1}{4H}$:
\begin{align}
\E&\nrm*{\bx_{KR/2+1} - x^*}^2 \nonumber\\
&\leq \prn*{1 - \frac{\lambda}{4H}}\E\nrm*{\bx_{KR/2} - x^*}^2 + \frac{\sigma^2}{16H^2M} + \frac{3K\sigma^2}{16H^2} + \frac{3K^2\bar{\zeta}^2}{16H^2} \\
&\leq \E\nrm*{\bx_0 - x^*}^2\prn*{1 - \frac{\lambda}{4H}}^{KR/2} + \prn*{\frac{\sigma^2}{16H^2M} + \frac{3K\sigma^2}{16H^2} + \frac{3K^2\bar{\zeta}^2}{16H^2}}\sum_{t=0}^{KR/2} \prn*{1 - \frac{\lambda}{4H}}^t \\
&\leq \E\nrm*{\bx_0 - x^*}^2\prn*{1 - \frac{\lambda}{4H}}^{KR/2} + \frac{\sigma^2}{4H\lambda M} + \frac{3K\sigma^2}{4H \lambda} + \frac{3K^2\bar{\zeta}^2}{4 H\lambda}\label{eq:hetero-local-sgd-eq2}
\end{align}
Now, we consider the weighted average iterate with $W = \sum_{t=0}^{KR} w_t = \sum_{t=KR/2+1}^{KR}(a+t) \geq \frac{K^2R^2}{4}$, and apply \eqref{eq:hetero-local-sgd-eq1}
\begin{align}
&\E F\prn*{\frac{1}{W}\sum_{t=0}^{KR} w_t \bx_t} - F^* \nonumber\\
&\leq \frac{1}{W}\sum_{t=0}^{KR} w_t \E\brk*{ F(\bx_t) - F^* }\\
&\leq \frac{1}{W}\sum_{t=KR/2+1}^{KR} w_t\brk*{\prn*{\frac{1}{\eta_t} - \lambda}\E\nrm*{\bx_t - x^*}^2 - \frac{1}{\eta_t}\E\nrm*{\bx_{t+1} - x^*}^2 + \frac{\eta_t\sigma^2}{M} + 24HK\sigma^2\eta_t^2 + 24HK^2\bar{\zeta}^2\eta_t^2} \\
&\leq \frac{w_{KR/2+1}\E\nrm*{\bx_{KR+1} - x^*}^2}{\eta_{KR/2+1}W} + \frac{1}{W}\sum_{t=KR/2+1}^{KR} \brk*{\frac{2\sigma^2}{\lambda M} + \frac{96HK\sigma^2}{\lambda^2(a+t+1)} + \frac{96HK^2\bar{\zeta}^2}{\lambda^2(a+t+1)}} \nonumber\\
&\qquad+ \frac{1}{W}\sum_{t=KR/2+2}^{KR} \prn*{\frac{w_t}{\eta_t} - \frac{w_{t-1}}{\eta_{t-1}} - w_t\lambda}\E\nrm*{\bx_t - x^*}^2 \\
&\leq \frac{2\lambda\prn*{a + \frac{KR}{2}+2}^2\E\nrm*{\bx_{KR+1} - x^*}^2}{K^2R^2} + \frac{4\sigma^2}{\lambda MKR} + \prn*{\frac{384H\sigma^2}{\lambda^2KR^2} + \frac{384H\bar{\zeta}^2}{\lambda^2R^2}}\log\prn*{\frac{e\prn*{a + KR + 1}}{\prn*{a + \frac{KR}{2} + 1}}} \\
&\leq \frac{162H^2\E\nrm*{\bx_{KR+1} - x^*}^2}{\lambda K^2R^2} + \frac{4\sigma^2}{\lambda MKR} + \prn*{\frac{384H\sigma^2}{\lambda^2KR^2} + \frac{384H\bar{\zeta}^2}{\lambda^2R^2}}\log\prn*{e + \frac{\lambda KR}{H}}\label{eq:hetero-local-sgd-eq3}
\end{align}
From here, we bound the first term by substituting \eqref{eq:hetero-local-sgd-eq2}:
\begin{align}
&\frac{162H^2\E\nrm*{\bx_{KR+1} - x^*}^2}{\lambda K^2R^2}\nonumber\\
&\leq \frac{162H^2}{\lambda K^2R^2}\prn*{\E\nrm*{\bx_0 - x^*}^2\prn*{1 - \frac{\lambda}{4H}}^{KR/2} + \frac{\sigma^2}{4H\lambda M} + \frac{3K\sigma^2}{4H \lambda} + \frac{3K^2\bar{\zeta}^2}{4 H\lambda}} \\
&\leq \frac{162H^2\E\nrm*{\bx_0 - x^*}^2}{\lambda K^2R^2}\prn*{1 - \frac{\lambda}{4H}}^{KR/2} + \frac{81H\sigma^2}{2\lambda^2 MK^2R^2} + \frac{243 H\sigma^2}{2 \lambda^2 KR^2} + \frac{243H\bar{\zeta}^2}{2\lambda^2R^2} \\
&\leq \frac{81\lambda\E\nrm*{\bx_0 - x^*}^2}{2}\exp\prn*{- \frac{\lambda KR}{8H}} + \frac{81\sigma^2}{4\lambda MKR} + \frac{243 H\sigma^2}{2 \lambda^2 KR^2} + \frac{243H\bar{\zeta}^2}{2\lambda^2R^2} \\
&\leq 81\Delta\exp\prn*{- \frac{\lambda KR}{8H}} + \frac{81\sigma^2}{4\lambda MKR} + \frac{243 H\sigma^2}{2 \lambda^2 KR^2} + \frac{243H\bar{\zeta}^2}{2\lambda^2R^2}
\end{align}
Combining this with \eqref{eq:hetero-local-sgd-eq3} completes the proof in the case that $KR > \frac{2H}{\lambda}$.

Consider now the case that $KR \leq \frac{2H}{\lambda}$. Then, we have $\eta_t = \eta = \frac{1}{4H}$ and $w_t = (1-\eta\lambda)^{-t-1}$ and the recurrence
\begin{align}
\E\brk*{F(\bx_t) - F^*}
&\leq \frac{1}{\eta}\prn*{1 - \eta\lambda}\E\nrm*{\bx_t - x^*}^2 - \frac{1}{\eta}\E\nrm*{\bx_{t+1} - x^*}^2  + \frac{\eta\sigma^2}{M} + 12HK\sigma^2\eta^2 + 12HK^2\eta^2\bar{\zeta}^2 
\end{align}
So, for $W = \sum_{t=0}^{KR}w_t$ we have 
\begin{align}
\E& F\prn*{\frac{1}{W}\sum_{t=0}^{KR} w_t \bx_t} - F^* \nonumber\\
&\leq \frac{1}{W}\sum_{t=0}^{KR} w_t \E\brk*{ F(\bx_t) - F^* }\\
&\leq \frac{1}{W}\sum_{t=0}^{KR} w_t\brk*{\frac{1}{\eta}\prn*{1 - \eta\lambda}\E\nrm*{\bx_t - x^*}^2 - \frac{1}{\eta}\E\nrm*{\bx_{t+1} - x^*}^2  + \frac{\eta\sigma^2}{M} + 12HK\sigma^2\eta^2 + 12HK^2\eta^2\bar{\zeta}^2} \\
&= \frac{1}{W}\sum_{t=0}^{KR} \brk*{\frac{1}{\eta}\prn*{1 - \eta\lambda}^{-t}\E\nrm*{\bx_t - x^*}^2 - \frac{1}{\eta}\prn*{1 - \eta\lambda}^{-(t+1)}\E\nrm*{\bx_{t+1} - x^*}^2} \nonumber\\
&\qquad+ \frac{\eta\sigma^2}{M} + 12HK\sigma^2\eta^2 + 12HK^2\eta^2\bar{\zeta}^2 \\
&\leq \frac{4H\E\nrm*{\bx_0 - x^*}^2}{W} + \frac{\sigma^2}{4HM} + \frac{3K\sigma^2}{4H} + \frac{3K^2\bar{\zeta}^2}{4H}
\end{align}
From here, we recall that $KR \leq \frac{2H}{\lambda}$ so
\begin{equation}
\E F\prn*{\frac{1}{W}\sum_{t=0}^{KR} w_t \bx_t} - F^* 
\leq \frac{4H\E\nrm*{\bx_0 - x^*}^2}{W} + \frac{\sigma^2}{2\lambda MKR} + \frac{3H\sigma^2}{\lambda^2 KR^2} + \frac{3H\bar{\zeta}^2}{\lambda^2R^2}
\end{equation}
Finally, we have
\begin{equation}
W = \sum_{t=0}^{KR}\prn*{1 - \frac{\lambda}{4H}}^{-t-1} \geq \prn*{1 - \frac{\lambda}{4H}}^{-KR-1}
\end{equation}
So, we conclude that when $KR \leq \frac{2H}{\lambda}$
\begin{align}
\E F\prn*{\frac{1}{W}\sum_{t=0}^{KR} w_t \bx_t} - F^* 
&\leq H\E\nrm*{\bx_0 - x^*}^2\prn*{1 - \frac{\lambda}{4H}}^{KR+1} + \frac{\sigma^2}{2\lambda MKR} + \frac{3H\sigma^2}{\lambda^2 KR^2} + \frac{3H\bar{\zeta}^2}{\lambda^2R^2} \\
&\leq \frac{2H\Delta}{\lambda}\exp\prn*{-\frac{\lambda KR}{4H}} + \frac{\sigma^2}{2\lambda MKR} + \frac{3H\sigma^2}{\lambda^2 KR^2} + \frac{3H\bar{\zeta}^2}{\lambda^2R^2}
\end{align}
This completes the proof.
\end{proof}

\subsection{Additional Details for \pref{fig:heterogeneous-local-sgd-experiments}}\label{app:heterogeneous-experiment-details}
The training set of MNIST (60,000 examples) was divided by digit into ten groups of equal size $n \approx 6,000$ (which required discarding some examples from the more common digits). PCA was used to reduce the dimensionality to 100, but no other preprocessing was used.

Then, for each of the 25 combinations ($i$,$j$) for even $i$ and odd $j$, a binary classification ``task'' was created, i.e.~classifying even ($+1$) versus odd ($-1$). These tasks were arbitrarily labelled task $1,2,\dots,25$.

For each $p \in [0.0, 0.2, 0.4, 0.6, 0.8, 1.0]$, machine $m$ was assigned data composed of $p\cdot 2n$ random examples from task $m$, and $(1-p)\cdot 2n$ random examples from a mixture of all the tasks. 

Local and Minibatch SGD were then used to optimize the logistic loss for each of the six described local datasets. The constant stepsize was tuned (from a log-scale grid of 10 points ranging from $e^{-6},\dots,e^{0}$ for Minibatch SGD, and a log-scale grid of 10 points ranging from $e^{-8},\dots,e^{-1}$ for Local SGD) for each value of $p$, $K$, and $R$ individually, and the average loss over four runs is reported for the best stepsize for each point in the plot. That is, each point in the plot represents the best possible performance of the algorithm for that $p$, $K$, and $R$ specifically.

Finally, we computed the value of $\zeta_*^2$ as a function of $p$ by using Newton's method to compute a very accurate estimate of the minimizer, and then explicitly calculating $\zeta_*^2(p)$ at that point.

\section{Proofs from \pref{sec:intermittent-communication-setting}}

Several of the lower bounds in \pref{sec:intermittent-communication-setting} are based on the same framework which we introduce here. For a vector $x$, we define its progress as
\begin{equation}
\prog{\alpha}(x) := \max\crl*{i\,:\, \abs{x_i} > \alpha}
\end{equation}
Our lower bound approach, as explained in \pref{subsec:high-level-lower-bound-approach}, is to show that any intermittent communication algorithm will fail to achieve a high amount of progress, even for randomized algorithms that leave the span of previous stochastic gradient queries. To formalize this, we define
\begin{definition}\label{def:robust-zero-chain-oracle}
A stochastic gradient oracle $g(x;z)$ is an $(\alpha,p,\delta)$-robust-zero-chain if there exists $\mc{Z}_0, \mc{Z}_1$ such that 
\begin{enumerate}
\item $\P(z \in \mc{Z}_0 \cup \mc{Z}_1) \geq 1-\delta$
\item $\P(z\in\mc{Z}_0 | z \in \mc{Z}_0 \cup \mc{Z}_1) \geq 1-p$
\item For all $z \in \mc{Z}_0$ and all $x$, $\prog{0}(g(x;z)) \leq \prog{\alpha}(x)$ and there exist functions $G_1,\dots,G_d$ such that
\[
\prog{\alpha}(x) \leq i \implies g(x;z) = G_i(x_1,x_2,\dots,x_i;z)
\]
\item For all $z \in \mc{Z}_1$ and all $x$, $\prog{0}(g(x;z)) \leq \prog{\alpha}(x) + 1$ and there exist functions $G_1,\dots,G_d$ such that
\[
\prog{\alpha}(x) \leq i \implies g(x;z) = G_i(x_1,x_2,\dots,x_{i+1};z)
\]
\end{enumerate}
\end{definition}
We now show that any algorithm that interacts with a robust zero chain will have a low amount of progress:
\begin{restatable}{lemma}{intermittentcommunicationprogress}\label{lem:intermittent-communication-progress}
Let $g(x;z)$ be an $(\alpha,p,\delta)$-robust-zero-chain, let $U\in\R^{D\times d}$ be a uniformly random orthogonal matrix with $U^\top U = I_{d\times{}d}$ for $D \geq d + \frac{2B^2}{\alpha^2}\log(32MKRd)$, and let $x^m_{k,r}$ be the $k\mathth$ oracle query on the $m\mathth$ machine during the $r\mathth$ round of communication for an intermittent communication algorithm that interacts with the stochastic gradient oracle $g_U(x;z) := U g(U^\top x;z)$. Then if $\max_{m,k,r} \nrm{x^m_{k,r}} \leq B$, the algorithm's output $\hat{x}$ will have progress
\[
\P\prn*{\prog{\alpha}(U^\top \hat{x}) \leq \min\crl*{KR,\ 8KRp + 12R\log M + 12R}} \geq \frac{5}{8} - 2MKR\delta
\]
\end{restatable}
The main ideas leading to \pref{lem:intermittent-communication-progress} stem from \cite{woodworth16tight} and \cite{carmon2017lower1}, who show that when a random rotation is applied to the objective and the dimension is sufficiently large, every algorithm behaves essentially as if its queries remained in the span of previously seen gradients. In the original arguments, the proof of this claim was extremely complicated and required a great deal of care due to subtleties with conditioning on the stochastic gradient oracle queries. Since then, the argument has gradually be refined and simplified, culminating in \citet{yairthesis} who presents the simplest argument to date. The proof of \pref{lem:intermittent-communication-progress} therefore resembles the proof of \citep[Proposition 2.4][]{yairthesis}, however, the arguments must be extended to accomodate the intermittent communication setting.

To facilitate our proofs, we introduce some notation. Recalling $\mc{Z}_0$ and $\mc{Z}_1$ from \pref{def:robust-zero-chain-oracle}, we define
\begin{equation}
S^m_{k,r} = \min\crl*{d,\ \sum_{k'=1}^{k-1} \indicator{z^m_{k',r} \in \mc{Z}_1} + \sum_{r'=1}^{r-1}\max_{1\leq m'\leq M}\sum_{k'=1}^K \indicator{z^{m'}_{k',r'} \in \mc{Z}_1}}
\end{equation}
We also define the following ``good events'' where the progress of the algorithm's oracle queries remains small
\begin{align}
\mc{G}^m_{k,r} &= \crl*{\prog{\alpha}(U^\top x^m_{k,r}) \leq S^m_{k,r}} \\
\bar{\mc{G}}^m_{k,r} &= \bigcap_{k'<k}\mc{G}^m_{k',r} \cap \bigcap_{r'<r}\bigcap_{m',k'} \mc{G}^{m'}_{k',r'}
\end{align}
We also define the event
\begin{equation}
Z = \crl*{\forall_{m,k,r}\ z^m_{k,r} \in \mc{Z}_0 \cup \mc{Z}_1}
\end{equation}
Finally, we use
\begin{equation}
U_{\leq i} = \brk*{U_1,U_2,\dots,U_i,0,\dots,0}
\end{equation}
to denote the matrix $U$ with the $(i+1)\mathth$ through $d\mathth$ columns replaced by zeros.

We begin by showing that when the good events $\bar{\mc{G}}^m_{k,r}$ happen, the algorithm's queries are determined by only a subset of the columns of $U$.
\begin{lemma}\label{lem:robust-zero-chain-u-dependence}
Let $g(x;z)$ be an $(\alpha,p,\delta)$-robust-zero-chain, let $U \in \R^{D\times d}$ be a uniformly random orthogonal matrix with $U^\top U = I_{d\times d}$, and let $x^m_{k,r}$ be the $k\mathth$ oracle query on the $m\mathth$ machine during the $r\mathth$ round of communication for an intermittent communication algorithm that interacts with the stochastic gradient oracle $g_U(x;z) := U g(U^\top x;z)$. Then conditioned on $\mc{S}$, the $\sigma$-algebra generated by the sequence $\crl{S^m_{k,r}}_{m,k,r}$; the event $\bar{G}^m_{k,r}$; and the event $Z$, the query $x^m_{k,r}$ is a measurable function of $\xi$ and $U_{\leq S^m_{k,r}}$. Similarly, conditioned on $\mc{S}$, $Z$, and $\bigcap_{m=1}^M\bar{G}^m_{K,R}$, the output of the algorithm, $\hat{x}$ is a measurable function of $\xi$ and $U_{\leq \max_m S^m_{K,R}}$.
\end{lemma}
\begin{proof}
By the definition of an intermittent communication algorithm, the oracle queries are generated according to query rules as
\begin{equation}
x^m_{k,r} = \mc{Q}^m_{k,r}\prn*{\crl*{x^m_{k',r}, g_U(x^m_{k',r},z^m_{k',r}):k'<k} \cup \crl*{x^{m'}_{k',r'},g_U(x^{m'}_{k',r'},z^{m'}_{k',r'}):r'<r}, \xi}
\end{equation}
The question is: upon which columns of $U$ does the righthand side of this equation depend when we condition on $\mc{S}$ and $\bar{\mc{G}}^m_{k,r}$? To answer this, we note that by \pref{def:robust-zero-chain-oracle}, if $z \in \mc{Z}_0$ then for any $x$
\begin{equation}
\prog{\alpha}(U^\top x) \leq i \implies 
\left\{
\begin{gathered}
\prog{0}(g(U^\top x;z)) \leq i \implies U g(U^\top x;z) = U_{\leq i} g(U^\top x;z) \\
g(U^\top x;z) = g(U_{\leq i}^\top x;z) 
\end{gathered}  
\right\}
\implies g_U(x;z) = g_{U_{\leq i}}(x;z)
\end{equation}
By a similar argument, even for $z \in \mc{Z}_1$
\begin{equation}
\prog{\alpha}(U^\top x) \leq i \implies g_U(x;z) = g_{U_{\leq i+1}}(x;z)
\end{equation}
Therefore, conditioned on the events $Z$ and $\bar{\mc{G}}^m_{k,r}$, for each $m',k',r'$ such that $m'=m,r'=r,k'<k$ or $r'<r$, let $i^{m'}_{k',r'} = \min\crl*{d,\ S^{m'}_{k',r'} + \indicator{z^{m'}_{k',r'} \in \mc{Z}_1}} \leq S^m_{k,r}$ then
\begin{equation}
g_U(x^{m'}_{k',r'}; z^{m'}_{k',r'}) = g_{U_{\leq i^{m'}_{k',r'}}}(x^{m'}_{k',r'};z^{m'}_{k',r'}) = g_{U_{\leq S^m_{k,r}}}(x^{m'}_{k',r'};z^{m'}_{k',r'})
\end{equation}
We conclude that
\begin{align}
x^m_{k,r} 
&= \mc{Q}^m_{k,r}\prn*{\crl*{x^m_{k',r}, g_U(x^m_{k',r},z^m_{k',r}):k'<k} \cup \crl*{x^{m'}_{k',r'},g_U(x^{m'}_{k',r'},z^{m'}_{k',r'}):r'<r}, \xi} \\
&= \mc{Q}^m_{k,r}\prn*{\crl*{x^m_{k',r}, g_{U_{\leq S^m_{k,r}}}(x^m_{k',r},z^m_{k',r}):k'<k} \cup \crl*{x^{m'}_{k',r'},g_{U_{\leq S^m_{k,r}}}(x^{m'}_{k',r'},z^{m'}_{k',r'}):r'<r}, \xi}
\end{align}
so conditioned on $\mc{S}$, $Z$, and $\bar{\mc{G}}^m_{k,r}$, $x^m_{k,r}$ is a measurable function of $U_{\leq S^m_{k,r}}$ and $\xi$.

We can apply the same argument to the algorithm's output
\begin{equation}
\hat{x} = \hat{X}\prn*{\crl*{x^m_{k,r}, g_U(x^m_{k,r},z^m_{k,r})}_{m,k,r}, \xi} = \hat{X}\prn*{\crl*{x^m_{k,r}, g_{U_{\leq \max_m S^m_{K,R}}}(x^m_{k,r},z^m_{k,r})}_{m,k,r}, \xi}
\end{equation}
which completes the proof.
\end{proof}

Next, we upper bound the random variables $S^m_{k,r}$: 
\begin{lemma}\label{lem:sum-max-binomials}
For any $(\alpha,p,\delta)$-robust-zero-chain,
\[
\P\prn*{\max_{m,k,r} S^m_{k,r} \geq \min\crl*{KR,\ 8KRp + 12R\log M + 12R}\,\middle|\,Z} \leq \frac{1}{4}
\]
\end{lemma}
\begin{proof}
The claim is equivalent to
\begin{equation}
\P\prn*{\sum_{r=1}^R\max_{1\leq m \leq M} \sum_{k=1}^K \indicator{z^m_{k,r} \in \mc{Z}_1} \geq \min\crl*{KR,\ 8KRp + 12R\log M + 12R}\,\middle|\,Z} \leq \frac{1}{4}
\end{equation}
Since $g(x;z)$ is an $(\alpha,p,\delta)$-robust-zero-chain, the random coins of the stochastic gradient oracles are independent, so, conditioned on $Z$ the indicators $\indicator{z^m_{k,r} \in \mc{Z}_1}$ are independent Bernoulli random variables with success probability at most $p$. It follows that for each $m$ and $r$, $\sum_{k=1}^K \indicator{z^m_{k,r} \in \mc{Z}_1}$ are independent $\textrm{Binomial}(K,p)$ random variables. 

Therefore, for each $r$, by the union bound and then the Chernoff bound, for any $c \geq 0$
\begin{equation}
\P\prn*{\max_{1\leq m \leq M} \sum_{k=1}^K \indicator{z^m_{k,r} \in \mc{Z}_1} \geq (1+c)Kp\,\middle|\,Z}
\leq M\exp\prn*{-\frac{c^2 Kp}{2 + c}}
\end{equation}
Furthermore, for any random variable $X \in [0,K]$, $\E X = \int_0^K \P\prn*{X \geq x} dx$. Therefore, for any $\epsilon > 0$
\begin{align}
\E\brk*{\max_{1\leq m \leq M} \sum_{k=1}^K \indicator{z^m_{k,r} \in \mc{Z}_1}\,\middle|\,Z}
&= \int_0^K \P\prn*{\max_{1\leq m \leq M} \sum_{k=1}^K \indicator{z^m_{k,r} \in \mc{Z}_1} \geq x\,\middle|\,Z} dx \\
&= Kp\int_{-1}^{\frac{1-p}{p}} \P\prn*{\max_{1\leq m \leq M} \sum_{k=1}^K \indicator{z^m_{k,r} \in \mc{Z}_1} \geq (1+c)Kp\,\middle|\,Z} dc \\
&\leq (1+\epsilon)Kp + MKp\int_{\epsilon}^{\frac{1-p}{p}} \exp\prn*{-\frac{cKp}{2 + c}} dc \\
&\leq (1+\epsilon)Kp + MKp\int_{\epsilon}^{\infty} \exp\prn*{-\frac{c\epsilon Kp}{2 + \epsilon}} dc \\
&= (1+\epsilon)Kp + \frac{M(2+\epsilon)}{\epsilon}\exp\prn*{-\frac{\epsilon^2Kp}{2+\epsilon}}
\end{align}
For the second line we used the change of variables $x \to (1+c)Kp$. We take $\epsilon = 1 + \frac{3}{Kp}\log M$ to conclude
\begin{equation}
\E\brk*{\max_{1\leq m \leq M} \sum_{k=1}^K \indicator{z^m_{k,r} \in \mc{Z}_1}\,\middle|\,Z}
\leq (1+\epsilon)Kp + \frac{M(2+\epsilon)}{\epsilon}\exp\prn*{-\frac{\epsilon^2Kp}{2+\epsilon}} 
\leq 2Kp + 3\log M + 3
\end{equation}
It follows that 
\begin{equation}
\E\brk*{\sum_{r=1}^R \max_{1\leq m \leq M} \sum_{k=1}^K \indicator{z^m_{k,r} \in \mc{Z}_1}\,\middle|\,Z}
\leq 2KRp + 3R\log M + 3R
\end{equation}
Markov's inequality along with the observation that $S^m_{k,r} \leq KR$ for all $m,k,r$ completes the proof.
\end{proof}

Using the previous lemmas, we prove the main result:
\intermittentcommunicationprogress*
\begin{proof}
We begin by conditioning on $Z$ and $\mc{S}$, the $\sigma$-algebra generated by $\crl{S^m_{k,r}}_{m,k,r}$ and bounding
\begin{align}
\P&\prn*{\prog{\alpha}(U^\top \hat{x}) > \max_m S^m_{K,R} \lor \exists_{m,k,r}\ \prog{\alpha}(U^\top x^m_{k,r}) > S^m_{k,r}\,\middle|\,Z,\mc{S}}\nonumber\\
&= \P\prn*{\crl*{\prog{\alpha}(U^\top \hat{x}) > \max_m S^m_{K,R}} \cup \bigcup_{m,k,r} \crl*{\prog{\alpha}(U^\top x^m_{k,r}) > S^m_{k,r}}\,\middle|\,Z,\mc{S}} \\
&= \P\prn*{\crl*{\crl*{\prog{\alpha}(U^\top \hat{x}) > \max_m S^m_{K,R}}\cap \bigcap_{m=1}^M\bar{\mc{G}}^m_{K,R}}  \cup \bigcup_{m,k,r}\crl*{\prog{\alpha}(U^\top x^m_{k,r}) > S^m_{k,r}} \cap \bar{\mc{G}}^m_{k,r}\,\middle|\,Z,\mc{S}} \\
&\leq \P\brk*{\prog{\alpha}(U^\top \hat{x}) > \max_m S^m_{K,R},\ \bigcap_{m=1}^M\bar{\mc{G}}^m_{K,R}\,\middle|\,Z,\mc{S}} + \sum_{m,k,r} \P\prn*{\crl*{\prog{\alpha}(U^\top x^m_{k,r}) > S^m_{k,r}} \cap \bar{\mc{G}}^m_{k,r}\,\middle|\,Z,\mc{S}} \\
&\leq \sum_{i > \max_m S^m_{K,R}} \P\prn*{\abs*{\inner{U_i}{\hat{x}}} > \alpha, \bigcap_{m=1}^M \bar{\mc{G}}^m_{K,R}\,\middle|\,Z,\mc{S}} + \sum_{m,k,r} \sum_{i > S^m_{k,r}} \P\prn*{\abs*{\inner{U_i}{x^m_{k,r}}} > \alpha, \bar{\mc{G}}^m_{k,r}\,\middle|\,Z,\mc{S}} 
\end{align}
By \pref{lem:robust-zero-chain-u-dependence}, there exist measurable functions $\mathsf{A}^m_{k,r}$ and $\mathsf{B}^m_{k,r}$ such that
\begin{equation}
x^m_{k,r} = \mathsf{A}^m_{k,r}(U_{\leq S^m_{k,r}},\xi) \indicator{Z,\bar{\mc{G}}^m_{k,r}} + \mathsf{B}^m_{k,r}(U,\xi) \indicator{\lnot Z \lor \lnot \bar{\mc{G}}^m_{k,r}}
\end{equation}
Therefore, 
\begin{align}
\P&\prn*{\prog{\alpha}(U^\top \hat{x}) > \max_m S^m_{K,R} \lor \exists_{m,k,r}\ \prog{\alpha}(U^\top x^m_{k,r}) > S^m_{k,r}\,\middle|\,Z,\mc{S}}\nonumber\\
&\leq \sum_{i > \max_m S^m_{K,R}} \P\prn*{\abs*{\inner{U_i}{\hat{\mathsf{A}}(U_{\leq \max_m S^m_{K,R}},\xi)}} > \alpha,\ \bigcap_{m=1}^M \bar{\mc{G}}^m_{K,R}\,\middle|\,Z,\mc{S}} \nonumber\\
&\qquad+ \sum_{m,k,r} \sum_{i > S^m_{k,r}} \P\prn*{\abs*{\inner{U_i}{\mathsf{A}^m_{k,r}(U_{\leq S^m_{k,r}},\xi)}} > \alpha, \bar{\mc{G}}^m_{k,r}\,\middle|\,Z,\mc{S}} \\
&\leq \sum_{i > \max_m S^m_{K,R}} \P\prn*{\abs*{\inner{U_i}{\hat{\mathsf{A}}(U_{\leq \max_m S^m_{K,R}},\xi)}} > \alpha\,\middle|\,Z,\mc{S}} \nonumber\\
&\qquad+ \sum_{m,k,r} \sum_{i > S^m_{k,r}} \P\prn*{\abs*{\inner{U_i}{\mathsf{A}^m_{k,r}(U_{\leq S^m_{k,r}},\xi)}} > \alpha\,\middle|\,Z,\mc{S}}
\end{align}
The algorithm's random coins, $\xi$, and the stochastic gradient oracles' random coins, $\crl{z^m_{k,r}}_{m,k,r}$ which determine $Z$ and $\mc{S}$, are independent of the random rotation $U$. Furthermore, for $i > S^m_{k,r}$, $U_i$ conditioned on $U_{\leq S^m_{k,r}}$ is a uniformly random vector on the $(D-S^m_{k,r})$-dimensional unit sphere orthogonal to the range of $U_{\leq S^m_{k,r}}$. Furthermore, by assumption $\nrm{\mathsf{A}^m_{k,r}(U_{\leq S^m_{k,r}},\xi)} \leq B$. Therefore, following \citet{yairthesis} concentration of measure on the sphere implies \citep{ball1997elementary}
\begin{equation}
\P\prn*{\abs*{\inner{U_i}{\mathsf{A}^m_{k,r}(U_{\leq S^m_{k,r}},\xi)}} > \alpha\,\middle|\,Z,\mc{S}}
\leq 2\exp\prn*{-\frac{(D - S^m_{k,r} + 1)\alpha^2}{2B^2}}
\end{equation}
Using the fact that $S^m_{k,r} \leq d$ and $D \geq d + \frac{2B^2}{\alpha^2}\log(32MKRd)$, we conclude that
\begin{align}
\P&\prn*{\prog{\alpha}(U^\top \hat{x}) > \max_m S^m_{K,R} \lor \exists_{m,k,r}\ \prog{\alpha}(U^\top x^m_{k,r}) > S^m_{k,r}\,\middle|\,Z,\mc{S}}\nonumber\\
&\leq 2(MKR + 1) d\exp\prn*{-\frac{(D - d + 1)\alpha^2}{2B^2}} \\
&\leq 2(MKR+1)d\exp\prn*{-\frac{(d + \frac{2B^2}{\alpha^2}\log(32MKRd) - d + 1)\alpha^2}{2B^2}} 
\leq \frac{1}{8} \label{eq:intermittent-communication-progress-lemma-term1}
\end{align}

To complete the proof of the lemma, we note that for
\begin{equation}
T = \min\crl*{KR,\ 8KRp + 12R\log M + 12R}
\end{equation}
we can upper bound
\begin{align}
&\P\prn*{\prog{\alpha}(U^\top \hat{x}) > T} \nonumber\\
&\leq \P\prn*{\prog{\alpha}(U^\top \hat{x}) > \max_{m,k,r} S^m_{k,r} \lor \exists_{m,k,r} \prog{\alpha}(U^\top x^m_{k,r}) > S^m_{k,r},\ \max_{m,k,r} S^m_{k,r} \leq T} + \P\prn*{\max_{m,k,r} S^m_{k,r} > T} \\
&\leq \P\prn*{\prog{\alpha}(U^\top \hat{x}) > \max_{m,k,r} S^m_{k,r} \lor \exists_{m,k,r} \prog{\alpha}(U^\top x^m_{k,r}) > S^m_{k,r}} + \P\prn*{\max_{m,k,r} S^m_{k,r} > T} \\
&= \P\prn*{\prog{\alpha}(U^\top \hat{x}) > \max_{m,k,r} S^m_{k,r} \lor \exists_{m,k,r} \prog{\alpha}(U^\top x^m_{k,r}) > S^m_{k,r}\,\middle|\,Z}\P(Z) + \P\prn*{\max_{m,k,r} S^m_{k,r} > T\,\middle|\,Z}\P(Z) \nonumber\\
&+ \P\prn*{\prog{\alpha}(U^\top \hat{x}) > \max_{m,k,r} S^m_{k,r} \lor \exists_{m,k,r} \prog{\alpha}(U^\top x^m_{k,r}) > S^m_{k,r}\,\middle|\,\lnot Z}\P(\lnot Z) + \P\prn*{\max_{m,k,r} S^m_{k,r} > T\,\middle|\,\lnot Z}\P(\lnot Z) \\
&\leq \P\prn*{\prog{\alpha}(U^\top \hat{x}) > \max_{m,k,r} S^m_{k,r} \lor \exists_{m,k,r} \prog{\alpha}(U^\top x^m_{k,r}) > S^m_{k,r}\,\middle|\,Z} + \P\prn*{\max_{m,k,r} S^m_{k,r} > T\,\middle|\,Z} + 2(1 - \P(Z))
\end{align}
By \eqref{eq:intermittent-communication-progress-lemma-term1}, the first term is bounded by $\frac{1}{8}$, by \pref{lem:sum-max-binomials} the second term is at most $\frac{1}{4}$, and by the union bound, 
\begin{equation}
\P(Z) \geq (1-\delta)^{MKR} \geq 1 - MKR\delta
\end{equation}
This completes the proof.
\end{proof}

\subsection{Proof of \pref{thm:homogeneous-convex-lower-bound}} \label{app:intermittent-communication-homogeneous-convex-lower-bound}

For a scalar function $\psi:\R\to\R$, $\zeta > 0$, and $N \geq 2$ to be defined later, we will construct the hard instance
\begin{equation}\label{eq:def-F-generic-psi-homogeneous-lower-bound}
F(x) = -\psi'(\zeta)x_1 + \psi(x_N) + \sum_{i=1}^{N-1}\psi(x_{i+1} - x_i)
\end{equation}
We also define a stochastic gradient oracle for $F$ 
\begin{equation}\label{eq:oracle-homogeneous-convex-lower-bound}
\begin{aligned}
z &= \begin{cases} 
0 & \textrm{with probability } (1-p)(1-\delta) \\ 
1 & \textrm{with probability } p(1-\delta) \\
2 & \textrm{with probability } \delta
\end{cases} \\
g(x;0) &= \sum_{i=1}^{\prog{\alpha}(x)}e_ie_i^\top \nabla F(\brk*{x_1,\dots,x_{\prog{\alpha}(x)},0,\dots,0}) \\
g(x;1) &= \frac{1}{p}\sum_{i=1}^{\prog{\alpha}(x) + 1}e_ie_i^\top\nabla F(\brk*{x_1,\dots,x_{\prog{\alpha}(x)+1},0,\dots,0}) - \frac{1-p}{p}g(x;0) \\
g(x;2) &= \frac{1}{\delta}\nabla F(x) - \frac{1-\delta}{\delta}\sum_{i=1}^{\prog{\alpha}(x) + 1}e_ie_i^\top\nabla F(\brk*{x_1,\dots,x_{\prog{\alpha}(x)+1},0,\dots,0})
\end{aligned}
\end{equation}
The following lemma relates the properties of $\psi$ to those of $F$ and $g$:
\begin{lemma}\label{lem:homogeneous-convex-generic-psi-construction}
Let $\psi$ be convex, twice-differentiable, and even, and let it satisfy $\psi(0) = \psi'(0) = 0$, and for all $x$, $\abs{\psi'(x)} \leq \ell_1$, and $\psi''(x) \leq \ell_2$. Then, for $F$ defined as in \eqref{eq:def-F-generic-psi-homogeneous-lower-bound}
\begin{enumerate}
\item $F$ is convex and $4\ell_2$-smooth
\item $\nrm{x^*}^2 \leq \zeta^2N^3$
\item $F(0) - F^* \leq N(\zeta \psi'(\zeta) - \psi(\zeta))$
\item For any $x$, $\prog{\alpha}(x) \leq \frac{N}{2} \implies F(x) - F^* \geq \psi'(\zeta)\brk*{N\zeta - \alpha - {\psi^*}'\prn*{\frac{2}{N}\psi'(\zeta)}} + \frac{N}{2}\brk*{\psi\prn*{{\psi^*}'\prn*{\frac{2}{N}\psi'(\zeta)}} - 2\psi(\zeta)}$
\item $\E_z g(x;z) = \nabla F(x)$
\item $\sup_x\E_z\nrm*{g(x;z) - \nabla F(x)} \leq \frac{6(1-p)\ell_1^2}{p} + \prn*{\frac{160}{p} + \frac{32}{\delta}}N\ell_2^2\alpha^2$
\item $g$ is an $(\alpha,p,\delta)$-robust-zero-chain
\end{enumerate}
\end{lemma}
\begin{proof}
We will prove each property one by one.

\textbf{1)} Because $\psi$ is convex, $F$ is the sum of convex functions and is thus convex. In addition,
\begin{equation}
\nabla^2 F(x) = \psi''(x_N)e_Ne_N^\top + \sum_{i=1}^{N-1}\psi''(x_{i+1} - x_i)(e_{i+1} - e_i)(e_{i+1} - e_i)^\top
\end{equation}
Therefore, for any unit vector $u$
\begin{align}
u^\top \nabla^2 F(x) u
&= \psi''(x_N)u_N^2 + \sum_{i=1}^{N-1}\psi''(x_{i+1} - x_i)(u_{i+1} - u_i)^2 \\
&\leq \ell_2\prn*{u_N^2 + \sum_{i=1}^{N-1}2u_{i+1}^2 + 2u_i^2} 
\leq 4\ell_2\nrm{u}^2 
= 4\ell_2
\end{align}
Therefore, $F$ is $4\ell_2$-smooth.

For properties 2 and 3, we will compute the minimizer of $F$, which satisfies $\nabla F(x^*) = 0$, i.e.
\begin{equation}
\begin{aligned}
0 &= -\psi'(\zeta) - \psi'(x^*_2 - x^*_1) \\
0 &= -\psi'(x^*_{i+1} - x^*_i) + \psi'(x^*_i - x^*_{i-1}) \qquad2 \leq i \leq N-1 \\
0 &= \psi'(x^*_N - x^*_{N-1}) + \psi'(x^*_N)
\end{aligned}
\end{equation}
Therefore, 
\begin{equation}
x^* = \zeta\sum_{i=1}^N (N - i + 1)e_i
\end{equation}
is a minimizer of $F$.

\textbf{2)} The solution has squared norm 
\begin{equation}
\nrm{x^*}^2 = \zeta^2 \sum_{i=1}^N (N - i + 1)^2 \leq \zeta^2 N^3
\end{equation}

\textbf{3)} Because $\psi(0) = 0$, the value of $F(0) = 0$, and the value at the optimum is
\begin{equation}
F(x^*) = -N\zeta \psi'(\zeta) + \psi(\zeta) + \sum_{i=1}^N \psi(-\zeta) = N(\psi(\zeta) - \zeta \psi'(\zeta))
\end{equation}

\textbf{4)} By Jensen's inequality and the convexity of $\psi$
\begin{equation}
\sum_{i=1}^{\lceil N/2 \rceil} \psi(x_{i+1} - x_i) = \left\lceil \frac{N}{2} \right\rceil \cdot \frac{1}{\lceil N/2 \rceil}\sum_{i=1}^{\lceil N/2 \rceil} \psi(x_{i+1} - x_i) \geq \frac{N}{2} \psi(x_{\lceil N/2 \rceil+1} - x_1)
\end{equation}
Therefore, for any $x$ with $\prog{\alpha}(x) \leq \frac{N}{2}$,
\begin{align}
F(x) 
&= -\psi'(\zeta)x_1 + \psi(x_N) + \sum_{n=1}^{N-1}\psi(x_{i+1} - x_i) \\
&\geq -\psi'(\zeta)x_1 + \frac{N}{2} \psi(x_1 - x_{\lceil N/2 \rceil+1}) \\
&\geq -\psi'(\zeta)x_1 + \frac{N}{2} \psi(x_1 - \alpha) \\
&\geq \inf_y -\psi'(\zeta)y + \frac{N}{2} \psi(y - \alpha) \\
&= -\psi'(\zeta)\brk*{\alpha + {\psi^*}'\prn*{\frac{2}{N}\psi'(\zeta)}} + \frac{N}{2} \psi\prn*{{\psi^*}'\prn*{\frac{2}{N}\psi'(\zeta)}}
\end{align}
where $\psi*$ is the Fenchel conjugate of $\psi$: $\psi^*(y) = \sup_x \inner{y}{x} - \psi(x)$, which satisfies ${\psi^*}'(\psi'(x)) = x$. Therefore,
\begin{equation}
F(x) - F^* \geq \psi'(\zeta)\brk*{N\zeta - \alpha - {\psi^*}'\prn*{\frac{2}{N}\psi'(\zeta)}} + \frac{N}{2}\brk*{\psi\prn*{{\psi^*}'\prn*{\frac{2}{N}\psi'(\zeta)}} - 2\psi(\zeta)}
\end{equation}

\textbf{5)} Let $\nabla_0 = \sum_{i=1}^{\prog{\alpha}(x)}e_ie_i^\top \nabla F(\brk*{x_1,\dots,x_{\prog{\alpha}(x)},0,\dots,0})$ and $\nabla_1 = \sum_{i=1}^{\prog{\alpha}(x)+1}e_ie_i^\top\nabla F(\brk*{x_1,\dots,x_{\prog{\alpha}(x)+1},0,\dots,0})$, then
\begin{equation}
\begin{aligned}
\E_z g(x;z) 
&= (1-p)(1-\delta)g(x;0) + p(1-\delta)g(x;1) + \delta g(x;2) \\
&= (1-p)(1-\delta)\nabla_0 + p(1-\delta)\prn*{\frac{1}{p}\nabla_1 - \frac{1-p}{p}\nabla_0} + \delta\prn*{\frac{1}{\delta}\nabla F(x) - \frac{1-\delta}{\delta}\nabla_1} = \nabla F(x)
\end{aligned}  
\end{equation}

\textbf{6)} Using the same $\nabla_0$ and $\nabla_1$ as above, we first expand
\begin{align}
&\E_z \nrm*{g(x;z) - \nabla F(x)}^2 \nonumber\\
&\leq (1-p)\nrm*{\nabla_0 - \nabla F(x)}^2 + \frac{1}{p}\nrm*{\nabla_1 - (1-p)\nabla_0 - p\nabla F(x)}^2 + \frac{(1-\delta)^2}{\delta}\nrm*{\nabla_1 - \nabla F(x)}^2 \\
&\leq \frac{3(1-p)}{p}\nrm*{\nabla_0 - \nabla F(x)}^2 + \prn*{\frac{2}{p} + \frac{1}{\delta}}\nrm*{\nabla_1 - \nabla F(x)}^2 
\end{align}
Because $\psi''(x) \leq \ell_2$, $\psi'$ is $\ell_2$-Lipschitz, and we also have $\abs{\psi'(x)} \leq \ell_1$. Therefore, for $j = \prog{\alpha}(x)$,
\begin{align}
&\nrm*{\nabla_0 - \nabla F(x)}^2 \nonumber\\
&= \left\|-\psi'(\zeta)e_1 + \sum_{i=1}^{j-1}\psi'(x_{i+1} - x_i)(e_{i+1} - e_i) + \psi'(-x_j)( - e_j)\right.\nonumber\\
&\qquad\qquad\qquad\left.+ \psi'(\zeta)e_1 - \psi'(x_N)e_N - \sum_{i=1}^{N-1}\psi'(x_{i+1} - x_i)(e_{i+1} - e_i)\right\|^2 \\
&= \nrm*{-e_j\psi'(-x_j) - \psi'(x_N)e_N - \sum_{i=j}^{N-1}\psi'(x_{i+1} - x_i)(e_{i+1} - e_i)}^2 \\
&\leq 2\nrm*{-\psi'(-x_j)e_j - \psi'(x_{j+1} - x_j)(e_{j+1} - e_j)}^2 + 2\nrm*{\psi'(x_N)e_N + \sum_{i=j+1}^{N-1}\psi'(x_{i+1} - x_i)(e_{i+1} - e_i)}^2 \\
&\leq 2(\ell_1^2 + \ell_2^2\alpha^2) + 32(N-j-1)\ell_2^2\alpha^2 \\
&\leq 2\ell_1^2 + 32N\ell_2^2\alpha^2
\end{align}
Similarly, 
\begin{align}
&\nrm*{\nabla_1 - \nabla F(x)}^2 \nonumber\\
&= \left\|-\psi'(\zeta)e_1 + \sum_{i=1}^{j}\psi'(x_{i+1} - x_i)(e_{i+1} - e_i) + \psi'(-x_{j+1})( - e_{j+1})\right.\nonumber\\
&\qquad\qquad\qquad\left.+ \psi'(\zeta)e_1 - \psi'(x_N)e_N - \sum_{i=1}^{N-1}\psi'(x_{i+1} - x_i)(e_{i+1} - e_i)\right\|^2 \\
&= \nrm*{-e_{j+1}\psi'(-x_{j+1}) - \psi'(x_N)e_N - \sum_{i=j+1}^{N-1}\psi'(x_{i+1} - x_i)(e_{i+1} - e_i)}^2 \\
&\leq 2\nrm*{-e_{j+1}\psi'(-x_{j+1}) - \psi'(x_{j+2} - x_{j+1})(e_{j+2} - e_{j+1})}^2 + 2\nrm*{\psi'(x_N)e_N + \sum_{i=j+2}^{N-1}\psi'(x_{i+1} - x_i)(e_{i+1} - e_i)}^2 \\
&\leq 10\ell_2^2\alpha^2 + 32(N-j-2)\ell_2^2\alpha^2 \\
&\leq 32N\ell_2^2\alpha^2
\end{align}
We conclude that
\begin{equation}
\E_z \nrm*{g(x;z) - \nabla F(x)}^2
\leq \frac{3(1-p)}{p}\prn*{2\ell_1^2 + 32N\ell_2^2\alpha^2} + \prn*{\frac{2}{p} + \frac{1}{\delta}}32N\ell_2^2\alpha^2
\leq \frac{6(1-p)\ell_1^2}{p} + \prn*{\frac{160}{p} + \frac{32}{\delta}}N\ell_2^2\alpha^2
\end{equation}

\textbf{7)} Comparing \eqref{eq:oracle-homogeneous-convex-lower-bound} to \pref{def:robust-zero-chain-oracle}, it is clear that $g$ is an $(\alpha,p,\delta)$-robust-zero-chain with $\mc{Z}_0 = \crl{0}$ and $\mc{Z}_1 = \crl{1}$.
\end{proof}

\homogeneousconvexlowerbound*
\begin{proof}
To prove the theorem, we instantiate $F$ as defined in \eqref{eq:def-F-generic-psi-homogeneous-lower-bound} using 
\begin{equation}
\psi(x) = \begin{cases}
\frac{\ell_2}{2}x^2 & \abs{x} \leq \frac{\ell_1}{\ell_2} \\
\ell_1\abs{x} - \frac{\ell_1^2}{2\ell_2} & \abs{x} > \frac{\ell_1}{\ell_2}
\end{cases}
\end{equation}
It is easy to confirm that this $\psi$ satisfies all of the conditions of \pref{lem:homogeneous-convex-generic-psi-construction}. We set 
\begin{align}
\zeta &= \frac{B}{N^{3/2}} \\
\ell_2 &= \frac{H}{4} \\
\ell_1 &= \ell_2 \zeta = \frac{HB}{4N^{3/2}} \\
p &\geq \frac{12\ell_1^2}{12\ell_1^2 + \sigma^2} = \frac{3H^2B^2}{3H^2B^2 + 4N^3\sigma^2} \\
\delta &= \frac{16}{MKR} \\
\alpha^2 &= \min\crl*{\frac{\sigma^2}{2N\ell_2^2\prn*{\frac{160}{p} + \frac{32}{\delta}}},\ \frac{B^2}{64N}}
\end{align}
By \pref{lem:homogeneous-convex-generic-psi-construction}, this ensures that $F$ is convex, $H$-smooth, and $\nrm{x^*}\leq B$. Furthermore, the stochastic gradient oracle variance is bounded by
\begin{equation}
\frac{6(1-p)\ell_1^2}{p} + \prn*{\frac{160}{p} + \frac{32}{\delta}}N\ell_2^2\alpha^2 \leq \frac{\sigma^2}{2} + \frac{\sigma^2}{2} = \sigma^2
\end{equation}
We also note that for $y \in [-\ell_1,\ell_1]$,
\begin{equation}
{\psi^*}'(y) = \frac{y}{\ell_2}
\end{equation}
Therefore, by \pref{lem:homogeneous-convex-generic-psi-construction} for $x$ such that $\prog{\alpha}(x) \leq \frac{N}{2}$, if $N > 2$ then
\begin{align}
F(x) - F^*
&\geq \psi'(\zeta)\brk*{N\zeta - \alpha - {\psi^*}'\prn*{\frac{2}{N}\psi'(\zeta)}} + \frac{N}{2}\brk*{\psi\prn*{{\psi^*}'\prn*{\frac{2}{N}\psi'(\zeta)}} - 2\psi(\zeta)} \\
&= \min\crl*{\ell_2\zeta, \ell_1}\brk*{N\zeta - \alpha - \frac{2\min\crl*{\zeta, \frac{\ell_1}{\ell_2}}}{N}} + \frac{N}{2}\brk*{\psi\prn*{\frac{2\min\crl*{\zeta, \frac{\ell_1}{\ell_2}}}{N}} - 2\psi(\zeta)} \\
&= \begin{cases}
\frac{N\ell_2\zeta^2}{2} - \ell_2\alpha\zeta - \frac{\ell_2\zeta^2}{N} & \zeta \leq \frac{\ell_1}{\ell_2} \\
\frac{N\ell_1^2}{2\ell_2} - \ell_1\alpha - \frac{\ell_1^2}{\ell_2 N} & \zeta > \frac{\ell_1}{\ell_2}
\end{cases} \\
&\geq \begin{cases}
\frac{N\ell_2\zeta^2}{4} - \ell_2\alpha\zeta & \zeta \leq \frac{\ell_1}{\ell_2} \\
\frac{N\ell_1^2}{6\ell_2} - \ell_1\alpha & \zeta > \frac{\ell_1}{\ell_2}
\end{cases} \\
&\geq \frac{N\ell_2\min\crl*{\zeta,\ \frac{\ell_1}{\ell_2}}^2}{6} - \ell_1\alpha 
= \frac{HB^2}{24N^2} - \frac{HB\alpha}{4N^{3/2}}
\geq \frac{HB^2}{96N^2}\label{eq:homogeneous-convex-lower-bound-eq1}
\end{align}

Because all of the algorithm's queries to the gradient oracle have norm bounded by $\gamma$, on the way to applying \pref{lem:intermittent-communication-progress} we introduce a uniformly random orthogonal matrix $U \in \R^{D\times N}$ for 
\begin{equation}
D = N + \frac{2\gamma^2}{\alpha^2}\log(32MKRN) \leq 4KR + \prn*{\frac{4\gamma^2KR}{B^2} + \frac{H^2\gamma^2KR\prn*{\frac{160}{p} + 512MKR}}{\sigma^2}}\log(128MK^2R^2)
\end{equation}
Then, since $g$ is an $(\alpha,p,\delta)$-robust-zero-chain, by \pref{lem:intermittent-communication-progress}, any intermittent communication algorithm that interacts with $Ug(U^\top x;z)$ will have progress at most
\begin{equation}
\prog{\alpha}(U^\top \hat{x}) \leq \min\crl*{KR,\ 8KRp + 12R\log M + 12R}
\end{equation}
with probability at least $\frac{5}{8} - 2MKR\delta = \frac{1}{2}$. We therefore take 
\begin{equation}
N = 2\ceil{\min\crl*{KR,\ 8KRp + 12R\log M + 12R}}
\end{equation}
which means that
\begin{equation}
\prog{\alpha}(U^\top \hat{x}) \leq \frac{N}{2}
\end{equation}
Therefore, by \eqref{eq:homogeneous-convex-lower-bound-eq1} we conclude
\begin{align}
F(U^\top \hat{x}) - F^*
&\geq \frac{HB^2}{96(2\min\crl*{KR,\ 8KRp + 12R\log M + 12R}+1)^2} \\
&\geq \frac{HB^2}{\min\crl*{864(KR)^2,\ 12288(KR)^2p^2 + 27648R^2(1+\log M)^2}} \\
&\geq \frac{HB^2}{1728(KR)^2} + \min\crl*{\frac{HB^2}{49152(KR)^2p^2},\ \frac{HB^2}{110592R^2(1+\log M)^2}} \label{eq:homogeneous-convex-lower-bound-eq2}
\end{align}
From here, we recall that $p$ needs to be chosen so that
\begin{equation}
p \geq \frac{3H^2B^2}{3H^2B^2 + 4N^3\sigma^2} 
\end{equation}
The difficulty here is that $N$ is defined in terms of $p$, however, we observe that
\begin{equation}
N \geq 2KRp
\end{equation}
therefore, choosing
\begin{equation}
p \geq \frac{3H^2B^2}{3H^2B^2 + 32\sigma^2K^3R^3p^3}
\end{equation}
satisfies the requirement on $p$. To that end, we set
\begin{equation}
p =\min\crl*{1,\ \prn*{\frac{3H^2B^2}{32\sigma^2K^3R^3}}^{1/4}}
\end{equation}
Therefore, returning to \eqref{eq:homogeneous-convex-lower-bound-eq2} we conclude
\begin{align}
F(U^\top \hat{x}) - F^*
&\geq \frac{HB^2}{1728(KR)^2} + \min\crl*{\frac{\sigma B}{15050\sqrt{KR}},\ \frac{HB^2}{110592R^2(1+\log M)^2}}
\end{align}
In addition, by \pref{lem:statistical-term-lower-bound}, the minimax error is also lower bounded by
\begin{equation}
F(\hat{x}) - F^* \geq c\cdot\min\crl*{\frac{\sigma B}{\sqrt{MKR}},\ HB^2}
\end{equation}
with probability at least $\frac{1}{4}$.
\end{proof}

\subsection{Proof of \pref{thm:homogeneous-convex-upper-bound}} \label{app:intermittent-communication-homogeneous-convex-upper-bound}

\homogeneousconvexupperbound*
\begin{proof}
The Accelerated SGD variant AC-SA \citep{lan2012optimal} run for $T$ iterations with stochastic gradient variance bounded by $\sigma^2$ guarantees \citep[Corollary 1][]{lan2012optimal}
\begin{equation}
\E F(\hat{x}) - F^* \leq c \cdot \frac{HB^2}{T^2} + c \cdot \frac{\sigma B}{\sqrt{T}}
\end{equation}
Therefore, Single-Machine Accelerated SGD, which corresponds to $KR$ steps with stochastic gradient variance bounded by $\sigma^2$, gives
\begin{equation}
\E F(\hat{x}) - F^* \leq c \cdot \frac{HB^2}{(KR)^2} + c \cdot \frac{\sigma B}{\sqrt{KR}}
\end{equation}
Likewise, Minibatch Accelerated SGD, which corresponds to $R$ steps with stochastic gradient variance bounded by $\frac{\sigma^2}{MK}$, gives
\begin{equation}
\E F(\hat{x}) - F^* \leq c \cdot \frac{HB^2}{R^2} + c \cdot \frac{\sigma B}{\sqrt{MKR}}
\end{equation}
Finally, by the smoothness of $F$,
\begin{equation}
F(0) - F^* \leq \frac{H}{2}\nrm{0 - x^*}^2 = \frac{HB^2}{2}
\end{equation}
Therefore, the rate claimed by the theorem can be achieved by using whichever of these methods has the smallest upper bound.
\end{proof}

\subsection{Proof of \pref{thm:homogeneous-strongly-convex-lower-bound}} \label{app:intermittent-communication-homogeneous-strongly-convex-lower-bound}

\homogeneousstronglyconvexlowerbound*
\begin{proof}
We prove this from \pref{thm:homogeneous-convex-lower-bound} using the reduction of \citet{zeyuan2016optimal} (see also \pref{thm:elad-reduction} and the discussion in \pref{subsec:lower-bounds-by-reduction}). In particular, we begin by supposing that there were an algorithm $\mc{A}$ which, for any $H$-smooth and $\lambda$-strongly convex objective, guarantees finding an $\epsilon$-suboptimal point using at most 
\begin{equation}\label{eq:homogeneous-sc-lower-bound-wrong}
R \leq \tm_\lambda(\epsilon,\Delta) := c \cdot\prn*{\sqrt{\frac{H}{K^2\lambda}}\log\frac{c'\Delta}{\epsilon} + \min\crl*{\frac{\sigma^2}{\lambda K\epsilon},\, \sqrt{\frac{H}{\lambda\log^2 M}}\log\frac{c'\Delta}{\epsilon}}}
\end{equation}
rounds of communication, when given an initial point $x_0$ with $\E F(x_0) - F^* \leq \Delta$. Here, $c$ and $c'$ are some universal constants. Then, \pref{thm:elad-reduction} implies the existence of an algorithm, $\eladreduction(\mc{A},e)$, that guarantees finding a point with expected suboptimality at most $\epsilon$ for any $H$-smooth and convex objective when given $x_0$ such that $\E\nrm{x_0 - x^*}^2 \leq B^2$ using the following number of rounds of communication:
\begin{align}
R 
&\leq \sum_{t=1}^{\ceil{\log \frac{4HB^2}{\epsilon}}} \tm_{H e^{1-t}}(HB^2e^{-t},HB^2e^{2-t}) \\
&= c\sum_{t=1}^{\ceil{\log \frac{4HB^2}{\epsilon}}} \brk*{\sqrt{\frac{1}{K^2 e^{1-t}}}\log(c'e^2) + \min\crl*{\frac{\sigma^2e^{2t}}{H^2B^2Ke},\, \sqrt{\frac{1}{e^{1-t} \log^2 M }}\log(c'e^2)}} \\
&\leq c\log(c'e^2)\brk*{\frac{1}{K}\sum_{t=1}^{\ceil{\log \frac{4HB^2}{\epsilon}}}e^{\frac{t}{2}} + \min\crl*{\frac{\sigma^2}{H^2B^2K} \sum_{t=1}^{\ceil{\log \frac{4HB^2}{\epsilon}}}e^{2t},\, \frac{1}{\log M}\sum_{t=1}^{\ceil{\log \frac{4HB^2}{\epsilon}}}e^{\frac{t}{2}} }} \\
&\leq \frac{16e^2\sqrt{e}c\log(c'e^2)}{\sqrt{e}-1}\brk*{\sqrt{\frac{HB^2}{K^2\epsilon}} + \min\crl*{\frac{\sigma^2 B^2}{K\epsilon^2},\, \sqrt{\frac{HB^2}{\epsilon \log^2 M}}}}
\end{align}
However, solving this expression for $\epsilon$, this implies that for some positive constant $c''(c,c')$, the algorithm $\eladreduction(\mc{A},e)$ will converge at a rate 
\begin{equation}
\E F(\hat{x}) - F^* \leq c''(c,c')\cdot\prn*{\frac{HB^2}{K^2R^2} + \min\crl*{\frac{\sigma B}{\sqrt{KR}},\, \frac{HB^2}{R^2\log^2 M}}}
\end{equation}
However, when the dimension is at least
\begin{equation}
D \geq c\cdot\prn*{KR + \prn*{\frac{\gamma^2KR}{B^2} + \frac{H^2\gamma^2KR\prn*{\frac{\sqrt{\sigma}KR}{\sqrt{HB}} + MKR}}{\sigma^2}}\log(MK^2R^2)}
\end{equation}
this contradicts the lower bound \pref{thm:homogeneous-convex-lower-bound} when $c''(c,c')$ is too small. We conclude that for some universal constants $c$ and $c'$, the guarantee \eqref{eq:homogeneous-sc-lower-bound-wrong} cannot hold. Solving for $\epsilon$, we conclude that any intermittent communication algorithm must have
\begin{equation}
\E F(\hat{x}) - F^* \geq c'\cdot\prn*{\Delta\exp\prn*{-\frac{c\sqrt{\lambda}KR}{\sqrt{H}}} + \min\crl*{\frac{\sigma^2}{\lambda KR},\, \Delta\exp\prn*{-\frac{c\sqrt{\lambda}R\log M}{\sqrt{H}}}}}
\end{equation}
Finally, by \pref{lem:statistical-term-lower-bound}, we also have that 
\begin{equation}
\E F(\hat{x}) - F^*  \geq c'\cdot\min\crl*{\frac{\sigma^2}{\lambda MKR},\,\Delta}
\end{equation}
in the worst case, which completes the proof of the lower bound.
\end{proof}

\subsection{Proof of \pref{thm:homogeneous-strongly-convex-upper-bound}} \label{app:intermittent-communication-homogeneous-strongly-convex-upper-bound}

\homogeneousstronglyconvexupperbound*
\begin{proof}
The Multi-stage AC-SA algorithm of \citet{ghadimi2013optimal}, run for $T$ iterations with stochastic gradient variance bounded by $\sigma^2$ guarantees \citep[Proposition 7][]{ghadimi2013optimal}
\begin{equation}
\E F(\hat{x}) - F^* \leq c \cdot \prn*{\Delta \exp\prn*{-\frac{c'\sqrt{\lambda}T}{\sqrt{H}}} + \frac{\sigma^2}{\lambda T}}
\end{equation}
Therefore, Single-Machine Accelerated SGD, which corresponds to $KR$ steps with stochastic gradient variance bounded by $\sigma^2$, gives
\begin{equation}
\E F(\hat{x}) - F^* \leq c \cdot \prn*{\Delta \exp\prn*{-\frac{c'\sqrt{\lambda}KR}{\sqrt{H}}} + \frac{\sigma^2}{\lambda KR}}
\end{equation}
Likewise, Minibatch Accelerated SGD, which corresponds to $R$ steps with stochastic gradient variance bounded by $\frac{\sigma^2}{MK}$, gives
\begin{equation}
\E F(\hat{x}) - F^* \leq c \cdot \prn*{\Delta \exp\prn*{-\frac{c'\sqrt{\lambda}R}{\sqrt{H}}} + \frac{\sigma^2}{\lambda MKR}}
\end{equation}
Finally, by assumption $F(0) - F^* \leq \Delta$. Therefore, the claimed rate can be achieved by using whichever of these methods has the smallest suboptimality upper bound.
\end{proof}

\subsection{Proof of \pref{thm:homogeneous-non-convex-lower-bound}} \label{app:intermittent-communication-homogeneous-non-convex-lower-bound}

Our strategy is to construct (1) a non-convex function $F$ such that $\prog{\alpha}(x) \leq T \implies \nrm{\nabla F(x)} \geq \epsilon$ for an appropriate $T$ and $\epsilon$ and (2) an $(\alpha,p,0)$-robust-zero-chain stochastic gradient oracle for $F$. Any intermittent communication algorithm using this oracle will therefore have small progress by \pref{lem:intermittent-communication-progress}, which implies that the gradient is large by the property (1).

Our construction originates with \citet{carmon2017lower1}, for each $T$ we define
\begin{equation}
F_T(x) = -\Psi(1)\Phi(x_1) + \sum_{i=2}^T \bigg[ \Psi(-x_{i-1})\Phi(-x_i) - \Psi(x_{i-1})\Phi(x_i) \bigg]
\end{equation}
where
\begin{align}
\Psi(x) &= \begin{cases} 0 & x \leq \frac{1}{2} \\ \exp\prn*{1 - \frac{1}{(2x-1)^2}} & x > \frac{1}{2} \end{cases} \\
\Phi(x) &= \sqrt{e}\int_{-\infty}^x e^{-\frac{1}{2}t^2}dt
\end{align}
The following lemma summarizes the relevant properties of $F_T$:
\begin{restatable}{lemma}{nonconvexconstruction}\label{lem:non-convex-homogeneous-construction-properties}
The function $F_T$ satisfies
\begin{enumerate}
\item $F_T(0) - \min_x F_T(x) \leq 12 T$
\item $F_T$ is $152$-smooth
\item $\sup_x \nrm{\nabla F(x)}_\infty \leq 23$
\item For all $x$, $\prog{0}(\nabla F_T(x)) \leq \prog{\frac{1}{2}}(x) + 1$
\item For all $x$, $\prog{\frac{1}{2}}(x) < T \implies \nrm{\nabla F_T(x)} > \abs*{\brk*{\nabla F_T(x)}_{\prog{\frac{1}{2}}(x)+1}} > 1$.
\end{enumerate}
\end{restatable}
\begin{proof}
Parts 1, 2, and 3 follow from the proof of \citep[Lemma 3][]{carmon2017lower1}. We derive the smoothness constant $152$ for part 2 by observing that the (symmetric) Hessian of $F_T$ is tri-diagonal and therefore by the Gershgorin circle theorem, for any $x$
\begin{align}
\nrm*{\nabla^2 F_T(x)}_{\textrm{op}}
&= \max_i \abs*{\lambda_i\prn*{\nabla^2 F_T(x)}} \\
&\leq \max_i \brk*{\abs*{\brk*{\nabla^2 F_T(x)}_{i,i-1}} + \abs*{\brk*{\nabla^2 F_T(x)}_{i,i}} + \abs*{\brk*{\nabla^2 F_T(x)}_{i,i+1}}} \\
&\leq \max_z\abs{\Psi''(z)}\max_z\abs{\Phi(z)} + \max_z\abs{\Psi(z)}\max_z\abs{\Phi''(z)} + 2\max_z\abs{\Psi'(z)}\max_z\abs{\Phi'(z)} \\
&\leq \frac{65}{2}\cdot \sqrt{2\pi e} + e\cdot 1 + 2 \cdot \sqrt{\frac{54}{e}} \cdot \sqrt{e} \leq 152
\end{align}
Finally, part 4 of the lemma follows from \citep[Observation 3][]{carmon2017lower1} and part 5 from \citep[Lemma 2][]{carmon2017lower1}.
\end{proof}

In addition to the function $F_T$, we also define a stochastic gradient oracle $g_T(x;z)$
\begin{align}
z &= \begin{cases} 0 & \textrm{with probability } 1-p \\ 1 & \textrm{with probability } p \end{cases} \\
g_T(x;0) &= \sum_{i=1}^{\prog{\frac{1}{2}(x)}}e_ie_i^\top \nabla F_T\prn*{\brk*{x_1,x_2,\dots,x_{\prog{\frac{1}{2}}(x)},0,\dots,0}} \\
g_T(x;1) &= \frac{1}{p}\nabla F_T(x) - \frac{1-p}{p}g_T(x;0)
\end{align}
The following lemma confirms that $g_T$ has the desired properties:
\begin{lemma}\label{lem:homogeneous-non-convex-oracle-properties}
For any $p$, the stochastic gradient oracle $g_T$ satisfies
\begin{enumerate}
\item $\E_z g_T(x;z) = \nabla F_T(x)$
\item $\E_z \nrm*{g_T(x;z) - \nabla F_T(x)}^2 \leq \frac{1058(1-p)}{p}$
\item $g_T$ is a $(\frac{1}{2},p,0)$-robust-zero-chain oracle.
\end{enumerate}
\end{lemma}
\begin{proof}
A simple calculation shows that
\begin{equation}
\E_z g_T(x;z) = (1-p)g_T(x;0) + p\prn*{\frac{1}{p}\nabla F_T(x) - \frac{1-p}{p}g_T(x;0)} = \nabla F_T(x)
\end{equation}
Furthermore, 
\begin{align}
\E_z &\nrm*{g_T(x;z) - \nabla F_T(x)}^2\nonumber\\
&= (1-p)\nrm*{g_T(x;0) - \nabla F_T(x)}^2 + p\nrm*{\frac{1}{p}\nabla F_T(x) - \frac{1-p}{p}g_T(x;0) - \nabla F_T(x)}^2 \\
&= \frac{1-p}{p}\nrm*{g_T(x;0) - \nabla F_T(x)}^2
\end{align}
Each coordinate of the gradient is given by
\begin{equation}
\brk*{\nabla F_T(x)}_i = - \Psi(-x_{i-1})\Phi'(-x_i) - \Psi(x_{i-1})\Phi'(x_i) - \Psi'(-x_i)\Phi(-x_{i+1}) - \Psi'(x_i)\Phi(x_{i+1})
\end{equation}
so the $i\mathth$ coordinate of $\nabla F_T$ only depends on the $i-1$, $i$ and $i+1$ coordinates of $x$. Furthermore, most of the coordinates of $g_T(x;0)$ are equal to the corresponding coordinates of $\nabla F_T(x)$. Specifically, let 
\begin{equation}
\tilde{x} = \brk*{x_1,x_2,\dots,x_{\prog{\frac{1}{2}}(x)},0,\dots,0}
\end{equation}
Then for $i < \prog{\frac{1}{2}}(x)$, $\tilde{x}_{i-1} = x_{i-1}$, $\tilde{x}_i = x_i$, and $\tilde{x}_{i+1} = x_{i+1}$, so $\brk*{g_T(x;0)}_i = \brk*{\nabla F_T(\tilde{x})}_i = \brk*{\nabla F_T(x)}_i$. Similarly, for $i > \prog{\frac{1}{2}}(x) + 1 \geq \prog{\frac{1}{2}}(\tilde{x}) + 1$, by part 4 of \pref{lem:non-convex-homogeneous-construction-properties}, $\brk*{g_T(x;0)}_i = \brk*{\nabla F_T(\tilde{x})}_i = \brk*{\nabla F_T(x)}_i = 0$. 

Therefore, $g(x;0)$ and $\nabla F_T(x)$ differ on at most two coordinates, the $\prog{\frac{1}{2}}(x)\mathth$ and $(\prog{\frac{1}{2}}(x) + 1)\mathth$, so part 3 of \pref{lem:non-convex-homogeneous-construction-properties} implies
\begin{align}
\E_z \nrm*{g_T(x;z) - \nabla F_T(x)}^2
= \frac{1-p}{p}\nrm*{g_T(x;0) - \nabla F_T(x)}^2
\leq \frac{2\cdot23^2(1-p)}{p}
\end{align}

Finally, for $\mc{Z}_0 = \crl{0}$, $\mc{Z}_1 = \crl{1}$, $g_T$ is a $(\frac{1}{2},p,0)$-robust zero chain because $\P(z \in \mc{Z}_0\cup\mc{Z}_1) = 1$ and $\P(z \in \mc{Z}_0\,|\,z\in\mc{Z}_0\cup\mc{Z}_1) = 1-p$. Furthermore, when $z = 0 \in \mc{Z}_0$ it is obvious that
\begin{equation}\label{eq:homogeneous-non-convex-g-is-robust-zc}
\prog{\frac{1}{2}}(x) \leq i \implies g(x;0) = g([x_1,\dots,x_i,0,\dots,0],0)
\end{equation}
Finally, when $z = 1 \in \mc{Z}_1$, $g(x;1) = \frac{1}{p}\nabla F_T(x) - \frac{1-p}{p}g_T(x;0)$. The second term depends only on $x_1,\dots,x_{\prog{\frac{1}{2}}(x)}$ by \eqref{eq:homogeneous-non-convex-g-is-robust-zc}. Furthermore,
\begin{equation}
\nabla F_T(x) = \nabla F_T([x_1,\dots,x_{\prog{\frac{1}{2}}(x) + 1},0,\dots,0])
\end{equation}
because the only the $\prog{\frac{1}{2}}(x) + 1,\dots,T$ coordinates of $\nabla F_T(x)$ depend on the $\prog{\frac{1}{2}}(x) + 2,\dots,T$ coordinates of $x$. By part 4 of \pref{lem:non-convex-homogeneous-construction-properties}, the $\prog{\frac{1}{2}}(x) + 2,\dots,T$ coordinates of $\nabla F_T(x)$ and $\nabla F_T([x_1,\dots,x_{\prog{\frac{1}{2}}(x) + 1},0,\dots,0])$ are all zero. Finally, for $i = \prog{\frac{1}{2}}(x)$, since $\Psi(z) = \Psi'(z) = 0$ for $z \leq \frac{1}{2}$,
\begin{align}
&\brk*{\nabla F_T([x_1,\dots,x_{i+1},0,\dots,0])}_{i+1} \nonumber\\
&= -\Psi(-x_i)\Phi'(-x_{i+1}) - \Psi(x_i)\Phi'(x_{i+1}) - \Psi'(-x_{i+1})\Phi(0) - \Psi'(x_{i+1})\Phi(0) \\
&= -\Psi(-x_i)\Phi'(-x_{i+1}) - \Psi(x_i)\Phi'(x_{i+1}) \\
&= -\Psi(-x_i)\Phi'(-x_{i+1}) - \Psi(x_i)\Phi'(x_{i+1}) - \Psi'(-x_{i+1})\Phi(x_{i+2}) - \Psi'(x_{i+1})\Phi(x_{i+2}) \\
&= \brk*{\nabla F_T(x)}_{i+1}
\end{align}
This establishes that $g_T$ is a $(\frac{1}{2},p,0)$-robust-zero-chain.
\end{proof}

We will proceed to combine \pref{lem:non-convex-homogeneous-construction-properties} and \pref{lem:homogeneous-non-convex-oracle-properties} with \pref{lem:intermittent-communication-progress} allows us to prove the lower bound. However, one of the conditions of \pref{lem:intermittent-communication-progress} is that the norm of the oracle queries is bounded. To enforce the we introduce an additional modification to the objective along with the random rotation. Specifically, we introduce the soft projection
\begin{equation}
\rho(x) = \frac{x}{\sqrt{1 + \frac{\nrm{x}^2}{\beta^2}}}
\end{equation}
where $\beta = \frac{240\sqrt{T}}{\zeta}$ and we define
\begin{align}
\hat{F}_{T,U}(x) &= \frac{\gamma}{\zeta^2} F_T(\zeta U^\top \rho(x)) + \frac{\gamma}{10\zeta^2}\nrm{\zeta x}^2 \\
\hat{g}_{T,U}(x;z) &= \frac{\gamma}{\zeta}\nabla\rho(x)U g_T(\zeta U^\top \rho(x);z) + \frac{\gamma}{5}x
\end{align}
We now verify that $\hat{F}_{T,U}$ and $\hat{g}_{T,U}(x;z)$ satisfy essentially the same properties as $F_T$ and $g_T$:
\begin{lemma}\label{lem:soft-projection-properties}
For any $T \geq 1$, $\gamma,\zeta \geq 0$, and $U$ with $U^\top U = I_{d\times{}d}$,
\begin{enumerate}
\item $\hat{F}_{T,U}(0) - \min_x \hat{F}_{T,U}(x) \leq \frac{12\gamma T}{\zeta^2}$
\item $F_{T,U}$ is $154\gamma$-smooth
\item $\E_z \hat{g}_{T,U}(x;z) = \nabla \hat{F}_{T,U}(x)$
\item $\E_z \nrm*{\hat{g}_{T,U}(x;z) - \nabla \hat{F}_{T,U}(x)}^2 \leq \frac{1058\gamma^2(1-p)}{\zeta^2 p}$
\item For any $x$, $\prog{\frac{1}{2}}(\zeta U^\top \rho(x)) < T \implies \nrm{\nabla \hat{F}_{T,U}(x)} \geq \frac{\gamma}{2\zeta}$
\end{enumerate}
\end{lemma}
\begin{proof}
For property 1, we note that $\hat{F}_{T,U}(0) = \frac{\gamma}{\zeta^2}F_T(0) \leq 0$ and by part 1 of \pref{lem:non-convex-homogeneous-construction-properties}
\begin{equation}
\min_x \hat{F}_{T,U}(x) = \min_x \frac{\gamma}{\zeta^2}F_T(\zeta U^\top \rho(x)) + \frac{\gamma}{10}\nrm{x}^2 \geq \min_{x:\nrm{x}\leq\zeta\beta} \frac{\gamma}{\zeta^2} F_T(x) \geq \frac{\gamma}{\zeta^2}\min_x F_T(x) \geq -\frac{12\gamma T}{\zeta^2}
\end{equation}
For property 2, we note that for any $x,y$
\begin{align}
&\nrm*{\nabla \hat{F}_{T,U}(x) - \nabla \hat{F}_{T,U}(y)}\nonumber\\
&= \nrm*{\frac{\gamma}{\zeta}\nabla \rho(x) U\nabla F_T(\zeta U^\top\rho(x)) - \frac{\gamma}{\zeta}\nabla \rho(y) U\nabla F_T(\zeta U^\top\rho(y)) + \frac{\gamma}{5}(x-y)} \\
&\leq \frac{\gamma}{\zeta}\nrm*{\nabla \rho(x) U\prn*{\nabla F_T(\zeta U^\top\rho(x)) - \nabla F_T(\zeta U^\top\rho(y))} + \prn*{\nabla \rho(x) - \nabla \rho(y)}U\nabla F_T(\zeta U^\top\rho(y))} + \frac{\gamma}{5}\nrm{x-y} \\
&\leq \frac{\gamma}{\zeta}\nrm*{\nabla \rho(x)}_{\textrm{op}}\nrm*{\nabla F_T(\zeta U^\top\rho(x)) - \nabla F_T(\zeta U^\top\rho(y))} + \frac{\gamma}{\zeta}\nrm*{\nabla \rho(x) - \nabla \rho(y)}_{\textrm{op}}\nrm*{U\nabla F_T(\zeta U^\top\rho(y))} + \frac{\gamma}{5}\nrm{x-y} \label{eq:thm-homogeneous-non-convex-thm-1}
\end{align}
First, 
\begin{equation}
\nrm*{\nabla \rho(x)}_{\textrm{op}} = \nrm*{\frac{1}{\sqrt{1 + \frac{\nrm{x}^2}{\beta^2}}}I - \frac{xx^\top}{\beta^2\prn*{1 + \frac{\nrm{x}^2}{\beta^2}}^{3/2}}}_{\textrm{op}} \leq 1
\end{equation}
Therefore, since $F_T$ is $152$-smooth by \pref{lem:non-convex-homogeneous-construction-properties},
\begin{equation}
\nrm*{\nabla F_T(\zeta U^\top\rho(x)) - \nabla F_T(\zeta U^\top\rho(y))} \leq 152\zeta\nrm{U^\top\rho(x) - U^\top\rho(y)} \leq 152\zeta \nrm{x-y}
\end{equation}
Next, we define $h(t) = \frac{1}{\sqrt{1+t^2}}$, which is $1$-Lipschitz, and bound:
\begin{align}
&\nrm*{\nabla \rho(x) - \nabla \rho(y)}_{\textrm{op}}\nonumber\\
&= \nrm*{h\prn*{\frac{\nrm{x}}{\beta}}I - h\prn*{\frac{\nrm{x}}{\beta}}\frac{\rho(x)\rho(x)^\top}{\beta^2} - h\prn*{\frac{\nrm{y}}{\beta}}I + h\prn*{\frac{\nrm{y}}{\beta}}\frac{\rho(y)\rho(y)^\top}{\beta^2}}_{\textrm{op}} \\
&\leq h\prn*{\frac{\nrm{y}}{\beta}}\nrm*{\frac{\rho(x)\rho(x)^\top}{\beta^2} - \frac{\rho(y)\rho(y)^\top}{\beta^2}}_{\textrm{op}} + \abs*{h\prn*{\frac{\nrm{x}}{\beta}} - h\prn*{\frac{\nrm{y}}{\beta}}}\nrm*{I - \frac{\rho(x)\rho(x)^\top}{\beta^2}}_{\textrm{op}} \\
&\leq \nrm*{\frac{\rho(x)\rho(x)^\top}{\beta^2} - \frac{\rho(y)\rho(y)^\top}{\beta^2}}_{\textrm{op}} + \abs*{\frac{\nrm{x}}{\beta} - \frac{\nrm{y}}{\beta}} \\
&\leq \frac{1}{\beta}\nrm{x-y} + \sup_{v:\nrm{v}\leq 1}\nrm*{\frac{\rho(x)\rho(x)^\top}{\beta^2}v - \frac{\rho(y)\rho(y)^\top}{\beta^2}v} \\
&= \frac{1}{\beta}\nrm{x-y} + \sup_{v:\nrm{v}\leq 1}\nrm*{\prn*{\frac{\rho(x)}{\beta} - \frac{\rho(y)}{\beta}}\inner{\frac{\rho(x)}{\beta}}{v} + \frac{\rho(y)}{\beta}\inner{\frac{\rho(x)}{\beta} - \frac{\rho(y)}{\beta}}{v}} \\
&\leq \frac{1}{\beta}\nrm{x-y} + \sup_{v:\nrm{v}\leq 1}\nrm*{\frac{\rho(x)}{\beta} - \frac{\rho(y)}{\beta}}\nrm*{\frac{\rho(x)}{\beta}}\nrm{v} + \nrm*{\frac{\rho(x)}{\beta} - \frac{\rho(y)}{\beta}}\nrm*{\frac{\rho(y)}{\beta}}\nrm*{v} \\
&\leq \frac{3}{\beta}\nrm{x-y}
\end{align}
Finally, we use \pref{lem:non-convex-homogeneous-construction-properties} to bound
\begin{equation}
\nrm*{U\nabla F_T(\zeta U^\top\rho(y))} \leq \sqrt{T}\nrm{\nabla F_T(\zeta U^\top\rho(y))}_\infty \leq 23\sqrt{T}
\end{equation}
Therefore, returning to \eqref{eq:thm-homogeneous-non-convex-thm-1}, we conclude that
\begin{align}
\nrm*{\nabla \hat{F}_{T,U}(x) - \nabla \hat{F}_{T,U}(y)}
&\leq 152\gamma\nrm{x-y} + \frac{69\gamma\sqrt{T}}{\zeta\beta}\nrm{x-y} + \frac{\gamma}{5}\nrm{x-y} \\
&\leq 154\gamma\nrm{x-y}
\end{align}
where we used that $\beta = \frac{240\sqrt{T}}{\zeta} > \frac{69\sqrt{T}}{\zeta}$. Therefore, we conclude that $\nabla \hat{F}_{T,U}(x)$ is $154\gamma$-smooth.

Property 3 follows immediately from \pref{lem:homogeneous-non-convex-oracle-properties}, and to show property 4 we bound
\begin{align}
\sup_x&\E_z \nrm*{\hat{g}_{T,U}(x;z) - \nabla \hat{F}_{T,U}(x)}^2\nonumber\\
&= \sup_x \E_z \nrm*{\frac{\gamma}{\zeta}\nabla \rho(x)U g_T(\zeta U^\top\rho(x);z) + \frac{\gamma}{5}x - \frac{\gamma}{\zeta}\nabla \rho(x)U\nabla F_T(\zeta U^\top\rho(x)) - \frac{\gamma}{5}x}^2 \\
&\leq \frac{\gamma^2}{\zeta^2}\sup_x \E_z \nrm*{g_T(\zeta U^\top\rho(x);z) - \nabla F_T(\zeta U^\top\rho(x))}^2 \\
&\leq \frac{\gamma^2}{\zeta^2}\sup_y \E_z \nrm*{g_T(y;z) - \nabla F_T(y)}^2 \\
&\leq \frac{1058\gamma^2(1-p)}{\zeta^2 p}
\end{align}
where the last line follows from \pref{lem:homogeneous-non-convex-oracle-properties}.

Finally, for property 5, we note that
\begin{equation}
\nrm*{\nabla \hat{F}_{T,U}(x)} = \nrm*{\frac{\gamma}{\zeta}\nabla \rho(x) U \nabla F_T(\zeta U^\top \rho(x)) + \frac{\gamma}{5}x}
\end{equation}
We now consider two cases. First, if $\nrm{x} \leq \frac{\beta}{2}$ then 
\begin{equation}
\nabla \rho(x) = \frac{1}{\sqrt{1 + \frac{\nrm{x}^2}{\beta^2}}}I - \frac{xx^\top}{\beta^2\prn*{1 + \frac{\nrm{x}^2}{\beta^2}}^{3/2}} \succeq \prn*{\inf_{z\in[0,1/2]}\frac{1}{\sqrt{1 + z^2}} - \frac{z^2}{(1 + z^2)^{3/2}}}I 
\succeq \frac{7}{10}I
\end{equation}
Let $j = \prog{\frac{1}{2}}(\zeta U^\top \rho(x)) + 1 \leq T$ so
\begin{align}
\nrm*{\nabla \hat{F}_{T,U}(x)}
&\geq \abs*{\inner{U_j}{\nabla \hat{F}_{T,U}(x)}} \\
&\geq \frac{7\gamma}{10\zeta}\abs*{\brk*{\nabla F_T(\zeta U^\top \rho(x))}_j} - \frac{\gamma}{5}\abs*{\inner{U_j}{x}} \\
&\geq \frac{7\gamma}{10\zeta} - \frac{\gamma}{5}\abs*{\inner{U_j}{x}}
\end{align}
where for the final inequality we used \pref{lem:non-convex-homogeneous-construction-properties}. In addition, $\nrm{x} \leq \frac{\beta}{2}$ and $\prog{\frac{1}{2}}(\zeta U^\top \rho(x)) < j$ implies
\begin{equation}
\abs*{\inner{U_j}{x}} \leq \frac{1}{2\zeta}\sqrt{1 + \frac{\nrm{x}^2}{\beta^2}} \leq \frac{6}{10\zeta}
\end{equation}
so if $\nrm{x} \leq \frac{\beta}{2}$ then
\begin{equation}
\nrm*{\nabla \hat{F}_{T,U}(x)} \geq \frac{\gamma}{2\zeta}
\end{equation}
On the other hand, if $\nrm{x} > \frac{\beta}{2}$ then by part 3 of \pref{lem:non-convex-homogeneous-construction-properties} and the choice $\beta = \frac{240\sqrt{T}}{\zeta}$,
\begin{align}
\nrm*{\nabla \hat{F}_{T,U}(x)} 
&= \nrm*{\frac{\gamma}{\zeta}\nabla \rho(x) U \nabla F_T(\zeta U^\top \rho(x)) + \frac{\gamma}{5}x} \\
&\geq \frac{\gamma}{5}\nrm{x} - \frac{\gamma}{\zeta}\nrm*{\nabla \rho(x) U \nabla F_T(\zeta U^\top \rho(x))} \\
&\geq \frac{\gamma\beta}{10} - \frac{\gamma}{\zeta}\nrm*{\nabla F_T(\zeta U^\top \rho(x))} \\
&\geq \frac{\gamma\beta}{10} - \frac{23\gamma\sqrt{T}}{\zeta} \\
&= \frac{24\gamma\sqrt{T}}{\zeta} - \frac{23\gamma\sqrt{T}}{\zeta} \\
&= \frac{\gamma\sqrt{T}}{\zeta} > \frac{\gamma}{2\zeta}
\end{align}
This completes the proof.
\end{proof}

\homogeneousnonconvexlowerbound*
\begin{proof}
We prove the lower bound using $\hat{F}_{T,U}$ and $\hat{g}_{T,U}$ for a uniformly random orthogonal $U \in \R^{D\times T}$ for 
\begin{equation}
D = T + 8\zeta^2\beta^2\log(32MKRT)
\end{equation}
Since
\begin{equation}
\hat{g}_{T,U}(x;z) = \frac{\gamma}{\zeta}\nabla\rho(x)U g_T(\zeta U^\top \rho(x);z) + \frac{\gamma}{5\zeta}x
\end{equation}
an intermittent communication algorithm that interacts with $\hat{g}_{T,U}(x;z)$ is equivalent to one that interacts with $U g_T(U^\top x;z)$ using queries of norm less than $\zeta\beta$ since $\rho(x)$, $\nabla \rho(x)$, and $x$ are invertible and ``known'' to the algorithm. Therefore, since $g_T$ is a $(\frac{1}{2},p,0)$-robust-zero-chain, \pref{lem:intermittent-communication-progress} ensures that with probability at least $\frac{5}{8}$, the output of the intermittent communication algorithm, $\hat{x}$, will have progress at most
\begin{equation}
\prog{\frac{1}{2}}(\zeta U^\top \rho(\hat{x})) \leq \min\crl*{KR,\ 8KRp + 12R\log M + 12R}
\end{equation}
We therefore take
\begin{equation}
T = \min\crl*{KR,\ 8KRp + 12R\log M + 12R} + 1
\end{equation} 
Therefore, by \pref{lem:soft-projection-properties}, with probability at least $\frac{5}{8}$,
\begin{equation}
\nrm*{\nabla \hat{F}_{T,U}(\hat{x})} \geq \frac{\gamma}{2\zeta}
\end{equation}
In light of \pref{lem:soft-projection-properties}, it is easy to confirm that if we take 
\begin{align}
\gamma &= \frac{H}{156} \\
\zeta &= \sqrt{\frac{12\gamma T}{\Delta}} = \sqrt{\frac{HT}{13\Delta}} \\
p &\geq \frac{\frac{1058\gamma^2}{\zeta^2}}{\sigma^2 + \frac{1058\gamma^2}{\zeta^2}} = \frac{529H\Delta}{936T\sigma^2 + 529H\Delta}
\end{align}
then $\hat{F}_{T,U}(0) - \min_x \hat{F}_{T,U}(x) \leq \Delta$, $\hat{F}_{T,U}$ is $H$-smooth, $\hat{g}_{T,U}$ is an unbiased estimate of $\nabla \hat{F}_{T,U}$, and $\E_z \nrm*{\hat{g}_{T,U}(x;z) - \nabla \hat{F}_{T,U}(x)}^2 \leq \sigma^2$, and the lower bound is
\begin{equation}
\nrm*{\nabla \hat{F}_{T,U}(\hat{x})} \geq \frac{\gamma}{2\zeta} = \frac{\sqrt{13H\Delta}}{312\sqrt{T}}
\end{equation}
with probability at least $\frac{5}{8}$. The difficulty is that $T$ is defined implicitly in terms of $p$ since we require that
\begin{align}
p &\geq \frac{529H\Delta}{936T\sigma^2 + 529H\Delta} \label{eq:homogeneous-non-convex-p-condition}\\
&= \max\crl*{\frac{529H\Delta}{936(KR+1)\sigma^2 + 529H\Delta},\ \frac{529H\Delta}{936(8KRp + 12R\log M + 12R + 1)\sigma^2 + 529H\Delta}}
\end{align}
We define $p_1$ to be the first term of the maximum, and $p_2$ to be the positive solution to the equation
\begin{equation}
p_2 = \frac{529H\Delta}{936(8KRp_2 + 12R\log M + 12R + 1)\sigma^2 + 529H\Delta}
\end{equation}
We note that since the right hand side is decreasing in $p_2$, 
\begin{equation}
\frac{529H\Delta}{936(8KRp_1 + 12R\log M + 12R + 1)\sigma^2 + 529H\Delta} < p_1
\iff p_2 < p_1
\end{equation}
and
\begin{equation}
\frac{529H\Delta}{936(8KRp_1 + 12R\log M + 12R + 1)\sigma^2 + 529H\Delta} \geq p_1
\iff p_2 \geq p_1
\end{equation}
We therefore take $p = \max\crl{p_1,p_2}$ and consider two cases:
\begin{align}
p_1 &> \frac{529H\Delta}{936(8KRp_1 + 12R\log M + 12R + 1)\sigma^2 + 529H\Delta}  \\
\iff \frac{K - 12 - 12\log M}{K} &\leq \frac{4224H\Delta}{936(KR+1)\sigma^2 + 529H\Delta} \label{eq:homogeneous-non-convex-p-inequality} \\
\iff K \leq 12(1 + \log M) \qquad&\textrm{or}\qquad \sigma^2 \leq \frac{4224H\Delta K}{936(KR+1)(K - 12 - 12\log M)} - \frac{529H\Delta}{936(KR+1)} \\
\implies K \leq 24(1 + \log M) \qquad&\textrm{or}\qquad \sigma^2 \leq \frac{10H\Delta}{KR}
\end{align}
In this case, since $p_1 > p_2$ we take 
\begin{equation}
p = p_1 = \frac{529H\Delta}{936(KR+1)\sigma^2 + 529H\Delta}
\end{equation}
which means $T = KR$ and the lower bound is
\begin{align}
\nrm*{\nabla \hat{F}_{T,U}(\hat{x})} 
&\geq \frac{\sqrt{13H\Delta}}{312\sqrt{KR}}
\geq \begin{cases}
\frac{\sqrt{H\Delta}}{424\sqrt{R(1+\log M)}} & K \leq 24(1 + \log M) \\
\frac{\sqrt{H\Delta}}{174\sqrt{KR}} + \frac{\sqrt{\sigma}(H\Delta)^{1/4}}{308(KR)^{1/4}} & \sigma^2 \leq \frac{10H\Delta}{KR}
\end{cases} \\
&\geq \frac{\sqrt{H\Delta}}{174\sqrt{KR}} + \frac{\sqrt{\sigma}(H\Delta)^{1/4}}{308(KR)^{1/4}},\ \frac{\sqrt{H\Delta}}{424\sqrt{R(1+\log M)}} \label{eq:homogeneous-non-convex-case-1}
\end{align}

Otherwise, $p_1 < p_2$ and we take 
\begin{align}
p &= p_2 \\ 
&= -\frac{936(12R(1+\log M) + 1)\sigma^2 + 529H\Delta}{16\cdot936KR\sigma^2} \nonumber\\
&\qquad+ \frac{\sqrt{\prn*{936(12R(1+\log M) + 1)\sigma^2 + 529H\Delta}^2 + 32\cdot529\cdot936 H\Delta KR\sigma^2}}{16\cdot936KR\sigma^2} \\
&= -\frac{3(1+\log M)}{4K} - \frac{1}{16KR} - \frac{529H\Delta}{14976KR\sigma^2} \nonumber\\
&\qquad+ \sqrt{\prn*{\frac{3(1+\log M)}{4K} + \frac{1}{16KR} + \frac{529H\Delta}{14976KR\sigma^2}}^2 + \frac{529H\Delta}{7488KR\sigma^2}}  \\
&\leq \frac{\sqrt{H\Delta}}{3\sigma\sqrt{KR}} 
\end{align}
Therefore,
\begin{align}
T 
&= 8KRp + 12R\log M + 12 R + 1 \\
&\leq \frac{3\sqrt{H\Delta KR}}{\sigma} + 12R\log M + 12 R + 1 \\
&\leq \max\crl*{\frac{6\sqrt{H\Delta KR}}{\sigma},\ 26R(1+\log M)}
\end{align}
This gives the lower bound
\begin{align}
\nrm*{\nabla \hat{F}_{T,U}(\hat{x})} 
&\geq \frac{\sqrt{13H\Delta}}{312\sqrt{T}} \\
&\geq \min\crl*{\frac{\sqrt{13H\Delta}}{312\sqrt{\frac{6\sqrt{H\Delta KR}}{\sigma}}},\ \frac{\sqrt{13H\Delta}}{312\sqrt{26R(1+\log M)}}} \\
&\geq \min\crl*{\frac{\sqrt{\sigma}(H\Delta)^{1/4}}{212(KR)^{1/4}},\ \frac{\sqrt{H\Delta}}{442\sqrt{R(1+\log M)}}} \label{eq:homogeneous-non-convex-case-2}
\end{align}
Finally, by the same argument that led to \eqref{eq:homogeneous-non-convex-p-inequality},
\begin{align} 
p_2 &\geq p_1 \\
\implies 1 &\geq \frac{K - 12 - 12 \log M}{K} \geq \frac{4224H\Delta}{936(KR+1)\sigma^2 + 529H\Delta} \\
\implies
\sigma^2 &\geq \frac{3695H\Delta}{936(KR+1)} \geq \frac{H\Delta}{KR} \\
\implies \frac{\sqrt{\sigma}(H\Delta)^{1/4}}{212(KR)^{1/4}} &\geq \frac{\sqrt{H\Delta}}{424\sqrt{KR}} + \frac{\sqrt{\sigma}(H\Delta)^{1/4}}{424(KR)^{1/4}}
\end{align}
This, combined with \eqref{eq:homogeneous-non-convex-case-1} and \eqref{eq:homogeneous-non-convex-case-2} establishes a lower bound of
\begin{equation}\label{eq:homogeneous-non-convex-lb-1}
\nrm*{\nabla \hat{F}_{T,U}(\hat{x})} \geq c\cdot \min\crl*{\frac{\sqrt{H\Delta}}{\sqrt{KR}} + \frac{\sqrt{\sigma}(H\Delta)^{1/4}}{(KR)^{1/4}},\ \frac{\sqrt{H\Delta}}{\sqrt{R(1+\log M)}}}
\end{equation}
This holds for any intermittent communication algorithm. We will now argue that there is an additional term to the lower bound resembling
\begin{equation}
\nrm*{\nabla \hat{F}_{T,U}(\hat{x})} \geq c\cdot \frac{\sqrt{\sigma}(H\Delta)^{1/4}}{(MKR)^{1/4}}
\end{equation}
The argument is simple: any intermittent communication algorithm with a given $M$, $K$, and $R$ can be implemented using $MKR$ sequential calls to the oracle, which is equivalent to a different intermittent communication setting with $M' = 1$, $K' = MKR$, and $R'= 1$. Applying our lower bound in this case gives
\begin{align}
\nrm*{\nabla \hat{F}_{T,U}(\hat{x})} 
&\geq c\cdot\min\crl*{\frac{\sqrt{H\Delta}}{\sqrt{K'R'}} + \frac{\sqrt{\sigma}(H\Delta)^{1/4}}{(K'R')^{1/4}},\ \frac{\sqrt{H\Delta}}{\sqrt{R'(1+\log M')}}} \\
&= c\cdot\min\crl*{\frac{\sqrt{H\Delta}}{\sqrt{MKR}} + \frac{\sqrt{\sigma}(H\Delta)^{1/4}}{(MKR)^{1/4}},\ \frac{\sqrt{H\Delta}}{\sqrt{1(1+\log 1)}}} \\
&\geq c\cdot\min\crl*{\frac{\sqrt{\sigma}(H\Delta)^{1/4}}{(MKR)^{1/4}},\ \sqrt{H\Delta}}
\end{align}
Combining this with \eqref{eq:homogeneous-non-convex-lb-1} completes the proof.
\end{proof}


\subsection{Proof of \pref{thm:homogeneous-non-convex-upper-bound}} \label{app:intermittent-communication-homogeneous-non-convex-upper-bound}

\homogeneousnonconvexupperbound*
\begin{proof}
Stochastic gradient descent using $T$ sequential steps when the variance is bounded by $\sigma^2$ guarantees \citep[Corollary 2.2][]{ghadimi2013stochastic}
\begin{equation}
\E\nrm*{\nabla F(\hat{x})} \leq c\cdot\frac{\sqrt{H\Delta}}{\sqrt{T}} + c\cdot\frac{\sqrt{\sigma}(H\Delta)^{1/4}}{T^{1/4}}
\end{equation}
Therefore, Single-Machine SGD, which is equivalent to $KR$ steps of SGD with variance bounded by $\sigma^2$, guarantees
\begin{equation}
\E\nrm*{\nabla F(\hat{x})} \leq c\cdot\frac{\sqrt{H\Delta}}{\sqrt{KR}} + c\cdot\frac{\sqrt{\sigma}(H\Delta)^{1/4}}{(KR)^{1/4}}
\end{equation}
Similarly, Minibatch SGD, which is equivalent to $R$ steps of SGD with variance bounded by $\frac{\sigma^2}{MK}$, guarantees
\begin{equation}
\E\nrm*{\nabla F(\hat{x})} \leq c\cdot\frac{\sqrt{H\Delta}}{\sqrt{R}} + c\cdot\frac{\sqrt{\sigma}(H\Delta)^{1/4}}{(MKR)^{1/4}}
\end{equation}
Finally, by the smoothness of $F$,
\begin{equation}
\nrm*{\nabla F(0)} \leq \sqrt{2H(F(0) - \min_x F(x))} = \sqrt{2H\Delta}
\end{equation}
Therefore, the rate claimed by the theorem can be achieved by using whichever of these three methods has the smallest upper bound.
\end{proof}

\section{Proofs from \pref{sec:breaking-the-lower-bounds}}

\subsection{Proof of \pref{thm:homogeneous-convex-lower-bound-third-order-smooth}}\label{app:intermittent-communication-homogeneous-convex-lower-bound-third-order-smooth}

\homogeneousconvexlowerboundthirdorder*
\begin{proof}
To prove the theorem, we instantiate $F$ as defined in \eqref{eq:def-F-generic-psi-homogeneous-lower-bound} using 
\begin{equation}
\psi(x) = \begin{cases}
\frac{\ell_2}{2}x^2 & \abs{x}\leq\zeta \\
\frac{\ell_2}{2}x^2 - \frac{\ell_3}{6}\prn*{\abs*{x} - \zeta}^3 & \abs{x} \in \brk*{\zeta,\ \zeta + \frac{\ell_2}{\ell_3}} \\
\prn*{\ell_2\zeta + \frac{\ell_2^2}{2\ell_3}}\abs{x} - \frac{\ell_2\zeta^2}{2} - \frac{\ell_2^2\zeta}{2\ell_3} - \frac{\ell_2^3}{6\ell_3^2} & \abs{x} > \zeta + \frac{\ell_2}{\ell_3}
\end{cases}
\end{equation}
It is easy to confirm that this $\psi$ satisfies all of the conditions of \pref{lem:homogeneous-convex-generic-psi-construction}, that is, $\psi$ is convex, twice-differentiable, and even, and $\psi(0) = \psi'(0) = 0$, $\abs{\psi'(x)} \leq \ell_1 := \ell_2\zeta + \frac{\ell_2^2}{2\ell_3}$, and $\psi''(x) \leq \ell_2$. 
In addition to these properties, $\psi$ has a bounded third derivative $\abs{\psi'''(x)} \leq \ell_3$. Computing the third derivative of $F$, we have for any unit vector $u$
\begin{align}
\abs*{\nabla^3 F(x)[u,u,u]}
&= \abs*{\psi'''(x_N)u_N^3 + \sum_{i=1}^{N-1}\psi'''(x_{i+1} - x_i)(u_{i+1} - u_i)^3} \\
&\leq \ell_3 \prn*{\abs{u_N}^3 + \sum_{i=1}^{N-1}\abs*{u_{i+1} - u_i}^3} \\
&\leq \ell_3 \prn*{\abs{u_N}^3 + 8\sum_{i=1}^{N-1}\abs*{u_{i+1}}^3 - \abs*{u_i}^3} \\
&\leq 16\ell_3\nrm{u}^3 = 16\ell_3 \label{eq:third-order-smooth-lower-bound}
\end{align}
We then set 
\begin{align}
\zeta &= \frac{B}{N^{3/2}} \\
\ell_2 &= \frac{H}{4} \\
\ell_3 &= \frac{\beta}{16} \\
p &\geq \frac{12\ell_1^2}{12\ell_1^2 + \sigma^2} \\
\delta &= \frac{1}{16MKR} \\
\alpha^2 &= \min\crl*{\frac{\sigma^2}{2N\ell_2^2\prn*{\frac{160}{p} + \frac{32}{\delta}}},\ \frac{\zeta^2}{N^2}}
\end{align}
By \pref{lem:homogeneous-convex-generic-psi-construction} and \eqref{eq:third-order-smooth-lower-bound}, this ensures that $F$ is convex, $H$-smooth, $\nrm{\nabla^3 F(x)} \leq \beta$, and $\nrm{x^*}\leq B$. Furthermore, the stochastic gradient oracle variance is bounded by
\begin{equation}
\frac{6(1-p)\ell_1^2}{p} + \prn*{\frac{160}{p} + \frac{32}{\delta}}N\ell_2^2\alpha^2 \leq \frac{\sigma^2}{2} + \frac{\sigma^2}{2} = \sigma^2
\end{equation}
We also note that for $y \in [-\zeta,\zeta]$,
\begin{equation}
{\psi^*}'(y) = \frac{y}{\ell_2}
\end{equation}
Therefore, by \pref{lem:homogeneous-convex-generic-psi-construction} for $x$ such that $\prog{\alpha}(x) \leq \frac{N}{2}$, then
\begin{align}
F(x) - F^*
&\geq \psi'(\zeta)\brk*{N\zeta - \alpha - {\psi^*}'\prn*{\frac{2}{N}\psi'(\zeta)}} + \frac{N}{2}\brk*{\psi\prn*{{\psi^*}'\prn*{\frac{2}{N}\psi'(\zeta)}} - 2\psi(\zeta)} \\
&= \ell_2\zeta\brk*{N\zeta - \alpha - \frac{2\zeta}{N}} + \frac{N}{2}\brk*{\frac{2\ell_2\zeta^2}{N^2} - \ell_2\zeta^2} \\
&\geq \frac{\ell_2 N \zeta^2}{6} = \frac{H B^2}{24N^2} \label{eq:third-order-smooth-lower-bound-eq1}
\end{align}
The last inequality uses that $N > 2$, $\alpha \leq \frac{\zeta}{N}$, and $\zeta = \frac{B}{N^{3/2}}$.

Finally, because all of the algorithm's queries to the gradient oracle have norm bounded by $\gamma$ so on the way to applying \pref{lem:intermittent-communication-progress} we introduce a uniformly random orthogonal matrix $U \in \R^{D\times N}$ for 
\begin{equation}
D = N + \frac{2\gamma^2}{\alpha^2}\log(32MKRN)
\end{equation}
Then, since $g$ is an $(\alpha,p,\delta)$-robust-zero-chain, by \pref{lem:intermittent-communication-progress}, any intermittent communication algorithm that interacts with $Ug(U^\top x;z)$ will have progress at most
\begin{equation}
\prog{\alpha}(U^\top \hat{x}) \leq \min\crl*{KR,\ 8KRp + 12R\log M + 12R}
\end{equation}
with probability at least $\frac{5}{8} - 2MKR\delta = \frac{1}{2}$. We therefore take 
\begin{equation}
N = 2\ceil{\min\crl*{KR,\ 8KRp + 12R\log M + 12R}}
\end{equation}
which means that
\begin{equation}
\prog{\alpha}(U^\top \hat{x}) \leq \frac{N}{2}
\end{equation}
Therefore, by \eqref{eq:third-order-smooth-lower-bound-eq1} we conclude
\begin{align}
F(U^\top \hat{x}) - F^*
&\geq \frac{HB^2}{24(2\min\crl*{KR,\ 8KRp + 12R\log M + 12R}+1)^2} \\
&\geq \frac{HB^2}{\min\crl*{216(KR)^2,\ 1944(KR)^2p^2 + 4056R^2(1+\log M)^2}} \\
&\geq \frac{HB^2}{432(KR)^2} + \min\crl*{\frac{HB^2}{7776(KR)^2p^2},\ \frac{HB^2}{16224R^2(1+\log M)^2}} \label{eq:homogeneous-convex-lower-bound-third-order-smooth-eq2}
\end{align}
From here, we recall that $p$ needs to be chosen so that
\begin{equation}
p \geq \frac{12\ell_1^2}{12\ell_1^2 + \sigma^2} 
= \frac{12\prn*{\ell_2\zeta + \frac{\ell_2^2}{2\ell_3}}^2}{12\prn*{\ell_2\zeta + \frac{\ell_2^2}{2\ell_3}}^2 + \sigma^2} 
= \frac{3\prn*{\frac{HB}{N^{3/2}} + \frac{2H^2}{\beta}}^2}{3\prn*{\frac{HB}{N^{3/2}} + \frac{2H^2}{\beta}}^2 + 4\sigma^2}
\end{equation}
The difficulty here is that $N$ is defined in terms of $p$, however, we observe that
\begin{equation}
N \geq KRp
\end{equation}
therefore, choosing
\begin{equation}
p \geq \frac{3\prn*{\frac{HB}{(KRp)^{3/2}} + \frac{2H^2}{\beta}}^2}{3\prn*{\frac{HB}{(KRp)^{3/2}} + \frac{2H^2}{\beta}}^2 + 4\sigma^2}
\end{equation}
satisfies the requirement on $p$. Therefore, we observe that since $x\mapsto \frac{x}{x + 4\sigma^2}$ is increasing in $x$,
\begin{equation}
\frac{3\prn*{\frac{HB}{(KRp)^{3/2}} + \frac{2H^2}{\beta}}^2}{3\prn*{\frac{HB}{(KRp)^{3/2}} + \frac{2H^2}{\beta}}^2 + 4\sigma^2}
\leq \max\crl*{\frac{3H^2B^2}{3H^2B^2 + 4\sigma^2K^3R^3p^3},\ \frac{48H^4}{48H^4 + 4\sigma^2\beta^2}}
\end{equation}
Therefore, we set
\begin{equation}
p =\min\crl*{1,\ \max\crl*{\prn*{\frac{3H^2B^2}{4\sigma^2K^3R^3}}^{1/4},\ \frac{12H^4}{\sigma^2\beta^2}}}
\end{equation}
Therefore, returning to \eqref{eq:homogeneous-convex-lower-bound-third-order-smooth-eq2} we conclude
\begin{align}
F&(U^\top \hat{x}) - F^* \nonumber\\
&\geq \frac{HB^2}{432(KR)^2} + \min\crl*{\frac{HB^2}{7776(KR)^2\frac{\sqrt{3}HB}{2\sigma K^{3/2}R^{3/2}}},\ \frac{HB^2}{7776(KR)^2\frac{144H^8}{\sigma^4\beta^4}},\ \frac{HB^2}{16224R^2(1+\log M)^2}} \\
&\geq \frac{HB^2}{432(KR)^2} + \min\crl*{\frac{\sigma B}{6735\sqrt{KR}},\ \frac{\sigma^4\beta^4B^2}{1119744(KR)^2H^7},\ \frac{HB^2}{16224R^2(1+\log M)^2}}
\end{align}
From here, we note that this lower bound could also be instantiated for some $H' \leq H$, which would still ensure that $F$ is $H$-smooth. Therefore, taking
\begin{equation}
H' = \min\crl*{H, \prn*{\frac{\sigma^4\beta^4(1+\log M)^2}{K^2}}^{1/8}}
\end{equation}
we conclude
\begin{align}
F&(U^\top \hat{x}) - F^* \nonumber\\
&\geq \frac{HB^2}{432(KR)^2} + \min\crl*{\frac{\sigma B}{6735\sqrt{KR}},\ \frac{\sqrt{\sigma\beta}B^2}{1119744K^{1/4}R^2(1+\log M)^{7/4}},\ \frac{HB^2}{16224R^2(1+\log M)^2}}
\end{align}
Finally, by \pref{lem:statistical-term-lower-bound}, the min-max error is also lower bounded by
\begin{equation}
F(\hat{x}) - F^* \geq c\cdot\min\crl*{\frac{\sigma B}{\sqrt{MKR}},\ HB^2}
\end{equation}
with probability at least $\frac{1}{4}$, which completes the proof.
\end{proof}

\subsection{Proof of \pref{thm:heterogeneous-convex-lower-bound}, \pref{thm:heterogeneous-strongly-convex-lower-bound}, and \pref{thm:bounded-heterogeneous-convex-lower-bound}}\label{app:heterogeneous-convex-lower-bound}


In this section, we prove \pref{thm:heterogeneous-convex-lower-bound}, \pref{thm:heterogeneous-strongly-convex-lower-bound}, and \pref{thm:bounded-heterogeneous-convex-lower-bound} using essentially the same argument. Our lower bound construction includes a term $\frac{\lambda}{2}\nrm{x}^2$, so the lower bound applies to strongly convex functions, and we also bound the heterogeneity in order to prove \pref{thm:bounded-heterogeneous-convex-lower-bound}. In fact, we begin by proving \pref{thm:bounded-heterogeneous-convex-lower-bound}, and then prove \pref{thm:heterogeneous-convex-lower-bound} and \pref{thm:heterogeneous-strongly-convex-lower-bound} as simple corollaries by taking $\sdiff^2$ sufficiently large.

For this lower bound, we will construct just two local functions $F_1$ and $F_2$. For the case $M > 2$, $F_1$ will be assigned to the first $\lfloor M/2\rfloor$ machines, and $F_2$ to the next $\lfloor M/2\rfloor$ machines. If there is an odd number of machines, we simply assign the last machine $F_3(x) = \frac{\lambda}{2}\nrm{x}^2$, which will reduce the lower bound by a factor of at most $\frac{M-1}{M}$. Therefore, we proceed by focusing on the case $M=2$. 

The construction is quite similar to other lower bounds, notably those of \citet{arjevani2015communication,nesterov2004introductory}.
For a function $\psi:\R\to\R$ and constant $C \in [0,1]$ to be defined later, let
\begin{align}
\tilde{F}(x) &= \frac{1}{2}\prn*{\tilde{F}_1(x) + \tilde{F}_2(x)} \\
\tilde{F}_1(x) &= -\psi'(\zeta)x_1 + C\psi(x_d) + \sum_{i=1}^{d/2-1}\psi\prn*{x_{2i+1} - x_{2i}} \\
\tilde{F}_2(x) &= \sum_{i=1}^{d/2}\psi\prn*{x_{2i} - x_{2i-1}}
\end{align}

The key property of these functions is that:
\begin{equation}\label{eq:heterogeneous-construction-gradient-progress-property}
\begin{aligned}
x_{\textrm{even}} \in \spn\crl*{e_1,e_2,\dots,e_{2i}} &\implies \left\{
\begin{aligned}
\nabla \tilde{F}_1(x_{\textrm{even}}) &\in \spn\crl*{e_1,e_2,\dots,e_{2i+1}} \\
\nabla \tilde{F}_2(x_{\textrm{even}}) &\in \spn\crl*{e_1,e_2,\dots,e_{2i}}
\end{aligned}\right. \\
x_{\textrm{odd}} \in \spn\crl*{e_1,e_2,\dots,e_{2i-1}} &\implies \left\{
\begin{aligned}
\nabla \tilde{F}_1(x_{\textrm{odd}}) &\in \spn\crl*{e_1,e_2,\dots,e_{2i-1}} \\
\nabla \tilde{F}_2(x_{\textrm{odd}}) &\in \spn\crl*{e_1,e_2,\dots,e_{2i}}
\end{aligned}\right.
\end{aligned}
\end{equation}
Therefore, roughly speaking, gradient queries to $\tilde{F}_1$ only allow for making progress when the progress of the query is even, and queries to $\tilde{F}_2$ only when the progress is odd. Of course, there is some chance that an arbitrary randomized algorithm might ``guess'' its way to additional progress, but we will shortly argue that this happens only with a very small probability for a suitably chosen stochastic gradient oracle.

We define stochastic gradient oracles for $F_1$ and $F_2$ as
\begin{align}\label{eq:heterogeneous-stochastic-gradient-oracle}
z &= \begin{cases} 0 & \textrm{with probability } 1-\delta \\ 1 & \textrm{with probability } \delta \end{cases} \\
\tilde{g}_1(x;0) &= \sum_{i=1}^{2\floor{\prog{\alpha}(x)/2} + 1}e_ie_i^\top \brk*{\nabla \tilde{F}_1\prn*{[x_1,\dots,x_{2\floor{\prog{\alpha}(x)/2} + 1},0,\dots,0]}} \\
\tilde{g}_1(x;1) &= \frac{1}{\delta}\nabla \tilde{F}_1(x) - \frac{1-\delta}{\delta}\tilde{g}_1(x;0) \\
\tilde{g}_2(x;0) &= \sum_{i=1}^{2\ceil{\prog{\alpha}(x)/2}}e_ie_i^\top \brk*{\nabla \tilde{F}_2\prn*{[x_1,\dots,x_{2\ceil{\prog{\alpha}(x)/2}},0,\dots,0]}} \\
\tilde{g}_2(x;1) &= \frac{1}{\delta}\nabla \tilde{F}_2(x) - \frac{1-\delta}{\delta}\tilde{g}_2(x;0) 
\end{align}
These oracles are very similar to the one used in the proof of \pref{thm:homogeneous-convex-lower-bound}. Furthermore, as in the proof of \pref{thm:homogeneous-convex-lower-bound}, our argument will rely on applying a random rotation to the argument of $F$:

Let $U \in \R^{D \times d}$ be a uniformly random matrix with $U^\top U = I_{d\times d}$ and define
\begin{equation}\label{eq:heterogeneous-rotation}
\begin{aligned}
F^U(x) &= \tilde{F}(U^\top x) + \frac{\lambda}{2}\nrm{x}^2 \\
F_1^U(x) &= \tilde{F}_1(U^\top x) + \frac{\lambda}{2}\nrm{x}^2 \\
F_2^U(x) &= \tilde{F}_2(U^\top x) + \frac{\lambda}{2}\nrm{x}^2 \\
g_1^U(x;z) &= U\tilde{g}_1(U^\top x;z) + \lambda x \\
g_2^U(x;z) &= U\tilde{g}_2(U^\top x;z) + \lambda x
\end{aligned}
\end{equation}
We will drop the superscript $U$ when it is clear from context. The following lemma is an analogue of \pref{lem:intermittent-communication-progress} for the heterogeneous setting. It shows that if $z = 0$ for every stochastic gradient oracle query made by the algorithm, then with high probability the algorithm will only gain one coordinate of progress per round of communication because of the property \eqref{eq:heterogeneous-construction-gradient-progress-property}.
\begin{restatable}{lemma}{heterogeneousprogress}\label{lem:heterogeneous-progress}
Let $\hat{x}$ be the output of any intermittent communication algorithm whose oracle queries all have norm bounded by $\gamma$ and where all machines interact with either $g^U_1$ or $g^U_2$, and let $D \geq R + \frac{2\gamma^2}{\alpha^2}\log(d/\delta)$. Then 
\[
\P\prn*{\prog{\alpha}(U^\top \hat{x}) \leq R} \geq 1 - 4MKR\delta
\]
over the randomness in the draw of $U$ and the stochastic gradient oracles.
\end{restatable}
\begin{proof}
First, let $Z$ denote the event that $z=0$ for every stochastic oracle query made by the algorithm. By the union bound, $\P(Z) \geq 1-MKR\delta$, and we will condition on this event for the rest of the proof.

Let $x^m_{k,r}$ denote the $k\mathth$ query during the $r\mathth$ round of communication on the $m\mathth$ machine. Furthermore, let $j(m) \in \crl{1,2}$ denote the index of the objective ($F_1$ or $F_2$) with which the $m\mathth$ machine interacts. In the context of the intermittent communication graph, we will use $\mc{X}^m_{k,r}$ to denote the queries made ancestors of the node $v^m_{k,r}$ and $\mc{G}^m_{k,r}$ to denote the stochastic gradient oracle responses.
\begin{align}
\mc{X}^m_{k,r} &= \crl*{x^m_{k',r}\,:\,k'<k} \cup \crl*{x^{m'}_{k',r'}\,:\,r'<r,k'\in[K],m'\in[M]} \\
\mc{G}^m_{k,r} &= \crl*{g^U_{j(m)}\prn*{x^m_{k',r};0}\,:\,k'<k} \cup \crl*{g^U_{j(m')}\prn*{x^{m'}_{k',r'};0}\,:\,r'<r,k'\in[K],m'\in[M]}
\end{align}

We will also abuse notation and write for $\mc{X}\subseteq\R^d$ 
\begin{equation}
\prog{\alpha}^U(\mc{X}) = \max\crl*{\prog{\alpha}(U^\top x)\,:\,x\in\mc{X}}
\end{equation}

We then define the following family of events
\begin{equation}
B_r = \bigcup_{m=1}^M \crl*{\prog{\alpha}^U\prn*{\mc{X}^m_{K,r}} > \begin{cases} 2\floor{r/2} + 1 & j(m) = 1 \\ 2\ceil{r/2} & j(m) = 2 \end{cases}}
\end{equation}
We will now argue that
$\P\prn*{\bigcup_{r=1}^R B_r\,\middle|\,Z}$
is small, which essentially proves the lemma. To prove this, we rewrite it with a disjoint union and apply the union bound:
\begin{align}
\P\prn*{\bigcup_{r=1}^R B_r\,\middle|\,Z}
&= \P\prn*{\bigcup_{r=1}^R \crl*{B_r \cap \brk*{\bigcup_{r'=1}^{r-1} B_{r'}}^c}\,\middle|\,Z} \\
&= \sum_{r=1}^R \P\prn*{B_r \cap \brk*{\bigcup_{r'=1}^{r-1} B_{r'}}^c\,\middle|\,Z} \\
&\leq \sum_{r=1}^R\sum_{m=1}^M \P\prn*{\prog{\alpha}^U\prn*{\mc{X}^m_{K,r}} > \begin{cases} 2\floor{r/2} + 1 & j(m) = 1 \\ 2\ceil{r/2} & j(m) = 2 \end{cases},\ \brk*{\bigcup_{r'=1}^{r-1} B_{r'}}^c\,\middle|\,Z}
\end{align}
We will show that each term in the final sum is small.

Fix $m$ and $r$. Conditioned on the event $Z$, each of gradient oracle responses corresponds to $z=0$. Let $x \in \mc{X}^m_{1,r}$ be an oracle query made on some machine $m'$, then under the event $\brk*{\bigcup_{r'=1}^{r-1} B_{r'}}^c$,
\begin{equation}
\prog{\alpha}^U(x) \leq \begin{cases} 2\floor{(r-1)/2} + 1 & j(m') = 1 \\ 2\ceil{(r-1)/2} & j(m') = 2 \end{cases}
\end{equation}
Consequently, using $U_{1:t} = \brk*{U_1,\dots,U_t,0,\dots,0}$ to denote the matrix $U$ with the $(t+1)\mathth$ through $D\mathth$ columns replaces by zeros, if $j(m') = 1$, we have
\begin{equation}\label{eq:heterogeneous-gradient-dependence}
\begin{aligned}
g^U_1(x;0) 
&= \lambda x + U\sum_{i=1}^{2\floor{\prog{\alpha}^U(x)/2} + 1}e_ie_i^\top \brk*{\nabla \tilde{F}_1\prn*{[U_1^\top x,\dots,U_{2\floor{\prog{\alpha}^U(x)/2} + 1}^\top x,0,\dots,0]}} \\
&= \lambda x + U_{1:2\floor{\prog{\alpha}^U(x)/2} + 1}\sum_{i=1}^{2\floor{\prog{\alpha}^U(x)/2} + 1}e_ie_i^\top \brk*{\nabla \tilde{F}_1\prn*{U_{1:2\floor{\prog{\alpha}^U(x)/2} + 1}^\top x}}
\end{aligned}  
\end{equation}
It is easy to see that is a measurable function of $x$ and $U_{1:2\floor{\prog{\alpha}^U(x)/2} + 1}$ and therefore also of $U_{1:2\floor{(r-1)/2} + 1}$. Similarly, it is straightforward to confirm that if $j(m') = 2$, then $g^U_2(x;0)$ is a measurable function of $U_{1:2\ceil{(r-1)/2}}$. 

By the definition of an intermittent communication algorithm, there exists a query functions $\mc{Q}^{m'}_{k',r'}$ such that for each $m',k',r'$
\begin{equation}
x^{m'}_{k',r'} = \mc{Q}^{m'}_{k',r'}\prn*{\mc{X}^{m'}_{k',r'}, \mc{G}^{m'}_{k',r'}, \xi}
\end{equation}
where $\xi$ are the random coins of the algorithm. It follows that conditional on $Z$ and $\brk*{\bigcup_{r'=1}^{r-1} B_{r'}}^c$, the queries $\mc{X}^m_{1,r}$ and gradients $\mc{G}^m_{1,r}$ are measurable functions of $\xi$, $U_{1:2\floor{(r-1)/2} + 1}$, and $U_{1:2\ceil{(r-1)/2}}$, and therefore of $\xi$ and $U_{1:r}$.

Suppose $j(m) = 1$. Then 
\begin{align}
\P&\prn*{\prog{\alpha}^U\prn*{\mc{X}^m_{K,r}} > 2\floor{r/2} + 1,\ \brk*{\bigcup_{r'=1}^{r-1} B_{r'}}^c\,\middle|\,Z} \nonumber\\
&= \P\prn*{\bigcup_{k=1}^K \brk*{\crl*{\prog{\alpha}(U^\top x^m_{k,r}) > 2\floor{r/2} + 1} \cap \bigcap_{k'=1}^{k-1} \crl*{\prog{\alpha}\prn*{U^\top x^m_{k',r}} \leq 2\floor{r/2} + 1}},\ \brk*{\bigcup_{r'=1}^{r-1} B_{r'}}^c\,\middle|\,Z} \\
&\leq \sum_{k=1}^K \P\prn*{\prog{\alpha}(U^\top x^m_{k,r}) > 2\floor{r/2} + 1,\ \bigcap_{k'=1}^{k-1} \crl*{\prog{\alpha}\prn*{U^\top x^m_{k',r}} \leq 2\floor{r/2} + 1},\ \brk*{\bigcup_{r'=1}^{r-1} B_{r'}}^c\,\middle|\,Z}
\end{align}
First, we note that $\mc{X}^m_{k,r} = \mc{X}^m_{1,r}\cup \crl*{x^m_{1,r},\dots,x^m_{k-1,r}}$ and $\mc{G}^m_{k,r} = \mc{G}^m_{1,r}\cup \crl*{g^U_1(x^m_{1,r};0),\dots,g^U_1(x^m_{k-1,r};0)}$. Furthermore, by the same argument as in \eqref{eq:heterogeneous-gradient-dependence} above, $\mc{X}^m_{k,r}$ and $\mc{G}^m_{k,r}$ are measurable functions of $\xi$ and $U_{1:2\floor{r/2} + 1}$ conditional on $Z$, $\bigcap_{k'=1}^{k-1} \crl*{\prog{\alpha}\prn*{U^\top x^m_{k',r}} \leq 2\floor{r/2} + 1}$ and $\brk*{\bigcup_{r'=1}^{r-1} B_{r'}}^c$. Therefore, there exists a measurable function $\mc{A}^m_{k,r}$ such that
\begin{align}
&\P\prn*{\prog{\alpha}(U^\top x^m_{k,r}) > 2\floor{r/2} + 1,\ \bigcap_{k'=1}^{k-1} \crl*{\prog{\alpha}\prn*{U^\top x^m_{k',r}} \leq 2\floor{r/2} + 1},\ \brk*{\bigcup_{r'=1}^{r-1} B_{r'}}^c\,\middle|\,Z} \nonumber\\
&= \P\prn*{\prog{\alpha}(U^\top \mc{A}^m_{k,r}(\xi,U_{1:2\floor{r/2} + 1})) > 2\floor{r/2} + 1,\ \bigcap_{k'=1}^{k-1} \crl*{\prog{\alpha}\prn*{U^\top x^m_{k',r}} \leq 2\floor{r/2} + 1},\ \brk*{\bigcup_{r'=1}^{r-1} B_{r'}}^c\,\middle|\,Z} \\
&\leq \P\prn*{\exists_{i > 2\floor{r/2} + 1}\abs*{\inner{U_i}{\mc{A}^m_{k,r}(\xi,U_{1:2\floor{r/2} + 1})}} > \alpha \,\middle|\,Z} \\
&\leq \sum_{i=2\floor{r/2} + 2}^d\P\prn*{\abs*{\inner{U_i}{\mc{A}^m_{k,r}(\xi,U_{1:2\floor{r/2} + 1})}} > \alpha \,\middle|\,Z}
\end{align}
From here, we observe that $U_i$ is independent of the algorithm's coins $\xi$ and the event $Z$, and conditional on $U_{1:2\floor{r/2} + 1}$, $U_i$ is uniformly distributed on the unit sphere in the $(D - 2\floor{r/2} - 1)$-dimensional subspace that is orthogonal to the range of $U_{1:2\floor{r/2} + 1}$. Therefore, for $\gamma \max\crl*{\nrm{x}\,:\,x\in\mc{X}^m_{k,r}}$, using standard concentration results on the sphere \citep{ball1997elementary},
\begin{equation}
\P\prn*{\abs*{\inner{U_i}{\mc{A}^m_{k,r}(\xi,U_{1:2\floor{r/2} + 1})}} > \alpha \,\middle|\,Z} \leq 2\exp\prn*{-\frac{(D - 2\floor{r/2})\alpha^2}{2\gamma^2}}
\end{equation}
We conclude that 
\begin{multline}
\P\prn*{\prog{\alpha}(U^\top x^m_{k,r}) > 2\floor{r/2} + 1,\ \bigcap_{k'=1}^{k-1} \crl*{\prog{\alpha}\prn*{U^\top x^m_{k',r}} \leq 2\floor{r/2} + 1},\ \brk*{\bigcup_{r'=1}^{r-1} B_{r'}}^c\,\middle|\,Z} \\
\leq 2d\exp\prn*{-\frac{(D - 2\floor{r/2})\alpha^2}{2\gamma^2}}
\end{multline}
and therefore,
\begin{equation}
\P\prn*{\prog{\alpha}^U\prn*{\mc{X}^m_{K,r}} > 2\floor{r/2} + 1,\ \brk*{\bigcup_{r'=1}^{r-1} B_{r'}}^c\,\middle|\,Z} 
\leq 2Kd\exp\prn*{-\frac{(D - 2\floor{r/2})\alpha^2}{2\gamma^2}}
\end{equation}
A nearly identical argument shows that when $j(m) = 2$,
\begin{equation}
\P\prn*{\prog{\alpha}^U\prn*{\mc{X}^m_{K,r}} > 2\ceil{r/2},\ \brk*{\bigcup_{r'=1}^{r-1} B_{r'}}^c\,\middle|\,Z} 
\leq 2Kd\exp\prn*{-\frac{(D - 2\ceil{r/2} + 1)\alpha^2}{2\gamma^2}}
\end{equation}
Therefore, we conclude
\begin{equation}
\P\prn*{\bigcup_{r=1}^R B_r\,\middle|\,Z} \leq 2MKRd\exp\prn*{-\frac{(D - R)\alpha^2}{2\gamma^2}}
\end{equation}
Finally, using the same arguements as above, conditional on the events $\brk*{\bigcup_{r=1}^R B_r}^c$ and $Z$, the algorithm's output
\begin{equation}
\hat{x} = \hat{X}\prn*{\bigcup_{m=1}^M\mc{X}^m_{K,R}, \bigcup_{m=1}^M\mc{G}^m_{K,R}, \xi}
\end{equation}
is a measurable function of $\xi$ and $U_{1:R}$ and therefore,
\begin{equation}
\P\prn*{\prog{\alpha}(U^\top\hat{x}) > R\,\middle|\,\brk*{\bigcup_{r=1}^R B_r}^c, Z} \leq 2d\exp\prn*{-\frac{(D - R + 1)\alpha^2}{2\gamma^2}}
\end{equation}
Therefore,
\begin{align}
\P\prn*{\prog{\alpha}(U^\top\hat{x}) > R} 
&\leq \P\prn*{\prog{\alpha}(U^\top\hat{x}) > R\,\middle|\,\brk*{\bigcup_{r=1}^R B_r}^c, Z} + \P\prn*{\bigcup_{r=1}^R B_r\,\middle|\,Z} + (1 - \P(Z)) \\
&\leq  3MKRd\exp\prn*{-\frac{(D - R)\alpha^2}{2\gamma^2}} + MKR\delta
\end{align}
Noting that $D \geq R + \frac{2\gamma^2}{\alpha^2}\log(d/\delta)$ completes the proof.
\end{proof}

We will use the fact that the algorithm's output will have progress less than $R$ in order to prove the lower bound. Before doing so, we will show that $F$ and the stochastic gradient oracles satisfy the necessary regularity properies:
\begin{lemma}\label{lem:heterogeneous-function-properties}
Let $\psi(x) = \frac{\ell_2}{2}x^2$ for $\ell_2 \leq \frac{H-\lambda}{4}$. In addition, let $\beta = \sqrt{1 + \frac{2\ell_2}{\lambda}}$, $q = \frac{\beta - 1}{\beta + 1}$, and $C = 1 - q$. Finally, let $\alpha^2 \leq \min\crl*{\frac{\zeta^2 q^{2d}}{4(1-q)^2},\ \frac{\sigma^2\delta}{8d\ell_2^2(1-\delta)}}$ and $d \geq R + \frac{\log 2}{2\log(1/q)}$. Then for all $U$
\begin{enumerate}
\item $F$, $F_1$, and $F_2$ are $\lambda$-strongly convex and $H$-smooth
\item $\nrm{x^*}^2 = \frac{\zeta^2q^2(1-q^{2d})}{(1-q)^2(1-q^2)}$
\item $F(0) - F^* \leq \frac{q\ell_2\zeta^2}{4(1-q)}$
\item $\prog{\alpha}(x) \leq R \implies F(x) - F^* \geq \frac{\zeta^2\lambda q^2}{16(1-q)^2(1-q^2)}q^{2R}$
\item $\E_z g_1(x;z) = \nabla F_1(x)$ and $\E_z g_2(x;z) = \nabla F_2(x)$
\item $\E_z\nrm*{g_1(x;z) - \nabla F_1(x)}^2 \leq \sigma^2$ and $\E_z\nrm*{g_2(x;z) - \nabla F_2(x)}^2 \leq \sigma^2$
\item $\frac{1}{2}\prn*{\nrm*{\nabla F_1(x^*)}^2 + \nrm*{\nabla F_2(x^*)}^2} \leq \frac{\zeta^2\ell_2^2(1+q)}{4(1-q)}$
\end{enumerate}
\end{lemma}
\begin{proof}
We will prove each point individually. 

\textbf{1)} 
For a unit vector $v$, let $v_i = \inner{U_i}{v}$ and $x_i = \inner{U_i}{x}$ for each $i \in [d]$. Then,
\begin{equation}
v^\top \nabla F_1(x) v 
= \ell_2 v_d^2 + \ell_2\sum_{i=1}^{d/2-1}\prn*{v_{2i+1} - v_{2i}}^2 + \lambda\nrm{v}^2 
\in [\lambda, 2\ell_2 + \lambda]
\end{equation}
Therefore, $F^U_1$ is $\lambda$-strongly convex and $(4\ell_2 + \lambda) \leq H$-smooth. It is also easy to confirm that $F^U_2$ is $\lambda$-strongly convex and $H$-smooth with the same argument. Finally, since $F$ is the mean of $F_1$ and $F_2$, it is also $\lambda$-strongly convex and $H$-smooth.

For the next points, we compute $x^* = \argmin_x F(x)$, which satisfies $\nabla F(U^\top x^*) = 0$. Letting $x^*_i = \inner{U_i}{x^*}$, this is equivalent to
\begin{equation}
\begin{aligned}\label{eq:heterogeneous-optimal-grad-zero}
0 &= -\frac{1}{2}\psi'(\zeta) + \lambda x^*_1 - \frac{1}{2}\psi'(x^*_2 - x^*_1) = -\frac{\ell_2\zeta}{2} + \prn*{\frac{\ell_2}{2} + \lambda} x^*_1 + \frac{\ell_2}{2} x^*_2 \\
0 &= \lambda x^*_i + \frac{1}{2}\psi'(x^*_{i} - x^*_{i-1}) - \frac{1}{2}\psi'(x^*_{i+1} - x^*_{i}) = -\frac{\ell_2}{2}x^*_{i-1} + \prn*{\ell_2 + \lambda}x^*_i - \frac{\ell_2}{2}x^*_{i+1} \qquad 2 \leq i \leq d-1 \\
0 &= \lambda x^*_d + \frac{C}{2}\psi'(x^*_{d}) + \frac{1}{2}\psi'(x^*_{d} - x^*_{d-1}) = -\frac{\ell_2}{2}x^*_{d-1} + \prn*{\frac{(1+C)\ell_2}{2} + \lambda}x^*_d
\end{aligned}
\end{equation}
Let $q$ be the smaller root of the quadratic equation
\begin{equation}\label{eq:heterogeneous-quadratic-equation}
-\frac{\ell_2}{2} + \prn*{\ell_2 + \lambda}q - \frac{\ell_2}{2}q^2 = 0
\end{equation}
so for $\beta = \sqrt{1 + \frac{2\ell_2}{\lambda}}$, 
\begin{equation}
q = 1 + \frac{\lambda}{\ell_2} - \sqrt{\prn*{1 + \frac{\lambda}{\ell_2}}^2 - 1} = \frac{\beta - 1}{\beta + 1}
\end{equation}
Returning to \eqref{eq:heterogeneous-optimal-grad-zero}, it is straightforward to confirm that 
\begin{equation}
x^* = \frac{\zeta}{1-q}\sum_{i=1}^d q^i U_i
\end{equation}
satisfies $\nabla F(x^*) = 0$, and therefore it is the minimizer of $F$.

\textbf{2)}
The minimizer has norm
\begin{equation}
\nrm{x^*}^2 = \frac{\zeta^2}{(1-q)^2} \sum_{i=1}^d q^{2i} = \frac{\zeta^2q^2(1-q^{2d})}{(1-q)^2(1-q^2)}
\end{equation}

\textbf{3)}
The minimum value of $F$ is
\begin{align}
F^*
&= -\frac{\ell_2\zeta^2q}{2(1-q)} + \frac{(1-q)\ell_2\zeta^2q^{2d}}{4(1-q)^2} + \frac{\ell_2\zeta^2}{4(1-q)^2}\sum_{i=1}^{d-1}\prn*{q^{i+1} - q^i}^2 + \frac{\lambda}{2}\nrm{x^*}^2 \\
&= \frac{\ell_2\zeta^2}{4(1-q)^2}\prn*{-2q(1-q) + (1-q)q^{2d} + (1-q)^2\sum_{i=1}^{d-1}q^{2i} + \frac{2\lambda}{\gamma}\sum_{i=1}^{d}q^{2i}} \\
&= \frac{\ell_2\zeta^2}{4(1-q)^2}\prn*{-2q(1-q) + (1-q)q^{2d} - (1-q)^2q^{2d} + \prn*{(1-q)^2 + \frac{2\lambda}{\gamma}}\sum_{i=1}^{d}q^{2i}} \\
&= \frac{\ell_2\zeta^2}{4(1-q)^2}\prn*{-2q(1-q) + (q-q^2)q^{2d} + \frac{(1-q)^2(1 + q)}{q}\cdot \frac{q^2(1 - q^{2d})}{1 - q^2}} \\
&= \frac{\ell_2\zeta^2}{4(1-q)^2}\prn*{-2q(1-q) + (q-q^2)q^{2d} + (q-q^2)(1 - q^{2d})} \\
&= \frac{-q\ell_2\zeta^2}{4(1-q)} 
\end{align}
for the fourth equality, we used that \eqref{eq:heterogeneous-quadratic-equation} implies $\frac{2\lambda}{\ell_2} = \frac{(1-q)^2}{q}$. Observing that $F(0) = 0$ proves claim 3.

\textbf{4)}
Let $x$ be a vector such that $\prog{\alpha}(x) \leq R$ and let $x_i = \inner{U_i}{x}$ for each $i \leq d$. Then by the $\lambda$-strong convexity of $F$,
\begin{align}
F(x) - F^* 
&\geq \frac{\lambda}{2}\nrm{x - x^*}^2 \\
&\geq \frac{\lambda}{2}\sum_{i=R+1}^{d}(x_i - x^*_i)^2 \\
&\geq \frac{\zeta^2\lambda}{8(1-q)^2}\sum_{i=R+1}^{d}q^{2i} \\
&= \frac{\zeta^2\lambda}{8(1-q)^2}\cdot\frac{q^2(q^{2R} - q^{2d})}{1 - q^2} \\
&\geq \frac{\zeta^2\lambda q^2}{16(1-q)^2(1-q^2)}q^{2R}
\end{align}
For the third inequality, we used that $\alpha \leq \frac{\zeta q^d}{2(1-q)} \leq x^*_d$. For the fourth inequality, we used that $d \geq R + \frac{\log 2}{2\log(1/q)}$ which implies $q^{2R} \geq 2q^{2d}$

\textbf{5)}
This is a straightforward computation
\begin{equation}
\E_z g^U_1(x;z) = (1-\delta)g^U_1(x;0) + \delta g^U_1(x;1) = (1-\delta)g^U_1(x;0) + \delta\brk*{\frac{1}{\delta}\nabla F^U_1(x) - \frac{1-\delta}{\delta}g^U_1(x;0)} = \nabla F^U_1(x)
\end{equation}
and
\begin{equation}
\E_z g^U_2(x;z) = (1-\delta)g^U_2(x;0) + \delta g^U_2(x;1) = (1-\delta)g^U_2(x;0) + \delta\brk*{\frac{1}{\delta}\nabla F^U_2(x) - \frac{1-\delta}{\delta}g^U_2(x;0)} = \nabla F^U_2(x)
\end{equation}

\textbf{6)}
For any $x$,
\begin{align}
\E_z \nrm*{g^U_1(x;z) - \nabla F^U_1(x)}^2 
&= (1-\delta)\nrm*{\tilde{g}_1(U^\top x;0) - \nabla \tilde{F}_1(U^\top x)}^2 + \delta\nrm*{\tilde{g}_1(U^\top x;1) - \nabla \tilde{F}_1(U^\top x)}^2 \\
&= \frac{1-\delta}{\delta}\nrm*{\tilde{g}_1(U^\top x;0) - \nabla \tilde{F}_1(U^\top x)}^2
\end{align}
Similarly,
\begin{align}
\E_z \nrm*{g^U_2(x;z) - \nabla F^U_2(x)}^2 
&= (1-\delta)\nrm*{\tilde{g}_2(U^\top x;0) - \nabla \tilde{F}_2(U^\top x)}^2 + \delta\nrm*{\tilde{g}_2(U^\top x;1) - \nabla \tilde{F}_2(U^\top x)}^2 \\
&= \frac{1-\delta}{\delta}\nrm*{\tilde{g}_2(U^\top x;0) - \nabla \tilde{F}_2(U^\top x)}^2
\end{align}
Let $x$ be such that $\prog{\alpha}(U^\top x) = j$, and let $x_i = \inner{U_i}{x}$ for each $i$. Then
\begin{align}
&\tilde{g}_1(U^\top x;0) - \nabla \tilde{F}_1(U^\top x)\nonumber\\
&= \sum_{i=1}^{2\floor{j/2} + 1} e_ie_i^\top \brk*{\nabla \tilde{F}_1\prn*{[x_1,\dots,x_{2\floor{j/2} + 1},0,\dots,0]}} - \tilde{F}_1(U^\top x) \\
&= -\sum_{i=2\floor{j/2} + 2}^d e_ie_i^\top \tilde{F}_1(U^\top x) \\
&= -\ell_2\sum_{i=\floor{j/2} + 1}^d \prn*{x_{2i+1} - x_{2i}}(e_{2i+1} - e_{2i})
\end{align}
Since $\prog{\alpha}(U^\top x) = j$, $\abs{x_{2i+1}} \leq \alpha$ and $\abs{x_{2i}} \leq \alpha$ for $i \geq \floor{j/2} + 1$. Therefore, (following a nearly identical argument for $g_2$)
\begin{align}
\E_z \nrm*{g^U_1(x;z) - \nabla F^U_1(x)}^2 
&\leq \frac{8d\ell_2^2\alpha^2(1-\delta)}{\delta} \leq \sigma^2 \\
\E_z \nrm*{g^U_2(x;z) - \nabla F^U_2(x)}^2 
&\leq \frac{8d\ell_2^2\alpha^2(1-\delta)}{\delta} \leq \sigma^2
\end{align}
Here, we used $\alpha^2 \leq \frac{\sigma^2\delta}{8d\ell_2^2(1-\delta)}$.

\textbf{7)}
First, we compute
\begin{align}
&\nrm*{\nabla F^U_2(x^*)}^2 \nonumber\\
&= \nrm*{\nabla \tilde{F}_2(U^\top x^*) + \lambda x^*}^2 \\
&= \frac{\zeta^2\ell_2^2}{(1-q)^2}\nrm*{(q-1)\sum_{i=1}^{d/2}q^{2i-1}(e_{2i} - e_{2i-1}) + \frac{\lambda}{\ell_2}\sum_{i=1}^d q^ie_i}^2 \\
&= \frac{\zeta^2\ell_2^2}{(1-q)^2}\nrm*{\sum_{i=1}^{d/2}\prn*{\frac{\lambda}{\ell_2}q + q - 1}q^{2i-1}e_{2i} + \sum_{i=1}^{d/2}\prn*{\frac{\lambda}{\ell_2} + 1 - q}q^{2i-1}e_{2i-1}}^2 \\
&= \frac{\zeta^2\ell_2^2}{(1-q)^2}\prn*{\prn*{\frac{\lambda}{\ell_2}q + q - 1}^2 + \prn*{\frac{\lambda}{\ell_2} + 1 - q}^2}\sum_{i=1}^{d/2}q^{4i-2} \\
&= \frac{\zeta^2\ell_2^2}{(1-q)^2}(1 + q^2)\prn*{\frac{\lambda}{\ell_2} + 1 - q}^2\frac{q^2(1 - q^{2d})}{1 - q^4} \\
&= \frac{\zeta^2\ell_2^2}{(1-q)^2}\prn*{\frac{(1-q)^2}{2q} + 1 - q}^2\frac{q^2(1 - q^{2d})}{1 - q^2} \\
&= \frac{\zeta^2\ell_2^2}{4}\prn*{1 + q}^2\frac{(1 - q^{2d})}{1 - q^2} \\
&\leq \frac{\zeta^2\ell_2^2(1+q)}{4(1-q)}
\end{align}
We used that \eqref{eq:heterogeneous-quadratic-equation} implies $\frac{\lambda}{\ell_2}q + q - 1 = -q\prn*{\frac{\lambda}{\ell_2} + 1-q}$ and $\frac{\lambda}{\ell_2} = \frac{(1-q)^2}{2q}$.
Also,
\begin{equation}
\frac{1}{2}\nabla F^U_1(x^*) + \frac{1}{2}\nabla F^U_2(x^*) = \nabla F^U(x^*) = 0 \implies \nrm*{\nabla F^U_1(x^*)}^2 = \nrm*{\nabla F^U_2(x^*)}^2
\end{equation}
This completes the proof.
\end{proof}
We are now prepared to prove the theorem:
\boundedheterogeneousconvexlowerbound*
\begin{proof}
Since the algorithm's queries have norm bounded by $\gamma$, we can take the dimension $D \geq R + \frac{2\gamma^2}{\alpha^2}\log(d/\delta)$ and conclude by \pref{lem:heterogeneous-progress} that with probability at least $1 - 4MKR\delta$, 
\begin{equation}
\prog{\alpha}(U^\top \hat{x}) \leq R
\end{equation}
We set some of the parameters of our construction as
\begin{align}
\delta &= \frac{1}{8MKR} \\
\beta &= \sqrt{1 + \frac{2\ell_2}{\lambda}} \\
q &= \frac{\beta - 1}{\beta + 1} \\
d &= R + \beta > R + \frac{\log 2}{2\log\prn*{\frac{\beta + 1}{\beta - 1}}} = R + \frac{\log 2}{2\log\prn*{1/q}} \\
\alpha^2 &= \min\crl*{\frac{\zeta^2q^{2d}}{4(1-q)^2},\ \frac{\sigma^2\delta}{8d\ell_2^2(1-\delta)}} \\
\ell_2 &\leq \frac{H-\lambda}{4} 
\end{align}
Therefore, by \pref{lem:heterogeneous-function-properties}, $F$, $F_1$, and $F_2$ are $\lambda$-strongly convex, $H$-smooth, and their associated stochastic gradient oracles are unbiased estimates of their gradients with variance at most $\sigma^2$. We now consider the convex and strongly convex cases separately, and establish lower bounds in each case by selecting the remaining parameters $\zeta$, $\ell_2$, and in the convex case $\lambda$.

\paragraph{The convex case:}
By \pref{lem:heterogeneous-function-properties}, in the event that $\prog{\alpha}(U^\top \hat{x}) \leq R$,
\begin{align}
F(\hat{x}) - F^* &\geq \frac{\lambda \zeta^2 q^2}{16(1-q)^2(1-q^2)}q^{2R} \\
\nrm{x^*}^2 &\leq \frac{\zeta^2q^2}{(1-q)^2(1-q^2)} \\
\frac{1}{2}\prn*{\nrm*{\nabla F_1(x^*)}^2 + \nrm*{\nabla F_2(x^*)}^2} &\leq \frac{\zeta^2\ell_2^2}{2(1-q)}
\end{align}
To establish the lower bound, we first note that when 
\begin{equation}
\zeta^2 = \min\crl*{\frac{(1-q)^2(1-q^2)B^2}{q^2},\ \frac{2(1-q)\sdiff^2}{\ell_2^2}}
\end{equation}
then $\nrm{x^*} \leq B$ and $\frac{1}{2}\prn*{\nrm*{\nabla F_1(x^*)}^2 + \nrm*{\nabla F_2(x^*)}^2} \leq \sdiff^2$ so that the objective satisfies all of the necessary regularity conditions. Furthermore with this choice, we have
\begin{align}
F(\hat{x}) - F^* 
&\geq \frac{\lambda \zeta^2 q^2}{16(1-q)^2(1-q^2)}q^{2R} \\
&= \min\crl*{\frac{\lambda B^2}{16},\ \frac{\lambda\sdiff^2q^2}{8\ell_2^2(1-q)(1-q^2)}}q^{2R} \\
&= \min\crl*{\frac{\lambda B^2}{16},\ \frac{\lambda\sdiff^2(\beta + 1)(\beta - 1)^2}{16\ell_2^2((\beta + 1)^2-(\beta - 1)^2)}}\exp\prn*{-2R\log\prn*{\frac{\beta + 1}{\beta - 1}}} \\
&\geq \min\crl*{\frac{\lambda B^2}{16},\ \frac{\lambda\sdiff^2(\beta - 1)^2}{64\ell_2^2}}\exp\prn*{-\frac{4R}{\beta - 1}}
\end{align}
From here, we will choose $\lambda = \frac{\ell_2}{64R^2}$ which ensures $\beta \geq 2$ and allows us to lower bound
\begin{align}
F(\hat{x}) - F^* 
&\geq \min\crl*{\frac{\ell_2 B^2}{1024R^2},\ \frac{\lambda\sdiff^2\beta^2}{256\ell_2^2}}\exp\prn*{-\frac{8R}{\beta}} \\
&\geq \min\crl*{\frac{\ell_2 B^2}{1024R^2},\ \frac{\sdiff^2}{128\ell_2}}\exp\prn*{-\frac{8R\sqrt{\lambda}}{\sqrt{\ell_2}}} \\
&\geq \min\crl*{\frac{\ell_2 B^2}{1024eR^2},\ \frac{\sdiff^2}{128e\ell_2}}
\end{align}
Finally, we set
\begin{equation}
\ell_2 = \min\crl*{\frac{H}{5}\ \frac{\sdiff R}{B}}
\end{equation}
which ensures that $\ell_2 \leq \frac{H - \lambda}{4}$ and that
\begin{equation}
F(\hat{x}) - F^* \geq \min\crl*{\frac{HB^2}{5120eR^2},\ \frac{\sdiff B}{1024e R}}
\end{equation}
We also note that by \pref{lem:statistical-term-lower-bound}, even when the objective is homogeneous, with probability at least $\frac{1}{4}$ any algorithm's output will have
\begin{equation}
F(\hat{x}) - F^* \geq c\cdot\min\crl*{\frac{\sigma B}{\sqrt{MKR}},\ HB^2}
\end{equation}
Finally, after plugging in all of the parameters, we have a lower bound on the necessary dimension of
\begin{equation}
D \geq R + c\cdot\gamma^2\max\crl*{\frac{R^3}{B^2},\, \frac{H^2R}{\sdiff},\, \frac{H^2MKR^2}{\sigma^2}}\log(MKR)
\end{equation}
Combining these lower bounds completes our proof in the convex case.

\paragraph{The strongly convex case:}
By \pref{lem:heterogeneous-function-properties}, in the event that $\prog{\alpha}(\hat{x}) \leq R$,
\begin{align}
F(\hat{x}) - F^* &\geq \frac{\lambda \zeta^2 q^2}{16(1-q)^2(1-q^2)}q^{2R} \\
F(0) - F^* &\leq \frac{\zeta^2 q \ell_2}{4(1-q)} \\
\frac{1}{2}\prn*{\nrm*{\nabla F_1(x^*)}^2 + \nrm*{\nabla F_2(x^*)}^2} &\leq \frac{\zeta^2\ell_2^2}{2(1-q)}
\end{align}
Therefore, we set
\begin{equation}
\zeta^2 = \min\crl*{\frac{4(1-q)\Delta}{q\ell_2},\ \frac{2(1-q)\sdiff^2}{\ell_2^2}}
\end{equation}
which ensures $F(0) - F^* \leq \Delta$, $\frac{1}{2}\prn*{\nrm*{\nabla F_1(x^*)}^2 + \nrm*{\nabla F_2(x^*)}^2} \leq \sdiff^2$, and
\begin{align}
F(\hat{x}) - F^* 
&\geq \frac{\lambda \zeta^2 q^2}{16(1-q)^2(1-q^2)}q^{2R} \\
&= \min\crl*{\frac{\lambda \Delta q}{4\ell_2(1-q)(1-q^2)},\ \frac{\lambda q^2 \sdiff^2}{8\ell_2^2(1-q)(1-q^2)}}  \exp\prn*{-2R\log\prn*{\frac{\beta + 1}{\beta - 1}}} \\
&= \min\crl*{\frac{\Delta}{2(1+q)},\ \frac{q \ell_2\sdiff^2}{4(1+q)}}\exp\prn*{-2R\log\prn*{\frac{\beta + 1}{\beta - 1}}} \\
&\geq \min\crl*{\frac{\Delta}{4},\ \frac{(\beta - 1)\ell_2\sdiff^2}{8\beta}}\exp\prn*{-\frac{4R}{\beta - 1}}
\end{align}
Using the fact that $9\lambda \leq H$, it is possible to choose $\ell_2$ such that $2\lambda \leq \ell_2 \leq \frac{H-\lambda}{4}$, which ensures that $\beta \geq 2$ and 
\begin{align}
F(\hat{x}) - F^* 
&\geq \min\crl*{\frac{\Delta}{4},\ \frac{\ell_2\sdiff^2}{16}}\exp\prn*{-\frac{8R\sqrt{\lambda}}{\sqrt{\ell_2}}}
\end{align}
Finally, we take $\ell_2 = \frac{H-\lambda}{4} \geq \frac{H}{5}$ from which we conclude that
\begin{equation}
F(\hat{x}) - F^* 
\geq \min\crl*{\frac{\Delta}{4},\ \frac{H\sdiff^2}{80}}\exp\prn*{-\frac{18R\sqrt{\lambda}}{\sqrt{H}}}
\end{equation}
Finally, by \pref{lem:statistical-term-lower-bound}, even when the objectives are homogeneous, with probability at least $1/4$, any algorithm's output will have
\begin{equation}
F(\hat{x}) - F^* \geq c\cdot\min\crl*{\frac{\sigma^2}{\lambda MKR},\ \Delta}
\end{equation}
Plugging in the various parameters gives a lower bound on the necessary dimension of
\begin{equation}
D \geq R + c\cdot\gamma^2\max\crl*{\frac{\sqrt{H\lambda}}{\Delta\prn*{1-\sqrt{\frac{\lambda}{H}}}^{2R}},\,\frac{H\sqrt{H\lambda}}{\sdiff^2\prn*{1-\sqrt{\frac{\lambda}{H}}}^{2R}},\,\frac{H^2MKR^2}{\sigma^2}}\log(MKR)
\end{equation}
This completes the proof.
\end{proof}

In light of \pref{thm:bounded-heterogeneous-convex-lower-bound}, \pref{thm:heterogeneous-convex-lower-bound} and \pref{thm:heterogeneous-strongly-convex-lower-bound} follow immediately by taking $\sdiff^2$ sufficiently large that the corresponding terms in the $\min$'s in \pref{thm:bounded-heterogeneous-convex-lower-bound} are never active.

\subsection{Proof of \pref{thm:homogeneous-non-convex-MSS-lower-bound}} \label{app:intermittent-communication-homogeneous-non-convex-MSS-lower-bound}

The proof of \pref{thm:homogeneous-non-convex-MSS-lower-bound} uses the same construction as for \pref{thm:homogeneous-non-convex-lower-bound}, just with a different stochastic gradient oracle that satisfies $L$-mean square smoothness. 
We recall the construction of \citet{carmon2017lower1}, for each $T$ we define
\begin{equation}
F_T(x) = -\Psi(1)\Phi(x_1) + \sum_{i=2}^T \bigg[ \Psi(-x_{i-1})\Phi(-x_i) - \Psi(x_{i-1})\Phi(x_i) \bigg]
\end{equation}
where
\begin{align}
\Psi(x) &= \begin{cases} 0 & x \leq \frac{1}{2} \\ \exp\prn*{1 - \frac{1}{(2x-1)^2}} & x > \frac{1}{2} \end{cases} \\
\Phi(x) &= \sqrt{e}\int_{-\infty}^x e^{-\frac{1}{2}t^2}dt
\end{align}
We also recall \pref{lem:non-convex-homogeneous-construction-properties} which summarizes the relevant properties of $F_T$:
\nonconvexconstruction*

In addition to the function $F_T$, we also define a stochastic gradient oracle $g_T(x;z)$ which satisfies $L$-mean square smoothness. To do so, we first introduce the smoothed indicator function $\Theta_i(x)$:
\begin{align}
\Lambda(t) &= \begin{cases}
0 & t \not\in \prn*{\frac{1}{4},\frac{1}{2}} \\
\exp\prn*{-\frac{1}{100\prn*{t-\frac{1}{4}}\prn*{\frac{1}{2}-t}}} & t \in \prn*{\frac{1}{4},\frac{1}{2}}
\end{cases} \\
\Gamma(t) &= \frac{\int_{1/4}^t \Lambda(\tau) d\tau}{\int_{1/4}^{1/2} \Lambda(\tau) d\tau} \\
\Theta_i(x) &= \Gamma\prn*{1 - \prn*{\sum_{j=i}^T \Gamma(\abs{x_j})^2}^{\frac{1}{2}}}
\end{align}
The important properties of $\Theta_i$ are summarized as follows:
\begin{lemma}\label{lem:theta-properties}
For any $i$ and $x$, $\Theta_i(x)$ satisfies
\begin{enumerate}
\item $\Theta_i(x) \in [0,1]$
\item $\Theta_i(x)$ is $36$-Lipschitz
\item $\indicator{i > \prog{\frac{1}{4}}(x)} \leq \Theta_i(x) \leq \indicator{i > \prog{\frac{1}{2}}(x)}$
\end{enumerate}
\end{lemma}
\begin{proof}
The fact that $\Lambda(t) \geq 0$ implies that $\Gamma(t) \geq 0$ and that $\int_{1/4}^t \Lambda(\tau)d\tau \leq \int_{1/4}^{1/2} \Lambda(\tau)d\tau$ so that $\Gamma(t) \leq 1$. Finally, $\Theta_i(x)$ is $\Gamma$ applied to some function of $x$, so $\Theta_i(x) \in [0,1]$ too.

First, we note that if $t \not\in \prn*{\frac{1}{4},\frac{1}{2}}$, then 
\begin{equation}
\frac{d}{dt}\Gamma(t) = \frac{\Lambda(t)}{\int_{1/4}^{1/2} \Lambda(\tau) d\tau} = 0
\end{equation}
Evaluating the integral, we have that
\begin{equation}
\int_{1/4}^{1/2} \exp\prn*{-\frac{1}{100\prn*{\tau-\frac{1}{4}}\prn*{\frac{1}{2}-\tau}}} d\tau 
\geq 0.088615
\end{equation}
Therefore, if $t \in \prn*{\frac{1}{4},\frac{1}{2}}$, then
\begin{align}
\frac{d}{dt}\Gamma(t) 
&= \frac{\Lambda(t)}{\int_{1/4}^{1/2} \Lambda(\tau) d\tau} \\
&\leq \frac{1}{0.088615}\exp\prn*{-\frac{1}{100\prn*{\frac{3}{8}-\frac{1}{4}}\prn*{\frac{1}{2}-\frac{3}{8}}}} \\
&= \frac{1}{0.088615}\exp\prn*{-\frac{64}{100}} \leq 6
\end{align}
Therefore, $\Gamma$ is $6$-Lipschitz, so for any $x,y$
\begin{align}
\abs{\Theta_i(x) - \Theta_i(y)}
&= \abs*{\Gamma\prn*{1 - \prn*{\sum_{j=i}^T \Gamma(\abs{x_j})^2}^{\frac{1}{2}}} - \Gamma\prn*{1 - \prn*{\sum_{j=i}^T \Gamma(\abs{y_j})^2}^{\frac{1}{2}}}} \\
&\leq 6\abs*{\prn*{\sum_{j=i}^T \Gamma(\abs{x_j})^2}^{\frac{1}{2}} - \prn*{\sum_{j=i}^T \Gamma(\abs{y_j})^2}^{\frac{1}{2}}} \\
&\leq 6\prn*{\sum_{j=i}^T \prn*{\Gamma(\abs{x_j}) - \Gamma(\abs{y_j})}^2}^{\frac{1}{2}} \\
&\leq 36\prn*{\sum_{j=i}^T \prn*{\abs{x_j} - \abs{y_j}}^2}^{\frac{1}{2}} \\
&\leq 36\prn*{\sum_{j=i}^T \prn*{x_j - y_j}^2}^{\frac{1}{2}} \\
&\leq 36\nrm{x-y}
\end{align}
For the second inequality, we applied the reverse triangle inequality for the L2 norm. This shows that $\Theta_i$ is $36$-Lipschitz.

Finally, $\Theta_i(x) \geq 0$ and
\begin{equation}
\prog{\frac{1}{4}}(x) < i \implies \forall_{j \geq i}\ \Gamma(\abs{x_j}) = 0 \implies \sum_{j=i}^T \Gamma(\abs{x_j})^2 = 0 \implies \Theta_i(x) = \Gamma(1) = 1 
\end{equation}
so $\Theta_i(x) \geq \indicator{i > \prog{\frac{1}{4}}(x)}$. Similarly, $\Theta_i(x) \leq 1$ and
\begin{equation}
\prog{\frac{1}{2}}(x) \geq i \implies \sum_{j=i}^T \Gamma(\abs{x_j})^2 \geq \Gamma(\abs{x_i})^2 \geq 1 \implies \Theta_i(x) = \Gamma(0) = 0
\end{equation}
so $\Theta_i(x) \leq \indicator{i > \prog{\frac{1}{2}}(x)}$.
\end{proof}

Using the smoothed indicator functions $\Theta_i$, we are ready to define the stochastic gradient oracle:
\begin{align}
z &= \begin{cases} 0 & \textrm{with probability } 1-p \\ 1 & \textrm{with probability } p \end{cases} \\
g_T(x;0) &= \sum_{i=1}^{T}(1 - \Theta_i(x)) e_ie_i^\top \nabla F_T\prn*{\brk*{(1 - \Theta_1(x))x_1,\dots,(1 - \Theta_i(x))x_i,\dots,(1 - \Theta_d(x))x_d}} \\
g_T(x;1) &= \frac{1}{p}\nabla F_T(x) - \frac{1-p}{p}g_T(x;0)
\end{align}
The following lemma confirms that $g_T$ has the desired properties:
\begin{lemma}\label{lem:homogeneous-non-convex-oracle-properties-MSS}
For any $p$, the stochastic gradient oracle $g_T$ satisfies
\begin{enumerate}
\item $\E_z g_T(x;z) = \nabla F_T(x)$
\item $\E_z \nrm*{g_T(x;z) - \nabla F_T(x)}^2 \leq \frac{4232(1-p)}{p}$
\item $g_T$ is a $(\frac{1}{4},p,0)$-robust-zero-chain oracle
\item $g_T$ is $\prn*{152^2 + \frac{1457187328(1-p)}{p}}$-mean square smooth
\end{enumerate}
\end{lemma}
\begin{proof}
A simple calculation shows that
\begin{equation}
\E_z g_T(x;z) = (1-p)g_T(x;0) + p\prn*{\frac{1}{p}\nabla F_T(x) - \frac{1-p}{p}g_T(x;0)} = \nabla F_T(x)
\end{equation}
Furthermore, 
\begin{align}
\E_z &\nrm*{g_T(x;z) - \nabla F_T(x)}^2\nonumber\\
&= (1-p)\nrm*{g_T(x;0) - \nabla F_T(x)}^2 + p\nrm*{\frac{1}{p}\nabla F_T(x) - \frac{1-p}{p}g_T(x;0) - \nabla F_T(x)}^2 \\
&= \frac{1-p}{p}\nrm*{g_T(x;0) - \nabla F_T(x)}^2
\end{align}
Each coordinate of the gradient is given by
\begin{equation}\label{eq:grad-coord}
\brk*{\nabla F_T(x)}_i = - \Psi(-x_{i-1})\Phi'(-x_i) - \Psi(x_{i-1})\Phi'(x_i) - \Psi'(-x_i)\Phi(-x_{i+1}) - \Psi'(x_i)\Phi(x_{i+1})
\end{equation}
so the $i\mathth$ coordinate of $\nabla F_T$ only depends on the $i-1$, $i$ and $i+1$ coordinates of $x$. Furthermore, most of the coordinates of $g_T(x;0)$ are equal to the corresponding coordinates of $\nabla F_T(x)$. Specifically, let 
\begin{equation}
\tilde{x} = \brk*{(1 - \Theta_1(x))x_1,\dots,(1 - \Theta_i(x))x_i,\dots,(1 - \Theta_d(x))x_d}
\end{equation}
By \pref{lem:theta-properties}, for $i \leq \prog{\frac{1}{2}}(x)$, $1-\Theta_i(x) \geq 1 - \indicator{i > \prog{\frac{1}{2}}(x)} = 1$, so $\tilde{x}_i = x_i$. Therefore, for $i < \prog{\frac{1}{2}}(x)$,
\begin{align}
g_T(x;0)_i &= (1-\Theta_i(x))\brk*{\nabla F_T(\tilde{x})}_i \\
&= - \Psi(-\tilde{x}_{i-1})\Phi'(-\tilde{x}_i) - \Psi(\tilde{x}_{i-1})\Phi'(\tilde{x}_i) - \Psi'(-\tilde{x}_i)\Phi(-\tilde{x}_{i+1}) - \Psi'(\tilde{x}_i)\Phi(\tilde{x}_{i+1}) \\
&= - \Psi(-x_{i-1})\Phi'(-x_i) - \Psi(x_{i-1})\Phi'(x_i) - \Psi'(-x_i)\Phi(-x_{i+1}) - \Psi'(x_i)\Phi(x_{i+1}) \\
&= \brk*{\nabla F_T(x)}_i
\end{align}
Furthermore, for $i > \prog{\frac{1}{2}}(x) + 1$, $\Psi(-x_{i-1}) = \Psi(x_{i-1}) = \Psi'(-x_{i-1}) = \Psi'(x_{i-1}) = 0$, and since $\prog{\frac{1}{2}}(\tilde{x}) \leq \prog{\frac{1}{2}}(x)$, we conclude $\brk*{\nabla F_T(x)}_i = \brk*{\nabla F_T(\tilde{x})}_i = 0$.

Therefore, $[g(x;0)]_i = \brk*{\nabla F_T(\tilde{x})}_i$ for all $i$ except possibly $i \in \crl*{\prog{\frac{1}{2}}(x),\prog{\frac{1}{2}}(x)+1}$. Recalling that $\sup_x\nrm{\nabla F_T(x)}_\infty \leq 23$ by \pref{lem:non-convex-homogeneous-construction-properties}, we have for $i = \prog{\frac{1}{2}}(x)$
\begin{align}
&\nrm*{g_T(x;0) - \nabla F_T(x)}^2\nonumber\\
&= \prn*{(1-\Theta_i(x))\brk*{\nabla F_T(\tilde{x})}_i - \brk*{\nabla F_T(x)}_i}^2 + \prn*{(1-\Theta_{i+1}(x))\brk*{\nabla F_T(\tilde{x})}_{i+1} - \brk*{\nabla F_T(x)}_{i+1}}^2 
\leq 2\cdot(2\cdot23)^2
\end{align}
We conclude that 
\begin{equation}
\E_z \nrm*{g_T(x;z) - \nabla F_T(x)}^2 \leq \frac{4232(1-p)}{p}
\end{equation}

We will now show that $g_T$ is a $(\frac{1}{4},p,0)$-robust zero chain.
For $\mc{Z}_0 = \crl{0}$, $\mc{Z}_1 = \crl{1}$, it is clear that $\P(z \in \mc{Z}_0\cup\mc{Z}_1) = 1$ and $\P(z \in \mc{Z}_0\,|\,z\in\mc{Z}_0\cup\mc{Z}_1) = 1-p$. Furthermore, by \pref{lem:theta-properties}, $1-\Theta_i(x) \leq 1 - \indicator{i>\prog{\frac{1}{4}}(x)}$, so
\begin{align}
g(x;0) 
&= \sum_{i=1}^{T}(1 - \Theta_i(x)) e_ie_i^\top \nabla F_T\prn*{\brk*{(1 - \Theta_1(x))x_1,\dots,(1 - \Theta_i(x))x_i,\dots,(1 - \Theta_d(x))x_d}}  \\
&= \sum_{i=1}^{\prog{\frac{1}{4}}(x)}(1 - \Theta_i(x)) e_ie_i^\top \nabla F_T\prn*{\brk*{(1 - \Theta_1(x))x_1,\dots,(1 - \Theta_{\prog{\frac{1}{4}}(x)}(x))x_{\prog{\frac{1}{4}}(x)},0,\dots,0}}
\end{align}
Furthermore, since $\Gamma(t) = 0$ for $t \leq \frac{1}{4}$,
\begin{equation}\label{eq:homogeneous-non-convex-g-is-robust-zc-MSS}
\Theta_i(x) = \Gamma\prn*{1 - \prn*{\sum_{j=i}^T \Gamma(\abs{x_j})^2}^{1/2}} = \Gamma\prn*{1 - \prn*{\sum_{j=i}^{\prog{\frac{1}{4}}(x)} \Gamma(\abs{x_j})^2}^{1/2}} = \Theta_i\prn*{\brk*{x_1,\dots,x_{\prog{\frac{1}{4}}(x)},0,\dots,0}}
\end{equation}
So, when $z = 0 \in \mc{Z}_0$, $\prog{0}(g(x;z)) \leq \prog{\frac{1}{4}}(x)$ and $g(x;0)$ depends only on $x_1,\dots,x_{\prog{\frac{1}{4}}(x)}$. Alternatively, when $z = 1 \in \mc{Z}_1$, $g(x;1) = \frac{1}{p}\nabla F_T(x) - \frac{1-p}{p}g_T(x;0)$. The second term depends only on $x_1,\dots,x_{\prog{\frac{1}{4}}(x)}$ by \eqref{eq:homogeneous-non-convex-g-is-robust-zc-MSS}. Furthermore, it was shown in the proof of \pref{lem:homogeneous-non-convex-oracle-properties} that $\prog{0}(\nabla F_T(x)) \leq \prog{\frac{1}{2}}(x) + 1$ and
\begin{equation}
\nabla F_T(x) = \nabla F_T([x_1,\dots,x_{\prog{\frac{1}{2}}(x) + 1},0,\dots,0])
\end{equation}
Since $\prog{\frac{1}{2}}(x) \leq \prog{\frac{1}{4}}(x) \leq \prog{0}(x)$, this implies that $\prog{\frac{1}{4}}(g(x;1)) \leq \prog{\frac{1}{4}}(x) + 1$ and $g(x;1)$ depends only on $x_1,\dots,x_{\prog{\frac{1}{4}}+1}$. We conclude that $g_T$ is a $\prn*{\frac{1}{4},p,0}$-robust zero chain.

Finally, we show that $g_T$ is mean square smooth. For arbitrary $x,y$, 
\begin{align}
\E_z\nrm*{g_T(x;z) - g_T(y;z)}^2
&= \E_z\nrm*{g_T(x;z) - \nabla F_T(x) - g_T(y;z) + \nabla F_T(y)}^2 + \nrm*{\nabla F_T(x) - \nabla F_T(y)}^2\\
&\leq \E_z\nrm*{g_T(x;z) - \nabla F_T(x) - g_T(y;z) + \nabla F_T(y)}^2 + 152^2\nrm{x-y}^2
\end{align}
For the final inequality, we used that $F_T$ is $152$-smooth by \pref{lem:non-convex-homogeneous-construction-properties}.
We recall from earlier in this proof that for any $i < \prog{\frac{1}{2}}(x)$ or $i > \prog{\frac{1}{2}}(x) + 1$, $[g_T(x;0)]_i = [\nabla F_T(x)]_i$, and therefore
\begin{equation}
[g_T(x;1)]_i = \frac{1}{p}[\nabla F_T(x)]_i - \frac{1-p}{p}[g_T(x;0)]_i = \frac{1}{p}[\nabla F_T(x)]_i - \frac{1-p}{p}[\nabla F_T(x)]_i = [\nabla F_T(x)]_i
\end{equation}
Let
\begin{align}
\tilde{x} &= \brk*{(1 - \Theta_1(x))x_1,\dots,(1 - \Theta_i(x))x_i,\dots,(1 - \Theta_d(x))x_d} \\
\tilde{y} &= \brk*{(1 - \Theta_1(y))y_1,\dots,(1 - \Theta_i(y))y_i,\dots,(1 - \Theta_d(y))y_d}
\end{align}
and $i_x = \prog{\frac{1}{2}}(x)$ and $i_y = \prog{\frac{1}{2}}(y)$. Then for each $i \in \crl*{i_x,i_x+1,i_y,i_y+1}$
\begin{align}
\E_z&\prn*{[g_T(x;z)]_i - [\nabla F_T(x)]_i - [g_T(y;z)]_i + [\nabla F_T(y)]_i}^2 \nonumber\\
&= (1-p)\prn*{[g(x;0)]_i - [\nabla F_T(x)]_i - [g(y;0)]_i + [\nabla F_T(y)]_i}^2 \nonumber\\
&+ \frac{(1-p)^2}{p}\prn*{[g(x;0)]_i - [\nabla F_T(x)]_i - [g(y;0)]_i + [\nabla F_T(y)]_i}^2 \\
&= \frac{1-p}{p}\prn*{[g(x;0)]_i - [\nabla F_T(x)]_i - [g(y;0)]_i + [\nabla F_T(y)]_i}^2 \\
&= \frac{1-p}{p}\prn*{(1-\Theta_i(x))[\nabla F_T(\tilde{x})]_i - [\nabla F_T(x)]_i - (1-\Theta_i(y))[\nabla F_T(\tilde{y})]_i + [\nabla F_T(y)]_i}^2 \\
&\leq \frac{2(1-p)}{p}\prn*{\prn*{(1-\Theta_i(x))[\nabla F_T(\tilde{x})]_i - (1-\Theta_i(y))[\nabla F_T(\tilde{y})]_i}^2 + \prn*{[\nabla F_T(x)]_i - [\nabla F_T(y)]_i}^2}
\end{align}
We now bound the first term
\begin{align}
&\prn*{(1-\Theta_i(x))[\nabla F_T(\tilde{x})]_i - (1-\Theta_i(y))[\nabla F_T(\tilde{y})]_i}^2\nonumber\\
&= \prn*{(\Theta_i(y)-\Theta_i(x))[\nabla F_T(\tilde{x})]_i + (1-\Theta_i(y))\prn*{[\nabla F_T(\tilde{x})]_i - [\nabla F_T(\tilde{y})]_i}}^2 \\
&\leq 2\prn*{[\nabla F_T(\tilde{x})]_i}^2(\Theta_i(y)-\Theta_i(x))^2 + 2(1-\Theta_i(y))^2\prn*{[\nabla F_T(\tilde{x})]_i - [\nabla F_T(\tilde{y})]_i}^2 \\
&\leq 2\cdot23^2\cdot36^2\cdot\nrm{x-y}^2 + 2\prn*{[\nabla F_T(\tilde{x})]_i - [\nabla F_T(\tilde{y})]_i}^2
\end{align}
Here, we used that $\nrm{\nabla F_T(x)}_\infty \leq 23$ by \pref{lem:non-convex-homogeneous-construction-properties} and that $\Theta_i(y) \in [0,1]$ and $\Theta_i$ is $36$-Lipschitz by \pref{lem:theta-properties}. To summarize so far, we have shown that
\begin{align}
&\E_z\nrm*{g_T(x;z) - g_T(y;z)}^2\nonumber\\
&\leq 152^2\nrm{x-y}^2 + \sum_{i\in\crl*{i_x,i_x+1,i_y,i_y+1}}\frac{2(1-p)}{p}\bigg(2\cdot23^2\cdot36^2\cdot\nrm{x-y}^2 \nonumber\\
&\qquad + 2\prn*{[\nabla F_T(\tilde{x})]_i - [\nabla F_T(\tilde{y})]_i}^2 + \prn*{[\nabla F_T(x)]_i - [\nabla F_T(y)]_i}^2\bigg) \\
&\leq \prn*{152^2 + \frac{16\cdot23^2\cdot36^2(1-p)}{p}}\nrm{x-y}^2 + \frac{2(1-p)}{p}\nrm*{\nabla F_T(x) - \nabla F_T(y)}^2 \nonumber\\
&\quad+\sum_{i\in\crl*{i_x,i_x+1,i_y,i_y+1}}\frac{4(1-p)}{p}\prn*{[\nabla F_T(\tilde{x})]_i - [\nabla F_T(\tilde{y})]_i}^2 \\
&\leq \prn*{152^2 + \frac{(16\cdot23^2\cdot36^2 + 4\cdot152^2)(1-p)}{p}}\nrm{x-y}^2 + \sum_{i\in\crl*{i_x,i_x+1,i_y,i_y+1}}\frac{4(1-p)}{p}\prn*{[\nabla F_T(\tilde{x})]_i - [\nabla F_T(\tilde{y})]_i}^2\label{eq:mss-bound-eq1}
\end{align}
We recall from \eqref{eq:grad-coord} that the $i\mathth$ coordinate of $\nabla F_T(\tilde{x})$ depends only on $\tilde{x}_{i-1}$, $\tilde{x}_i$, and $\tilde{x}_{i+1}$ and similarly for $\nabla F_T(\tilde{x})$. This, coupled with the fact that $\nabla F_T$ is $152$-smooth by \pref{lem:non-convex-homogeneous-construction-properties}, implies that for each $i$
\begin{align}
\prn*{[\nabla F_T(\tilde{x})]_i - [\nabla F_T(\tilde{y})]_i}^2
&\leq 152^2\sum_{j=i-1}^{i+1}\prn*{\tilde{x}_j - \tilde{y}_j}^2 \\
&= 152^2\sum_{j=i-1}^{i+1}\prn*{(1-\Theta_j(x))x_j - (1-\Theta_j(y))y_j}^2
\end{align}
We now consider the quantity $\prn*{(1-\Theta_j(x))x_j - (1-\Theta_j(y))y_j}^2$ in three cases:

Case 1: If $\Theta_j(x) = \Theta_j(y)$, then 
\begin{equation}
\prn*{(1-\Theta_j(x))x_j - (1-\Theta_j(y))y_j}^2 = (1-\Theta_j(x))^2\prn*{x_j - y_j}^2 \leq \prn*{x_j - y_j}^2
\end{equation}

Case 2: If $\Theta_j(y) = 1$ and  $\Theta_j(x) < 1$, then
\begin{align}
\prn*{(1-\Theta_j(x))x_j - (1-\Theta_j(y))y_j}^2 
&= (1-\Theta_j(x))^2\prn*{x_j}^2 \\
&\leq \frac{1}{4}(1-\Theta_j(x))^2 \\
&\leq \frac{1}{4}(0 + 36\nrm{x-y})^2 \\
&= 324\nrm{x-y}^2
\end{align}
For the first inequality, we used that $\Theta_j(x) < 1 \implies \abs{x_j} < \frac{1}{2}$. For the second inequality, we used that $\Theta_j$ is $36$-Lipschitz by \pref{lem:theta-properties}. If $\Theta_j(y) < 1$ and  $\Theta_j(x) = 1$ the same arguement with $x$ and $y$ switched shows the same upper bound.

Case 3: If $\Theta_j(y) = 0$ and $1 > \Theta_j(x) > 0$, then
\begin{align}
\prn*{(1-\Theta_j(x))x_j - (1-\Theta_j(y))y_j}^2 
&= \prn*{(1-\Theta_j(x))x_j - y_j}^2 \\
&\leq 2(x_j - y_j)^2 + \frac{1}{4} \Theta_j(x)^2 \\
&\leq 2(x_j - y_j)^2 + \frac{1}{4} \prn*{0 + 36\nrm{x-y}}^2 \\
&\leq 2(x_j - y_j)^2 + 324\nrm{x-y}^2
\end{align}
For the first inequality, we used that $\Theta_j(x) < 1 \implies \abs{x_j} < \frac{1}{2}$. For the second inequality, we used that $\Theta_j$ is $36$-Lipschitz by \pref{lem:theta-properties}. Again, the same argument applies when $x$ and $y$ are reversed.

We conclude that in any case,
\begin{align}
\prn*{[\nabla F_T(\tilde{x})]_i - [\nabla F_T(\tilde{y})]_i}^2
&\leq 152^2\sum_{j=i-1}^{i+1}326\nrm{x-y}^2 \\
&\leq 152^2\cdot326\cdot12\nrm{x-y}^2
\end{align}
Plugging this back into \eqref{eq:mss-bound-eq1}, we conclude that
\begin{align}
&\E_z\nrm*{g_T(x;z) - g_T(y;z)}^2\nonumber\\
&\leq \prn*{152^2 + \frac{(16\cdot23^2\cdot36^2 + 4\cdot152^2 + 16\cdot152^2\cdot326\cdot12)(1-p)}{p}}\nrm{x-y}^2 \\
&= \prn*{152^2 + \frac{1457187328(1-p)}{p}}\nrm{x-y}^2
\end{align}
This completes the proof.
\end{proof}

We will proceed to combine \pref{lem:non-convex-homogeneous-construction-properties} and \pref{lem:homogeneous-non-convex-oracle-properties-MSS} with \pref{lem:intermittent-communication-progress} allows us to prove the lower bound. However, one of the conditions of \pref{lem:intermittent-communication-progress} is that the norm of the oracle queries is bounded. To enforce the we introduce an additional modification to the objective along with the random rotation. Specifically, we introduce the soft projection
\begin{equation}
\rho(x) = \frac{x}{\sqrt{1 + \frac{\nrm{x}^2}{\beta^2}}}
\end{equation}
where $\beta = \frac{240\sqrt{T}}{\zeta}$ and we define
\begin{align}
\hat{F}_{T,U}(x) &= \frac{\gamma}{\zeta^2} F_T(\zeta U^\top \rho(x)) + \frac{\gamma}{10\zeta^2}\nrm{\zeta x}^2 \\
\hat{g}_{T,U}(x;z) &= \frac{\gamma}{\zeta}\nabla\rho(x)U g_T(\zeta U^\top \rho(x);z) + \frac{\gamma}{5}x
\end{align}
We now verify that $\hat{F}_{T,U}$ and $\hat{g}_{T,U}(x;z)$ satisfy essentially the same properties as $F_T$ and $g_T$:
\begin{lemma}\label{lem:soft-projection-properties-MSS}
For any $T \geq 1$, $\gamma,\zeta \geq 0$, and $U$ with $U^\top U = I_{d\times{}d}$,
\begin{enumerate}
\item $\hat{F}_{T,U}(0) - \min_x \hat{F}_{T,U}(x) \leq \frac{12\gamma T}{\zeta^2}$
\item $F_{T,U}$ is $154\gamma$-smooth
\item $\E_z \hat{g}_{T,U}(x;z) = \nabla \hat{F}_{T,U}(x)$
\item $\E_z \nrm*{\hat{g}_{T,U}(x;z) - \nabla \hat{F}_{T,U}(x)}^2 \leq \frac{4232\gamma^2(1-p)}{\zeta^2 p}$
\item $\hat{g}_{T,U}$ is $2\gamma^2\prn*{152^2 + \frac{1457187329(1-p)}{p}}$-mean square smooth
\item For any $x$, $\prog{\frac{1}{2}}(\zeta U^\top \rho(x)) < T \implies \nrm{\nabla \hat{F}_{T,U}(x)} \geq \frac{\gamma}{2\zeta}$
\end{enumerate}
\end{lemma}
\begin{proof}
Points 1, 2, and 6 follow immediately from \pref{lem:soft-projection-properties}. Point 3 follows from \pref{lem:homogeneous-non-convex-oracle-properties-MSS} and the chain rule.
For property 4, we bound
\begin{align}
\sup_x&\E_z \nrm*{\hat{g}_{T,U}(x;z) - \nabla \hat{F}_{T,U}(x)}^2\nonumber\\
&= \sup_x \E_z \nrm*{\frac{\gamma}{\zeta}\nabla \rho(x)U g_T(\zeta U^\top\rho(x);z) + \frac{\gamma}{5}x - \frac{\gamma}{\zeta}\nabla \rho(x)U\nabla F_T(\zeta U^\top\rho(x)) - \frac{\gamma}{5}x}^2 \\
&\leq \frac{\gamma^2}{\zeta^2}\sup_x \E_z \nrm*{g_T(\zeta U^\top\rho(x);z) - \nabla F_T(\zeta U^\top\rho(x))}^2 \\
&\leq \frac{\gamma^2}{\zeta^2}\sup_y \E_z \nrm*{g_T(y;z) - \nabla F_T(y)}^2 \\
&\leq \frac{4232\gamma^2(1-p)}{\zeta^2 p}
\end{align}
where the last line follows from \pref{lem:homogeneous-non-convex-oracle-properties-MSS}.

For property 5, we note that for any $x$
\begin{equation}
\nrm*{\nabla \rho(x)}_{\textrm{op}} = \nrm*{\frac{1}{\sqrt{1 + \frac{\nrm{x}^2}{\beta^2}}}I - \frac{xx^\top}{\beta^2\prn*{1 + \frac{\nrm{x}^2}{\beta^2}}^{3/2}}}_{\textrm{op}} \leq 1
\end{equation}
Likewise, we define $h(t) = \frac{1}{\sqrt{1+t^2}}$, which is $1$-Lipschitz, and bound:
\begin{align}
&\nrm*{\nabla \rho(x) - \nabla \rho(y)}_{\textrm{op}}\nonumber\\
&= \nrm*{h\prn*{\frac{\nrm{x}}{\beta}}I - h\prn*{\frac{\nrm{x}}{\beta}}\frac{\rho(x)\rho(x)^\top}{\beta^2} - h\prn*{\frac{\nrm{y}}{\beta}}I + h\prn*{\frac{\nrm{y}}{\beta}}\frac{\rho(y)\rho(y)^\top}{\beta^2}}_{\textrm{op}} \\
&\leq h\prn*{\frac{\nrm{y}}{\beta}}\nrm*{\frac{\rho(x)\rho(x)^\top}{\beta^2} - \frac{\rho(y)\rho(y)^\top}{\beta^2}}_{\textrm{op}} + \abs*{h\prn*{\frac{\nrm{x}}{\beta}} - h\prn*{\frac{\nrm{y}}{\beta}}}\nrm*{I - \frac{\rho(x)\rho(x)^\top}{\beta^2}}_{\textrm{op}} \\
&\leq \nrm*{\frac{\rho(x)\rho(x)^\top}{\beta^2} - \frac{\rho(y)\rho(y)^\top}{\beta^2}}_{\textrm{op}} + \abs*{\frac{\nrm{x}}{\beta} - \frac{\nrm{y}}{\beta}} \\
&\leq \frac{1}{\beta}\nrm{x-y} + \sup_{v:\nrm{v}\leq 1}\nrm*{\frac{\rho(x)\rho(x)^\top}{\beta^2}v - \frac{\rho(y)\rho(y)^\top}{\beta^2}v} \\
&= \frac{1}{\beta}\nrm{x-y} + \sup_{v:\nrm{v}\leq 1}\nrm*{\prn*{\frac{\rho(x)}{\beta} - \frac{\rho(y)}{\beta}}\inner{\frac{\rho(x)}{\beta}}{v} + \frac{\rho(y)}{\beta}\inner{\frac{\rho(x)}{\beta} - \frac{\rho(y)}{\beta}}{v}} \\
&\leq \frac{1}{\beta}\nrm{x-y} + \sup_{v:\nrm{v}\leq 1}\nrm*{\frac{\rho(x)}{\beta} - \frac{\rho(y)}{\beta}}\nrm*{\frac{\rho(x)}{\beta}}\nrm{v} + \nrm*{\frac{\rho(x)}{\beta} - \frac{\rho(y)}{\beta}}\nrm*{\frac{\rho(y)}{\beta}}\nrm*{v} \\
&\leq \frac{3}{\beta}\nrm{x-y}
\end{align}
Therefore,
\begin{align}
&\E_z\nrm*{\hat{g}_{T,U}(x;z) - \hat{g}_{T,U}(y;z)}^2\nonumber\\
&= \E_z\nrm*{\frac{\gamma}{\zeta}\nabla \rho(x)g_T(\zeta U^\top \rho(x);z) - \frac{\gamma}{\zeta}\nabla \rho(y)g_T(\zeta U^\top \rho(y);z)}^2 \\
&\leq \frac{2\gamma^2}{\zeta^2}\prn*{\E_z\nrm*{\prn*{\nabla \rho(x) - \nabla \rho(y)}g_T(\zeta U^\top \rho(x);z)}^2 + \E_z\nrm*{\nabla \rho(y)\prn*{g_T(\zeta U^\top \rho(x);z) - g_T(\zeta U^\top \rho(y);z)}}^2} \\
&\leq  \frac{2\gamma^2}{\zeta^2}\prn*{\frac{9\nrm{x-y}^2}{\beta^2}\E_z\nrm*{g_T(\zeta U^\top \rho(x);z)}^2 + \E_z\nrm*{g_T(\zeta U^\top \rho(x);z) - g_T(\zeta U^\top \rho(y);z)}^2} \\
&\leq \frac{2\gamma^2}{\zeta^2}\prn*{\frac{9\nrm{x-y}^2}{\beta^2}\E_z\nrm*{g_T(\zeta U^\top \rho(x);z)}^2 + \prn*{152^2 + \frac{1457187328(1-p)}{p}}\nrm*{\zeta U^\top \rho(x) - \zeta U^\top \rho(y)}^2} \\
&\leq \frac{2\gamma^2}{\zeta^2}\prn*{\frac{9\nrm{x-y}^2}{\beta^2}\E_z\nrm*{g_T(\zeta U^\top \rho(x);z)}^2 + \prn*{152^2\zeta^2 + \frac{1457187328(1-p)\zeta^2}{p}}\nrm*{x - y}^2}
\end{align}
For the final inequality, we used \pref{lem:homogeneous-non-convex-oracle-properties-MSS}. Finally, we have
\begin{align}
\E_z\nrm*{g_T(\zeta U^\top \rho(x);z)}^2
&\leq \sup_x \E_z\nrm*{g_T(x;z)}^2 \\
&= \sup_x \brk*{(1-p)\nrm*{g(x;0)}^2 + p\nrm*{\frac{1}{p}\nabla F_T(x) - \frac{1-p}{p}g(x;0)}^2} \\
&\leq \sup_x \brk*{\frac{3(1-p)}{p}\nrm*{g(x;0)}^2 + \frac{1}{p}\nrm*{\nabla F_T(x)}^2} \\
&\leq \sup_x \frac{3(1-p)}{p}\nrm*{\nabla F_T(x)}^2 + \frac{1}{p}\sup_x \nrm*{\nabla F_T(x)}^2 \\
&\leq \frac{4(1-p)}{p}\cdot23^2 T
\end{align}
Therefore, using that $\beta = \frac{240\sqrt{T}}{\zeta}$, we have
\begin{align}
&\E_z\nrm*{\hat{g}_{T,U}(x;z) - \hat{g}_{T,U}(y;z)}^2\nonumber\\
&\leq \frac{2\gamma^2}{\zeta^2}\prn*{\frac{19044\gamma^2T(1-p)}{\zeta^2\beta^2p} + 152^2\zeta^2 + \frac{1457187328(1-p)\zeta^2}{p}}\nrm{x-y}^2 \\
&\leq 2\gamma^2\prn*{152^2 + \frac{1457187329(1-p)}{p}}\nrm{x-y}^2
\end{align}
This completes the proof.
\end{proof}

\homogeneousnonconvexMSSlowerbound*
\begin{proof}
We prove the lower bound using $\hat{F}_{T,U}$ and $\hat{g}_{T,U}$ for a uniformly random orthogonal $U \in \R^{D\times T}$ for 
\begin{equation}
D = T + 32\zeta^2\beta^2\log(32MKRT)
\end{equation}
For the oracle
\begin{equation}
\hat{g}_{T,U}(x;z) = \frac{\gamma}{\zeta}\nabla\rho(x)U g_T(\zeta U^\top \rho(x);z) + \frac{\gamma}{5}x
\end{equation}
an intermittent communication algorithm that interacts with $\hat{g}_{T,U}(x;z)$ is precisely equivalent to one that interacts with $U g_T(U^\top x;z)$ using queries of norm less than $\zeta\beta$ since $\rho(x)$, $\nabla \rho(x)$, and $x$ are invertible and ``known'' to the algorithm. Therefore, since $g_T$ is a $(\frac{1}{4},p,0)$-robust-zero-chain by \pref{lem:homogeneous-non-convex-oracle-properties-MSS}, \pref{lem:intermittent-communication-progress} ensures that with probability at least $\frac{5}{8}$, the output of the intermittent communication algorithm, $\hat{x}$, will have progress at most
\begin{equation}
\prog{\frac{1}{2}}(\zeta U^\top \rho(\hat{x})) \leq \prog{\frac{1}{4}}(\zeta U^\top \rho(\hat{x})) \leq \min\crl*{KR,\ 8KRp + 12R\log M + 12R}
\end{equation}
We therefore take
\begin{equation}
T = \min\crl*{KR,\ 8KRp + 12R\log M + 12R} + 1
\end{equation} 
Therefore, by \pref{lem:soft-projection-properties-MSS}, with probability at least $\frac{5}{8}$,
\begin{equation}
\nrm*{\nabla \hat{F}_{T,U}(\hat{x})} \geq \frac{\gamma}{2\zeta}
\end{equation}
In light of \pref{lem:soft-projection-properties-MSS}, it is easy to confirm that if we take 
\begin{align}
\gamma &= \frac{L\sqrt{p}}{53986} \\
\zeta^2 &= \frac{12\gamma T}{\Delta} \\
p &\geq \frac{L\Delta\sqrt{p}}{149\sigma^2 T + L\Delta\sqrt{p}} \label{eq:mss-p-constraint}
\end{align}
then $\hat{F}_{T,U}(0) - \min_x \hat{F}_{T,U}(x) \leq \Delta$, $\hat{F}_{T,U}$ is $L$-smooth, $\hat{g}_{T,U}$ is an unbiased estimate of $\nabla \hat{F}_{T,U}$, the variance of $\hat{g}_{T,U}$ is bounded by $\sigma^2$, and $\hat{g}_{T,U}$ is $L^2$-mean square smooth. With these parameters, the lower bound is
\begin{equation}
\E\nrm*{\nabla \hat{F}_{T,U}(\hat{x})} 
\geq \frac{5\gamma}{16\zeta} 
\geq \frac{\sqrt{\gamma\Delta}}{12\sqrt{T}} 
\geq \frac{p^{1/4}\sqrt{L\Delta}}{2789\sqrt{T}}
\end{equation}

We now consider several cases:

Case 1: If $K \leq 24(1 + \log M)$, then we take $p=1$ which satisfies the constraint \eqref{eq:mss-p-constraint}. In this case, we upper bound $T \leq KR+1$ and conclude that for some constant $c$ (which may change from line to line)
\begin{equation}
\E\nrm*{\nabla \hat{F}_{T,U}(\hat{x})} 
\geq c\frac{\sqrt{L\Delta}}{\sqrt{KR}} 
\geq c\frac{\sqrt{L\Delta}}{\sqrt{24R(1+\log M)+1}} 
\geq c\frac{\sqrt{L\Delta}}{\sqrt{R(1+\log M)}}
\end{equation}

Case 2: If $K > 24(1 + \log M)$ and $\sigma^2 \leq \frac{L\Delta}{74(KR+1)}$, then we again take $p=1$ and lower bound for some constant $c$ (which may change from line to line)
\begin{align}
\E\nrm*{\nabla \hat{F}_{T,U}(\hat{x})} 
&\geq c\frac{\sqrt{L\Delta}}{\sqrt{KR+1}} \\
&\geq c\frac{\sqrt{L\Delta}}{\sqrt{KR}} + c\prn*{\frac{L\Delta}{74(KR+1)}}^{1/6}\prn*{\frac{L\Delta}{KR+1}}^{1/3} \\
&\geq c\frac{\sqrt{L\Delta}}{\sqrt{KR}} + c\prn*{\frac{L\sigma\Delta}{KR}}^{1/3}
\end{align}

Case 3: If $K > 24(1 + \log M)$ and $\sigma^2 > \frac{L\Delta}{74(KR+1)}$, then we note that
\begin{equation}
K > 24(1 + \log M) \implies \frac{1}{2}KRp + 12R(1+\log M) + 1 \leq T
\end{equation}
Therefore, $p$ satisfies \eqref{eq:mss-p-constraint} if
\begin{align}
\frac{L\Delta\sqrt{p}}{149\sigma^2 \prn*{\frac{1}{2}KRp + 12R(1+\log M) + 1} + L\Delta\sqrt{p}}  & \leq p \\
\iff \sqrt{p}\prn*{\frac{1}{2}KRp + 12R(1+\log M) + 1} &\geq \frac{L\Delta}{149\sigma^2}(1-p)
\end{align}
We take $p = p(\mu) = \mu^4\prn*{\frac{L\Delta}{74\sigma^2KR}}^{2/3}$ for a parameter $\mu \geq 1$. We note that $p(1) \leq 1$ because $\sigma^2 > \frac{L\Delta}{74(KR+1)}$. This satisfies this inequality for any $\mu \geq 1$ and gives the lower bound for constant $c$ (which may change from line to line)
\begin{align}
\E\nrm*{\nabla \hat{F}_{T,U}(\hat{x})} 
&\geq c\frac{p^{1/4}\sqrt{L\Delta}}{\sqrt{8KRp + 12R(1+\log M) + 1}} \\
&= c\cdot \mu\prn*{\frac{L\Delta}{\sigma^2KR}}^{1/6} \frac{\sqrt{L\Delta}}{\sqrt{KR\mu^4\prn*{\frac{L\Delta}{\sigma^2KR}}^{2/3} + R(1+\log M)}} \\
&\geq c\cdot \mu\prn*{\frac{L\Delta}{\sigma^2KR}}^{1/6} \min\crl*{\frac{\sqrt{L\Delta}}{\mu^2\sqrt{KR}\prn*{\frac{L\Delta}{\sigma^2KR}}^{1/3}},\, \frac{\sqrt{L\Delta}}{\sqrt{R(1+\log M)}}} \\
&= c\cdot\min\crl*{\frac{\prn*{L\sigma\Delta}^{1/3}}{\mu\prn*{KR}^{1/3}},\, \frac{\mu\prn*{L\Delta}^{2/3}}{\sigma^{1/3}K^{1/6}R^{2/3}\sqrt{1+\log M}}}
\end{align}
Taking 
\begin{equation}
\mu = \min\crl*{1,\, \prn*{\frac{L\sigma \Delta}{KR}}^{1/6}\cdot\frac{\sigma^{1/6}K^{1/12}R^{1/3}(1+\log M)^{1/4}}{(L\Delta)^{1/3}}}
\end{equation}
allows us to further lower bound this as
\begin{align}
\E\nrm*{\nabla \hat{F}_{T,U}(\hat{x})} 
&\geq c\cdot\min\crl*{\frac{\prn*{L\sigma\Delta}^{1/3}}{\prn*{KR}^{1/3}},\, \frac{\sqrt{L\Delta}}{K^{1/4}\sqrt{R}(1+\log M)^{1/4}}} \\
&\geq c\cdot\min\crl*{\frac{\sqrt{L\Delta}}{\sqrt{KR}} + \frac{\prn*{L\sigma\Delta}^{1/3}}{\prn*{KR}^{1/3}},\, \frac{\sqrt{L\Delta}}{K^{1/4}\sqrt{R}(1+\log M)^{1/4}}}
\end{align}
For the second inequality, we used that $\sigma^2 > \frac{L\Delta}{74(KR+1)}$.

Together, cases 1-3 imply a lower bound of
\begin{equation}\label{eq:homogeneous-non-convex-MSS-lb-1}
\E\nrm*{\nabla \hat{F}_{T,U}(\hat{x})} \geq c\cdot \min\crl*{\frac{\sqrt{L\Delta}}{\sqrt{R(1+\log M)}},\,\frac{\sqrt{L\Delta}}{\sqrt{KR}} + \frac{\prn*{L\sigma\Delta}^{1/3}}{\prn*{KR}^{1/3}},\, \frac{\sqrt{L\Delta}}{K^{1/4}\sqrt{R}(1+\log M)^{1/4}}}
\end{equation}
This holds for any intermittent communication algorithm. We will now argue for one additional term in the lower bound.
The argument is simple: any intermittent communication algorithm with a given $M$, $K$, and $R$ can be implemented using $MKR$ sequential calls to the oracle, which is equivalent to a different intermittent communication setting with $M' = 1$, $K' = MKR$, and $R'= 1$. Querying the oracle $MKR$ times sequentially is clearly more powerful, so lower bounds in this setting apply also in ours.  Our lower bound for $M' = 1$, $K' = MKR$, and $R'= 1$ gives
\begin{align}
\nrm*{\nabla \hat{F}_{T,U}(\hat{x})} 
&\geq c\cdot \min\crl*{\sqrt{L\Delta},\,\frac{\sqrt{L\Delta}}{\sqrt{MKR}} + \frac{\prn*{L\sigma\Delta}^{1/3}}{\prn*{MKR}^{1/3}},\, \frac{\sqrt{L\Delta}}{(MKR)^{1/4}}} \\
&\geq c\cdot \min\crl*{\frac{\prn*{L\sigma\Delta}^{1/3}}{\prn*{MKR}^{1/3}},\, \frac{\sqrt{L\Delta}}{(MKR)^{1/4}}}
\end{align}
Combining this with \eqref{eq:homogeneous-non-convex-MSS-lb-1} completes the proof.
\end{proof}

\end{document}